 \documentclass[a4paper,11pt]{article}
\usepackage{pdfsync}

 \usepackage[plainpages=false]{hyperref}
 \usepackage{amsfonts,amsmath,amssymb,amsthm}
 \usepackage{latexsym,lscape,rawfonts,mathrsfs}

 \usepackage[dvips]{color}
 \usepackage{multicol}

 \usepackage[all]{xy}
 \usepackage{eufrak}
 \usepackage{makeidx}
 \usepackage{graphicx,psfrag}
 \usepackage{pstool}

 \usepackage{array,tabularx}

 \usepackage{setspace}

 \usepackage{appendix}

\usepackage{txfonts}

 \newcommand{\C}{\ensuremath{\mathbb{C}}}
 \newcommand{\D}[2]{\ensuremath{ \frac{\partial{#1}}{\partial{#2}}}}
 \newcommand{\E}{\ensuremath{\mathbb{E}}}

 \newcommand{\R}{\ensuremath{\mathbb{R}}}
 \newcommand{\Z}{\ensuremath{\mathbb{Z}}}
 \newcommand{\CP}{\ensuremath{\mathbb{CP}}}

 \newcommand{\ba}{\begin{align*}}
 \newcommand{\ea}{\end{align*}}

 \DeclareMathOperator{\supp}{supp}

 \DeclareMathOperator{\Vol}{Vol}
 \DeclareMathOperator{\diam}{diam}

 \DeclareMathOperator{\diag}{diag}

 \newcommand{\norm}[2]{{ \ensuremath{\left\|} #1 \ensuremath{\right\|}}_{#2}}
 \newcommand{\snorm}[2]{{ \ensuremath{\left |} #1 \ensuremath{\right |}}_{#2}}

 \makeatletter
 \def\ExtendSymbol#1#2#3#4#5{\ext@arrow 0099{\arrowfill@#1#2#3}{#4}{#5}}
 \newcommand\longeq[2][]{\ExtendSymbol{=}{=}{=}{#1}{#2}}
 \makeatother

 \makeatletter
 \def\ExtendSymbol#1#2#3#4#5{\ext@arrow 0099{\arrowfill@#1#2#3}{#4}{#5}}
 \newcommand\longright[2][]{\ExtendSymbol{-}{-}{\rightarrow}{#1}{#2}}
 \makeatother

 \definecolor{hao}{rgb}{1,0.5,0}
 \definecolor{miao}{cmyk}{0.5,0,0.2,0.2}
 \definecolor{qiao}{gray}{0.96}

\newtheorem{prop}{Proposition}[section]

\newtheorem{proposition}[prop]{Proposition}
\newtheorem{theoremin}[prop]{Theorem}
\newtheorem{theorem}[prop]{Theorem}

\newtheorem{lemma}[prop]{Lemma}
\newtheorem{claim}[prop]{Claim}
\newtheorem{corollaryin}[prop]{Corollary}
\newtheorem{corollary}[prop]{Corollary}

\newtheorem{remark}[prop]{Remark}

\newtheorem{definition}[prop]{Definition}

\newtheorem{conjecture}[prop]{Conjecture}

\numberwithin{equation}{section}

 \title{Space of Ricci flows (II)}
 \author{Xiuxiong Chen\footnote{Supported by NSF grant DMS-1211652.}\;,  Bing Wang\footnote{Supported by NSF grant DMS-1312836.}}
\date{}
 
\begin{document}
  \maketitle
 \begin{abstract}
  Based on the compactness of the moduli of non-collapsed Calabi-Yau spaces with mild singularities,
  we set up a structure theory for polarized K\"ahler Ricci flows with proper geometric bounds.
  Our theory is a generalization of the structure theory of non-collapsed K\"ahler Einstein manifolds.
  As applications, we prove the Hamilton-Tian conjecture and the partial-$C^0$-conjecture of Tian.
 \end{abstract}

\tableofcontents

\section{Introduction}

This paper is the continuation of the study in (\cite{CW3}) and (\cite{CW4}). In~\cite{CW3}, we developed
a weak compactness theory for non-collapsed Ricci flows with bounded scalar curvature and bounded half-dimensional curvature integral.  This weak compactness theory is applied in~\cite{CW4} to prove the Hamilton-Tian conjecture of
complex dimension 2 and its geometric consequences.
However, the assumption of half dimensional curvature integral is restrictive.
It is not available for high dimensional anti-canonical K\"ahler Ricci flow, i.e., K\"ahler Ricci flow on a Fano manifold $(M,J)$, in the class $2\pi c_1(M,J)$.
In this paper,  by taking advantage of the extra structures from K\"ahler geometry,  we drop this curvature integral condition.

The present paper is inspired by two different sources.
One source is the structure theory of K\"ahler Einstein manifolds which was developed over last 20 years by many people,
notably, Anderson, Cheeger, Colding, Tian and more recently,  Naber, Donaldson and Sun.  The recent progress of the structure theory of
K\"ahler Einstein manifolds supplies many additional tools for our approach.
The other source is the seminal work of Perelman on the Ricci flow(c.f.~\cite{Pe1},~\cite{SeT}).
Actually, it was pointed out by Perelman already that his idea in \cite{Pe1} can be applied to study K\"ahler Ricci flow.
He said that

``\textit{present work has also some applications to the
Hamilton-Tian conjecture concerning K\"ahler-Ricci flow on K\"ahler manifold with
positive first Chern class: these will be discussed in a separate paper}".

\noindent We cannot help to wonder how far he will push the subject of Ricci flow if he continued to publicize
his works on arxiv.
Although ``\textit{this separate paper}" never appears,  his fundamental estimates of K\"ahler Ricci flow on Fano manifolds is the base of our present
research.  Besides Perelman's estimates,
we also note that the following technical results in the Ricci flow are important to the formation of this paper
over a long period of time: the Sobolev constant estimate by Q.S. Zhang(\cite{Zhq1}) and R. Ye (\cite{Ye}),
and the volume ratio upper bound estimate by Q.S. Zhang(\cite{Zhq3}) and  Chen-Wang(\cite{CW5}). Some other important estimates can be found in the summary of~\cite{CW2}. 

Our key observation is that there is a ``canonical neighborhood" theorem for anti-canonical K\"ahler Ricci flows.
The idea of ``canonical neighborhood"  originates from Theorem 12.1 of Perelman's paper~\cite{Pe1}.
For every 3-dimensional Ricci flow, Perelman showed that the space-time neighborhood of a high curvature point can be approximated by
a $\kappa$-solution, which is a model Ricci flow solution.
To be precise, a $\kappa$-solution is a 3-dimensional, $\kappa$-noncollapsed,  ancient Ricci flow solution with bounded, nonnegative curvature operator.
 By definition, it is not clear at all that the moduli of $\kappa$-solutions has compactness under (pointed-) smooth topology (modulo diffeomorphisms).
Perelman genuinely proved the compactness by delicate use of Hamilton-Ivey estimate and the geometry of nonnegatively curved 3-manifolds.
In light of the compactness of the moduli of $\kappa$-solutions,  by a maximum principle type argument,
Perelman developed the ``canonical neighborhood" theorem,
which is of essential importance to his celebrated
solution of the Poincar\'{e} conjecture(c.f.~\cite{KL},~\cite{MT},~\cite{CZ}).

The idea of  ``canonical neighborhood"  is universal and  can be applied in many different geometric settings.
In particular,  there is a ``canonical neighborhood" theorem
for the anti-canonical K\"ahler Ricci flows, where estimates of many quantities, including scalar curvature, Ricci potential and Sobolev constant,
are available.  Clearly, a ``canonical neighborhood" should be a neighborhood in space-time, behaving like a model space-time,
which is more or less the blowup limit of the given flow.
Therefore, it is natural to expect that the model space-time is the scalar flat Ricci flow solutions,
which must be Ricci flat,  due to the
equation $ \frac{\partial}{\partial t} R= \Delta R + 2|Ric|^2$, satisfied by the scalar curvature $R$.
For this reason, the model space and model space-time can be identified, since the evolution on time direction is trivial.
It is also natural to expect that the model space has some K\"ahler structure.  In other words, the model space should
be K\"ahler Ricci flat space, or Calabi-Yau space.
Now the first essential difficulty  appears.  A good model space should have a compact moduli.
For example, in the case of 3-dimensional Ricci flow, the moduli space of $\kappa$-solutions, which are the model space-times, has compactness in the smooth topology.
However, the moduli space of all the non-collapsed smooth Calabi-Yau space-times
is clearly not compact under the smooth topology.
A blowdown sequence of Eguchi-Hanson metrics is an easy example.
For the sake of compactness,
we need to replace the smooth topology by a weaker topology,
the pointed-Cheeger-Gromov topology.
At the same time,
we also need to enlarge the class of model spaces from complete Calabi-Yau manifolds to the Calabi-Yau spaces with mild singularities(c.f.~Definition~\ref{dfn:SC27_1}),
which we denote by $\widetilde{\mathscr{KS}}(n,\kappa)$.
Similar to the compactness theorem of Perelman's $\kappa$-solutions,
we have the compactness of   $\widetilde{\mathscr{KS}}(n,\kappa)$.

\begin{theoremin}[\textbf{Compactness of model moduli}]
$\widetilde{\mathscr{KS}}(n,\kappa)$ is compact under the pointed Cheeger-Gromov topology.
Moreover, each space $X \in \widetilde{\mathscr{KS}}(n,\kappa)$ is a Calabi-Yau conifold.
\label{thmin:HE21_1}
\end{theoremin}

The notion of conifold is well known to string theorist as
some special Calabi-Yau 3-folds with singularities(c.f.~\cite{Green}).
In this paper, by abusing notation,
we use it to denote a space whose singular part admits cone type tangent spaces.
The precise definition is given in Definition~\ref{dfn:HD20_1}.
Note that Calabi-Yau conifold is a generalization of Calabi-Yau orbifold.
The strategy to prove the compactness of
$\widetilde{\mathscr{KS}}(n,\kappa)$ follows the same route of the weak compactness theory
of K\"ahler Einstein manifolds, developed by Cheeger, Gromoll, Anderson, Colding, Tian, Naber, etc.
However, the analysis foundation on the singular spaces need to be carefully checked, which is done in section 2.
Theorem~\ref{thmin:HE21_1} is motivated by section 11 of Perelman's seminal paper~\cite{Pe1},
where Perelman proved the compactness of moduli space of $\kappa$-solutions and showed that
$\kappa$-solutions have many properties which are not obvious from definition.

By trivial extension, each
$X \in \widetilde{\mathscr{KS}}(n,\kappa)$ can be understood as a space-time $X \times (-\infty, \infty)$
satisfying Ricci flow equation.
Intuitively, the rescaled space-time structure in a given anti-canonical K\"ahler Ricci flow
should behave similar to that of $X \times (-\infty, \infty)$ for some $X \in \widetilde{\mathscr{KS}}(n,\kappa)$,
when the rescaling factor is large enough.
In order to make sense that two space-times are close to each other,
we need the Cheeger-Gromov topology for space-times, a slight generalization of the Cheeger-Gromov
topology for metric spaces.
When restricted on each time slice, this topology is the same as the usual Cheeger-Gromov topology.
Between every two different time slices, there is a natural homeomorphism map connecting them.
Therefore, the above intuition can be realized if we can show a blowup sequence of Ricci flow
space-times from a given K\"ahler Ricci flow converges to a limit space-time $X \times (-\infty, \infty)$,
in the pointed Cheeger-Gromov topology for space-times.
However,  it is not easy to obtain the homeomorphism maps between different time slices in the limit.
Although it is quite obvious to guess that the homeomorphism maps among different time slices
are the limit  of identity maps,
there exists serious technical difficulty to show the existence and regularity of the limit maps.
The difficulty boils down to a fundamental improvement of Perelman's pseudolocality theorem(Theorem 10.1 of~\cite{Pe1}).
Recall that Perelman's pseudolocality theorem says that Ricci flow cannot ``quickly" turn an almost Euclidean region into a very curved one.
It is a  short-time, one-sided estimate in nature. We need to improve it  to a long-time, two-sided estimate.
Not surprisingly, the rigidity of K\"ahler geometry plays an essential role for such an improvement.
The two-sided, long-time pseudolocality is an estimate in the time direction.
Modulo this time direction estimate and the weak compactness in the space direction, we can take limit for
a sequence of Ricci flows blown up from a given flow.
Then the canonical neighborhood theorem can be set up if we can show that the limit space-time locates in
$\widetilde{\mathscr{KS}}(n,\kappa)$,
following the same route as that in the proof of Theorem 12.1 of~\cite{Pe1}. \\

From the above discussion, it is clear that the strategy to prove the canonical neighborhood theorem is simple.
However, the technical difficulty hidden behind this simple strategy is not that simple.
We observe that the anti-canonical
K\"ahler Ricci flow has many additional structures, all of them should be used to carry out the proof of
the canonical neighborhood theorem.
In particular, over every anti-canonical K\"ahler Ricci flow, there is a natural anti-canonical polarization, which should play an important role, as done in~\cite{CW4}.
Although we are aiming at the anti-canonical case,
in this paper, however, we shall consider flows with more general polarizations.
We call  $\mathcal{LM}=\left\{ (M^n, g(t), J, L, h(t)),  t \ \in (-T,T) \subset \R \right\}$  a polarized K\"ahler Ricci flow if
\begin{itemize}
\item $\mathcal{M}=\left\{ (M^n, g(t), J),  t \ \in (-T,T) \right\}$ is a K\"ahler Ricci flow solution.
\item $L$ is a Hermitian line bundle over $M$, $h(t)$ is a family of smooth metrics on $L$ whose curvature is $\omega(t)$, the metric form compatible with
$g(t)$ and the complex structure $J$.
\end{itemize}
Clearly, the first Chern class of $L$ is $[\omega(t)]$, which does not depend on time.  So a polarized K\"ahler Ricci flow stays in a fixed integer K\"ahler class.
The evolution equation of $g(t)$ can be written as
\begin{align}
     \frac{\partial}{\partial t} g_{i\bar{j}}=-R_{i\bar{j}}+ \lambda g_{i\bar{j}},  \label{eqn:K07_1}
\end{align}
where $\lambda=\frac{c_1(M)}{c_1(L)}$.
Since the flow stays in the fixed class, we can let $\omega_t=\omega_0 + \sqrt{-1} \partial \bar{\partial} \varphi$. Then $\dot{\varphi}$
is the Ricci potential, i.e.,
\begin{align*}
  \sqrt{-1} \partial \bar{\partial} \dot{\varphi} = -Ric + \lambda g.
\end{align*}
Note the choice of $\varphi$ is unique up to adding a constant. So we can always modify the choice of $\varphi$ such that  $\displaystyle \sup_M \dot{\varphi}=0$.
For simplicity, we denote $\mathscr{K}(n,A)$ as the collection of all the polarized K\"ahler Ricci flows
$\mathcal{LM}$ satisfying the following estimate
\begin{align}
\begin{cases}
  &T \geq 2, \\
  &C_S(M)+\frac{1}{\Vol(M)}+|\dot{\varphi}|_{C^0(M)} + |R-n\lambda|_{C^0(M)} \leq A,
  \quad \textrm{for every time}  \quad t \in (-T,T).
\end{cases}
 \label{eqn:SK20_1}
\end{align}
Here $C_S$ means the Sobolev constant, $A$ is a uniform constant.  In this paper, we study the structure of
polarized K\"ahler Ricci flows locating in the space $\mathscr{K}(n,A)$.
The motivation behind (\ref{eqn:SK20_1}) arises from the fundamental estimate of diameter, scalar curvature,
$C^1$-norm of Ricci potential, and Sobolev constant
along the anti-canonical K\"ahler Ricci flows(c.f.~\cite{SeT},~\cite{Zhq1},~\cite{Ye}).
Every polarized K\"ahler Ricci flow solution in $\mathscr{K}(n,A)$ has at least three structures: the metric space structure, the flow structure, the line bundle structure.
Same structures can be discussed on the model space-time in $\widetilde{\mathscr{KS}}(n,\kappa)$.
All the structures of a flow in $\mathscr{K}(n,A)$ can be modeled by the corresponding structures
in $\widetilde{\mathscr{KS}}(n,\kappa)$, which is the same meaning as the ``canonical neighborhood theorem".
We shall compare these structures term by term. \\

Under the (pointed-)Cheeger-Gromov topology at time $0$,  let us compare the metric structure of a flow in
 $\mathscr{K}(n,A)$ with a Calabi-Yau conifold in $\widetilde{\mathscr{KS}}(n,\kappa)$.
We shall show that  $\mathscr{K}(n,A)$ and  $\widetilde{\mathscr{KS}}(n,\kappa)$ behaves almost the same in this
perspective. Intuitively, one can think that the weak compactness theory of Ricci-flat manifolds and Einstein manifolds are almost the same. 
For simplicity of notation,  we use $\stackrel{G.H.}{\longrightarrow}$ to denote the convergence in Gromov-Hausdorff topology. We use $\stackrel{\hat{C}^k}{\longrightarrow}$ to denote
the $C^k$-Cheeger-Gromov topology, i.e., the convergence is in the Gromov-Hausdorff topology,
and can be improved to be in $C^k$-topology (modulo diffeomorphisms) away from singularities.
We call a point being regular if it has a neighborhood with smooth manifold structure and call a point being singular if it is not regular(c.f. Proposition~\ref{prn:HA08_1} and Remark~\ref{rmk:HA07_1}).

\begin{theoremin}[\textbf{Metric space estimates}]
 Suppose $\mathcal{LM}_i \in \mathscr{K}(n,A)$. By taking subsequence if necessary, we have
\begin{align*}
  (M_i, x_i, g_i(0)) \longright{\hat{C}^{\infty}} (\bar{M}, \bar{x}, \bar{g}).
\end{align*}
 The limit space $\bar{M}$ has a classical regular-singular decomposition $\mathcal{R} \cup \mathcal{S}$ with the following properties.
 \begin{itemize}
     \item  $\left(\mathcal{R}, \bar{g}\right)$ is a smooth, open Riemannian manifold. Moreover,
            $\mathcal{R}$ admits a limit K\"ahler structure $\bar{J}$ such that  $\left(\mathcal{R}, \bar{g}, \bar{J} \right)$ is an open K\"ahler manifold.
     \item  $\mathcal{S}$ is a closed set and $\dim_{\mathcal{M}} \mathcal{S} \leq 2n-4$, where $\mathcal{M}$ means Minkowski dimension(c.f.~Definition~\ref{dfn:HE08_1}).
     \item  Every tangent space of $\bar{M}$ is an irreducible metric cone.
     \item  Let $\mathrm{v}$ be the volume density, i.e.,
            \begin{align}
              \mathrm{v}(y)=\limsup_{r \to 0} \omega_{2n}^{-1}r^{-2n}|B(y,r)|
            \end{align}
            for every point $y \in \bar{M}$.
            Then a point is regular if and only if $\mathrm{v}(y)=1$,  a point is singular if and only if
            $\mathrm{v}(y) \leq 1-2\delta_0$, where $\delta_0$ is a dimensional
            constant determined by Anderson's gap theorem.
\end{itemize}
\label{thmin:SC24_1}
\end{theoremin}

It is important to note the difference between $\widetilde{\mathscr{KS}}(n,\kappa)$ and $\mathscr{K}(n,A)$. 
We use $\widetilde{\mathscr{KS}}(n,\kappa)$ to denote the space of possible bubbles, or blowup limits.
Therefore, every metric space in it is a non-compact one.   However, each time slice of flows in $\mathscr{K}(n,A)$ is a compact manifold.  
The limit space $\bar{M}$ of Theorem~\ref{thmin:SC24_1} maybe compact and does not belong to $\widetilde{\mathscr{KS}}(n,\kappa)$.

In the study of the line bundle structure of $\mathscr{K}(n,A)$,  the Bergman function plays an important role. 
Actually, for every positive integer $k$ large enough such that $L^k$ is globally generated, we define the Bergman function $\mathbf{b}^{(k)}$ as follows
\begin{align}
   \mathbf{b}^{(k)}(x,t)=\log \sum_{i=0}^{N_k} \norm{S_i^{(k)}}{h(t)}^2(x,t),
\label{eqn:HA03_3}
\end{align}
where $N_k=\dim_{\C} H^0(M, L^k)-1$, $\left\{S_i^{(k)} \right\}_{i=0}^{N_k}$ are orthonormal basis of $H^0(M, L^k)$ under the natural metrics $\omega(t)$ and $h(t)$.
Theorem~\ref{thmin:SC24_1} means that the metric structure of the center time slice of a K\"ahler Ricci flow in $\mathscr{K}(n,A)$ can be modeled by
non-collapsed Calabi-Yau manifolds with mild singularities.  In particular, each tangent space of a point in the limit space is a metric cone.
The trivial line bundle structure on metric cone then implies an estimate of line bundle structure of the original manifold, due to
delicate use of H\"omander's $\bar{\partial}$-estimate, as done by Donaldson and Sun(c.f.~\cite{DS}).

\begin{theoremin}[\textbf{Line bundle estimates}]
  Suppose $\mathcal{LM} \in \mathscr{K}(n,A)$, then
  \begin{align*}
    \inf_{x \in M} \mathbf{b}^{(k_0)}(x,0) \geq -c_0
  \end{align*}
  for some positive number $c_0=c_0(n,A)$, and positive integer $k_0=k_0(n,A)$.
\label{thmin:HC08_1}
\end{theoremin}

In other words, Theorem~\ref{thmin:HC08_1} states that there is a uniform partial-$C^0$-estimate at time $t=0$.  This estimate then implies variety structure of
limit space, as discussed in~\cite{Tian12} and~\cite{DS}.
Theorem~\ref{thmin:HC08_1} can be understood that the line bundle structure of $\mathscr{K}(n,A)$ is modeled after
that of $\widetilde{\mathscr{KS}}(n,\kappa)$.

Theorem~\ref{thmin:SC24_1} and Theorem~\ref{thmin:HC08_1} deal only with one time slice.
In order to make sense of limit K\"ahler Ricci flow, we have to compare the limit spaces of different time slices.
For example, we choose $x_i \in M_i$, then we have
\begin{align*}
    (M_i, x_i, g_i(0)) \longright{\hat{C}^{\infty}} (\bar{M}, \bar{x}, \bar{g}), \quad (M_i, x_i, g_i(-1)) \longright{\hat{C}^{\infty}} (\bar{M}', \bar{x}', \bar{g}').
\end{align*}
How are $\bar{M}$ and $\bar{M}'$ related?  If $\bar{x}$ is a regular point of $\bar{M}$,  can we say $\bar{x}'$ is a regular point of $\bar{M}'$?
Note that Perelman's pseudolocality theorem cannot answer this question, due to its short-time, one-sided property.   In order to relate different time slices,
we need to improve Perelman's pseudolocality theorem to the following long-time, two-sided estimate,
which is the technical core of the current paper.

\begin{theoremin}[\textbf{Time direction estimates}]
  Suppose $\mathcal{LM} \in \mathscr{K}(n,A)$.
  Suppose $x_0 \in M$, $\Omega=B_{g(0)}(x_0,r)$, $\Omega'=B_{g(0)}(x_0, \frac{r}{2})$ for some $r \in (0,1)$. 
  At time $t=0$, suppose the isoperimetric constant estimate $\mathbf{I}(\Omega) \geq (1-\delta_0) \mathbf{I}(\C^n)$ holds for $\delta_0=\delta_0(n)$,
  the same constant in Theorem~\ref{thmin:SC24_1}.   Then we have
  \begin{align*}
     |\nabla^k Rm|(x,t) \leq C_k,  \quad \forall \;  k \in \Z^{\geq 0}, \quad x \in \Omega',
     \quad t \in [-1,1],
  \end{align*}
  where $C_k$ is a constant depending on $n,A,r$ and $k$.
\label{thmin:HC06_1}
\end{theoremin}

Theorem~\ref{thmin:HC06_1} holds trivially on each space in $\widetilde{\mathscr{KS}}(n,\kappa)$, when regarded as
a static Ricci flow solution.  Therefore, it can be understood as the time direction structure, or the flow structure of
$\mathcal{LM} \in \mathscr{K}(n,A)$ is similar to that of $\widetilde{\mathscr{KS}}(n,\kappa)$.
Theorem~\ref{thmin:HC06_1} removes the major stumbling block for defining a limit K\"ahler Ricci flow,
since it guarantees that the regular-singular decomposition of the limit space is independent of time.
Therefore, there is a natural induced K\"ahler Ricci flow structure on the regular part of the limit space.   We denote its completion
by a limit K\"ahler Ricci flow solution, in a weak sense.  Clearly, the limit K\"ahler Ricci flow naturally inherits a  limit line
bundle structure, or a limit polarization, on the regular part.
Moreover, the limit underlying space does have a variety structure due to Theorem~\ref{thmin:HC08_1}.
With these structures in hand,  we are ready to discuss the convergence theorem of polarized K\"ahler Ricci flows,
which is the main structure theorem of this paper(c.f. section~\ref{sec:sflow} for meaning of the notations).

\begin{theoremin}[\textbf{Weak compactness of polarized flows}]
  Suppose $\mathcal{LM}_i \in \mathscr{K}(n,A)$, $x_i \in M_i$ satisfying $\diam_{g_i(0)}(M_i)<C$ uniformly or $\sup_{\mathcal{M}_i} |R| \to 0$.
  By passing to subsequence if necessary, we have
  \begin{align*}
       \left(\mathcal{LM}_i, x_i \right) \stackrel{\hat{C}^{\infty}}{\longrightarrow} \left(\overline{\mathcal{LM}}, \bar{x}\right),
  \end{align*}
  where $\overline{\mathcal{LM}}$ is a polarized K\"ahler Ricci flow solution on an analytic normal variety $\bar{M}$,
  whose singular set $\mathcal{S}$ has Minkowski codimension at least $4$, with respect to each $\bar{g}(t)$.
  Moreover, if $\bar{M}$ is compact, then it is a projective normal variety with at most log-terminal singularities.
\label{thmin:HC06_2}
\end{theoremin}

As it is developed for, our structure theory has  applications in the study of anti-canonical K\"ahler Ricci flows.
Due to the fundamental estimate of Perelman and the monotonicity of his $\mu$-functional along each anti-canonical
K\"ahler Ricci flow, we can apply Theorem~\ref{thmin:HC06_2} directly and obtain the following theorem.

  \begin{theoremin}[\textbf{Hamilton-Tian conjecture}]
    Suppose $\left\{ (M^n, g(t)), 0 \leq t <\infty \right\}$ is an anti-canonical K\"ahler Ricci flow solution
    on a Fano manifold $(M,J)$.   For every $s>1$, define
    \begin{align*}
      &g_s(t) \triangleq g(t+s), \\
      &\mathcal{M}_{s} \triangleq \{(M^n, g_s(t)), -s \leq t \leq s\}.
    \end{align*}
    Then for every sequence $s_i \to \infty$, by taking subsequence if necessary, we have
    \begin{align}
      \left(\mathcal{M}_{s_i}, g_{s_i}\right) \longright{\hat{C}^{\infty}} \left( \bar{\mathcal{M}}, \bar{g} \right),   \label{eqn:SC13_111}
    \end{align}
    where the limit space-time $\bar{\mathcal{M}}$ is a K\"ahler Ricci soliton flow solution on a $Q$-Fano normal variety $(\bar{M},\bar{J})$.
    Moreover,  with respect to each $\bar{g}(t)$, there is a uniform $C$ independent of time such that
    the $r$-neighborhood of the singular set $\mathcal{S}$ has measure not greater than $Cr^4$.
 \label{thmin:SC24_3}
  \end{theoremin}

 Theorem~\ref{thmin:SC24_3} confirms the famous Hamilton-Tian conjecture, with more information
 than that was conjectured(c.f.~Conjecture 9.1. of~\cite{Tian97} for the precise statement).
 The two dimensional case was confirmed by the authors in~\cite{CW3}.
 We note that in a recent paper \cite{TZZ2},
 another approach to attack Hamilton-Tian conjecture in complex dimension  $3$,
 based on $L^4$-bound of Ricci curvature, was presented by Z.L. Zhang and G. Tian.
 Their work in turn depends on the comparison geometry with integral Ricci bounded developed
 by G.F. Wei and P. Petersen(\cite{PW1}).
 For other important progress in K\"ahler Ricci flow, we refer interested readers  to the following papers(far away from being complete): \cite{Sesum}, \cite{Zhq1}, \cite{Ye}, \cite{TZ2}, \cite{Zhq2}, \cite{SongWeinkove}, \cite{TZ2}, \cite{SongTian}, \cite{PhongSturm}, \cite{Tosa}, \cite{SzeKrf}, as well as references listed therein.\\

 As corollaries of Theorem~\ref{thmin:SC24_3}, we can affirmatively answer some problems raised in~\cite{CW4}.

\begin{corollaryin}
  Every anti-canonical K\"ahler Ricci flow is tamed, i.e., partial-$C^0$-estimate holds along the flow.
 \label{clyin:SC24_4}
\end{corollaryin}

\begin{corollaryin}
 Suppose $\{(M^n, g(t)), 0 \leq t < \infty \}$ is an anti-canonical K\"ahler Ricci flow on a Fano manifold $M$.
 Then the flow converges to a K\"ahler Einstein metric
 if one of the following conditions hold for every large positive integer $\nu$.
 \begin{itemize}
  \item $\alpha_{\nu,1}>\frac{n}{n+1}$.
  \item $\alpha_{\nu,2}>\frac{n}{n+1}$ and $\alpha_{\nu,1} > \frac{1}{2- \frac{n-1}{(n+1) \alpha_{\nu,2}}}$.
 \end{itemize}
 \label{clyin:SC24_5}
\end{corollaryin}

 Corollary~\ref{clyin:SC24_5} give rise to a method for searching Fano K\"ahler Einstein metrics in high dimension,
 which generalize the 2-dimensional case due to Tian(c.f.~\cite{Tian90}). The quantities $\alpha_{\nu,k}$ are some algebro-geometric invarariant.
 The interested readers are referred to~\cite{Tian90} for the precise definition.

Our structure theory can be applied to study a family of K\"ahler Ricci flows with some uniform initial conditions.
In this perspective, we have the following theorem.

\begin{theoremin}[\textbf{Partial-$C^0$-conjecture of Tian}]
 For every positive constants $R_0,V_0$, there exists a positive integer $k_0$
 and a positive constant $c_0$ with the following properties.

 Suppose $(M,\omega, J)$ is a K\"ahler manifold satisfying
 $Ric \geq R_0$ and $\Vol(M) \geq V_0$, $[\omega]=2\pi c_1(M,J)$.
 Then we have
 \begin{align*}
    \inf_{x \in M} \mathbf{b}^{(k_0)}(x) > -c_0.
 \end{align*}
 \label{thmin:HA01_1}
 \end{theoremin}

 Theorem~\ref{thmin:HA01_1} confirms the partial-$C^0$-conjecture of Tian(c.f.~\cite{Tian90Kyo},\cite{Tian12}).
 The low dimension case ($n \leq 3$) was proved by Jiang(\cite{Jiang}), depending on the partial-$C^0$-estimate along the flow,
 developed by  Chen-Wang(\cite{CW3},\cite{CW4}) in complex dimension 2
 and Tian-Zhang(\cite{TZZ2}) in complex dimension 3.
 In fact, a more general version of Theorem~\ref{thmin:HA01_1} is proved(c.f. Theorem~\ref{thm:HA01_3}).
 As a corollary of Theorem~\ref{thmin:HA01_1}, we have

 \begin{corollaryin}(c.f.~\cite{Sze})
  The partial-$C^0$-estimate holds along the classical continuity path.
 \label{clyin:HA01_2}
 \end{corollaryin}

 Following Corollary~\ref{clyin:SC24_4},  we obtain the following result, which was originally proved by
 G. Sz\'{e}kelyhidi(c.f.~\cite{Sze})  along the classical continuity path.
\begin{corollaryin}
  Suppose  $(M,J)$  is a Fano manifold with $Aut(M,J)$ discrete.  If it is stable in the sense of S.Paul(c.f.~\cite{Paul12}), then it admits a K\"ahler Einstein metric.
\label{clyin:HA03_1}
\end{corollaryin}

An important application of our structure theory is devoted to the study of the relationships among different stabilities.
By the work of Chen, Donaldson and Sun(c.f.~\cite{CDS0},~\cite{CDS1},~\cite{CDS2} and~\cite{CDS3}),  a long standing stability conjecture, going back to Yau(c.f. Problem 65 of~\cite{Yau93}) and critically contributed by Tian(c.f~\cite{Tian97}) and Donaldson(c.f.~\cite{Do02}), was confirmed.
We now know a Fano manifold is K-stable if and only if it admits K\"ahler Einstein metrics.  A posteriori, we see that
the K-stability is equivalent to Paul's stability if the underlying manifold has discrete automorphism group.
It is an interesting problem  to prove this equivalence a priori,
which will be discussed in a separate paper(c.f.~\cite{CSW}). \\

 Let us quickly go over the relationships among the theorems.
 Theorem~\ref{thmin:HE21_1} is the structure theorem of the model space $\widetilde{\mathscr{KS}}(n,\kappa)$.
 Theorem~\ref{thmin:SC24_1}, Theorem~\ref{thmin:HC08_1} and Theorem~\ref{thmin:HC06_1}
 combined together give the canonical neighborhood structure of the polarized K\"ahler Ricci flow
  in $\mathscr{K}(n,A)$, in a strong sense.
 The main structure theorem in this paper is Theorem~\ref{thmin:HC06_2}, the weak compactness theorem of polarized K\"ahler Ricci flows.
 It is clear that Theorem~\ref{thmin:SC24_3} and Theorem~\ref{thmin:HA01_1} are
 direct applications of Theorem~\ref{thmin:HC06_2}.  The proof of Theorem~\ref{thmin:HC06_2} is based on the combination of
 Theorem~\ref{thmin:SC24_1}, Theorem~\ref{thmin:HC08_1} and Theorem~\ref{thmin:HC06_1}.   These three theorems deal with different structures
 of $\mathscr{K}(n,A)$,  including the Ricci flow structure, metric space structure, line bundle structure and variety structure.
 The importance of these structures decreases in order, for the purpose of developing compactness.
 However, all these structures are intertwined together.
 Paradoxically, the proof of the compactness of these structures does not follow the same order, due to the lack of precise estimate of Bergman functions.
 Instead of proving them in order, we define a concept called ``polarized canonical radius", which guarantees the convergence of all these structures
 under this radius.  The only thing we need to do then is to show that this radius cannot be too small.  Otherwise,  we can apply a maximum principle argument
 to obtain a contradiction, which essentially arise from the monotonicity of Perelman's reduced volume and localized $W$-functional.\\

 This paper is organized as follows.   In section 2, we discuss the model space $\widetilde{\mathscr{KS}}(n,\kappa)$, which consists of non-collapsed Calabi-Yau spaces with mild singularities.   By checking analysis foundation and repeating the weak compactness theory of K\"ahler Einstein manifolds, we prove the compactness of
 $\widetilde{\mathscr{KS}}(n,\kappa)$ and show that every space in it is a conifold.
 In other words, we prove Theorem~\ref{thmin:HE21_1} at the end of section 2.
 We also develop  some a priori estimates, which will be  essentially used in the following sections.
 In section 3, we define the ``canonical radius" and discuss the convergence of metric structures when canonical radius is uniformly bounded from below.
 In section 4, we first set up a forward, long-time pseudolocality theorem based on the existence of
 partial-$C^0$-estimate. Motivated by this pseudolocality  theorem,  we then refine the ``canonical radius" to ``polarized canonical radius"  and discuss the convergence of
 flow structure and line bundle structure under the assumption that polarized canonical radius  is uniformly bounded from below. Finally,  at the end of section 4,  we use a maximum principle argument to
 show that there is an a priori bound of the polarized canonical radius.
 In section 5, we prove Theorem~\ref{thmin:SC24_1}-\ref{thmin:HC06_2}, together with some other
 more detailed properties of the space $\mathscr{K}(n,A)$.
 Up to this section, everything is developed for general polarized K\"ahler Ricci flow.
 At last, in section 6, we focus on the anti-canonical K\"ahler Ricci flows.
 Applying the general structure theory,  we prove Theorem~\ref{thmin:SC24_3} and Theorem~\ref{thmin:HA01_1}. \\

\noindent {\bf Acknowledgment}
 Both authors are very grateful to professor Simon Donaldson for his constant support.
 The second author would like to thank Song Sun and Shaosai Huang for helpful discussions and suggestions. 
 Thanks also go to Weiyong He, Haozhao Li,Yuanqi Wang,  Guoqiang Wu, Chengjian Yao, Hao Yin, Kai Zheng
 for their valuable comments.

\section{Model space---Calabi-Yau Space with mild singularities}

The Model space of a polarized K\"ahler Ricci flow in $\mathscr{K}(n,A)$ consists of the space-time blowup limits from
flows in $\mathscr{K}(n,A)$.   In this section, we shall discuss the properties of the model space, from the perspective of
metric space structure and the intrinsic Ricci flow structure.

\subsection{Singular Calabi-Yau space $\widetilde{\mathscr{KS}}(n,\kappa)$}

Let $\mathscr{KS}(n)$ be the collection of all the complete $n$-dimensional Calabi-Yau (K\"ahler Ricci flat) manifolds.
By Bishop-Gromov comparison, it is clear that the asymptotic volume ratio is well defined for every manifolds in the moduli space $\mathscr{KS}(n)$.
The gap theorem of Anderson (c.f. Gap Lemma 3.1 of~\cite{An90}) implies that the asymptotic volume ratio is strictly less than $1-2\delta_0$ whenever the underlying manifold is not the flat $\C^n$, where
$\delta_0$ is a dimensional constant.  We fix this constant and call it as Anderson constant in this paper.

Let $\mathscr{KS}(n,\kappa)$ be a subspace of $\mathscr{KS}(n)$, with every element has asymptotic volume ratio at least $\kappa$.
Clearly, $\mathscr{KS}(n,\kappa)$  is not compact under the pointed-Gromov-Hausdorff topology.
It can be compactified as a space $\overline{\mathscr{KS}}(n,\kappa)$.
 However, this may not be the largest space that one can develop weak-compactness theory.  So we extend the space $\overline{\mathscr{KS}}(n,\kappa)$ further to a possibly bigger compact space $\widetilde{\mathscr{KS}}(n,\kappa)$, which is defined as follows.

 \begin{definition}
 Let $\widetilde{\mathscr{KS}}(n,\kappa)$ be the collection of length spaces $(X,g)$ with the following properties.
\begin{enumerate}
  \item  $X$ has a disjoint regular-singular decomposition $X=\mathcal{R} \cup \mathcal{S}$, where $\mathcal{R}$ is the regular part,  $\mathcal{S}$ is the singular part.
   A point is called regular if it has a neighborhood which is isometric to a totally geodesic convex domain of  some smooth Riemannian manifold.  A point is called singular
   if it is not regular.
  \item  The regular part $\mathcal{R}$ is a nonempty, open Ricci-flat manifold of real dimension $m=2n$.
         Moreover, there exists a complex structure $J$ on $\mathcal{R}$ such that $(\mathcal{R}, g, J)$ is a K\"ahler manifold.
  \item  $\mathcal{R}$ is weakly convex, i.e., for every point $x \in \mathcal{R}$, there exists a measure ($2n$-dimensional Hausdorff measure) zero set
    $\mathcal{C}_x \supset \mathcal{S}$ such that every point in $X \backslash \mathcal{C}_x$ can be connected to $x$
    by a unique shortest geodesic in $\mathcal{R}$.   For convenience, we call $\mathcal{C}_x$ as the cut locus of $x$.
  \item $\dim_{\mathcal{M}} \mathcal{S} < 2n-3$, where $\mathcal{M}$ means Minkowski dimension.
  \item Let $\mathrm{v}$ be the volume density function,i.e.,
        \begin{align}
          \mathrm{v}(x) \triangleq \lim_{r \to 0} \frac{|B(x,r)|}{\omega_{2n} r^{2n}}   \label{eqn:SC16_1}
        \end{align}
        for every $x \in X$. Then $\mathrm{v}\equiv 1$ on $\mathcal{R}$ and $\mathrm{v} \leq 1-2\delta_0$ on $\mathcal{S}$.
        In other words, the function $\mathrm{v}$ is a criterion function for singularity.
        Here $\delta_0$ is the Anderson constant.
  \item The asymptotic volume ratio $\mathrm{avr}(X) \geq \kappa$. In other words, we have
   \begin{align*}
      \lim_{r \to \infty} \frac{|B(x,r)|}{\omega_{2n}r^{2n}} \geq \kappa
   \end{align*}
    for every $x \in X$.
\end{enumerate}
 Let $\widetilde{\mathscr{KS}}(n)$ be the collection of metric spaces $(X,g)$ with all the above properties except the last one.  Since Euclidean space is a special element,
 we define
 \begin{align*}
      \widetilde{\mathscr{KS}}^{*}(n) \triangleq \widetilde{\mathscr{KS}}(n) \backslash \{(\C^n, g_{\E})\},
      \quad \widetilde{\mathscr{KS}}^{*}(n,\kappa) \triangleq \widetilde{\mathscr{KS}}(n,\kappa) \backslash \{(\C^n, g_{\E})\}.
 \end{align*}
\label{dfn:SC27_1}
\end{definition}

Note that the $\kappa$ in $\widetilde{\mathscr{KS}}(n)$ means the asymptotic area ratio is at least $\kappa$.  If we drop $\kappa$,  the space $\widetilde{\mathscr{KS}}(n)$ may contain
compact spaces.  The default measure is always the $2n$-dimensional Hausdorff measure, unless we mention otherwise. 
We use $\dim_{\mathcal{H}}$ to denote Hausdorff dimension, $\dim_{\mathcal{M}}$ to denote Minkowski dimension, or the
box-counting dimension.    Since Minkowski dimension is not as often used as Hausdorff dimension,  let us recall the definition of it quickly(c.f.~\cite{Falco}).
\begin{definition}
  Suppose $E$ is a bounded subset of $X$.
  $E_r$ is the $r$-neighborhood of $E$ in $X$.  Then
  the upper Minkowski dimension of $E$ is defined as the limit:
  $\displaystyle \dim_{\mathcal{H}} X-\liminf_{r \to 0^+}  \frac{\log |E_r|}{\log r}$.
  We say $\dim_{\mathcal{M}} E \leq \dim_{\mathcal{H}} X-k$
  if the upper Minkowski dimension of $E$ is not greater than $2n-k$.
  Namely, we have
  \begin{align*}
    \liminf_{r \to 0^+}  \frac{\log |E_r|}{\log r} \geq k.
  \end{align*}
  If $E$ is not a bounded set, we say $\dim_{\mathcal{M}} E \leq \dim_{\mathcal{H}} X-k$ if
  $\dim_{\mathcal{M}} E \cap B \leq \dim_{\mathcal{H}} X-k$ for each unit geodesic ball $B \subset X$
  satisfying $B\cap E \neq \emptyset$.
\label{dfn:HE08_1}
\end{definition}

In general, it is known that Hausdorff dimension is not greater than Minkowski dimension. Hence, we always have
$\dim_{\mathcal{H}}\mathcal{S} \leq \dim_{\mathcal{M}} \mathcal{S}$.
In our discussion,  $X$ clearly has Hausdorff dimension $2n$. Therefore, $\dim_{\mathcal{M}} \mathcal{S}<2n-3$
implies  that for each nonempty intersection $B(x_0,1) \cap \mathcal{S}$,  its $r$-neighborhood has measure
$o(r^{3})$ for sufficiently small $r$.  By virtue of the high codimension of $\mathcal{S}$ and the Ricci-flatness of $\mathcal{R}$,
in many aspects, each metric space $X \in \widetilde{\mathscr{KS}}(n,\kappa)$ can be treated as an intrinsic Ricci-flat space.  We shall see that the geometry of $X$ is almost the same as that of Calabi-Yau manifold.

\begin{proposition}[\textbf{Bishop-Gromov volume comparison}]
Suppose $x_0 \in X$, $0<r_a<r_b<\infty$,  and $\delta>0$. Then we have
\begin{align}
 &\frac{|B(x_0,r_a)|}{r_a^{2n}} \geq \frac{|B(x_0,r_b)|}{r_b^{2n}},  \label{eqn:GD08_2}\\
 &\frac{|B(x_0,r_a+\delta)|-|B(x_0,r_a)|}{(r_a+\delta)^{2n}-r_a^{2n}} \geq \frac{|B(x_0,r_b+\delta)|-|B(x_0,r_b)|}{(r_b+\delta)^{2n}-r_b^{2n}}.   \label{eqn:GD08_3}
\end{align}
\label{prn:HD19_1}
\end{proposition}
\begin{proof}
 We first prove (\ref{eqn:GD08_2}) for the case $x_0 \in \mathcal{R}$.
 Away from the cut locus $\mathcal{C}_{x_0}$, which is measure-zero, every point can be connected to $x_0$ by a unique smooth geodesic. 
 Therefore, every point $y \in X \backslash \mathcal{C}_{x_0}$ can be identified with a point $(\gamma'(0), L) \in \R^{2n}$, where $\gamma$ is
 the shortest geodesic connecting $x_0$ and $y$, with $\gamma(0)=x_0$, $L$ is the length of $\gamma$.  In this way, we constructed a
 polar coordinate system around $x_0$.  Since $|B(x_0,r)|=|B(x_0,r) \backslash \mathcal{C}_{x_0}|$, by calculating the volume element evolution
 along each $\gamma$ in polar coordinate, we obtain the volume comparison same as the Riemannian case.
 This is more or less standard. For example, one can check the details from~\cite{ZhuSH}, or the survey~\cite{Wei}. 
 Now we show (\ref{eqn:GD08_2}) for $x_0 \in \mathcal{S}$.  Let $x_i \in \mathcal{R}$ and $x_i \to x_0$.  Fix $r>0$.  Note that 
 \begin{align}
    \lim_{i \to \infty} |B(x_i,r)| = |B(x,r)|.   \label{eqn:GD01_0}
 \end{align}
 Actually, for each $\epsilon>0$ and large $i$, we have $B(x_i,r-\epsilon) \subset B(x,r) \subset B(x_i, r+\epsilon)$ and hence
 \begin{align}
  &|B(x_i,r-\epsilon)|-|B(x,r)| \leq |B(x_i,r)| -|B(x,r)|\leq |B(x_i,r+\epsilon)|-|B(x,r)|, \notag \\
  &  | |B(x_i,r)| -|B(x,r)| | \leq |B(x_i,r+\epsilon)-B(x_i,r-\epsilon)|. \label{eqn:GD01_1}
 \end{align}
 Note that $x_i$ is a regular point for each $i>1$, by standard Bishop-Gromov comparison, we have
 \begin{align*}
  |B(x_i,r+\epsilon)-B(x_i,r-\epsilon)| \leq 2n\omega_{2n} \{(r+\epsilon)^{2n}-(r-\epsilon)^{2n}\} \leq C(n,r) \epsilon. 
 \end{align*}
 Therefore, taking limit of (\ref{eqn:GD01_1}) as $i \to \infty$ and then let $\epsilon \to 0$, we obtain (\ref{eqn:GD01_0}). 
 Consequently, we have
 \begin{align}
     \lim_{i \to \infty} \omega_{2n}^{-1}r_a^{-2n}|B(x_i,r_a)|= \omega_{2n}^{-1}r_a^{-2n}|B(x_0,r_a)|,   \quad
     \lim_{i \to \infty} \omega_{2n}^{-1}r_b^{-2n}|B(x_i,r_b)|=\omega_{2n}^{-1}r_b^{-2n}|B(x_0,r_b)|.   \label{eqn:GD01_2}
 \end{align}
 Again, $x_i$ is a regular point for each $i>1$, so (\ref{eqn:GD08_2}) was proved for $x_i$ and can be written as
 \begin{align*}
   \omega_{2n}^{-1}r_a^{-2n}|B(x_i,r_a)| \geq  \omega_{2n}^{-1}r_b^{-2n}|B(x_i,r_b)|. 
 \end{align*}
 Plugging the above inequality into (\ref{eqn:GD01_2}), we obtain (\ref{eqn:GD08_2}) for the singular point $x_0$. 
 
 The proof of (\ref{eqn:GD08_3}) is similar.  We first prove (\ref{eqn:GD08_3}) for regular point $x_0$ and then use approximation to prove it for
 singular $x_0$.   For regular $x_0$,  in polar coordinates, (\ref{eqn:GD08_3}) can be proved the same as the smooth Riemannian manifold
 case(c.f. Theorem 3.1 of~\cite{ZhuSH}).   In the approximation step, it is important to have volume continuity of annulus. However, this can be proved similar
 to (\ref{eqn:GD01_0}), by using triangle inequalities.  
\end{proof}

\begin{corollary}[\textbf{Volume doubling}]
 $X$ is a volume doubling metric space. More precisely, for every $x_0 \in X$ and $r>0$, we have
 \begin{align*}
    \frac{|B(x_0,2r)|}{|B(x_0,r)|} \leq \kappa^{-1}.
 \end{align*}
\label{cly:HD20_1}
\end{corollary}

\begin{corollary}[\textbf{``Area ratio" monotonicity}]
For each $x_0 \in X$, there is a function $A(r)$, the ``area ratio", defined almost everywhere on $(0,\infty)$ such that
 \begin{align}
     &|B(x_0,r)|=\int_0^r A(s) s^{2n-1} ds, \quad \forall \; r>0.   \label{eqn:GD08_4}\\
     & \frac{|B(x_0,r_b)|}{r_b^{2n}}-\frac{|B(x_0,r_a)|}{r_a^{2n}}=\int_{r_a}^{r_b} \frac{2n}{r} \left( \frac{A(r)}{2n} - \frac{|B(x_0,r)|}{r^{2n}}\right) dr, \quad \forall \; 0<r_a<r_b.   \label{eqn:GD08_5}
 \end{align}
 Furthermore, $A$ is non-increasing on its domain.  In other words, we have $ A(r_a) \geq A(r_b)$ whenever $A(r_a)$, $A(r_b)$ are well defined and $0<r_a<r_b$. 
\label{cly:GD08_1} 
\end{corollary}

\begin{proof}
 From the approximation process in the proof of Proposition~\ref{prn:HD19_1}, we see that even for $x_0 \in \mathcal{S}$, the inequalities
 \begin{align*}
  0 \leq  \frac{d}{dr} |B(x_0, r)| \leq 2n \omega_{2n} r^{2n-1},  \quad
  -\frac{2n}{r}     \leq     \frac{d}{dr} \left\{\frac{|B(x_0, r)|}{\omega_{2n} r^{2n}} \right\} \leq 0,
 \end{align*}
 hold in the barrier sense.  In particular,  $ |B(x_0, r)|$ and $\omega_{2n}r^{-2n}|B(x_0,r)|$ are monotone, uniformly Lipschitz functions of $r$ on each compact sub-interval of $(0,\infty)$. 
 Therefore, they have bounded derivatives almost everywhere. By abuse of notation, we denote the derivatives of $|B(x_0,r)|$ by $|\partial B(x_0, r)|$. 
 Let $A(r)$ be  $r^{1-2n}|\partial B(x_0, r)|$.  Clearly, $A(r)$ is defined almost everywhere on $(0,\infty)$. 
 Intuitively, $A(r)$ is the area ratio of geodesic sphere.  By absolute continuity of $|B(x_0,r)|$ and $r^{-2n}|B(x_0,r)|$, (\ref{eqn:GD08_4}) and (\ref{eqn:GD08_5}) are nothing but the Newton-Leibniz formula.

 We now show the monotonicity of $A$. Actually, suppose $A(r_a)$ and $A(r_b)$ are well defined.  Then we have
\begin{align*}
  &A(r_a)=\lim_{\epsilon \to 0^{+}}   \frac{|B(x_0,r_a+\epsilon)|-|B(x_0,r_a)|}{r_a^{2n-1}\epsilon}=\lim_{\epsilon \to 0^{+}}   \frac{2n \left\{|B(x_0,r_a+\epsilon)|-|B(x_0,r_a)| \right\}}{(r_a+\epsilon)^{2n}-r_a^{2n}},\\
  &A(r_b)=\lim_{\epsilon \to 0^{+}}   \frac{|B(x_0,r_b+\epsilon)|-|B(x_0,r_b)|}{r_b^{2n-1}\epsilon}=\lim_{\epsilon \to 0^{+}}   \frac{2n \left\{|B(x_0,r_b+\epsilon)|-|B(x_0,r_b)| \right\}}{(r_b+\epsilon)^{2n}-r_b^{2n}}. 
\end{align*}
Following from (\ref{eqn:GD08_3}) and the above identities, we obtain $A(r_a) \geq A(r_b)$ by taking limits.  
\end{proof}

\begin{proposition}[\textbf{Segment inequality}]
  For every nonnegative function $f \in L_{loc}^1(X)$,  define
 \begin{align*}
   \mathcal{F}_f(x_1,x_2) \triangleq \inf_{\gamma} \int_0^{l} f(\gamma(s))ds,
 \end{align*}
 where the infimum is taken over all minimal geodesics $\gamma$, from $x_1$ to $x_2$ and $s$ denotes the arc length.
 Suppose $p \in X$,  $r>0$, $A_1$, $A_2$ are two subsets of $B(p,r)$.
 Then we have
 \begin{align}
  \int_{A_1 \times A_2} \mathcal{F}_{f}(x_1,x_2) \leq 4^n r (|A_1|+|A_2|) \int_{B(p, 3r)} f.
 \label{eqn:HD17_3}
 \end{align}
\label{prn:HD17_1}
\end{proposition}

\begin{proof}
 Fix a smooth point $x_1$, then away from cut locus, every point can be connected to $x_1$ by a unique geodesic.
 Since $X \times X$ is equipped with the product measure, it is clear that away from a measure-zero set, every point
 $(x_1, x_2) \in X \times X$ has the property that $x_1$ and $x_2$ are smooth and  can be joined by
 a unique smooth shortest geodesic.
 Then the proof of (\ref{eqn:HD17_3})  is reduced to the same status as the Riemannian manifold case.
 The interested readers can find the details in the work of  Cheeger and Colding in~\cite{CCWarp}.
 \end{proof}

Due to the work of Cheeger and Colding (c.f. Remark 2.82 of~\cite{CCWarp}),
the  segment inequality implies the $(1,2)$-Poincar\'{e} inequality in general.
In our particular case, the Poincar\'{e} constant can be understood more precisely.

\begin{proposition}[\textbf{Bound of Poincar\'{e} constant}]
Suppose $f \in L_{loc}^1(X)$, $h$ is an upper gradient of $f$ in the sense of Cheeger(c.f.~Definition~\ref{dfn:HD18_1}).
Then for every geodesic ball $B(p,r) \subset X$ and real number $q \geq 1$, we have
\begin{align}
    \fint_{B(p,r)} |f-\underline{f}| \leq  2 \cdot 6^{2n} \cdot r \left(\fint_{B(p,3r)} h^q \right)^{\frac{1}{q}},
 \label{eqn:HD18_1}
 \end{align}
where $\fint$ means the average, $\underline{f}$ is the average of $f$ on $B(p,r)$.
In particular, there is a uniform  $(1,2)$-Poincar\`e constant on $X$.
\label{prn:HC29_2}
\end{proposition}

\begin{proof}
 This is standard. For example, one can check~\cite{CCWarp} and references therein for the details. 
\end{proof}

\begin{proposition}[\textbf{Bound of Sobolev constant}]
There is a uniform isoperimetric constant on $X$.
Consequently, a uniform $L^2$-Sobolev inequality hold on $X$.
\label{prn:HC29_3}
\end{proposition}

\begin{proof}
 Due to the uniform non-collapsing condition and the weak convexity and Ricci-flatness of $\mathcal{R}$, the argument of Croke (c.f.~\cite{Croke}) applies.
 So there is a uniform isoperimetric constant on $X$.
 The $L^2$-Sobolev constant follows from the isoperimetric constant(c.f.~\cite{SchYau}).
\end{proof}

Note that for each $X \in \widetilde{\mathscr{KS}}(n,\kappa)$, we lose smooth structure around $\mathcal{S}$.
In orbifold case, one can recover the smooth structure at a local ``covering" space.  For our $X$, it is not known whether one has such a property.    However,  the good news is that the smooth structure does not play an essential role in many aspects.  In the next subsection, we shall see that  the  analysis on $X$ is almost the same as that on manifold.

\subsection{Sobolev space, Dirichlet form and heat semigroup}

On a metric measure space, one can define
Sobolev space $H_{1,2}(X)$ following Cheeger(\cite{Cheeger99}), or $N^{1,2}(X)$
 following Shanmugalingam(\cite{ShanNew}).
However, these two definitions coincide whenever volume doubling property and uniform $(1,2)$-Poincar\'{e} inequality holds,
in light of  Theorem 4.10 of~\cite{ShanNew}, or the discussion on page 440 of~\cite{Cheeger99}.
In particular, for the space $(X,g,d\mu)$ which we are interested in, we have $N^{1,2}(X)=H_{1,2}(X)$ as Banach spaces.
For simplicity, we shall only use the notation $N^{1,2}(X)$ and follow the route of Cheeger.

\begin{definition}
Suppose $\Omega \subset X$. Let $f: \Omega \to [0,\infty]$ be an extended function.
 An extended real  function $h: \Omega \to [0,\infty]$ is called an upper gradient of $f$ on $\Omega$
if for every two points
$z_1,z_2 \in \Omega$ and all continuous rectifiable curves $c: [0,l] \to \Omega$, parameterized by arc length $s$, with $z_1,z_2$ end points, we have
\begin{align*}
   |f(z_1)-f(z_2)| \leq \int_0^l h(c(s))ds.
\end{align*}
\label{dfn:HD18_1}
\end{definition}

\begin{definition}
 The Sobolev space $N^{1,2}(X)$ is the subspace of $L^2(X)$ consisting of functions $f$ for which the norm
 \begin{align}
   \norm{f}{N^{1,2}}^2=\norm{f}{L^2}^2 + \inf_{f_i} \liminf_{i \to \infty}\norm{h_i}{L^2}^2 < \infty,
 \label{eqn:HD13_1}
\end{align}
where the limit infimum is taken over all upper gradients $h_i$ of the functions $f_i$, which satisfies $\norm{f_i -f}{L^2(X)} \to 0$.
\end{definition}

Note that the above $N^{1,2}$-norm  is equivalent to Cheeger's definition(c.f. equation (2.1) of ~\cite{Cheeger99}).
With this norm, we know $N^{1,2}(X)$
is complete(c.f. Theorem 2.7 of~\cite{Cheeger99}).
Clearly, it follows directly from the definition that zero function $f \in L^2(X)$ is the zero function in $N^{1,2}(X)$.
It is not surprising that $N^{1,2}(X)$ is the classical Sobolev space whenever $X$ is a smooth manifold.
This can be easily proved following the same argument of Theorem 4.5 of~\cite{ShanNew}, where the same conclusion was proved whenever
$X$ is a domain of Euclidean space.  In particular, as Banach spaces, we have
\begin{align}
   N^{1,2}(\mathcal{R}) \cong W^{1,2}(\mathcal{R}),
\end{align}
where $W^{1,2}(\mathcal{R})$ is the classical Sobolev space on the smooth manifold $\mathcal{R}$.

\begin{proposition}[\textbf{Smooth approximation}]
Suppose $\Omega$ is an open set of $X$,  $f \in N^{1,2}(\Omega)$.
Then there is a sequence of $f_i \in C^{\infty}(\Omega \backslash \mathcal{S}) \cap N^{1,2}(\Omega)$,
$\supp f_i \subset \Omega \backslash \mathcal{S}$ such that
 \begin{align}
     \lim_{i \to \infty}  \norm{f_i-f}{N^{1,2}(\Omega)}=0.
  \label{eqn:HD15_1}
  \end{align}
Moreover, if $f$ is also nonnegative, we can choose the approximation $f_i$ nonnegative.
If $\Omega$ is bounded,  then $\supp f_i$ is  a compact subset of $\overline{\Omega} \backslash \mathcal{S}$.
\label{prn:HD04_1}
\end{proposition}

\begin{proof}
It suffices to show the proof for the case when  both $\diam(\Omega)$ and $\norm{f}{L^{\infty}}$ are bounded.
For otherwise, we can apply the bounded result for the truncated function $\min\{k, \max\{-k, f\}\}$ on $\Omega \cap B(x_0,k)$ for each $k$ and then use a
standard diagonal sequence argument, to reduce to the bounded case.

Since $\mathcal{S}$ has measure zero, $\Omega \backslash \mathcal{S}$ is a smooth manifold, we have
  \begin{align*}
    \norm{f_i-f}{N^{1,2}(\Omega)}=\norm{f_i-f}{N^{1,2}(\Omega \backslash \mathcal{S})}
    =\norm{f_i-f}{W^{1,2}(\Omega \backslash \mathcal{S})}.
  \end{align*}
  Therefore, (\ref{eqn:HD15_1})  is equivalent to
  \begin{align}
     \lim_{i \to \infty}  \norm{f_i-f}{W^{1,2}(\Omega \backslash \mathcal{S})}=0.
  \label{eqn:HD15_2}
  \end{align}
  However, this sequence of $f_i$ can be constructed following a standard method,
  as indicated by the proof of Theorem 2 of section 5.3.2 of Evans' book~\cite{Evans}.
  For the convenience of the readers, we include a detailed construction of $f_i$ here.

  For each positive integer $i$, define
   \begin{align*}
      \Omega_i \triangleq \{y \in \Omega| d(y,\mathcal{S})>2^{-i}\}, \quad
      V_i \triangleq \Omega_{i+3} \backslash \overline{\Omega}_{i+1}, \quad
      W_i \triangleq \Omega_{i+4} \backslash \overline{\Omega}_{i}.
   \end{align*}
   Also, choose open sets $V_0$ and $W_0$ such that
   \begin{align*}
   \overline{\Omega}_4 \cap \Omega \supset  V_0  \supset \overline{\Omega}_2 \cap \Omega,   \quad W_0 \supset \overline{\Omega}_6 \cap \Omega \supset V_0.
   \end{align*}
   Then we have
   \begin{align*}
    \Omega \backslash \mathcal{S}= \bigcup_{i=0}^{\infty} V_i=\bigcup_{i=0}^{\infty} W_i,  \quad \overline{V}_i \cap \Omega_i \subset W_i, \quad \forall \; i \geq 0.
   \end{align*}
   Clearly, by composing with $d(\cdot, \mathcal{S})$, we can choose Lipschitz cutoff functions $\zeta_i$ such that $\zeta_i=1$ on $V_i$
   and $\supp \zeta_i \subset W_i$,   $|\nabla \zeta_i|<2^{i+5}$.
   Set $ \eta_i \triangleq \frac{\zeta_i}{\sum_j \zeta_j}$. 
   Clearly, $\eta_i$ is a kind of partition of unity subordinate to the covering $\bigcup_{i} W_i$. In other words, we have
   \begin{align*}
      \begin{cases}
       0\leq \eta_i \leq 1, & \eta_i \in C_c^{1}(W_i),  \quad \forall \; i \geq 1, \\
       \sum_{i} \eta_i =1, & \textrm{on} \; \Omega \backslash \mathcal{S}.
      \end{cases}
   \end{align*}
   Note that $\eta_0$ is special. It is only in $C^{1}(W_0)$ in general.   However, it vanishes around $\partial W_0 \cap \Omega$.
   For each $i \geq 0$, note that
   $V_i \cap V_j =\emptyset$ if $|i-j| \geq 2$,  $W_i \cap W_j= \emptyset$ if $|i-j| \geq 4$.  Therefore, we have
   \begin{align*}
     0\leq \eta_i<1,  \quad  |\nabla \eta_i|<2^{i+10}.
   \end{align*}
   For each $i \geq 1$, we see that $\eta_i f \in W_0^{1,2}(W_i)$. Note that $W_i \subset \mathcal{R}$.
   Applying convolution with smooth mollifiers(c.f. Theorem 1 of section 5.3.1 of~\cite{Evans}),  we can choose a smooth function
   $h_i \in C_c^{\infty}(W_i)$ such that
   \begin{align*}
        \norm{h_i-\eta_i f}{W^{1,2}(\Omega \backslash \mathcal{S})}^2=\norm{h_i-\eta_i f}{W^{1,2}(W_i)}^2 < 9^{-i-1}\epsilon^2.
   \end{align*}
   For $i=0$, we can choose $h_0 \in C^{\infty}(W_0)$ which vanishes around $\partial W_0 \cap \Omega$ such that the above inequality hold.
   For each large $k$, we define $ H_k \triangleq \sum_{i=0}^k h_i$.  Then $H_k \in C^{\infty}(\cup_{i=0}^k W_i)
    \subset C^{\infty}(\Omega \backslash \mathcal{S})$.
    Moreover, we have estimate
    \begin{align}
       \norm{H_k-f}{W^{1,2}(\Omega \backslash \mathcal{S})}
       &=\norm{\sum_{i=0}^k h_i -\sum_{i=1}^{\infty}\eta_i f}{W^{1,2}(\Omega  \backslash \mathcal{S})}
        =\norm{\sum_{i=0}^k (h_i-\eta_i f) -\sum_{i=k+1}^{\infty}\eta_i f}{W^{1,2}(\Omega  \backslash \mathcal{S})} \notag\\
         &\leq  \sum_{i=0}^k \norm{h_i-\eta_i f}{W^{1,2}(\Omega  \backslash \mathcal{S})}
         + \norm{\sum_{i=k+1}^{\infty}\eta_i f}{W^{1,2}(\Omega  \backslash \mathcal{S})}.
    \label{eqn:HD03_2}
    \end{align}
  However, the first term on the right hand side of the above inequality can be bounded as follows.
  \begin{align}
    \sum_{i=0}^k \norm{h_i-\eta_i f}{W^{1,2}(\Omega \backslash \mathcal{S})}
    <\sum_{i=0}^{k} 3^{-i-1} \epsilon<\frac{1}{2}\epsilon.
  \label{eqn:HD03_3}
  \end{align}
  On the other hand, note that $\sum_{i=k+1}^{\infty} \eta_i =1$ on $\bigcup_{i=k+5}^{\infty}W_i$, and it is supported on
  $\bigcup_{i=k+1}^{\infty}W_i$. Thus, we have
  \begin{align}
    \norm{\sum_{i=k+1}^{\infty}\eta_i f}{W^{1,2}(\Omega  \backslash \mathcal{S})}^2
    &\leq \norm{\sum_{i=k+1}^{\infty}\eta_i f}{W^{1,2}(\bigcup_{i=k+5}^{\infty}W_i)}^2
    +\norm{\sum_{i=k+1}^{\infty}\eta_i f}{W^{1,2}(\bigcup_{i=k+1}^{k+4}W_i)}^2 \notag\\
    &=\norm{f}{W^{1,2}(\bigcup_{i=k+5}^{\infty}W_i)}^2
    +\norm{\sum_{i=k+1}^{k+8}\eta_i f}{W^{1,2}(\bigcup_{i=k+1}^{k+4}W_i)}^2.
  \label{eqn:HD03_1}
  \end{align}
  For simplicity of notation, define $\chi_k \triangleq \sum_{i=k+1}^{k+8} \eta_i$. Clearly,  $0\leq \chi_k \leq 1$.
  We have
  \begin{align*}
    \norm{\sum_{i=k+1}^{k+8}\eta_i f}{W^{1,2}(\bigcup_{i=k+1}^{k+4}W_i)}^2
    &=\norm{\chi_k f}{W^{1,2}(\bigcup_{i=k+1}^{k+4}W_i)}^2
      =\int_{\bigcup_{i=k+1}^{k+4}W_i} \chi_k^2 f^2 + \left| \left \langle \chi_k \nabla f + f \nabla \chi_k \right \rangle \right|^2\\
    &\leq \int_{\bigcup_{i=k+1}^{k+4}W_i}  f^2 + 2\chi_k^2  |\nabla f|^2 + 2f^2|\nabla \chi_k|^2\\
    &\leq \left( 2 \int_{\bigcup_{i=k+1}^{k+4}W_i}  f^2+|\nabla f|^2 \right)
      +2\norm{f}{L^{\infty}(\Omega)}^2 \int_{\bigcup_{i=k+1}^{k+4}W_i} |\nabla \chi_k|^2.
  \end{align*}
  It is easy to see that $|\nabla \chi_k|<2^{k+20}$ by estimate of $\eta_k$.
  By virtue of Minkowski codimension assumption, we obtain
  \begin{align*}
     \left|\bigcup_{i=k+1}^{k+4}W_i \right|< \left|\bigcup_{i=k+1}^{\infty}W_i \right|< C2^{-3k}
       <C \left(2^{-k+5} \right)^3,
  \end{align*}
  which in turn implies that
  \begin{align*}
     \norm{\sum_{i=k+1}^{k+8}\eta_i f}{W^{1,2}(\bigcup_{i=k+1}^{k+4}W_i)}^2
     \leq 2\norm{f}{W^{1,2}(\bigcup_{i=k+1}^{k+4}W_i)}^2
     + C\norm{f}{L^{\infty}(\Omega)}^2 2^{-k}.
  \end{align*}
  Plug the above inequality into (\ref{eqn:HD03_1}), we obtain
  \begin{align*}
      \norm{\sum_{i=k+1}^{\infty}\eta_i f}{W^{1,2}(\Omega)}^2
      \leq 2 \norm{f}{W^{1,2}(\bigcup_{i=k+1}^{\infty}W_i)}^2 + C\norm{f}{L^{\infty}(\Omega)}^2 2^{-k}.
  \end{align*}
  Together with (\ref{eqn:HD03_2}) and (\ref{eqn:HD03_3}), the above inequality implies that
  \begin{align*}
       \norm{H_k-f}{W^{1,2}(\Omega)} <\frac{1}{2}\epsilon
       +2 \norm{f}{W^{1,2}(\bigcup_{i=k+1}^{\infty}W_i)}^2 + C\norm{f}{L^{\infty}(\Omega)}^2 2^{-k}.
    \end{align*}
   Recall that $f \in W^{1,2}(\Omega  \backslash \mathcal{S})$, $|\bigcup_{i=k+1}^{\infty}W_i| \to 0$ as $k \to \infty$.  So we can choose $k$ large
   enough such that
   \begin{align*}
     \norm{H_k-f}{W^{1,2}(\Omega  \backslash \mathcal{S})} <\epsilon.
   \end{align*}
   Let $\epsilon=\frac{1}{i}$,  we denote the corresponding $H_k$ in the above inequality by $f_i$.
   Clearly, $f_i$ is supported on $\Omega \backslash \mathcal{S}$ and is smooth.
   Moreover, (\ref{eqn:HD15_2}), consequently (\ref{eqn:HD15_1}),  follows from the above inequality.

   It follows from the construction that $f_i \geq 0$ whenever $ f \geq 0$.  Also, from the construction,
    if $\Omega$ is bounded, $\supp f_i$ is a compact subset of $\overline{\Omega} \backslash \mathcal{S}$.
\end{proof}

\begin{corollary}[\textbf{Smooth functions with compact supports}]
   $C_c^{\infty}(\mathcal{R}) \cap N^{1,2}(X)$ is dense in $N^{1,2}(X)$.
\label{cly:HE10_1}
\end{corollary}
\begin{proof}
 Fix $f \in N^{1,2}(X)$, without loss of generality, we may assume that $f \in C^{\infty}(\mathcal{R})$ and $f$ vanishes around $\mathcal{S}$,
 by Proposition~\ref{prn:HD04_1}.
 Fix $x_0 \in \mathcal{R}$ and let $r(x)=d(x,x_0)$.
 For each large $k$, let $\phi_k=\phi(r(x)-k)$, where $\phi$ is a smooth cutoff function on real axis such that $\phi \equiv 1$ on $(-\infty, 0)$ and
 $\phi \equiv 0$ on $(1,\infty)$.   Moreover, $|\phi'| \leq 2$.
 Note that $\supp f \cap \overline{B(x_0, k+1)}$ is a compact subset of $B(x_0,k+2) \backslash \mathcal{S}$.  By convolution with mollifier if necessary,
 we can assume $\phi_k$ is smooth and on $\supp f \cap \supp \phi_k$, $\supp \phi_k \subset B(x_0,k+2)$, $\phi_k \equiv 1$ on $B(x_0,k-1)$.
 Moreover,  $|\nabla \phi_k|<4$ and $0 \leq \phi_k<2$.
 Therefore, $\phi_k f \in C_c^{\infty}(\mathcal{R})$.  It is easy to calculate
 \begin{align*}
   \norm{f-\phi_k f}{N^{1,2}(X)}^2&=\int_X (1-\phi_k)^2f^2 d\mu + \int_X \left| \nabla \{(1-\phi_k) f\} \right|^2 d\mu\\
      &\leq \int_{X \backslash B(x_0,k-1)} (1-\phi_k)^2 f^2 du + 2\int_{X \backslash B(x_0, k-1)} \left\{(1-\phi_k)^2 |\nabla f|^2 + f^2|\nabla \phi_k|^2 \right\} d\mu\\
      &\leq \int_{X \backslash B(x_0,k-1)}  f^2 du + 2\int_{X \backslash B(x_0, k-1)} \left\{|\nabla f|^2 +16 f^2 \right\} d\mu\\
      &\leq 33 \int_{X \backslash B(x_0, k-1)} \left\{|\nabla f|^2 + f^2 \right\} d\mu.
 \end{align*}
 Clearly, the right hand side of the above inequality goes to $0$ as $k \to \infty$, since $f \in N^{1,2}(X)$.  Therefore, every $f \in  N^{1,2}(X)$ can
 be approximated by smooth functions with compact supports.
\end{proof}

In light of Proposition~\ref{prn:HD04_1}, we can define $N_0^{1,2}(\Omega)$ as the completion of all the functions in
$C_c^{\infty}(\Omega \backslash \mathcal{S}) \cap N^{1,2}(X)$, under the $N^{1,2}(\Omega)$-norm.
Note that a function $f$ in $N_0^{1,2}(\Omega)$ may not have compact support, with respect to $\Omega$.
However, $f|_{\partial \Omega}=0$, in the sense of traces.

\begin{proposition}[\textbf{Global continuous approximation}]
For each $f \in C_c(X)$, i.e., a continuous function with compact support, there exists a sequence of $f_i \in C_c(X) \cap N^{1,2}(X)$ such that
$\displaystyle \lim_{i \to \infty} \norm{f_i-f}{C(X)} \to 0$.
\label{prn:HD24_1}
\end{proposition}

\begin{proof}
For each $\epsilon>0$, $x \in X$, define $\phi_{\epsilon,x}$ to be the character equation of the geodesic ball
 $B(x,\epsilon)$.
In other words, $\phi_{\epsilon, x} \equiv 1$ on $B(x,\epsilon)$ and $0$ on $X \backslash B(x,\epsilon)$.  Define
$\psi_{\epsilon,x}$ to be $\frac{\phi_{\epsilon,x}}{|B(x,\epsilon)|}$. Clearly, we have
 \begin{align}
    \int_X \psi_{\epsilon,x}(y)d\mu_y=1.
 \label{eqn:HD26_1}
 \end{align}
 Similar to Euclidean case, we define approximation functions as convolution of $f$ and $\psi_{\epsilon,\cdot}$ as follows:
 \begin{align*}
    f_{\epsilon}(x) \triangleq (\psi_{\epsilon}* f)(x)= \int_X f(y) \psi_{\epsilon, x}(y) d\mu_y.
 \end{align*}
 Fix $\epsilon>0$. Suppose $x_1,x_2$ are two points in $X$ with distance $\rho \in (0,\epsilon)$. Then we calculate
 \begin{align*}
    |f_{\epsilon}(x_1)-f_{\epsilon}(x_2)| &\leq \int_X |f|(y) \left| \psi_{\epsilon,x_1}(y)-\psi_{\epsilon,x_2}(y)\right| d\mu_y\\
       &\leq  \norm{f}{C(X)} \int_X  \left| \frac{\phi_{\epsilon,x_1}}{|B(x_1,\epsilon)|} - \frac{\phi_{\epsilon,x_2}}{|B(x_2,\epsilon)|}\right| d\mu_y\\
       &=\frac{\norm{f}{C(X)} }{|B(x_1,\epsilon)||B(x_2,\epsilon)|} \int_X \left\vert\phi_{\epsilon,x_1} |B(x_2,\epsilon)|
       -\phi_{\epsilon,x_2} |B(x_1,\epsilon)| \right\vert d\mu_y\\
   &\leq C(n,\kappa) \norm{f}{C(X)} \epsilon^{-4n}  
   \int_X \left\vert\phi_{\epsilon,x_1} |B(x_2,\epsilon)|
       -\phi_{\epsilon,x_2} |B(x_1,\epsilon)| \right\vert d\mu_y. 
 \end{align*}
 Notice that
 \begin{align*}
  &\quad \int_X \left\vert\phi_{\epsilon,x_1} |B(x_2,\epsilon)|
       -\phi_{\epsilon,x_2} |B(x_1,\epsilon)| \right\vert d\mu_y\\
  &=\int_X \left \vert \phi_{\epsilon,x_1}\left\{  |B(x_2,\epsilon)|-|B(x_1,\epsilon)| \right\}
      +|B(x_1,\epsilon)| \cdot (\phi_{\epsilon,x_1}-\phi_{\epsilon,x_2}) \right\vert d\mu_y\\
  &\leq  \int_X\phi_{\epsilon,x_1} \left \vert   |B(x_2,\epsilon)|-|B(x_1,\epsilon)| \right\vert d\mu_y
      +|B(x_1,\epsilon)| \int_X |\phi_{\epsilon,x_1} -\phi_{\epsilon,x_2}| d\mu_y \\
  &=|B(x_1,\epsilon)| \left\{  \left \vert   |B(x_2,\epsilon)|-|B(x_1,\epsilon)|  \right\vert
    +  \int_X |\phi_{\epsilon,x_1} -\phi_{\epsilon,x_2}| \right\}\\
  &=|B(x_1,\epsilon)| \left\{  \left \vert   |B(x_2,\epsilon)|-|B(x_1,\epsilon)|  \right\vert
    +|B(x_1,\epsilon) \backslash B(x_2,\epsilon)|+|B(x_2,\epsilon) \backslash B(x_1,\epsilon)|  \right\}\\
  &\leq 2  |B(x_1,\epsilon)| \left\{  
    |B(x_1,\epsilon) \backslash B(x_2,\epsilon)|+|B(x_2,\epsilon) \backslash B(x_1,\epsilon)|  \right\}.     
 \end{align*}
By Bishop-Gromov volume comparison and non-collapsing condition, we have
 \begin{align*}
   &|B(x_2,\epsilon) \backslash B(x_1,\epsilon)| 
   \leq  \left| B(x_1,\epsilon+\rho) \backslash B(x_1,\epsilon-\rho)\right| \leq  C(n,\kappa) \epsilon^{2n-1} \rho, \\
   &|B(x_1,\epsilon) \backslash B(x_2,\epsilon)| 
   \leq  \left| B(x_2,\epsilon+\rho) \backslash B(x_2,\epsilon-\rho)\right| \leq  C(n,\kappa) \epsilon^{2n-1} \rho. 
 \end{align*}
 Thus, for each $\rho \in (0,\epsilon)$, we have estimate
 \begin{align*}
    |f_{\epsilon}(x_1)-f_{\epsilon}(x_2)| 
     \leq \frac{C(n,\kappa) \norm{f}{C(X)}}{\epsilon} \rho,
 \end{align*}
 which means that the Lipschitz constant of $f_{\epsilon}$ is uniformly bounded, for each fixed $\epsilon$.
 In particular, $f_{\epsilon}$ locates in $C_c(X) \cap N^{1,2}(X)$.
 It follows from (\ref{eqn:HD26_1}) that
 \begin{align*}
     \left| f_{\epsilon}(x) -f(x) \right|&=\left| \int_X \{f(y)-f(x)\} \psi_{\epsilon, x}(y) d\mu_y \right| 
    \leq \int_{B(x,\epsilon)} |f(y)-f(x)| \psi_{\epsilon,x}(y) d\mu_y
    \leq \sup_{y \in B(x,\epsilon)} |f(y)-f(x)|. 
 \end{align*}
 Note that $f$ is uniformly continuous since $f$ is continuous and $\supp(f)$ is contained in a compact subset of $X$.
 Hence the right hand side of the above inequality converges to zero uniformly as $\epsilon \to 0$.
 Therefore, $\psi_{2^{-i}} *f$ is a sequence of functions in $C_c(X) \cap N^{1,2}(X)$ and converges to $f$ in $C(X)$-norm.
\end{proof}

For each open set $\Omega \subset X$,
there is a restriction map $\pi: N^{1,2}(\Omega) \to N^{1,2}(\Omega \backslash \mathcal{S})$ in the obvious way.
Note that $\Omega \backslash \mathcal{S}=\Omega \cap \mathcal{R}$ is a smooth manifold,
hence $N^{1,2}(\Omega \backslash \mathcal{S})=W^{1,2}(\Omega \backslash \mathcal{S})$.
In general, the map $\pi$ is not surjective. However, in our special setting, $\mathcal{S}$ has high codimension,
we have much more information.

\begin{proposition}(\textbf{Identity is isometry})
Suppose $\Omega$ is an open set of $X$, then
the restriction map $\pi: N^{1,2}(\Omega) \to N^{1,2}(\Omega \backslash \mathcal{S})=W^{1,2}(\Omega \backslash \mathcal{S})$ is an isomorphic  isometry.
\label{prn:HD13_1}
\end{proposition}

\begin{proof}
 If we proved $\pi$ is an isomorphism, it is clear that $\pi$ is an isometry since $\mathcal{S}$ has measure zero.
 Thus, we only need to focus on the proof of isomorphism. For simplicity, we assume $\Omega=X$.
 Then $\Omega \backslash \mathcal{S}=X \backslash \mathcal{S}=\mathcal{R}$.

 \textit{Injectivity}: Suppose $\pi(f)=0$.  Then $\norm{f}{L^2(X)}=0$ since $\mathcal{S}$ has measure zero.
 Due to the fact $f \in N^{1,2}(X)$, $\norm{f}{L^2(X)}=0$ implies that $\norm{f}{N^{1,2}(X)}=0$.
 Therefore, $f$ is the zero element in $N^{1,2}(X)$.

 \textit{Surjectivity}:  For every $0 \neq \tilde{f} \in W^{1,2}(\mathcal{R})$,  from the proof of Proposition~\ref{prn:HD04_1},
 there is a sequence of smooth functions $\tilde{f}_i$ supported on $\mathcal{R}$ such that
 $\displaystyle \norm{\tilde{f}_i -\tilde{f}}{W^{1,2}(\mathcal{R})} \to 0$.
 In particular, $\tilde{f}_i$ is a Cauchy sequence in $W^{1,2}(\mathcal{R})$.
 Since $\displaystyle \norm{\tilde{f}_i -\tilde{f}}{N^{1,2}(X)}=\norm{\tilde{f}_i -\tilde{f}}{W^{1,2}(\mathcal{R})}$,
 it is clear that $\tilde{f}_i$ is a Cauchy sequence in $N^{1,2}(X)$.
 Therefore, there is a function $f \in N^{1,2}(X)$, as the limit of $\tilde{f}_i$, by completeness of $N^{1,2}(X)$.
 After we obtain $f$, it is clear that $\norm{f_i-f}{N^{1,2}(X)} \to 0$, which forces that
 \begin{align*}
   \norm{f_i-\pi(f)}{W^{1,2}(\mathcal{R})} \to 0.
 \end{align*}
 Therefore, $\pi(f)=\tilde{f}$.
\end{proof}

In light of Proposition~\ref{prn:HD13_1}, we can regard $N^{1,2}(X)$ as the same Banach space as $W^{1,2}(\mathcal{R})$.
However, $W^{1,2}(\mathcal{R})$ is a Hilbert space. This induces a natural inner product structure on $W^{1,2}(\mathcal{R})$ as follows:
\begin{align*}
 \langle\langle f_1, f_2\rangle\rangle
 = \int_{\mathcal{R}} \left\{ \pi(f_1) \pi(f_2) +\langle \nabla \pi(f_1), \nabla \pi(f_2)\rangle \right\}d\mu, \quad \forall \; f_1, f_2 \in N^{1,2}(X).
\end{align*}
For simplicity of notation, we shall not differentiate $f$ and $\pi(f)$. Under this convention, we have
\begin{align*}
 \langle\langle f_1, f_2\rangle\rangle
 = \int_{\mathcal{R}} \left\{ f_1f_2 +\langle \nabla f_1, \nabla f_2 \rangle \right\}d\mu, \quad \forall \; f_1, f_2 \in N^{1,2}(X).
\end{align*}
Therefore, $N^{1,2}(X)$ is isomorphic to $W^{1,2}(\mathcal{R})$ as a Hilbert space.  For every $f_1,f_2 \in N^{1,2}(X)$, we define
a nonnegative, symmetric, bilinear form $\mathscr{E}$ as follows
\begin{align}
    \mathscr{E}(f_1,f_2) \triangleq \int_{\mathcal{R}} \langle  \nabla f_1, \nabla f_2 \rangle d\mu.
\label{eqn:HD26_2}
\end{align}
We want to show that $\mathscr{E}$ is a Dirichlet form.
Actually, it is clear that $\norm{f}{N^{1,2}(X)}^2=\norm{f}{L^2(X)}^2+\mathscr{E}(f,f)$.
Since $N^{1,2}(X)$ is complete, we know that $\mathscr{E}$ is closed by definition.
On the other hand,  since $W^{1,2}(\mathcal{R})$ is dense
in $L^2(\mathcal{R})=L^2(X)$, $\mathcal{S}$ has measure zero,
it follows directly that $N^{1,2}(X)$ is dense in $L^2(X)$.
Furthermore, it is clear that
\begin{align}
    \mathscr{E}(\min\{1,\max\{0,f\}\}, \min\{1,\max\{0,f\}\}) \leq \mathscr{E}(f,f), \quad \forall \; f \in N^{1,2}(X).
\label{eqn:HE26_1}
\end{align}
Therefore,  $\mathscr{E}$ is a closed, nonnegative, symmetric, bilinear form on $N^{1,2}(X)$,
which is a dense subspace of $L^2(X)$, with unit contraction property (\ref{eqn:HE26_1}).
It follows from a standard definition (c.f.~\cite{FuOsTa} for definition of Dirichlet form) that $\mathscr{E}$ is a Dirichlet form.
Not surprisingly, this Dirichlet form $\mathscr{E}$ is much better than general Dirichlet form since
the underlying space $X$ has rich geometry.   In fact, suppose $u \in N_0^{1,2}(\Omega)$ for some open set $\Omega \subset X$,
it is clear that $u \equiv 0$ on $\Omega$ if and only if $ \mathscr{E}(u,u)=0$.
This means that $\mathscr{E}$ is irreducible by direct definition.
Also, for every constant $c$, we have $\mathscr{E}(u,v)=0$, whenever $v \equiv c$ in a neighborhood of the support set of $u$. 
This means that $\mathscr{E}$ is strongly local.
Furthermore,  it follows from Corollary~\ref{cly:HE10_1} that $N^{1,2}(X) \cap C_c(X)$ is dense in $N^{1,2}(X)$ with $N^{1,2}$-norm.
On the other hand, Proposition~\ref{prn:HD24_1} implies that $N^{1,2}(X) \cap C_c(X)$ is dense in $C_c(X)$ with uniform supreme norm.
Consequently, $N^{1,2}(X) \cap C_c(X)$ is a core of $\mathscr{E}$ and $\mathscr{E}$ is a regular Dirichlet form,  following from the definition verbatim.
Putting all the above information together, we obtain the following property.

\begin{proposition}[\textbf{Existence of excellent Dirichlet form}]
 On the Hilbert space $L^2(X)$, there exists a Dirichlet form $\mathscr{E}$ defined on a dense subspace
 $N^{1,2}(X) \subset L^2(X)$, by formula (\ref{eqn:HD26_2}).  Furthermore, the Dirichlet form $\mathscr{E}$
 is irreducible, strongly local and regular.
\label{prn:HD26_1}
\end{proposition}

With respect to the Dirichlet form $\mathscr{E}$, one can obtain much geometric and analytic information. A good reference is
the nice paper~\cite{KoZhou}, by P. Koskela and Y. Zhou.  We now focus on some elementary properties. 
Note that there is a unique generator (c.f. Chapter 1 of~\cite{FuOsTa}) of $\mathscr{E}$, which we denote by $\mathcal{L}$.
In other words, $\mathcal{L}$ is a self-adjoint and non-positive definite operator in $L^2(X)$ with domain $Dom(\mathcal{L})$
which is dense in $N^{1,2}(X)$ such that
\begin{align}
     \mathscr{E}(f,h)=-\int_X h \cdot \mathcal{L}f d\mu, \quad \forall \; f \in Dom(\mathcal{L}), \; h\in N^{1,2}(X).
\label{eqn:HD27_1}
\end{align}
Note that $C_c^{\infty}(\mathcal{R})$ is a dense subset of $Dom(\mathcal{L})$.  Suppose $f \in C_c^{\infty}(\mathcal{R})$, $h \in N^{1,2}(X)=N^{1,2}(\mathcal{R})$,
it is clear that
\begin{align*}
   \mathscr{E}(f,h)=\int_{\mathcal{R}} \langle \nabla f, \nabla h\rangle d\mu= -\int_X h \cdot \Delta f d\mu.
\end{align*}
Therefore, $\mathcal{L}$ is nothing but the extension of the classical Laplacian operator, with domain as the largest dense subset of $N^{1,2}(X)$ such that the integration by parts, i.e., equation (\ref{eqn:HD27_1}), holds.
For this reason, we shall just denote $\mathcal{L}$ by $\Delta$ in the future.

Based on the generator operator $\Delta$, there is an associated heat semigroup $\left(P_t\right)_{t \geq 0}= \left( e^{t\Delta}\right)_{t \geq 0}$, which acts on
$L^2(X)$ with the following properties(c.f. Chapter 1 of~\cite{FuOsTa}).
\begin{itemize}
\item Semi-group: $P_0=Id$; $P_t \circ P_s=P_{t+s}$, for every $t, s  \geq 0$.
\item Generator: $\displaystyle \lim_{t \to 0^+} \norm{\frac{1}{t}(P_t f-f) -\Delta f}{L^2(X)}=0$, for every $f \in L^2(X) \cap Dom(\Delta)$.
\item $L^2$-contractive: $\norm{P_t f}{L^2(X)}^2 \leq \norm{f}{L^2(X)}^2$, for every $f \in L^2(X)$, $t>0$.
\item Strong continuous: $\displaystyle \lim_{t \to 0^{+}}\norm{P_t f-f}{L^2(X)}=0$, for every $f \in L^2(X)$.
\item Markovian: $\norm{P_t f}{L^{\infty}(X)} \leq \norm{f}{L^{\infty}(X)}$, for every $f \in L^2(X)\cap L^{\infty}(X)$, $t>0$.
\item Heat solution: $\Delta P_t f= \frac{\partial}{\partial t} P_t f$, for every $f \in L^2(X)$ and $t>0$.
\end{itemize}
The above properties are well known in semigroup theory on Banach spaces(c.f. Section 7.4 of \cite{Evans}).  Actually, for every $f \in L^2(X)$,
one can also show that $P_t f$ is the unique square-integrable solution with initial value $f$(c.f. Proposition 1.2 of~\cite{Stu95} and references therein).
We call $(P_t)_{t \geq 0}$ as the heat semigroup as usual.
Associated with this heat semigroup, there exists a nonnegative kernel function, or fundamental solution, $p(t,x,y)$,
such that
\begin{align*}
  P_t(f)(y)=\int_X f(x) p(t,x,y)d\mu_x, \quad \forall \; f \in L^2(X), \; t>0.
\end{align*}
Moreover, $p$ satisfies the symmetry $p(t,x,y)=p(t,y,x)$.   Interested readers are referred to Proposition 2.3 and the discussion in Section 2.4(C) of~\cite{Stu95}
for more detailed information.
As usual, we call $p(t,x,y)$ as the heat kernel.

\begin{definition}
Suppose $u \in N_{loc}^{1,2}(\Omega)$. Define
\begin{align}
   \int_{\Omega} \varphi \Delta u \triangleq -\mathscr{E}(u,\varphi)
\label{eqn:HD29_2}
\end{align}
for every $\varphi \in N_c^{1,2}(\Omega)$, i.e., $\varphi \in N^{1,2}(\Omega)$ and has compact support set in $\Omega$.
Similarly, (\ref{eqn:HD29_2}) can be applied if $u \in N^{1,2}(\Omega)$ and $\varphi \in N_0^{1,2}(\Omega)$.
\label{dfn:HD29_1}
\end{definition}

Suppose $u \in N_{loc}^{1,2}(\Omega) \cap C^{2}(\Omega \backslash \mathcal{S})$, then $\Delta u|_{\Omega \backslash \mathcal{S}}$
is  a continuous function.
By taking value $\infty$ on $\mathcal{S}$, we can regard $\Delta u$ as an extended function on $\Omega$. Suppose $\Delta u \in L_{loc}^2(\Omega)$, then
for every smooth test function $\varphi_i \in N_c^{1,2}(\Omega)$, we have
\begin{align*}
    \int_{\Omega \backslash \mathcal{S}} \varphi_i \Delta u=-\int_{\Omega \backslash \mathcal{S}} \langle \nabla u, \nabla \varphi_i \rangle.
\end{align*}
Let $\varphi$ be the limit of $\varphi_i$ in $N^{1,2}(\Omega)$. Taking limit of the above equation shows that
\begin{align*}
  \int_{\Omega} \varphi \Delta u=-\int_{\Omega} \langle \nabla u, \nabla \varphi \rangle.
\end{align*}
Therefore, whenever $u \in N_{loc}^{1,2}(\Omega) \cap C^{2}(\Omega \backslash \mathcal{S})$ and the classical $\Delta u$ is in $L_{loc}^2(\Omega)$,
we see that the LHS and RHS of (\ref{eqn:HD29_2}) holds in the classical sense.  Similar argument applies 
if $u \in N^{1,2}(\Omega) \cap C^2(\Omega \backslash \mathcal{S})$, $\Delta u \in L^2(\Omega)$, $\varphi \in N_0^{1,2}(\Omega)$. 
Therefore, Definition~\ref{dfn:HD29_1} is justified. 

  Now we assume $u \in N_c^{1,2}(\Omega)$. 
  Then in the weak sense, for every $\varphi \in N_c^{1,2}(\Omega)$, we can define
  $\int_{\Omega} \varphi \Delta u$.  It is not hard to see that $\int_{\Omega} \varphi \Delta u$ makes sense even if $\varphi$ is in $N^{1,2}(\Omega)$ only. 
  In fact,  let $\chi$ be a cutoff function with value $1$ on $\Omega'$ and vanishes
  around $\partial \Omega$, where $\Omega'$ contains the support of $u$.
  By Definition~\ref{dfn:HD29_1}, we have
  \begin{align*}
     \int_{\Omega} (\chi \varphi) \Delta u= -\mathscr{E}(u, \chi \varphi)=-\int_{\Omega} \langle \nabla u, \nabla (\chi \varphi)\rangle
      =-\int_{\Omega'} \langle \nabla u, \nabla \varphi \rangle=-\int_{\Omega} \langle \nabla u, \nabla \varphi \rangle
      =-\mathscr{E}(u,\varphi).
  \end{align*}
  The above calculation does not depend on the particular choice of $\chi$.   Consequently,  we can define
  $\int_{\Omega} \varphi \Delta u$ as $-\mathscr{E}(u,\varphi)$.
  Summarizing the above discussion, we have the following property.

\begin{proposition}[\textbf{Integration by parts}]
  Suppose $\Omega$ is a domain in $X$, $f_1 \in N_c^{1,2}(\Omega)$, $f_2 \in N^{1,2}(\Omega)$.
  Then we have
  \begin{align}
       \int_{\Omega} f_2 \Delta f_1 d\mu = -\int_{\Omega} \langle \nabla f_1, \nabla f_2 \rangle d\mu=\int_{\Omega} f_1 \Delta f_2 d\mu.
  \label{eqn:HD13_2}
  \end{align}
  Furthermore, if $f_2 \in N^{1,2}(\Omega) \cap C^2(\Omega \backslash \mathcal{S})$ and $\Delta f_2 \in L^2(\Omega \backslash \mathcal{S})$
  as a classical function,     then we can understand the integral
  $\int_{\Omega} f_1 \Delta f_2 d\mu=\int_{\Omega \backslash \mathcal{S}} f_1 \Delta f_2 d\mu$ in the classical sense.
  Similarly, if $f_1 \in N_c^{1,2}(\Omega) \cap C^2(\Omega \backslash \mathcal{S})$,
  $\Delta f_1 \in L^2(\Omega \backslash \mathcal{S})$ as a classical function,  then
  $\int_{\Omega} f_2 \Delta f_1 d\mu=\int_{\Omega \backslash \mathcal{S}} f_2 \Delta f_1 d\mu$ can be understood in
  the classical sense.  If both $f_1$ and $f_2$ locate in $N_0^{1,2}(\Omega)$, then (\ref{eqn:HD13_2}) also holds. 
\label{prn:HD13_2}
\end{proposition}

\begin{definition}
  Suppose $u \in N_{loc}^{1,2}(\Omega)$, $f \in L_{loc}^2(\Omega)$, we say $\Delta u \geq f$ in the weak sense whenever
  \begin{align}
   \int_{\Omega}(-\Delta u +f)\varphi= \mathscr{E}(u,\varphi) + \int_{\Omega} f\varphi \leq 0
  \label{eqn:HD29_1}
  \end{align}
  for every nonnegative test function $\varphi \in N_c^{1,2}(\Omega)$.  We call $u$ subharmonic if $\Delta u \geq 0$ in the weak sense.
  We call $u$ superharmonic if $-u$ is subharmonic. We call $u$ harmonic if $u$ is both subharmonic and superharmonic.
\label{dfn:HD29_2}
\end{definition}

 Due to Proposition~\ref{prn:HD04_1}, for a $u \in  N_{loc}^{1,2}(\Omega)$, in order to check (\ref{eqn:HD29_1}) for all $\varphi \in N_c^{1,2}(\Omega)$,
  it suffices to check all smooth nonnegative test functions with supports in $\Omega \backslash \mathcal{S}$.
It is important to notice that the restriction of $\Delta$ on $\mathcal{R}$ is the classical Laplacian on Riemannian manifold.
In fact,  if a function $u$ is harmonic in the above sense, then  $u |_{\mathcal{R}}$ is harmonic function in the distribution sense.
By standard improving regularity theory of elliptic equations, we know our $u$ is smooth and $\Delta u=0$ in the classical sense.

 Similarly, one can follow the standard route to define heat solution (sub solution, super solution) for the heat operator
 $\square= \left( \frac{\partial}{\partial t} -\Delta \right)$ in the weak sense.  We leave these details to interested readers.
 It is quite clear that a weak heat solution is a smooth function when restricted on $\mathcal{R} \times (0,T]$, by
 standard improving regularity theory of heat equations(c.f. Chapter 7 of~\cite{Evans}).

\subsection{Harmonic functions and heat flow solutions on model space}

Suppose $K$ is a compact subset of $\Omega \backslash \mathcal{S}$,  it is clear that $K$ is also a compact subset of $\Omega$.
However, the reverse is not true.  If $K$ is a compact subset of $\Omega$, then $K\backslash \mathcal{S}$ may not be a compact
subset of $\Omega \backslash \mathcal{S}$.
For this reason, we see that  $N_{loc}^{1,2}(\Omega) \subset N_{loc}^{1,2}(\Omega \backslash \mathcal{S})$ and not equal if $\mathcal{S} \neq \emptyset$, even if
$\mathcal{S}$ has very high codimension.
However, if we restrict our attention only on bounded subharmonic functions, then the above difference will vanish.

\begin{proposition}[\textbf{Extension of bounded subharmonic functions}]
  Suppose $\Omega$ is a bounded open domain in $X$, $u$ is a bounded  subharmonic function on $\Omega \backslash \mathcal{S}$.
  Then $u \in N_{loc}^{1,2}(\Omega)$ and it is  subharmonic on $\Omega$.
\label{prn:HD16_2}
\end{proposition}

\begin{proof}
It suffices to prove that $u \in N_{loc}^{1,2}(\Omega)$.
Note that by definition, we only have $u \in N_{loc}^{1,2}(\Omega \backslash \mathcal{S})$, which is a superset of $N_{loc}^{1,2}(\Omega)$.
In fact, for each small $r>0$, one can construct a Lipschitz cutoff function
\begin{align}
   \chi(x)=\phi \left(\frac{d(x,\mathcal{S})}{r} \right),
\label{eqn:HD16_1}
\end{align}
where $\phi$ is a cutoff function on $[0,\infty)$ which is equivalent to $1$ on $[0,1]$, $0$ on $[2, \infty)$, and $|\phi'| \leq 2$.
By the assumption of Minkowski codimension of $\mathcal{S}$,  we have
\begin{align}
   | \Omega \cap \supp \chi| \leq Cr^3,  \quad \int_{\Omega} |\nabla \chi|^2 \leq C r^{-2}\int_{\Omega \cap \{\nabla \chi \neq 0\}} \chi^2 \leq Cr.
\label{eqn:HD16_2}
\end{align}
Fix a relatively compact subset $\Omega' \subset \Omega$, we can find a cutoff function $\eta$ which is identically $1$ on $\Omega'$ and vanishes around
$\partial \Omega$. Moreover, $|\nabla \eta| \leq C$, which depends on $\Omega'$ and $\Omega$.

By adding a constant if necessary, we can assume $u \geq 0$.
Note that $u$ is subharmonic on $\Omega \backslash \mathcal{S}$,  $u\eta^2 (1-\chi)^2$ can be chosen as a test function. It follows from definition that
\begin{align*}
   0&\leq  \int_{\Omega \backslash \mathcal{S}} (\Delta u)  u\eta^2(1-\chi)^2
       =-\int_{\Omega \backslash \mathcal{S}} \langle \nabla u, \nabla (u\eta^2(1-\chi)^2) \rangle \\
   & =-\int_{\Omega \backslash \mathcal{S}} |\nabla u|^2\eta^2 (1-\chi)^2
   + \int_{\Omega \backslash \mathcal{S}} u \langle \nabla u, -2(1-\chi)^2\eta \nabla \eta+ 2\eta^2(1-\chi) \nabla \chi \rangle.
\end{align*}
Note that $u,\eta,\nabla \eta$ are bounded. Then H\"{o}lder inequality applies.
\begin{align*}
   \frac{1}{2} \int_{\Omega \backslash \mathcal{S}} |\nabla u|^2\eta^2 (1-\chi)^2
   &\leq  C+ C\int_{\Omega \backslash \mathcal{S}}\eta^2 (1-\chi) |\nabla u||\nabla \chi|\\
   &\leq  C+ \frac{1}{4}\int_{\Omega \backslash \mathcal{S}} |\nabla u|^2\eta^2 (1-\chi)^2 + C\int_{\Omega \backslash \mathcal{S}}  \eta^2 |\nabla \chi|^2.
\end{align*}
Recall the definition of $\chi$ in (\ref{eqn:HD16_1}) and estimate (\ref{eqn:HD16_2}).
Let $r \to 0$,  the above inequality yields that
\begin{align*}
     \int_{\Omega \backslash \mathcal{S}} |\nabla u|^2\eta^2 \leq C,
\end{align*}
which forces that $\displaystyle \int_{\Omega'  \backslash \mathcal{S}} |\nabla u|^2 \leq C$.
Hence $u \in W^{1,2}(\Omega' \backslash \mathcal{S})$ since $u$ is bounded.  This is the same to say $u \in N^{1,2}(\Omega')$.
By the arbitrary choice of $\Omega'$, we have proved that $u \in N_{loc}^{1,2}(\Omega)$.
\end{proof}

We now move on to the discussion of heat kernels.

\begin{proposition}[\textbf{Heat Kernel estimates}]
The exists a unique heat kernel  $p(t,x,y)$ of $X$, with respect to the Dirichlet form $\mathscr{E}=\langle \nabla \cdot, \nabla \cdot \rangle$.
Moreover, $p(t,x,y)$ satisfies the following properties.
\begin{itemize}
\item Stochastically completeness. In other words, we have
\begin{align*}
   \int_X p(t,x,y) d\mu_x=1
\end{align*}
for every $x \in X$.
\item The Gaussian estimate holds.  In other words, there exists a constant $C$ depending only on $n,\kappa$ such that
\begin{align}
  \frac{1}{C}t^{-n} e^{-\frac{d^2(x,y)}{3t}} \leq p(t,x,y) \leq Ct^{-n} e^{-\frac{d^2(x,y)}{5t}}
\label{eqn:HC31_7}
\end{align}
for every $x,y \in X$ and $t>0$.
\item For each positive integer $j$, there is a constant $C=C(n,\kappa,j)$ such that
\begin{align}
    \left| \left( \frac{\partial}{\partial t} \right)^j p(t,x,y)\right| \leq Ct^{-n-j} e^{-\frac{d^2(x,y)}{5t}}
\label{eqn:HD08_1}
\end{align}
for every $x,y \in X$ and $t>0$.
\label{prn:HC29_7}
\end{itemize}
\label{prn:HD07_1}
\end{proposition}
\begin{proof}
 Since $X$ satisfies the doubling property and has a uniform $(1,2)$-Poincar\'{e} constant,  by Corollary~\ref{cly:HD20_1}
 and Proposition~\ref{prn:HC29_2},
 the existence of the heat Kernel follows from the work of Sturm (c.f.Proposition 2.3 of~\cite{Stu95}).
 The uniqueness of the heat kernel follows from the uniqueness of the heat semigroup.
 The stochastic completeness is guaranteed by the doubling property, see Theorem 4 and the following remarks of~\cite{Stu94}.

  The Gaussian estimate follows from Corollary 4.2 and Corollary 4.10 in Sturm's paper~\cite{Stu96}, where $C$
   depends on the volume doubling condition   and the $(1,2)$-Poincar\'{e} constant.

  The heat kernel derivative estimate, inequality (\ref{eqn:HD08_1}), follows from Corollary 2.7. of ~\cite{Stu95}, whose proof follows the same line
  as Theorem 6.3 of~\cite{Saloff}.
 \end{proof}

 In general, the estimates of heat kernel only hold almost everywhere with respect to the measure $d\mu$.
 However, in the current situation, when restricted on $\mathcal{R} \times (0,\infty)$, $p$ is clearly a smooth function.
 Therefore, (\ref{eqn:HC31_7}) and (\ref{eqn:HD08_1}) actually hold true everywhere away from $\mathcal{S}$.
 Note that  $\Delta p=\frac{\partial}{\partial t} p$ clearly locates in $L^2(X) \cap C^{\infty}(\mathcal{R})$.
 Hence integration by parts (Proposition~\ref{prn:HD13_2}) applies.  Then by standard radial cutoff function construction and direct calculation,
 Proposition~\ref{prn:HD07_1}
 yields the following estimates immediately.

 \begin{corollary}[\textbf{Off-diagonal integral estimates of heat kernel}]
    For every $r>0, t>0$, and $x_0 \in X$, we have
    \begin{align}
     &\int_{X \backslash B(x_0,2r)} p^2 d\mu_x<Ct^{-n}e^{-\frac{r^2}{5t}}, \label{eqn:HD08_7}\\
     &\int_{X \backslash B(x_0,2r)} |\nabla p(t,x,x_0)|^2 d\mu_x
       < C \left( \frac{1}{t} + \frac{1}{r^2}\right) t^{-n} e^{-\frac{r^2}{5t}}, \label{eqn:HD08_4}
    \end{align}
    for some $C=C(n,\kappa)$.  Consequently, we have
    \begin{align}
    \int_0^{t} \int_{X \backslash B(x_0,2r)}  \left( p^2+|\nabla p|^2 \right) d\mu_x ds
      < C \int_0^t \left(1+\frac{1}{s} + \frac{1}{r^2}\right) s^{-n} e^{-\frac{r^2}{5s}} ds.
    \label{eqn:HD08_3}
    \end{align}
 \end{corollary}

By virtue of Proposition~\ref{prn:HD04_1}, smooth functions are dense in $N^{1,2}(X)$.  Then it is easy to see that $(X,g,d\mu)$, together with the heat process,
has non-negative Ricci curvature in the sense of Bakry-Emery, i.e., $X \in CD(0,\infty)$ by the notation of Bakry-Emery(c.f.~\cite{BaEm},~\cite{Bak}).
The following Proposition is nothing but part of Proposition 2.1 of~\cite{Bak}.
The rigorous proof is tedious and is postponed in the appendix. 

\begin{proposition}[\textbf{Weighted Sobolev inequality}]
For every function $f \in N^{1,2}(X)$, every $t>0$ and every $y \in X$, we have
 \begin{align}
      \int_X f^2(x) p(t,x,y) d\mu_x - \left(\int_X f(x) p(t,x,y)d\mu_x \right)^2 \leq 2t \int_X |\nabla f|^2(x) p(t,x,y) d\mu_x.
   \label{eqn:HD05_3}
   \end{align}
 In other words, for every $t>0$, with respect to the probability measure $p(t,x,y)d\mu_x$, $L^2$-Sobolev inequality holds
 with the uniform Sobolev constant $\frac{1}{2t}$.
\label{prn:HD07_2}
\end{proposition}

On a Riemannian manifold with proper geometry bound, the heat kernel can be regarded as a solution starting from a $\delta$-function.
This property also holds for every $X \in \widetilde{\mathscr{KS}}(n,\kappa)$.

\begin{proposition}[\textbf{$\delta$-function property of heat kernel}]
  Suppose $w$ is a function on $[0,t] \times X$, differentiable along the time direction,
  $w(s,\cdot) \in N_c^{1,2}(X)$ for each $s \in [0,t]$. 
  Moreover,   we assume $\displaystyle \limsup_{s \to 0^{+}} \norm{w(s,\cdot)}{L^1(X)}<\infty$, $w$ is continuous at $(0,x_0)$. 
  Then we have
  \begin{align}
   &-w(0,x_0) +\int_{X} w(t,x)p(t,x,x_0)d\mu_x=\int_0^{t} \int_X \left\{ \left(\frac{\partial}{\partial s}+\Delta \right) w(s,x) \right\} p(s,x,x_0) d\mu_x ds.
        \label{eqn:HD09_1}
 \end{align}
 Consequently,  equation (\ref{eqn:HD09_1}) holds for functions $w(s,x)+a(s)$ where $a$ is a differentiable function of time.
\label{prn:HD09_1}
\end{proposition}

\begin{proof}
Clearly, (\ref{eqn:HD09_1}) holds if $w(s,\cdot) \equiv a(s)$. Therefore, it suffices to show (\ref{eqn:HD09_1}) when $w(s,\cdot) \in N_c^{1,2}(X)$ for each $s$.  
For simplicity of notation,  we denote $p(t,x,x_0)$ by $p$ and assume $d\mu_x$ as the default measure.  It follows from integration by parts that
\begin{align*}
  \frac{d}{dt} \int_{X} wp
  =\int_X \left\{ \left(\frac{\partial}{\partial t}+\Delta \right) w \right\} p   + \int_X w \left\{\left(\frac{\partial}{\partial t}-\Delta \right) p\right\} 
  =\int_X \left\{ \left(\frac{\partial}{\partial t}+\Delta \right) w \right\} p.  
\end{align*}
For each $\epsilon>0$, we can find $\delta$ small enough such that $|w(s,x)-w(0,x_0)|<\epsilon$ whenever $0<s<\delta^2$ and $d(x,x_0)<\delta$.
Then the heat kernel estimate implies that
\begin{align*}
  &\quad \left|-w(0,x_0)+\lim_{t \to 0^+} \int_X w(t,x)p(t,x,x_0)d\mu_x \right|
  = \left|\lim_{t \to 0^+} \int_X  \left\{w(t,x)-w(0,x_0)\right\}p(t,x,x_0)d\mu_x \right|\\
   &\leq \lim_{t \to 0^{+}} \left\{ |w(0,x_0)|\int_{X \backslash B(x_0,\delta)} p(t,x,x_0)d\mu_x
     + \int_{X \backslash B(x_0,\delta)} |w(t,x)| p(t,x,x_0)d\mu_x +\epsilon \right\}
    \leq \epsilon.
\end{align*}
By arbitrary choice of $\epsilon$,  we have $\displaystyle \lim_{t \to 0^+} \int_X w(t,x)p(t,x,x_0)d\mu_x =w(0,x_0)$.  Plugging this relationship into the integration of previous equation, 
we obtain (\ref{eqn:HD09_1}). 
\end{proof}

Based on the excellent properties of heat kernels, from Proposition~\ref{prn:HC29_7} to Proposition~\ref{prn:HD09_1},  we are ready to generalize the celebrated Cheng-Yau estimate (c.f.~\cite{ChYau}) to our setting.
We basically follow the paper~\cite{KoRaSh}.
However, due to the essential importance of this estimate and the excellent geometry of our underlying space, we write down a simplified proof here.

\begin{proposition} [\textbf{Cheng-Yau type gradient estimate}]
Suppose $\Omega=B(x_0,4r)$ for some $r>0$ and $x_0 \in \mathcal{R}$.
Suppose $u \in L^{\infty}(\Omega) \cap N^{1,2}(\Omega)$ and $u$ satisfies the equation
\begin{align}
    \Delta u=h
\label{eqn:HE03_3}
\end{align}
for some $h \in C^{\frac{1}{2}}(\Omega)$.
Then we have
\begin{align}
    |\nabla u|(x_0) \leq \frac{C}{r}  \left( \norm{u}{L^{\infty}(\Omega)} + r^{\frac{5}{2}}[h]_{C^{\frac{1}{2}}(\Omega)} + r^2|h(x_0)| \right)
 \label{eqn:HD01_1}
\end{align}
for a constant $C=C(n,\kappa)$, where $\displaystyle [h]_{C^{\frac{1}{2}}(\Omega)}=\sup_{x,y \in \Omega} \frac{|h(x)-h(y)|}{d^{\frac{1}{2}}(x,y)}$. 
\label{prn:HC29_6}
\end{proposition}

 \begin{proof}
   Without loss of generality, we assume $r=1$ and $\norm{u}{L^{\infty}(\Omega)}=1$.
   Let $\chi$ be a Lipschitz cutoff function such that $\chi \equiv 1$ on $B(x_0,1)$ and vanishes outside $B(x_0, 2)$ such that $|\nabla \chi| \leq 2$.
   Define
   \begin{align*}
         a(t) &\triangleq P_t(u\chi)(x_0),    &J(t) &\triangleq \frac{1}{t} \int_0^{t} \int_X |\nabla w(s,x)|^2 p(s,x,x_0) d\mu_x ds, \quad \forall \; t>0,\\
      w(t,x) &\triangleq u(x)\chi(x)-a(t), & J(0) &\triangleq \lim_{t \to 0^{+}} J(t)=|\nabla w(0, x_0)|^2=|\nabla u(x_0)|^2.
   \end{align*}
   From the definition of $w(t,x)$, it is clear that  $\displaystyle \int_X w(t,x) p(t,x,x_0)d\mu_x=0$.  Applying (\ref{eqn:HD05_3}), we have
   \begin{align}
    &\int_X w^2(t,x) p(t,x,x_0)d\mu_x \leq 2t\int_X |\nabla w|^2p(t,x,x_0) d\mu_x,  \label{eqn:HD09_3}\\
    &|a(t)|= \left|\int_X u\chi p(t,x,x_0)d\mu_x\right| \leq \int_X |u\chi| p(t,x,x_0)d\mu_x   \leq \int_X  p(t,x,x_0)d\mu_x \leq 1, \label{eqn:HF20_1}\\
    &|w(t,x)|= |u(x)\chi(x)  -a(t)| \leq |u(x)\chi(x)|+|a(t)| \leq 2, \quad \forall \; x \in X.   \label{eqn:HD06_5}
   \end{align}
   It follows from the definition of $J$ that
   \begin{align}
    |\nabla u(x_0)|^2=J(0)=-\int_0^{1} J'(t) dt +J(1).
   \label{eqn:HD06_6}
   \end{align}
  However,  in light of (\ref{eqn:HD09_3}), we have
   \begin{align*}
     J'(t)&=-\frac{1}{t} J(t)+\frac{1}{t} \int_X |\nabla w|^2p(t,x,x_0) d\mu_x
      \geq-\frac{1}{t} J(t)+\frac{1}{2t^2} \int_X w^2(t,x) p(t,x,x_0)d\mu_x \notag\\
      &=\frac{1}{t^2} \left(-tJ(t) +\frac{1}{2}\int_X w^2(t,x) p(t,x,x_0)d\mu_x \right)
      =\frac{F(t)}{t^2},  
  \end{align*}
  where we used the definition $F(t) \triangleq -tJ(t) +\frac{1}{2}\int_X w^2(t,x) p(t,x,x_0)d\mu_x$. 
  Therefore, by the previous inequalities,  equation (\ref{eqn:HD06_6}) can be rewritten as
   \begin{align}
    |\nabla u(x_0)|^2 \leq -\int_0^1 \frac{F(s)}{s^2} ds -F(1) +\frac{1}{2} \int_X w^2p
    \leq 2-\int_0^1 \frac{F(s)}{s^2} ds -F(1).
   \label{eqn:HD09_7}
   \end{align}
  Therefore,  $|\nabla u(x_0)|$ follows from the estimate of $F(t)$.

  We now focus on the estimate of $F(t)$.    Applying (\ref{eqn:HD09_1}) to $w^2$, we obtain 
    \begin{align}
      \int_0^t \int_X \left\{ \left( \frac{\partial}{\partial s} + \Delta \right) w^2(s,x)\right\}p(s,x,x_0)d\mu_x ds
      =\int_X w^2(t,x)p(t,x,x_0) d\mu_x,
      \label{eqn:HD06_7}
   \end{align}
   since $w(0,x_0)=0$. Note that the application of (\ref{eqn:HD09_1}) can be justified.
   Actually, from its definition, $w^2(x,t)=u\chi (u\chi -2a(t)) +a^2(t)$.   The first part of the right hand side of this equation is a function in $N_c^{1,2}(X)$ for each $t$. It is Lipschitz continuous around $(0,x_0)$,
  since $x_0$ is a smooth point and the standard improving regularity theory of elliptic functions applies here.
  The second part is a differentiable function of time.
  Therefore, (\ref{eqn:HD09_1}) applies for $w^2$.
  On the other hand, the fact that $u$ is a weak solution of (\ref{eqn:HE03_3}) implies that
   \begin{align*}
     |\nabla w|^2=\frac{1}{2}\Delta w^2-w\Delta w
        =\frac{1}{2}\Delta w^2-w(u\Delta \chi +\chi h+2 \langle \nabla u, \nabla \chi \rangle)
   \end{align*}
   in the weak sense.
   Plugging the above equation and (\ref{eqn:HD06_7}) into the definition of $J(t)$, we have
   \begin{align*}
    tJ(t)
    &=\int_0^t\int_X \left\{ \frac{1}{2}\left( \frac{\partial}{\partial s}+\Delta\right) w^2
       -w(u\Delta \chi+\chi h +2 \langle \nabla u, \nabla \chi \rangle) \right\} p(s,x,x_0)d\mu_x ds \notag\\
    &=\int_0^t\int_X \left\{ \frac{1}{2}w^2
       -w(u\Delta \chi +\chi h+2 \langle \nabla u, \nabla \chi \rangle) \right\} p(s,x,x_0)d\mu_x ds.
   \end{align*}
   For the simplicity of notation, we will  denote $p(s,x,x_0)$ by $p$ only. Also, we will  drop integration elements when they are clear. 
   From the definition of $F(t)$,  the above equation can be written as 
   \begin{align*}
     &\quad F(t)-\int_0^t\int_X w\chi h p=\int_0^{t}\int_X w(u\Delta \chi+2 \langle \nabla u, \nabla \chi \rangle) p.
   \end{align*}
   Recall that $w=u\chi -a$ and $\nabla w=u\nabla \chi + \chi \nabla u$.   Integrating by parts gives us   
   \begin{align*}  
     F(t)-\iint  w\chi h p =\iint \langle -up\nabla w-uw \nabla p +pw \nabla u, \nabla \chi  \rangle
           =\iint \langle -ap \nabla u -uw \nabla p, \nabla \chi \rangle - pu^2|\nabla \chi|^2. 
   \end{align*}
  By  (\ref{eqn:HF20_1}) and (\ref{eqn:HD06_5}), we have $|a| \leq 1$ and $|uw| \leq 2$. 
   By the choice of $\chi$ and the H\"{o}lder inequality, it is clear that
   \begin{align}
         &\quad \left|F(t) +\int_0^t\int_X ( u^2|\nabla \chi|^2 -w\chi h) p \right| 
            \leq 2 \int_0^t \int_{B(x_0,2) \backslash B(x_0,1)} (p|\nabla u| + |\nabla p|)  \notag\\
         &\leq C\left( \int_0^t \int_{B(x_0,2) \backslash B(x_0,1)} \left(|\nabla u| ^2+1\right) \right)^{\frac{1}{2}}
             \cdot \left( \int_0^t \int_{B(x_0,2) \backslash B(x_0,1)} \left(|\nabla p| ^2+p^2\right) \right)^{\frac{1}{2}} \notag\\
         &\leq C\left(1+ \norm{h}{L^{\infty}(\Omega)} \right) \sqrt{t}\left( \int_0^t \int_{B(x_0,2) \backslash B(x_0,1)} \left(|\nabla p| ^2+p^2\right) \right)^{\frac{1}{2}} .  \label{eqn:GD05_1}
   \end{align}
   Note that in the last step of the above inequality, we used the following Caccioppoli-type inequality:
   \begin{align*}
      \int_{B(x_0,2)} |\nabla u|^2  \leq C \int_{B(x_0,4)} (u^2+h^2)  \leq C(n,\kappa) \left(1+\norm{h}{L^2(\Omega)} \right)^2
      \leq C(n,\kappa) \left(1+\norm{h}{L^{\infty}(\Omega)} \right)^2, 
   \end{align*}
   which can be proved by multiplying equation (\ref{eqn:HE03_3}) on both sides by $\tilde{\chi}^2 u$ and doing integration by parts, for some cutoff function $\tilde{\chi}$. 
   By inequality (\ref{eqn:HD08_3}), the last term in (\ref{eqn:GD05_1}) can be controlled by $C(\kappa, \beta) t^{\beta}$ for any positive number $\beta$.  
   For the simplicity of later calculation, we choose $\beta=\frac{3}{4}$. Note that $\norm{h}{L^{\infty}(\Omega)}<|h(x_0)|+[h]_{C^{\frac{1}{2}}(\Omega)}$. 
   Let $L=1+ |h(x_0)|+[h]_{C^{\frac{1}{2}}(\Omega)}$, then we have
   \begin{align}
    |F(t)| &\leq  \left| \int_0^t\int_X w\chi h p \right| + \left| \int_0^t\int_{B(x_0,2) \backslash B(x_0,1)} u^2|\nabla \chi|^2 p \right|+ CLt^{\frac{5}{4}} \notag\\
            &\leq \left| \int_0^t\int_X w\chi h p \right| +4 \int_0^t\int_{B(x_0,2) \backslash B(x_0,1)}  p +CLt^{\frac{5}{4}}. 
   \label{eqn:HD09_9} 
   \end{align} 
  The second term on the right hand side of the above inequality can be absorbed by the last term, due to the exponential decay of $p$ and Euclidean volume growth condition(c.f. Proposition~\ref{prn:HD07_1} and Proposition~\ref{prn:HD19_1}). 
  On the other hand,  since $pw$ has zero integeral, we have
  \begin{align*}
    \left| \int_0^t\int_X w\chi h p \right|&= \left| \int_0^t\int_X (\chi h-h(x_0)) pw \right|
    \leq  \left| \int_0^t\int_{B(x_0,2)} (\chi h-h(x_0)) pw \right| +  \left| \int_0^t\int_{X \backslash B(x_0,2)} h(x_0)pw \right|. 
  \end{align*}
  However, as $\chi h-h(x_0)$ vanishes at $x_0$, we have $ |\chi h-h(x_0)|  \leq [h]_{C^{\frac{1}{2}}(\Omega)} \cdot d^{\frac{1}{2}}(x,x_0)$. 
  Consequently, we have
  \begin{align*}
    &\left| \int_0^t\int_{B(x_0,2)} (\chi h-h(x_0)) pw d\mu_x ds\right| < C  [h]_{C^{\frac{1}{2}}(\Omega)} \int_0^t s^{\frac{1}{4}} ds< C [h]_{C^{\frac{1}{2}}(\Omega)} t^{\frac{5}{4}} < CLt^{\frac{5}{4}},\\
    &\left| \int_0^t\int_{X \backslash B(x_0,2)} h(x_0)pw d\mu_x ds\right|<2|h(x_0)| \int_0^t\int_{X \backslash B(x_0,2)} p d\mu_x ds < C|h(x_0)| t^{\frac{5}{4}}<CLt^{\frac{5}{4}}. 
  \end{align*}
  Plugging the above inequalities into (\ref{eqn:HD09_9}), we obtain  
  \begin{align}
    |F(t)|<CLt^{\frac{5}{4}},  \quad  \int_0^{1}\frac{|F(s)|}{s^2} ds< CL \int_0^1 s^{-\frac{3}{4}} ds<CL.
  \label{eqn:HD09_5}
  \end{align}
  Recall that $L=1+ |h(x_0)|+[h]_{C^{\frac{1}{2}}(\Omega)}$.  Plugging the above inequalities into (\ref{eqn:HD09_7}), we obtain the desired estimate of $|\nabla u(x_0)|$.
\end{proof}

Combining Proposition~\ref{prn:HC29_6} and the $C^{\alpha}$-estimate (c.f. Theorem 4.1 of~\cite{Saloffnote}) for bounded heat solutions, one can derive the Li-Yau type gradient estimate.  
Alternatively, for bounded heat solution $u$, one can follow the proof of Lemma~\ref{lma:HH10_1} to obtain the uniform bound of $|\nabla u|$ by De-Giorgi iteration process. 
Another interesting application of Cheng-Yau type inequality is the following Liouville theorem.

\begin{corollary}[\textbf{Liouville theorem}]
  Suppose $u$ is a bounded harmonic function on $\mathcal{R}$, then $u \equiv C$.
\label{cly:HE08_1}
\end{corollary}
\begin{proof}
  By virtue of Proposition~\ref{prn:HD16_2}, the extension property of subharmonic functions, we know that $u \in N_{loc}^{1,2}(X)$.
  In light of Proposition~\ref{prn:HD04_1}, the dense property of smooth functions in $N_0^{1,2}(\Omega)$,  it is clear that $\Delta u=0$ on $X$ in the weak sense.
  Fix $x_0 \in \mathcal{R}$ and a large $r>0$, by Cheng-Yau estimate in Proposition~\ref{prn:HC29_6},
  we have
  \begin{align*}
     |\nabla u|(x_0) < \frac{C}{r}
  \end{align*}
  for some uniform constant $C$. Let $r \to \infty$, we see that $|\nabla u|(x_0)=0$.  It follows that $\nabla u \equiv 0$ on $\mathcal{R}$
  by the arbitrary choice of $x_0 \in \mathcal{R}$.  Consequently, $u \equiv C$ on $\mathcal{R}$.
\end{proof}

\begin{proposition}[\textbf{Estimates for Dirichlet problem solution}]
Suppose $\Omega$ is a bounded open set of $X$, $f $ is a continuous function in $N^{1,2}(\Omega)$. Then we have the following properties.

\begin{itemize}
\item  There is a unique solution $u \in N^{1,2}(\Omega)$ solving the Dirichlet problem
\begin{align}
  \Delta u=0, \quad \textrm{in} \; \Omega; \quad
  (u-f)|_{\partial \Omega}=0
\label{eqn:HE03_2}
\end{align}
in the weak sense of traces. In other words, $\Delta u=0$ in the weak sense and $u-f \in N_0^{1,2}(\Omega)$.
\item Weak maximum principle holds for $u$, i.e.,
       \begin{align}
          \sup_{x \in \Omega} u(x)=\sup_{x \in \partial \Omega} u(x), \quad  \inf_{x \in \Omega} u(x)=\inf_{x \in \partial \Omega} u(x).
       \end{align}
\item  Strong maximum principle holds for $u$, i.e.,   if there is an interior point $x_0 \in \Omega$ such that
      $\displaystyle u(x_0)= \sup_{x \in \Omega} u(x)$ or $\displaystyle u(x_0)=\inf_{x \in \Omega} u(x)$, then $u$ is a constant.
\end{itemize}
\label{prn:HD11_1}
\end{proposition}

\begin{proof}
   The existence and uniqueness of the Dirichlet problem follows from Theorem 7.12 and Theorem 7.14 of Cheeger's work~\cite{Cheeger99},
   where a much more general case was considered.     The weak maximum principle follows from the uniqueness.
   Also, the weak maximum principle was proved by Shanmugalligam in~\cite{ShanHar}.
   The strong maximum principle follows from elliptic Harnack estimates, which is a consequence of the volume doubling and $(1,2)$-Poincar\'{e} inequality. 
   This is due to the work of K.T. Sturm in~\cite{Stu96}.  The manifold case was obtained by A. Grigor'yan in~\cite{Grigor}, and L. Saloff-Coste in~\cite{Saloffnote}. 
   More information can be found in the beautiful survey~\cite{Saloffsur} by L. Saloff-Coste. 
   
   We write down an elementary proof here for the convenience of the readers, based on the excellent underlying geometry.
   Here we follow Corollary 6.4 of~\cite{KiSh}. 
   Without loss of generality, we assume $u \geq 0$ on $\partial \Omega$ and $u$ is not a constant.
   It suffices to show that $u>0$ in $\Omega$.  Clearly, by classical harmonic function theory on Riemannian manifold and continuity of
   $u$(c.f.~Proposition~\ref{prn:HD16_1}), it is clear that
   $u>0$ on $\Omega \cap \mathcal{R}$.   Therefore, we only need to show that $u>0$ on $\Omega \cap \mathcal{S}$.
   We argue by contradiction.  If this statement were wrong, we can find a point $y_0 \in \Omega \cap \mathcal{S}$ such that
   $u(y_0)=0$.  Choose $r$ small enough such that $B(y_0,2r) \subset \Omega$.
   For each small $\epsilon$, choose $\tau$ small enough such that
   \begin{align*}
       \left| \Omega_{\tau} \cap B(y_0,2r) \right| < \epsilon |B(y_0,2r)|,
   \end{align*}
   where $\Omega_{\tau}=\{x \in \Omega| u(x) \leq \tau\}$.  Note that $\tau$ can be chosen since $\Omega_0 \cap B(y_0,2r)$ is a subset of
   $\mathcal{S}$ which has zero measure, and $u$ is continuous. 
   Now consider the function $\tau-u$, which is obviously harmonic.  Let $(\tau-u)^{+}$ be $\max\{\tau-u, 0\}$.  Then
   $(\tau-u)^{+}$ is a bounded, continuous, subharmonic function in $N^{1,2}(B(y_0,2r))$. So Moser(or Nash-Moser-De-Giorgi) iteration applies to obtain
   \begin{align*}
     \sup_{B(y_0,r)} |(\tau-u)^{+}|^2 \leq C \int_{B(y_0,2r)} |(\tau-u)^{+}|^2=C \int_{\Omega_{\tau} \cap B(y_0,2r)} |(\tau-u)^{+}|^2 \leq C\epsilon r^{2n} \tau^2
   \end{align*}
   for some $C=C(n,\kappa)$.  Choose $\epsilon$ small enough such that $C\epsilon^2 r^{2n}<\frac{1}{4}$.  Then we have
   \begin{align*}
     \sup_{B(y_0,r)} |(\tau-u)^{+}| <\frac{\tau}{2},
   \end{align*}
   which implies that $u>\frac{\tau}{2}$ on $B(y_0,r)$.  In particular, $u(y_0) \geq \frac{\tau}{2}>0$, which contradicts the assumption $u(y_0)=0$.
\end{proof}

Clearly, the essential stuff in the proof of the strong maximum principle of Proposition~\ref{prn:HD11_1} is a delicate use of elliptic Moser iteration.
In equation (\ref{eqn:HE03_2}),  if we replace the operator $\Delta$ by $\square$, the heat operator, then one can easily obtain a strong maximum principle for heat equation solutions, based on a parabolic Moser iteration. The details are left to the interested readers.
On the other hand, if we replace the right hand side of equation (\ref{eqn:HE03_2}) by a function $h$, we can also obtain uniqueness and existence of solutions.

\begin{proposition}[\textbf{Existence and uniqueness of Poisson equation solution}]
Suppose $\Omega$ is a bounded open set of $X$, $f \in N^{1,2}(\Omega)$, $h \in L^2(\Omega)$.
Then there exists a unique $u \in N^{1,2}(\Omega)$ such that
\begin{align}
  \Delta u=h, \quad \textrm{in} \; \Omega; \quad
   (u-f)|_{\partial \Omega}=0.
\label{eqn:HE03_1}
\end{align}
\label{prn:HE03_1}
\end{proposition}
\begin{proof}
  First, let us consider the Poisson equation
 \begin{align*}
 \Delta v=h, \; \textrm{in} \; \Omega;   \quad v \in N_0^{1,2}(\Omega).
\end{align*}
By standard functional analysis, the existence of the above equation is guaranteed by Riesz representation theorem.
The uniqueness follows from the irreducibility of $\mathscr{E}$.
Second, it is obvious that there is a bijective map between the solution $u$ of (\ref{eqn:HE03_1})
and harmonic solution $w$ of (\ref{eqn:HE03_2}) by $u=w+v$.   Therefore, the existence and uniqueness of
(\ref{eqn:HE03_1}) follows from Proposition~\ref{prn:HD11_1}.
\end{proof}

Combining the strong maximum principle for harmonic functions in Proposition~\ref{prn:HD11_1} with the heat kernel estimates, we obtain the
strong maximum principle for subharmonic functions.

\begin{proposition} [\textbf{Strong maximum principle for subharmonic functions}]
Suppose $\Omega$ is a bounded domain of $X$, $u$ is a continuous subharmonic function in $N^{1,2}(\Omega)$. Then we have
\begin{align}
    \sup_{x \in \Omega} u(x) = \sup_{x \in \partial \Omega} u(x).
\label{eqn:HC29_2}
\end{align}
In other words,  the weak maximum principle holds for subharmonic functions. Moreover, if $\displaystyle \sup_{x \in \Omega}u(x)$ is achieved at
some point $x_0 \in \Omega$, then $u$ is a constant. Namely, the strong maximum principle holds for subharmonic functions.
\label{prn:HC29_5}
\end{proposition}

\begin{proof}
The weak maximum principle is well known in literature.
For example, see Lemma 4 of~\cite{Stu94} and the reference therein.

The strong maximum principle can be proved as that in Proposition~\ref{prn:HD11_1}.  Actually, $-w$ is a nonnegative superharmonic function on $\Omega$.
If $w$ is not a constant, then we can regard $-w$ as $u$ in the second part of the proof of Proposition~\ref{prn:HD11_1}.
Then everything goes through since only Moser iteration for subharmonic function is used there. 
\end{proof}

\begin{proposition}[\textbf{Removing singularity of harmonic functions}]
Suppose $\Omega$ is an open domain in $X$, $u$ is a  bounded harmonic function on $\Omega \backslash \mathcal{S}$.
Then $u$ can be regarded as a harmonic function on $\Omega$. Moreover, on each compact subset of $\Omega$, $u$ is uniformly Lipschitz continuous.
In particular, $u$ can be extended continuously over the singular set $\Omega \cap \mathcal{S}$.
\label{prn:HD16_1}
\end{proposition}
\begin{proof}
 In light of Proposition~\ref{prn:HD16_2}, we see that $u \in N_{loc}^{1,2}(\Omega)$.  Since $\Delta u=0$ on $\Omega \backslash \mathcal{S}$, we see that
 \begin{align*}
    \int_{\Omega} \langle \nabla u, \nabla \varphi \rangle=0
 \end{align*}
 for every smooth test function supported on $\Omega \backslash \mathcal{S}$.  However, such functions are dense in $N_0^{1,2}(\Omega)$, so the above
 equation actually holds form every $\varphi \in N_c^{1,2}(\Omega)$. Therefore, $u$ is harmonic in $\Omega$ by definition.

 The Lipschitz continuity follows from Proposition~\ref{prn:HC29_6} and the density and weak convexity of $\mathcal{R}$.
\end{proof}

Note that the weak convexity of $\mathcal{R}$ is important for that $u$ can be extended over singularities. For otherwise, the limit of $u(x_i)$ for $x_i \to x_0$ may depends
on the choice of  sequence $\{x_i\}$, where $x_0 \in \mathcal{S}, x_i \in \mathcal{R}$.   If $\mathcal{R}$ is not convex, there is an easy counter example
of Proposition~\ref{prn:HD16_1}.  Let $X$ be the union of two cones $C(S^3/\Gamma)$, for some finite group $\Gamma \in ISO(S^3)$, by identifying two vertices.
In this case, $\mathcal{S}$ is the isolated vertex $O$.  Let $u$ be $1$ on  one branch and $0$ on the other, then it is clear that $u$ is a harmonic function
on $X \backslash \mathcal{S}$.  However, $u$ can not take a value at $O$ so that $u$ is continuous.  Of course, convexity of $\mathcal{R}$ is only a sufficient
condition to guarantee the continuity extension.  It can be replaced by other weaker conditions.
Moreover, based on Proposition~\ref{prn:HC29_6}, one can obtain uniform gradient estimate of $|\nabla p(t,\cdot, x_0)|$, which depends only
on $t,n$ and $\kappa$.  Hence the heat kernel $p(t,\cdot, x_0)$ is a continuous function on $X \times (0,\infty)$.
Therefore,  by approximation, the estimate in Proposition~\ref{prn:HD07_1} holds on every point on
$X$, even if this point is singular.

\subsection{Approximation  functions of distance}

\begin{proposition}[\textbf{Almost super-harmonicity of distance function}] Suppose $x_0 \in X$,  $r(x)=d(x,x_0)$. Then we have
\begin{align}
\Delta r^2 \leq 4n
\label{eqn:HC29_1}
\end{align}
in the weak sense. In other words, for every nonnegative $\chi \in N_c^{1,2}(X)$, we have
\begin{align}
      -\int_X \left\langle \nabla r^2, \nabla \chi \right\rangle \leq \int_X 4n \chi.
\label{eqn:HD16_3}
\end{align}
\label{prn:HC29_1}
\end{proposition}

\begin{proof}
Let us first assume $x_0 \in \mathcal{R}$.
Clearly, away from the generalized cut locus, we have
$\Delta r^2 \leq 4n$ in the classical sense.  Therefore, $\Delta r^2 \leq 4n$ on $\mathcal{R}$ in the distribution sense,
same as the smooth Riemannian manifold case.   Since smooth cutoff functions supported on $\mathcal{R}$
are dense in $N_0^{1,2}(X)$(c.f. Corollary~\ref{cly:HE10_1} and Proposition~\ref{prn:HD04_1}, where $codim(\mathcal{S})>2$ is essentially used),
we see that for every $\chi \in N_c^{1,2}(X)$,  inequality (\ref{eqn:HD16_3}) holds true.

 Now suppose $x_0 \in \mathcal{S}$, we can choose regular points $x_i \to x_0$.  Let  $r_i=d(x_i, \cdot)$, then for each nonnegative function $\chi \in N_c^{1,2}(X)$, we have
\begin{align*}
  -\int_X \left\langle \nabla r_i^2, \nabla \chi \right\rangle \leq \int_X 4n \chi.
\end{align*}
Let $\Omega$ be a bounded open set containing the support of $\chi$.
Then $r_i^2$ weakly converges to a unique limit in $N^{1,2}(\Omega)$,
$r_i^2$ strongly converges to $r^2$ in $L^2(\Omega)$. This means that
$r^2$ is the weak limit of $r_i^2$ in $N^{1,2}(\Omega)$. It follows that
\begin{align*}
    -\int_X \left\langle \nabla r^2, \nabla \chi \right\rangle
    =-\lim_{i \to \infty}  \int_X \left\langle \nabla r_i^2, \nabla \chi \right\rangle
    \leq \int_X 4n \chi.
\end{align*}
\end{proof}

In view of Proposition~\ref{prn:HC29_1}, we can obtain many rigidity theorems.

\begin{lemma}[\textbf{Cheeger-Gromoll type splitting}]
Suppose $X$ contains a straight line $\gamma$. Then there is a length space $N$ such that
$X$ is isometric to $N \times \R$ as metric space product.
\label{lma:HE04_1}
\end{lemma}
\begin{proof}
The proof is almost the same as the classical one.  However, we shall take this as an opportunity to check the analysis tools developed in previous subsections.
Actually, fix $x_0 \in \gamma$, we can divide $\gamma$ into two rays $\gamma^{+}$ and $\gamma^{-}$.
Accordingly, there are Buseman functions $b^{+}$ and $b^{-}$.  Proposition~\ref{prn:HC29_1} implies that $\Delta r \leq \frac{2n-1}{r}$ in the weak sense,
which in turn forces both $b^{+}$ and $b^{-}$ to be subharmonic functions.  By triangle inequality, we know $b^{+}+b^{-} \geq 0$ globally and
achieve $0$ on $x_0$.  It follows from strong maximum principle, by Proposition~\ref{prn:HC29_5},  that $b^{+}+b^{-} \equiv 0$.
Consequently, $b^{+}$ is harmonic.  Then Weitzenb\"{o}ck formula implies that in the weak sense, we have
\begin{align*}
  0=\frac{1}{2}\Delta |\nabla b^{+}|^2= |Hess_{b^{+}}|^2.
\end{align*}
Since $b^{+}$ is harmonic, it is harmonic on $\mathcal{R}=X \backslash \mathcal{S}$.  By standard improving regularity theory of harmonic functions on
smooth manifold, we see that $b^{+}$ is a smooth function on $\mathcal{R}$ satisfying
\begin{align*}
  |\nabla b^{+}|^2 \equiv 1, \quad  |Hess_{b^{+}}|^2 \equiv 0.
\end{align*}
Up to this step, everything is the same as the classical case.
However, since the regular part $\mathcal{R}$ is not complete, the following argument is slightly different.
On $\mathcal{R}$, since $\mathcal{L}_{\nabla b^{+}} g=2Hess_{b^{+}}=0$, we see that $\nabla b^{+}$ is a Killing field.
The flow generated by $\nabla b^{+}$ preserves metrics, and in particular the volume element.   
By the high codimension of $\mathcal{S}$, weak convexity of $\mathcal{R}$ and the essential gap of volume density between regular and singular points, one can obtain that the flow generated by $\nabla b^{+}$ preserves regularity.
The full details will be explained as follows. 

Let $\varphi_t$ be the time $t$ flow map generated by $\nabla b^{+}$ when it is well defined. In other words,  we have $\frac{d}{dt} \varphi_t(x)=\left. \nabla b^{+} \right|_{\varphi_t(x)}$ whenever $\varphi_t(x) \in \mathcal{R}$. 

\begin{claim}[\textbf{Existence of flows away from small sets}]
 For each fixed $x_0 \in X$ and $A>0$, there is a set $E_A$ such that $\varphi_t(x)$ exists and locates in $\mathcal{R}$ for all
 $x \in B(x_0, A) \backslash E_A$ and $t \in [-A, A]$.   Moreover, we have
 \begin{align}
   \dim_{\mathcal{M}} E_A \leq \dim_{\mathcal{M}} \mathcal{S} +1<2n-2.  \label{eqn:MC14_4}
 \end{align}
 Consequently, there is a measure-zero set $E$ such that $\varphi_t(x)$ exists and locates in $\mathcal{R}$ for all $x \in X \backslash E$ and $t \in (-\infty, \infty)$.  
\label{clm:MA20_1} 
\end{claim}

Fix $x_0 \in X$ and choose $\xi$ to be a very small positive number, $A$ to be a large positive number.  
Let $q_0$ be the Minkowski codimension of $\mathcal{S}$, i.e., $q_0=2n-\dim_{\mathcal{M}} \mathcal{S}$,   
$\epsilon$ be a very small number to be used in the volume estimate related to Minkowski codimension. 
Note that $d(\cdot, \mathcal{S})$ is a Lipshitz function with Lipshitz constant $1$.
By perturbing $d(\cdot, \mathcal{S})$, we can find a smooth hyper surface $\Sigma_{\xi}$ (c.f. Corollary~\ref{cor:MA20_1} for more details)
such that
\begin{align*}
 &\left|  B(x_0, 3 A) \cap \Sigma_{\xi} \right|_{\mathcal{H}^{2n-1}} \leq C \xi^{q_0-\epsilon-1},  \\
 & \frac{1}{H} \xi <d(x, \mathcal{S})< H \xi, \quad \forall \; x \in \Sigma_{\xi} \cap B(x_0, 2A). 
\end{align*}
Note that the $\xi^{q_0-\epsilon}$ in the first inequality comes from the fact that $\dim_{\mathcal{M}}\mathcal{S}=2n-q_0$ and the application of co-area formula. 
The constant $C$ in the first inequality depends both on $\epsilon$ and the set $B(x_0, 3A)$. 
The constant $H$ in the second inequality depends on $\kappa, n$ and comes from the perturbation technique. $H$ will be fixed in the following discussion.
$C$ may vary from line to line, as usual. 

Define a set
\begin{align}
  E_{A,\xi} \triangleq \left\{ x\in \overline{B(x_0, A)} \left| d\left( \varphi_t(x), \mathcal{S} \right) \leq \frac{1}{H} \xi, \quad \textrm{for some} \; t \in [-A, A] \right. \right\}.  
\label{eqn:MA20_0}  
\end{align}
Now we decompose $E_{A,\xi}$ into two parts $I$ and $II$ as follows.
\begin{align*}
   I= \{x \in \overline{B(x_0, A)} | d(x, \mathcal{S}) \leq 10 \xi\} \cap E_{A,\xi}, \quad
   II=\{x \in \overline{B(x_0, A)} | d(x, \mathcal{S}) >10 \xi \} \cap E_{A,\xi}. 
\end{align*}
By the Minkowski dimension assumption of $\mathcal{S}$, we know that 
\begin{align}
 |I| \leq    \left| \left\{ \left. x \in \overline{B(x_0, A)} \right| d(x, \mathcal{S}) \leq 10 \xi \right\} \right| \leq C \xi^{q_0-\epsilon}, 
\label{eqn:MA20_1} 
\end{align}
where $|\cdot|$ means the $2n$-dimensional Hausdorff measure. 
Taking every point $y \in II$, one can flow it to a point on $\Sigma_{\xi}$ at some time $t \in [-A, A]$ by the definition of $E_{A,\xi}$. 
Since $|\nabla b^{+}|=1$, it is clear that $d(y, \varphi_t(y)) \leq |t|$. Then triangle inequality implies that
\begin{align*}
   d(x_0, \varphi_t(y)) < d(x_0, y) + d(y, \varphi_t(y)) \leq A + |t| \leq 2A < 3A
\end{align*}
Therefore, the set $II$ can be locally regarded as a bundle over $\Sigma_{\xi} \cap B(x_0, 3A)$.  Note that along the flow line of the Killing field $\nabla b^+$,  $\varphi_t$ preserves local isometry. 
We equip the set $\left\{\Sigma_{\xi} \cap B(x_0, 3A) \right\} \times [-A, A]$ with the obvious product measure. Consider the map
\begin{align*}
  \varphi: \Omega \subset \left\{\Sigma_{\xi} \cap B(x_0, 3A) \right\} \times [-A, A]  &\mapsto  X, \\
              (x,t) &\mapsto \varphi_t(x).    
\end{align*} 
Here $\Omega$ is the  maximal subset of $\left\{\Sigma_{\xi} \cap B(x_0, 3A) \right\} \times [-A, A]$ such that $\varphi(x,t)=\varphi_t(x)$ is well defined. 
It is clear that $\varphi$ decrease volume whenever the flow line is not perpendicular to $\Sigma_{\xi}$.   It follows that
\begin{align*}
 |II| \leq  |\Omega|  \leq \left| \left\{\Sigma_{\xi} \cap B(x_0, 3A) \right\}  \right|_{\mathcal{H}^{2n-1}} \cdot 3A \leq CA \xi^{q_0-\epsilon-1}. 
\end{align*}
Combining the above inequality with (\ref{eqn:MA20_1}), we have
\begin{align}
    |E_{A,\xi}| \leq |I|+|II| \leq  C \xi^{q_0-\epsilon} + C A \xi^{q_0-\epsilon-1} \leq C\xi^{q_0-\epsilon-1},   \label{eqn:MA20_3}
\end{align}
where the last $C$ depends on $n,\kappa$ and $B(x_0, 3A)$.

We observe that
\begin{align}
   \left\{ x \left| d(x, E_{A, \xi}) < H^{-1}\xi  \right. \right\} \subset E_{2A, 2\xi}.  \label{eqn:MC14_0}
\end{align}
In fact, if $x$ locates in the $H^{-1}\xi$-neighborhood of $E_{A, \xi}$,  then we can find a point $y \in E_{A, \xi}$ such that
$d(x, E_{A, \xi})=d(x,y)=\delta<H^{-1}\xi$.  So we can find a shortest geodesic connecting $x$ to $y$ satisfying
$\gamma(0)=x$ and $\gamma(\delta)=y$, with $|\gamma|=\delta$.  
Note that we can assume $\gamma$ is a smooth geodesic. For otherwise, we have
$d(x, \mathcal{S})\leq \delta <H^{-1}\xi$,  which automatically implies $x \in E_{2A, 2\xi}$ by the definition in (\ref{eqn:MA20_0}). 
For the same reason, triangle inequality guarantees us to assume
\begin{align}
  d(\gamma, \mathcal{S})>H^{-1}\xi.   \label{eqn:MC14_1}
\end{align}
As $y \in E_{A, \xi}$, following its definition in (\ref{eqn:MA20_0}),  we can find a $t_0 \in [-A, A]$ such that 
\begin{align}
d(\varphi_{t_0}(y), \mathcal{S}) \leq H^{-1}\xi.  \label{eqn:MC14_2}
\end{align}
Let $s_0$ be the smallest positive value such that $d(\varphi_{s_0}(\gamma), \mathcal{S}) \leq H^{-1}\xi$. 
The combination of (\ref{eqn:MC14_1}) and (\ref{eqn:MC14_2}) then yields that $0<|s_0| \leq |t_0|$. 
Note that $\varphi_{s}(\gamma)$ is well defined on $(-|s_0|, |s_0|)$. So the length of $\varphi_s(\gamma)$ is the same as the length of $\gamma$
for each $s \in (-|s_0|, |s_0|)$.
Now we are ready to estimate the distance 
between $\varphi_{s_0}(x)$ and the singular set $\mathcal{S}$:
\begin{align*}
  d(\varphi_{s_0}(x), \mathcal{S}) \leq |\varphi_{s_0}(\gamma)| + d(\varphi_{s_0}(\gamma), \mathcal{S})
   =|\gamma| + d(\varphi_{s_0}(\gamma), \mathcal{S})
   <H^{-1} \xi + H^{-1} \xi=\frac{2\xi}{H}. 
\end{align*}
Note that $E_{A, \xi} \subset \overline{B(x_0, A)}$.  Triangle inequality implies that $x \in \overline{B(x_0, 2A)}$. 
Recall also that $s_0 \in [-A, A] \subset [-2A, 2A]$.  In light of definition equation (\ref{eqn:MA20_0}), the above inequality 
implies that $x \in E_{2A, 2\xi}$.  So we finish the proof of (\ref{eqn:MC14_0}).

Let $\xi=\xi_i \to 0$ and define 
\begin{align}
   E_{A} \triangleq \cap_{i=1}^{\infty} E_{A,\xi_i}.    \label{eqn:MC14_3}
\end{align}
We obtain a set $E_{A} \subset \overline{B(x_0, A)}$ such that $\varphi_t(x) \in \mathcal{R}$ for each point 
$x \in \overline{B(x_0, A)} \backslash E_{A}$ and $t \in [-A,A]$. 
Furthermore, from defintion equation (\ref{eqn:MC14_3}), it is clear that $E_A \subset E_{A, \xi}$, which together with (\ref{eqn:MC14_0})
implies that the $H^{-1}\xi$ neighborhood of $E_A$ is contained in $E_{2A, 2\xi}$. 
Consequently, it follows from (\ref{eqn:MA20_3}) (replacing $A$ by $2A$, $\xi$ by $2\xi$) that
\begin{align*}
  \left|  \left\{ x \left |  d(x, E_A) < H^{-1} \xi \right.\right\}  \right|_{\mathcal{H}^{2n}} \leq C \xi^{q_0-\epsilon-1}=C \left( H^{-1} \xi \right)^{q_0-\epsilon-1},
\end{align*}
where the last $C$ depends on $n,\kappa, H, \epsilon$ and the set $B(x_0, 3A)$, but independt of the choice of $\xi$. 
Since $\epsilon$ can be any small number, the above inequality implies (\ref{eqn:MC14_4}) by the definition of
Minkowski dimension(c.f. Definition~\ref{dfn:HE08_1}).

Then we set $A=A_i \to \infty$ and define $E \triangleq \cup_{i=1}^{\infty} E_{A_i}$. As a union of countably many measure-zero sets, $E$ is clearly a measure-zero set. 
Clearly, for each $x \in X \backslash E$,  and each $t \in (-\infty, \infty)$, we can always find a large $A_i$ such that $x \in B(x_0, A_i)$ and $t \in (-A_i, A_i)$.
Then it follows that $\varphi_t(x)$ exists and locates in $B(x_0, 2A_i) \cap \mathcal{R} \subset \mathcal{R}$. 
So we finish the proof of Claim~\ref{clm:MA20_1}.

\begin{claim}[\textbf{Flow lines preserve regularity}]
 Suppose $x_0$ is a regular point, then the whole flow line of $\nabla b^+$ initiated from $x_0$ is defined for all time and stays in $\mathcal{R}$. 
\label{clm:GD11_1}  
\end{claim}

Suppose Claim~\ref{clm:GD11_1} was wrong. Without loss of generality,  we can assume that the flow image of $x_0$ hits singularity the first time at $T_0>0$.
Let $\varphi_{t}$ be the time $t$ flow map generated by $\nabla b^{+}$.  Fix $\epsilon$ very small, then $y_0=\varphi_{T_0-\epsilon}(x_0)$ is a regular point. 
By Claim~\ref{clm:MA20_1}, we see that $\varphi_{T_0-\epsilon}$ is well defined away from a measure-zero set $E$. 
Furthermore, it preserves volume element and length element. For each $r \in (0,1)$, we define $\chi_r$ as
\begin{align*}
  \chi_r(x)=
  \begin{cases}
   r-d(x,x_0),  & \textrm{if} \; d(x,x_0)<r \\
   0, &\textrm{if} \; d(x,x_0) \geq r. 
  \end{cases}
\end{align*}
The function $f_{r,\epsilon}=\chi_r \circ \varphi_{T_0-\epsilon}^{-1}$ is defined on $X \backslash E$.  Actually, by choosing $A >>T_0+1$, we know that
$f_{r,\epsilon}$ is defined on $B(y_0, A) \backslash E_{A}$ and vanishes outside $B(y_0, 0.5A)$.  This is a simple application of triangle inequality.
If $x \in B(y_0, A) \backslash E_{A}$ and $d(x,y_0)>0.5A>5(T_0+1)$, then we have
\begin{align*}
    d(\varphi_{T_0-\epsilon}^{-1}(x), x_0) &\geq -d(\varphi_{T_0-\epsilon}^{-1}(x), x) + d(x, x_0)
    \geq -d(\varphi_{T_0-\epsilon}^{-1}(x), x) -d(y_0, x_0)+d(x, y_0)\\
    &\geq -2(T_0-\epsilon) + 5 (T_0+1)> 3T_0+5>r. 
\end{align*}
Then $\chi_r(\varphi_{T_0-\epsilon}^{-1}(x))=0$ by definition of $\chi_r$. 
By the local isometry property of $\varphi_{T_0-\epsilon}$, we obtain that $|\nabla (\chi_r \circ \varphi_{T_0-\epsilon}^{-1})| \leq 1$ on $B(y_0, A) \backslash E_{A}$.
Recall that $E_A$ has codimension at least $2$ by (\ref{eqn:MC14_4}). Then it is clear that $f_{r,\epsilon} \in N_0^{1,2}(X)$. 
Note that $f_{r,\epsilon}$
has a version $\tilde{f}_{r,\epsilon}$ which is globally Lipshitz with Lipshitz constant $1$.  In other words, one can find a measure-zero set $F$ such that $\tilde{f}_{r,\epsilon}=f_{r,\epsilon}$
on $X \backslash F$.  The function $\tilde{f}_{r,\epsilon}$ can be obtained as follows. Let $f_{r,\epsilon}(t)$ be the heat flow solution initiated from $f_{r,\epsilon} \in N^{1,2}(X)$. 
Then $|\nabla f_{r,\epsilon}(t)|$ is a heat subsolution.  
Note that here we used the condition $\dim_{\mathcal{M}} \mathcal{S}<2n-3$ and the weak convexity of $\mathcal{R}$, to guarantee that $|\nabla f_{r,\epsilon}(t)| \in N_{loc}^{1,2}(X)$.
Further details can be found in Appendix~\ref{app:A}. 
By maximum principle(on the space possibly has singularities), we see that $|\nabla f_{r,\epsilon}(t)| \leq 1$ for each 
$t>0$. Let $t_i \to 0$, then the limit of $f_{r,\epsilon}(t_i)$ can be chosen as $\tilde{f}_{r,\epsilon}$. 
Under the help of $\tilde{f}_{r,\epsilon}$, we shall see that $\varphi_{T_0-\epsilon}$ is an isometry from $X \backslash E$ to $X \backslash E$, by further ajusting $E$ with an extra measure-zero set if necessary.
Actually,  if $x \in \partial B(x_0,r) \backslash E$, then $y=\varphi_{T_0-\epsilon}(x)$ is a regular point.  Note that $\tilde{f}_{r,\epsilon}(y_0)=r$
and $\tilde{f}_{r,\epsilon}(y)=0$.  Since $\tilde{f}_{r,\epsilon}$ has uniform global Lipshitz constant $1$, we have $d(y,y_0) \geq r$.  
Therefore,  $\varphi_{T_0-\epsilon}$ is a distance-expanding map from  $X \backslash E$ to $X \backslash E$.   By reversing the position of $x_0,y_0$ and using $-\nabla b^{+}$ to generate flow, it is clear that
$\varphi_{T_0-\epsilon}$ is distance-shrinking.    Combining these two directions, we obtain 
\begin{align*}
    \varphi_{T_0-\epsilon}^{-1} \left( B(y_0, r) \backslash E \right) = B(x_0, r) \backslash E. 
\end{align*}
In particular, we see that 
\begin{align*}
    |B(x_0,r)|=|B(y_0,r)|=|B(\varphi_{T_0-\epsilon}(x_0), r)|. 
\end{align*}
Using triangle inequality and letting $\epsilon \to 0$, we have  $|B(x_0,r)|=|B(\varphi_{T_0}(x_0), r)|$ for each $r \in (0, 1)$. 
However, $x_0$ is regular.  For some $r \in (0, 1)$, we have $\omega_{2n}^{-1}r^{-2n}|B(x_0,r)|>1-\frac{\delta_0}{100}$. 
Same volume ratio estimate hold for the ball $B(\varphi_{T_0}(x_0),r)$. 
Therefore,  $\varphi_{T_0}(x_0)$ is forced to be a regular point by Anderson's gap theorem(c.f. Corollary~\ref{cly:SL23_1}). 
This contradicts the assumption of $T_0$.  Therefore, the proof of Claim~\ref{clm:GD11_1} is complete. \\

Let $N$ be the level set $b^{+}=0$, $N'=N \cap \mathcal{R}$. Then it is clear that
$\mathcal{R}=N' \times \R$. Taking metric completion on both sides, we obtain $X= N \times \R$ as metric space product. 
\end{proof}

Lemma~\ref{lma:HE04_1} should be a special case of Gigli~\cite{Gigli13}.    Its local version is the following lemma. 

\begin{lemma}[\textbf{Metric cone rigidity}]
 Suppose $x_0 \in X$, $\Omega=B(x_0,1)$. Then the following conditions are equivalent.
 \begin{enumerate}
 \item Volume ratio same on scales $0.5$ and $1$, i.e., we have
     \begin{align}
        2^{2n}|B(x_0,0.5)|=|B(x_0,1)|.   \label{eqn:HE10_1}
      \end{align}
 \item $\Omega$ is a volume cone, i.e., for every $0<r_1<r_2<1$, we have
      \begin{align}
        r_1^{-2n}|B(x_0,r_1)|=r_2^{-2n}|B(x_0,r_2)|.   \label{eqn:HE04_1}
      \end{align}
  \item $\frac{r^2}{2}$ is the unique weak solution of the Poisson equation
      \begin{align}
        \Delta u=2n, \; \textrm{in} \; \Omega; \quad \left.\left(u-\frac{r^2}{2} \right)\right|_{\partial \Omega}=0.
      \label{eqn:HE03_4}
      \end{align}
 \item  $\frac{r^2}{2}$ induces local metric cone structure on $\Omega$. In other words, on $\Omega \backslash \mathcal{S}$, we have
      \begin{align*}
          Hess_{\frac{r^2}{2}} -g \equiv 0.
      \end{align*}
 \end{enumerate}
\label{lma:HE04_2}
\end{lemma}

\begin{proof}
 \textit{1 $\Rightarrow$ 2:}
Let $A(r)$ be the  ``area" ratio function in Corollary~\ref{cly:GD08_1}, i.e.,  $ A(r)= \frac{|\partial B(x_0,r)|}{r^{2n-1}}$. 
By Corollary~\ref{cly:GD08_1}, we know $A(r)$ is defined almost everywhere and is non-increasing on its domain. 
Note that
\begin{align*}
  \frac{d}{dr} \left( \frac{|B(x_0,r)|}{r^{2n}}\right)=\frac{|\partial B(x_0,r)|}{r^{2n}} -\frac{2n}{r} \frac{|B(x_0,r)|}{r^{2n}}
    =\frac{2n}{r} \left\{\frac{A(r)}{2n} -\frac{|B(x_0,r)|}{r^{2n}} \right\} \leq 0.
\end{align*}
Combining  (\ref{eqn:HE10_1})  and (\ref{eqn:GD08_5}), we have $ \frac{A(r)}{2n}-\frac{|B(x_0,r)|}{r^{2n}} \equiv 0$ for a.e. $r \in (0.5, 1)$. 
In particular, we have
\begin{align*}
   |B(x_0,1)|=\frac{A(1)}{2n}=\int_0^{1} A(1) r^{2n-1}dr,
\end{align*}
where $A(1)$ is understood as $\displaystyle \lim_{r \to 1^{-}} A(r)$. 
On the other hand, it follows from (\ref{eqn:GD08_4}) that
\begin{align*}
  |B(x_0,1)|=\int_0^{1} A(r) r^{2n-1} dr.
\end{align*}
So we have
\begin{align*}
  \int_0^{1} (A(r)-A(1)) r^{2n-1} dr=0.
\end{align*}
Note that $A(r)$ is a non-increasing function. So the above equality means that
\begin{align*}
 A(1) \equiv A(r) \equiv \lim_{r \to 0^{+}} A(r).
\end{align*}
It follows that $|B(x_0,r)| = \frac{A(1)}{2n}r^{2n}$ for every $0<r<1$. In particular,  $B(x_0,1)$ is a volume cone.

 \textit{2 $\Rightarrow$ 3:}   Suppose $u$ is the unique solution of the Poisson equation (\ref{eqn:HE03_4}), we need to show that
 $u\equiv \frac{r^2}{2}$. By uniqueness of weak solutions, it suffices to show that $\frac{r^2}{2}-u$ is harmonic on $\Omega$, i.e., for every
 $\varphi \in N_c^{1,2}(\Omega)$, we have $\int_{\Omega} \varphi \Delta \left( \frac{r^2}{2}-u\right)=0$. 
 By rescaling, we can also assume $0 \leq \varphi \leq 1$. 
 Fix such a $\varphi$, we can choose $\epsilon$ small such that the support of $\varphi$ is contained in $B(x_0, 1-\epsilon)$.  Define
 \begin{align*}
    \eta(x)=
 \begin{cases}
  1, & \textrm{if}  \quad d(x,x_0)<1-\epsilon, \\
  \frac{1-d(x,x_0)}{\epsilon}, & \textrm{if} \quad 1-\epsilon \leq d(x,x_1) \leq 1. 
 \end{cases}    
 \end{align*}
 Note that $\eta \in N_0^{1,2}(\Omega)$, $\frac{r^2}{2}-u$ is superharmonic on $\Omega$. It follows from integration by parts that
 \begin{align*}
   \int_X \eta \Delta\left( \frac{r^2}{2} -u\right) =-2n\int_{\Omega} \eta +\frac{1}{\epsilon} \int_{B(x_0,1) \backslash B(x_0,1-\epsilon)} r \geq -O(\epsilon),
 \end{align*}
 where we used volume cone condition in the last step. 
 Thus, we have
 \begin{align*}
  0 \geq  \int_{\Omega} \varphi \Delta \left( \frac{r^2}{2}-u\right)
  =\int_{\Omega} \eta \Delta \left( \frac{r^2}{2}-u\right)+\int_{\Omega} (\varphi-\eta) \Delta \left( \frac{r^2}{2}-u\right)
  \geq \int_{\Omega} \eta \Delta \left( \frac{r^2}{2}-u\right) \geq -O(\epsilon). 
 \end{align*}
 Let $\epsilon \to 0$, we obtain $\int_{\Omega} \varphi \Delta \left( \frac{r^2}{2}-u\right)=0$. Consequently, $\frac{r^2}{2}-u$ is harmonic by the arbitrary choice 
 of $\varphi$.

  \textit{3 $\Rightarrow$ 4:}
   Since $\frac{r^2}{2}$ solves the Poisson equation with right hand side a constant,  by standard bootstrapping argument for  elliptic equation,
   we see that $\frac{r^2}{2}$ is a smooth function on $\Omega \backslash \mathcal{S}$.
   Clearly, we have $\left|\nabla \frac{r^2}{2}\right|^2=\frac{r^2}{2}$. Taking  Laplacian on both sides, Weitzenb\"{o}ck formula  yields that
   $\displaystyle \left|Hess_{\frac{r^2}{2}} \right|^2=2n$, which in turn implies that
 \begin{align*}
    \left|Hess_{\frac{r^2}{2}} -g\right|^2= \left|Hess_{\frac{r^2}{2}} \right|^2-2\Delta \frac{r^2}{2}+2n
   = \left|Hess_{\frac{r^2}{2}} \right|^2-2n=0.
 \end{align*}
 Therefore, on $\Omega \backslash \mathcal{S}$, we have $Hess_{\frac{r^2}{2}} -g\equiv 0$ in the classical sense.
 Consequently,  $\nabla \frac{r^2}{2}$ is a conformal Killing field.
 Similar to the proof of Claim~\ref{clm:GD11_1}, one can show that the flow generated by $\nabla \frac{r^2}{2}$ preserves regularity.  
 Hence it is clear that $\Omega \backslash \mathcal{S}$ has a local metric cone structure, whose completion implies that $\overline{\Omega}$ is a unit ball in a metric cone.

 \textit{4 $\Rightarrow$ 1:}
  For each $0<r<1$, note that $|B(x_0,r)|=|B(x_0,r) \backslash \mathcal{S}|$.  Note the flow generated by $\nabla \frac{r^2}{2}$ preserves regularity.
  More precisely, we have
  \begin{align*}
      \mathcal{L}_{\nabla \frac{r^2}{2}} g=g, \quad \mathcal{L}_{\nabla \frac{r^2}{2}} d\mu= 2n d\mu.
  \end{align*}
  Then (\ref{eqn:HE04_1}) follows from the integration of the above equation along flow lines.
\end{proof}

\begin{lemma}[\textbf{K\"ahler cone splitting}]
 Suppose $X \in \widetilde{\mathscr{KS}}^*(n,\kappa)$ is a metric cone with vertex $x_0$.  Then we can find a metric cone $C(Z)$ with vertex $z^*$ such that
 \begin{align*}
     X= \C^{n-k} \times C(Z), \quad x_0=(0, z^*), \quad   2 \leq k \leq n.
 \end{align*}
 Moreover, there is no straight line in $C(Z)$ passing through $z^*$.
\label{lma:HD30_1}
\end{lemma}

\begin{proof}
   It suffices to show that if $X$ splits off a real straight line $\R$, then it splits off a complex line $\C$.
   In fact, if there is a straight line passing through $x_0$, we can find a function $h$ which is the Buseman function determined by the line.
   Therefore, $\nabla h$ is a parallel vector field with $|\nabla h| \equiv 1$.   The K\"ahler condition implies that
   $J \nabla h$ is another parallel vector field satisfying $|J \nabla h| \equiv 1$ on $\mathcal{R}=X \backslash \mathcal{S}$.
   On the regular set,  define function
   \begin{align*}
     u=\left\langle J\nabla h, \nabla \frac{r^2}{2} \right\rangle,
   \end{align*}
   where $r$ is the distance to the vertex $x_0$.
   Metric cone condition implies that $Hess_{\frac{r^2}{2}}=g$. Since $J\nabla h$ is parallel, we see that
   \begin{align*}
     \nabla u= Hess_{\frac{r^2}{2}} (J\nabla h, \cdot)=J \nabla h.
   \end{align*}
   Recall that $Hess_h \equiv 0$. Taking gradient of the above equation implies that $Hess_u \equiv 0$.  This forces that $\nabla u=J \nabla h$
   is also a splitting direction.  Note that although $u$ is only defined on $\mathcal{R}$, which is not complete.
   However, we can bypass this difficulty as done in the proof of  Lemma~\ref{lma:HE04_1}, since $\nabla u$ is a Killing field preserving regularity.
   Therefore, we obtain a splitting factor $\C$.  Since $J\nabla u=-\nabla h$, the space spanned by $\nabla u$ and
   $J \nabla u$ is closed under the $J$-action.  This induces the $J$-action closedness of the split linear space, which then must be $\C^{n-k}$
   for some integer $k$.  Because $X$ is not $\C^n$, we know the singular set is not empty, whose dimension restriction forces that $k \geq 2$.
\end{proof}

For each rigidity property in Lemma~\ref{lma:HE04_1}-Lemma~\ref{lma:HD30_1},  there should exist an ``almost" version.
For example,  Lemma~\ref{lma:HE04_2} basically says that a volume cone implies a metric cone.
Hence the ``almost" version is that for a unit geodesic ball $B(x_0,1)$ whose volume ratio function $r^{-2n}|B(x_0,r)|$ is very close to a constant
function on $[0,1]$, then after proper rescaling, each ball $B(x_0, r)$ is very close to $B(x_0,1)$ in the Gromov-Hausdorff topology.
The basic idea is expressed clearly in~\cite{CCWarp}.
We only interpret what they did.  Actually, if volume ratio is almost a constant, then it is expected that $|Hess_{\frac{r^2}{2}}-g|$ has a small
$L^2$-norm.  However, since the regularity of distance function $r$ is bad, one should replace $\frac{r^2}{2}$ by an approximation function, which is
very close to $\frac{r^2}{2}$ in $N^{1,2}$-norm on one hand, and has excellent regularity on the other hand.
Such approximation function is nothing but the solution of the Poisson equation  (\ref{eqn:HE03_4}).
For the purpose of developing ``almost" rigidity properties, one need some technical preparation, which will be listed as Lemmas.
Note that the space $\widetilde{\mathscr{KS}}(n,\kappa)$ has scaling invariance.
Therefore, we can always let the scale we are interested in to be $1$, to simplify the notations.

In view of Proposition~\ref{prn:HC29_1}, we can define many auxiliary radial functions, as in the classical case for Riemannian manifold(c.f.~\cite{CCWarp}).
For each $0<r<R<\infty$, define
\begin{align}
  &\underline{U}(r) \triangleq \frac{r^2}{4n},
  \quad
  \underline{G}(r) \triangleq \frac{r^{2-2n}}{2n(2n-2)\omega_{2n}}, \\
  &\underline{U}_R \triangleq \frac{r^2-R^2}{4n}, \quad  \underline{G}_R \triangleq \frac{r^{2-2n}-R^{2-2n}}{2n(2n-2)\omega_{2n}}, \\
  &\underline{L}_R \triangleq \frac{r^{2-2n}R^{2n}-R^{2}}{2n(2n-2)} + \frac{r^2-R^2}{4n}.
\label{eqn:HE08_1}
\end{align}
Then by Proposition~\ref{prn:HC29_1} and direct calculation, we have the following lemma.

\begin{lemma}[\textbf{Existence of good radial comparison functions}]
Suppose $x_0 \in X$.  Let $r(x)=d(x,x_0)$ and
define $\underline{U}_1(x)=\underline{U}_1(r(x))$, $\underline{G}_1=\underline{G}_1(r(x))$ and $\underline{L}_1(x)=\underline{L}(r(x))$
as done in (\ref{eqn:HE08_1}).  Then we have
\begin{align*}
 &\Delta \underline{U}_1 \leq 1, \textrm{on} \; X;  \quad \underline{U}_1|_{\partial B(x_0,1)}=0.  \\
 &\Delta \underline{G}_1 \geq 0, \textrm{on} \; B(x_0,1) \backslash \{x_0\}; \quad \underline{G}_1 |_{\partial B(x_0,1)}=0. \\
 &\Delta \underline{L}_1 \geq 0, \textrm{on} \; B(x_0,1) \backslash \{x_0\}; \quad \underline{L}_1 |_{\partial B(x_0,1)}=0.
\end{align*}
\label{lma:HE08_1}
\end{lemma}

Lemma~\ref{lma:HE08_1} is used to improve the maximum principle.  Same as that done by Abresch-Gromoll (c.f. Proposition 2.3 of~\cite{AbGr}), we
obtain the following estimate of excess  function.

\begin{lemma}[\textbf{Abresch-Gromoll type estimate}]
Suppose $x_0 \in X$,
 $\gamma$ is a line segment centered at $x_0$ with length $2$, end points $p_{+}$ and $p_{-}$.
 Let $e(x)$ be the excess function $d(x,p_{+})+d(x,p_{-})-2$.
 Then we have
 \begin{align}
    \sup_{x \in B(x_0,\epsilon)} e(x) \leq C \epsilon^{\frac{2n}{2n-1}}   \label{eqn:HE08_2}
 \end{align}
 for each $\epsilon \in (0, 1)$ and some universal constant $C=C(n)$.
\label{lma:HE05_1}
\end{lemma}

Lemma~\ref{lma:HE08_1} can also be applied to construct good cutoff functions.

\begin{lemma}[\textbf{Cutoff functions on annulus}]
Suppose $x_0 \in X$, $0<\rho<1<\infty$.
Then there exists a function $\phi: X \to [0,1]$ such that
\begin{align*}
  &\phi \in C^{\infty}(B(x_0,1) \backslash \mathcal{S}),  \quad \supp \phi \Subset B(x_0, 1), \quad \phi \equiv 1 \; \textrm{on} \; B(x_0,\rho), \\
  &|\nabla \phi| \leq c(n, \rho), \qquad \quad |\Delta \phi| \leq c(n, \rho), \quad \textrm{on} \; B(x_0,1) \backslash \mathcal{S}.
\end{align*}
Furthermore, for each pair $\rho_1, \rho_2$ satisfying $0<\rho_1<\frac{1}{2}<2<\rho_2<\infty$, there exists a function $\phi: X \times [0,1]$ such that
\begin{align*}
  &\phi \in C^{\infty}((B(x_0,\rho_2) \backslash \overline{B(x_0, \rho_1 )}) \cap \mathcal{R}),
  \quad \supp \phi \Subset B(x_0, \rho_2) \backslash \overline{B(x_0,\rho_1)}, \\
  &\phi \equiv 1 \; \textrm{on} \; B\left(x_0,\frac{\rho_2}{2} \right) \backslash \overline{B(x_0, 2\rho_1)}, \\
  &|\nabla \phi| \leq c(n, \rho_1,\rho_2), \qquad \quad |\Delta \phi| \leq c(n, \rho_1,\rho_2),
   \quad \textrm{on} \; (B(x_0,\rho_2) \backslash \overline{B(x_0, \rho_1 )}) \cap \mathcal{R}.
\end{align*}
\label{lma:HE04_3}
\end{lemma}

The proof of Lemma~\ref{lma:HE04_3} is based on the maximum principle, solvability of Poisson equation and the fact that  $\Delta \underline{L}_1 \geq 1$
and $\Delta \underline{U}_{R'} \geq 1$ for each $R'>0$.
With these properties, one can compare  $\underline{L}_1$ with the Poisson equation solution $f$ which has same boundary value as $\underline{L}_1$.
Then construct cutoff function based on the value of $f$.  Since the proof follows that of~\cite{CCWarp} verbatim, we omit the details here.

\begin{lemma}[\textbf{Harmonic approximation of local Buseman function}]
There exists a constant $c=c(n)$ with the following properties.

 Suppose $x_0 \in X$,
 $\gamma$ is a line segment centered at $x_0$ with length $2$, end points $p_{+}$ and $p_{-}$, $\epsilon$ is an arbitrary small positive number,
 say $0<\epsilon<0.1$.
 In the ball $B(x_0,4\epsilon)$, define local Buseman functions
 \begin{align*}
    b_{+}(x)=d(x, p_{+})-d(x_0, p_{+}),   \quad b_{-}(x)=d(x,p_{-})-d(x_0, p_{-}).
 \end{align*}
 Let $u_{\pm}$ be the harmonic functions in $B(x_0,4\epsilon)$ such that $\left.\left( u_{\pm}-b_{\pm} \right) \right|_{\partial B(x_0,4\epsilon)}=0$.
 Let $u$ be one of $u_{\pm}$ and $b$ be the corresponding $b_{\pm}$ respectively. Then we have
 \begin{itemize}
 \item  $|u-b| \leq c\epsilon^{1+\alpha}$.
 \item  $\fint_{B(x,\epsilon)} |\nabla (u-b)|^2 \leq c\epsilon^{\alpha}$.
 \item  $\fint_{B(x,\epsilon)} |Hess_{u}|^2 < c\epsilon^{-2+\alpha}$.
 \end{itemize}
 Here $\alpha=\alpha(n)$ is a universal constant, which can be chosen as $\frac{1}{2n-1}$.
\label{lma:HE05_2}
\end{lemma}

\begin{proof}
   For simplicity, we assume $u=u_{+}$ and $b=b_{+}$.

   The pointwise estimate of $|u-b|$ follows from maximum principle and the the excess estimate (\ref{eqn:HE08_2}), same as traditional case.

  We proceed to show the integral estimate of $|\nabla (u-b)|$.
  Note that $b(x)=r(x)-d(x_0,p_{+})$, where $r(x)=d(x, p_{+})$.  It follows from rescaling that
   \begin{align*}
        \fint_{B(x_0,4\epsilon)} |\Delta b| = \fint_{B(x_0,4\epsilon)} |\Delta r| < C \epsilon^{-1}.
   \end{align*}
   Actually, by the fact $\Delta r \leq \frac{2n-1}{r}$, the estimate of $\int |\Delta r|$ is reduced to the estimate of
   $\int \Delta r$. However,  $\int \Delta r$ can be bounded by integration by parts,
   modulo some technical discussion around  the generalized cut locus and singular set $\mathcal{S}$.  Due to the high
   codimension of $\mathcal{S}$, the integral of $\Delta r$ around of $\mathcal{S}$ can be ignored. Then we return
   to the smooth manifold case, which is discussed clearly in~\cite{Cheeger01}.

   Clearly, $u-b \in N_0^{1,2}(B(x_0,4\epsilon))$. Hence integration by parts, Proposition~\ref{prn:HD13_2},
   applies and we have
   \begin{align*}
     \fint_{B(x_0,4\epsilon)} |\nabla (u-b)|^2 &= \fint_{B(x_0,4\epsilon)} (u-b) \Delta (b-u)= \fint_{B(x_0,4\epsilon)} (u-b) \Delta b
         \leq  C\epsilon^{1+\alpha} \fint_{B(x_0,4\epsilon)} |\Delta b|  < C\epsilon^{\alpha}.
   \end{align*}
Note that $u$ is harmonic in $B(x_0,4\epsilon)$.    Weitzenb\"{o}ck formula implies that
\begin{align*}
  \frac{1}{2} \Delta \left(|\nabla u|^2-1\right)=\frac{1}{2} \Delta |\nabla u|^2=|Hess_{u}|^2 \geq 0
\end{align*}
in the classical sense on $B(x_0, 4\epsilon) \backslash \mathcal{S}$.  By extension property of subharmonic function, Proposition~\ref{prn:HD16_2}, we see that
$|\nabla u|^2 \in N_{loc}^2(B(x_0, 4\epsilon))$.  Let $\phi$ be a cutoff function vanishes on $\partial B(x_0, 4\epsilon)$ and equivalent to $1$ on $B(x_0,\epsilon)$,
with $\epsilon |\nabla \phi|$
and $\epsilon^2 |\Delta \phi|$ bounded as in Lemma~\ref{lma:HE04_3}.   Clearly, $\phi \in N_c^{1,2}(B(x_0, 4\epsilon))$.
Therefore, it follows from  integration by parts, Proposition~\ref{prn:HD13_2}, that
\begin{align*}
2\fint_{B(x_0, 4\epsilon)} \phi |Hess_u|^2=\fint_{B(x_0, 4\epsilon)} \phi \Delta \left(|\nabla u|^2-1\right)=\fint_{B(x_0, 4\epsilon)} \left(|\nabla u|^2-1 \right) \Delta \phi.
\end{align*}
Consequently, we obtain
\begin{align*}
   \fint_{B(x_0,\epsilon)} |Hess_{u}|^2 \leq \fint_{B(x_0, 4\epsilon)} \phi |Hess_u|^2 \leq C \epsilon^{-2} \fint_{B(x_0, 4\epsilon)} \left||\nabla u|^2-1\right| \leq C\epsilon^{\alpha-2}.
\end{align*}
\end{proof}

Note that Lemma~\ref{lma:HE05_2} implies almost splitting property already.  Therefore, it is generalization of Lemma~\ref{lma:HE04_1}, the splitting property.
Not surprisingly, one can use Lemma~\ref{lma:HE05_2} to prove Lemma~\ref{lma:HE04_1}, at least formally.
Actually, if there is a line with length $2L$ centered at $x_0$, then in the unit ball $B(x_0,1)$,  it follows from Lemma~\ref{lma:HE05_2} that
\begin{align*}
  |u-b|<cL^{-\alpha}, \quad  \fint_{B(x_0,1)} |\nabla(u-b)|<cL^{-\alpha},   \quad \fint_{B(x_0,1)} |Hess_u|^2<cL^{-\alpha}.
\end{align*}
Let $L \to \infty$, we see that $Hess_{u} \equiv 0$ on $\mathcal{R}$.

From the proof of Lemma~\ref{lma:HE05_2}, it is clear that the key to  obtain smallness of $|Hess_u|^2$ is the integration by parts, which is checked in our case.
For smooth Riemannian manifold, the approximation in Lemma~\ref{lma:HE05_2} was improved by Colding and Naber in~\cite{ColdNa}.
The essential difference is that they chose parabolic approximation functions, instead of harmonic approximations.
Suppose $\gamma$ is a line segment with length $2$, centered around $x_0$, with end points $p_{+}$ and $p_{-}$.
 Then one can construct cutoff functions $\psi$ such that it vanishes outside
$B(x_0, 8)$ and inside $B(p_{+}, 0.1)$ and $B(p_{-}, 0.1)$, and equals $1$ on $B(x_0,4) \backslash (B(p_{+}, 0.2) \cup B(p_{-}, 0.2))$. Moreover, we have pointwise
bound of $|\Delta \psi|$ and $|\nabla \psi|$.  Then for $b_{\pm}$, we can run heat flow starting from $\psi b_{\pm}$ to obtain solution $h_{t,\pm}$.
Then the function $h_{t,\pm}$ is a better approximation function of  $b_{\pm}$ on the scale around $\sqrt{t}$.
The extra technical tools needed for Colding-Naber's argument beyond the harmonic approximation consists of an a priori bound of heat kernel, and the construction of
cutoff function with the properties as mentioned above.  However, in light of Proposition~\ref{prn:HD07_1} and Lemma~\ref{lma:HE04_3}, both tools are available in our setting.
Therefore, we can develop our version of the parabolic approximation estimate, Theorem 2.19 of~\cite{ColdNa},  in the current case.

\begin{lemma}[\textbf{Parabolic approximation of local Buseman function}]
There exist two constants $c=c(n), \bar{\epsilon}=\bar{\epsilon}(n)$ with the following properties.

Suppose $x_0 \in X$, $\gamma$ is a line segment whose center point locates in $B(x_0,0.2)$, with end points $p_{+}$ and $p_{-}$, with length $2$.
Let $h_{t}$ be the heat approximation of $b$ which is one of $b_{\pm}$.
Suppose the excess value $d(x_0,p_{+})+d(x_0,p_{-})-2<\epsilon^2$ for some $\epsilon \in (0,\bar{\epsilon})$.
Then there exists $\lambda \in [0.5, 2]$ such that
 \begin{itemize}
 \item $|h_{\lambda \epsilon^2} - b| \leq c \epsilon^2$.
 \item $\fint_{B(x,\epsilon)} ||\nabla h_{\lambda \epsilon^2}|^2-1| \leq c\epsilon$.
 \item $\int_{0.1}^{1.9} \fint_{B(x,\epsilon)} ||\nabla h_{\lambda \epsilon^2}|^2-1| \leq c\epsilon^2$.
 \end{itemize}
Most importantly, we have
 \begin{align*}
  \int_{0.1}^{1.9} \fint_{B(\gamma(s), \epsilon)} |Hess_{h_{\lambda \epsilon^2}}|^2 \leq c.
 \end{align*}
\label{lma:HE04_4}
\end{lemma}

Note that we did not formulate the parabolic approximation in the most precise way. For example,  $\gamma$ need not to be a geodesic, an $\epsilon$-geodesic
suffices.  Interested readers are referred to \cite{ColdNa} for the most general version.

According to the discussion form Lemma~\ref{lma:HE04_3} to Lemma~\ref{lma:HE04_4}, it is quite clear that the integral estimate of approximation functions
can be obtained in the same way as the Riemannian manifold case, provided the following properties.
\begin{itemize}
\item Almost super-harmonicity of distance functions, Proposition~\ref{prn:HC29_1}.
\item Bishop-Gromov volume comparison, Proposition~\ref{prn:HD19_1}.
\item Strong maximum principle for subharmonic functions, Proposition~\ref{prn:HC29_5}.
\item Integration by parts, Proposition~\ref{prn:HD13_2}.
\item Existence of excellent cutoff function, Lemma~\ref{lma:HE04_3}.
\end{itemize}
Since all of these properties are checked in our situation, we can follow the route of Cheeger-Colding to obtain the following properties, almost line by line.

\begin{lemma}[\textbf{Approximation slices}]
Suppose $x_0 \in X$, $\gamma_1, \gamma_2, \cdots, \gamma_k$ are $k$ line segments with length $2L>>2$ such that the center point of
$\gamma_k$ locates in $B(x_0,1)$ for each $k$.  Furthermore, these lines are almost perpendicular to each other, i.e., the Gromov-Hausdorff distance between
$\gamma_1 \cup \gamma_2 \cdots \cup \gamma_k$ and $\tilde{\gamma}_1 \cup \tilde{\gamma}_2 \cup \cdots \tilde{\gamma}_k$ is bounded by
$L\psi(L^{-1})$, where $\tilde{\gamma}_i$ is the line segment on the $i$-th coordinate axis of $\R^k$, centered at the origin and with length $2L$, $\psi$ is a nonnegative monotonically increasing function satisfying $\psi(0)=0$. 
 Suppose the end points of $\gamma_k$ are $p_{i,+}$ and $p_{i,-}$.
Let $b_{i,\pm}$ be the corresponding local Buseman functions with respect to $\gamma_i$.   Let $u_i$ be the harmonic function
on $B(x_i,4)$ with the same value as $b_{i,+}$ on $\partial B(x_0,4)$.  Then we have
\begin{align*}
    \int_{B(x_0,1)} \left\{ \sum_{1 \leq i \leq k} |\nabla u_i -1|^2 + \sum_{1 \leq i<j\leq k}|\langle \nabla u_i, \nabla u_j\rangle| + \sum_{1 \leq i \leq k} |Hess_{u_i}|^2\right\}
    \leq  \bar{\psi}(L^{-1}), 
  \end{align*}
 where $\bar{\psi}$ is also a nonnegative monotonically increasing function satisfying $\bar{\psi}(0)=0$, depending on $\psi$.  
\label{lma:HE07_1}
\end{lemma}

Let $\vec{u}=(u_1, u_2, \cdots, u_k)$, we can regard $\vec{u}$ as an almost submersion from $B(x_0,1)$ to its image on
$\R^k$. Consequently,  slice argument can be set up as that in~\cite{CCT}.  The slice argument together with the Chern-Simons theory can improve the behavior of the singular set $\mathcal{S}$.  A more fundamental application of the slice argument is to set up the following volume convergence property, as done in~\cite{ColdVol}.

\begin{proposition}[\textbf{Volume continuity}]
For every $(X,x_0,g) \in  \widetilde{\mathscr{KS}}(n,\kappa)$ and $\epsilon>0$, there is a constant $\xi=\xi(X,\epsilon)$ such that
\begin{align*}
   \left|  \log \frac{|B(y_0,1)|}{|B(x_0,1)|} \right|<\epsilon
\end{align*}
for any $(Y,y_0,h) \in \widetilde{\mathscr{KS}}(n,\kappa)$ satisfying $d_{PGH}((X,x_0,g), (Y,y_0,h))<\xi$.
\label{prn:HD22_1}
\end{proposition}

Recall that $d_{PGH}$ means the pointed-Gromov-Hausdorff distance. In Proposition~\ref{prn:HD22_1}, the inequality $d_{PGH}((X,x_0,g), (Y,y_0,h))<\xi$ means that the Gromov-Hausdorff distance between $B(x_0, \xi^{-1}) \subset X$ and $B(y_0, \xi^{-1}) \subset Y$ is less than $\xi$. 

Applying the same argument as in~\cite{CCWarp}, we obtain the rigidity of almost volume cones.

\begin{proposition}[\textbf{Almost volume cone implies almost metric cone}]
For each $\epsilon>0$, there exists $\xi=\xi(n,\epsilon)$ with the following properties.

Suppose $(X,x_0,g) \in  \widetilde{\mathscr{KS}}(n,\kappa)$ satisfies $\displaystyle  \frac{|B(x_0,2)|}{ |B(x_0,1)|} \geq (1-\epsilon)2^{2n}$,
then there exists a metric cone over a length space $Z$,  with vertex $z^*$ such that
\begin{align*}
  \diam(Z) < \pi + \xi, \quad
  d_{GH} \left( B(x_0,1), B(z^*, 1)\right) <\xi.
\end{align*}
Furthermore, $\displaystyle \lim_{\epsilon \to 0} \xi(n,\epsilon)=0$. 
\label{prn:HD22_2}
\end{proposition}

Similar to  Lemma 9.14 of~\cite{CCT},   we obtain the almost K\"ahler cone splitting, based on Proposition~\ref{prn:HD22_2}.

\begin{proposition}[\textbf{Almost K\"ahler cone splitting}]
For each $\epsilon>0$, there exists $\xi=\xi(n,\epsilon)$ with the following properties.

 Suppose $X \in \widetilde{\mathscr{KS}}(n,\kappa)$, $x_0 \in X$,  $b$ is a Lipschitz function on $B(x_0,2)$ satisfying
 \begin{align*}
  \sup_{B(x_0,2) \backslash \mathcal{S}} |\nabla b| \leq 2,  \quad \fint_{B(x_0,2) \backslash \mathcal{S}} |Hess_b|^2 \leq \epsilon^2.
   \end{align*}
 Suppose also $\displaystyle \frac{|B(x_0,2)|}{|B(x_0,1)|} \geq (1-\epsilon) 2^{2n}$, i.e., $B(x_0,1)$ is an almost volume cone.
  Then there exists a Lipschitz function $\tilde{b}$ on $B(x_0,1)$ such that
  \begin{align*}
     \sup_{B(x_0,1) \backslash \mathcal{S}} \left|\nabla \tilde{b}\right| \leq 3,
     \quad \fint_{B(x_0,1) \backslash \mathcal{S}} \left|\nabla \tilde{b}-J \nabla b\right|^2 \leq \xi.
  \end{align*}
\label{prn:HE08_1}
\end{proposition}

\subsection{Volume radius}
Anderson's gap theorem implies that one can improve regularity of the  very interior part of a geodesic ball whenever the volume ratio of the geodesic ball is very close
to the Euclidean one.    This suggests us to define the volume radius as follows.

 \begin{definition}
Let $\delta_0$ be  the Anderson constant.
Suppose $X \in \widetilde{\mathscr{KS}}(n,\kappa)$, $x_0 \in X$.Then we define
  \begin{align*}
    &\Omega_{x_0} \triangleq \left\{ r| r>0, r^{-2n}|B(x_0,r)|\geq (1-\delta_0) \omega_{2n} \right\}.\\
    &\mathbf{vr}(x_0) \triangleq
    \begin{cases}
      \sup \Omega_{x_0}, &\textrm{if} \; \Omega_{x_0} \neq \emptyset, \\
      0, &\textrm{if} \; \Omega_{x_0}=\emptyset.
    \end{cases}
  \end{align*}
We call $\mathbf{vr}(x_0)$ the volume radius of the point $x_0$.
\label{dfn:SC17_1}
\end{definition}
According to this definition,  a point is regular if and only if its volume radius is positive.
On the other hand, if the space is not $\C^n$, then every point has a finite volume radius by a generalized Anderson's gap theorem.

\begin{proposition}[\textbf{Euclidean space by $\mathbf{vr}$}]
  Suppose $X \in \widetilde{\mathscr{KS}}(n)$ and $\mathbf{vr}(x_0)=\infty$ for some $x_0 \in X$, then $X$ is isometric to the Euclidean space $\C^{n}$.
\label{prn:SC17_1}
\end{proposition}
\begin{proof}
 Fix an arbitrary point $x \in X$, then volume comparison implies that
 \begin{align*}
   \mathrm{v}(x) \geq \lim_{r \to \infty} \omega_{2n}^{-1} r^{-2n} |B(x,r)|=\mathrm{avr}(X) \geq 1-\delta_0.
 \end{align*}
 Therefore, $x$ is a regular point. Since $x$ is arbitrarily chosen, we see that $X \in \mathscr{KS}(n)$.  Then the statement follows from Anderson's gap theorem.
\end{proof}

A local version of Proposition~\ref{prn:SC17_1} is the following local Harnack inequality of $\mathbf{vr}$.

\begin{proposition}[\textbf{Local Harnack inequality of volume radius}]
There is a constant $\tilde{K}=\tilde{K}(n)$ with the following properties.

Suppose $x \in X \in \widetilde{\mathscr{KS}}^{*}(n)$, $r=\mathbf{vr}(x)>0$, then we have
\begin{align}
   \tilde{K}^{-1}r  \leq \mathbf{vr} \leq \tilde{K}r
\label{eqn:SC26_8}
\end{align}
in the ball $B(x, \tilde{K}^{-1}r)$.  Moreover, for every $\rho \in (0, \tilde{K}^{-1}r)$, $y \in B(x, \tilde{K}^{-1}r)$, we have
\begin{align}
   &\omega_{2n}^{-1}\rho^{-2n} |B(y,\rho)| \geq 1-\frac{\delta_0}{100}, \label{eqn:SC26_9}\\
   & |Rm|(y) \leq \tilde{K}^{2} r^{-2}, \label{eqn:SC27_3}\\
   & inj(y) \geq \tilde{K}^{-1} r. \label{eqn:SC27_4}
\end{align}
\label{prn:SC26_7}
\end{proposition}

\begin{proof}
 It follows from Bishop volume comparison, Anderson's gap theorem and a compactness argument.
 Actually, by adjusting $\tilde{K}$ if necessary, it suffices to show (\ref{eqn:SC27_3}).  
 We argue by contradiction. Suppose (\ref{eqn:SC27_3}) were wrong, by point-selecting and rescaling, we can find a sequence of $L_i \to \infty$
 and Ricci-flat spaces $(X_i, x_i, g_i)  \in \widetilde{\mathscr{KS}}^{*}(n)$ such that 
 \begin{align*}
      |Rm|(x_i)=1,  \quad \sup_{x \in B(x_i, L_i)} |Rm|(x) \leq 2,  \quad \omega_{2n}^{-1} L_i^{-2n} |B(x_i, L_i)| \geq 1-\delta_0.
 \end{align*}
 Improving regularity property of Ricci-flat metrics implies higher order estimate of $Rm$ in the balls $B(x_i, L_i-1)$.
 Therefore, we can take smooth convergence limit(c.f.~\cite{Ha95}): 
 \begin{align*}
    (X_i, x_i, g_i)  \longright{C^{\infty}-Cheeger-Gromov}  (X_{\infty}, x_{\infty}, g_{\infty}). 
 \end{align*}
 The limit space satisfying $|Rm|(x_{\infty})=1$ and  $\textrm{avr}(X_{\infty}) \geq 1-\delta_0$, which is impossible by Anderson's gap theorem or Proposition~\ref{prn:SC17_1}. 
\end{proof}

On a Ricci-flat geodesic ball, it is well known that $|Rm|$ bound implies
bound of $|\nabla^k Rm|$ for each positive integer $k$ in a smaller geodesic ball.
So (\ref{eqn:SC27_3}) immediately yields the following corollary.

\begin{corollary}[\textbf{Improving regularity property of volume radius}]
 There is a small positive constant $c_a=c_a(n)$ with the following properties.

 Suppose $x \in X \in \widetilde{\mathscr{KS}}^{*}(n)$, $\mathbf{vr}(x) \geq r>0$, then we have
 \begin{align}
   r^{2+k}|\nabla^k Rm|(y) \leq c_a^{-2}, \quad \forall \; y \in B(x,c_a r), \quad 0 \leq k \leq 5.
 \label{eqn:SL23_8}
 \end{align}
\label{cly:SL23_1}
\end{corollary}

 In Ricci bounded geometry, harmonic radius (c.f.~\cite{An90}) plays an important role.   A point $x$ is defined to have harmonic radius at least $r$ if on the smooth geodesic ball $B(x,r)$,  there exists a harmonic diffeomorphism
 $\Psi=(u_1, u_2, \cdots, u_{2n}): B(x,r) \to \Omega \subset \R^{2n}$ such that
 \begin{align*}
     \frac{1}{2}\delta_{ij} \leq g_{ij}=g(\nabla u_i, \nabla u_j) \leq 2 \delta_{ij},  \quad
     r^{\frac{3}{2}}\norm{g_{ij}}{C^{1,\frac{1}{2}}} \leq 2.
 \end{align*}
 Then harmonic radius is defined as the supreme of all the possible $r$'s mentioned above.  For convenience, we use
 $\mathbf{hr}$ to denote harmonic radius.   This definition can be easily moved to our case when the underlying space
 is in $\widetilde{\mathscr{KS}}^*(n)$. We define $\mathbf{hr}$ to be $0$ on the singular part of the underlying space.
 It is clear from the definitions and Proposition~\ref{prn:SC26_7} that volume radius and harmonic radius can bound each other, i.e., they are equivalent. The following Proposition is obvious.

 \begin{proposition}[\textbf{Equivalence of volume and harmonic radius}]
 Suppose $x \in X \in \widetilde{\mathscr{KS}}^{*}(n)$, then we have
 \begin{align*}
   \frac{1}{C} \mathbf{hr}(x) \leq \mathbf{vr}(x) \leq C \mathbf{hr}(x)
 \end{align*}
 for some uniform constant $C=C(n)$.
 \label{prn:HA09_1}
 \end{proposition}

 Note that the regularity requirement of the underlying space to define volume radius is much weaker than that to define harmonic radius a priori. Therefore, Proposition~\ref{prn:HA09_1} already implies a regularity improvement.
We shall set up the compactness theory based on volume radius,
since volume radius may be applicable to more general metric measure spaces.

  Let $X \in \widetilde{\mathscr{KS}}^*(n)$ and decompose it as $X=\mathcal{R} \cup \mathcal{S}$. Then $\mathbf{vr}$ is a positive finite function on $\mathcal{R}$ and equals $0$ on $\mathcal{S}$.

\begin{proposition}[\textbf{Rigidity of volume ratio}]
  Suppose $X \in \widetilde{\mathscr{KS}}(n)$. If for two concentric geodesic balls $B(x_0,r_1) \subset B(x_0,r_2)$
  centered at a regular point $x_0$,  we have
\begin{align}
     \omega_{2n}^{-1}r_1^{-2n}|B(x_0,r_1)|= \omega_{2n}^{-1}r_2^{-2n}|B(x_0,r_2)|,
\label{eqn:SL23_1}
\end{align}
then the ball $B(x_0,r_2)$ is isometric to a geodesic ball of radius $r_2$ in $\C^n$.
Furthermore, if $X \in \mathscr{KS}(n)$, then we can further conclude that $X$ is Euclidean.
\label{prn:SC17_3}
\end{proposition}

\begin{proof}
  From the proof of Lemma~\ref{lma:HE04_2}, it is clear that $B(x_0,r_2)$ is a volume cone with constant volume
  ratio $\omega_{2n}$.
  Observe the change of volume element along each smooth geodesic emanating from $x_0$, in the polar coordinate.
  By the volume density gap between regular and singular points,
  the optimal volume ratio of $B(x_0,1)$ forces that it does not contain any singular point.
  Then the situation is the same as the smooth Riemannian case.   Clearly, a smooth Ricci-flat geodesic ball with volume
  ratio $\omega_{2n}$ is isometric to a Euclidean ball of the same radius.

 If $X\in \mathscr{KS}(n)$, by analyticity of metric tensor, it is clear that $X$ is flat and hence $\C^n$ due to its non-collapsing property at infinity.
\end{proof}

\begin{proposition}[\textbf{Continuity of volume radius}]
  $\mathbf{vr}$ is a continuous function on $X$ whenever $X \in \widetilde{\mathscr{KS}}(n)$.
\label{prn:HE07_2}
\end{proposition}

\begin{proof}
Since $\mathbf{vr} \equiv \infty$ on $\C^n$, which is obvious continuous. So we can assume $X \in \widetilde{\mathscr{KS}}^{*}(n)$ without loss of generality.
By Proposition~\ref{prn:SC17_1}, we know $\mathbf{vr}$ is a finite function on $X$.

So we assume $\mathbf{vr}$ is a function with value in $[0,\infty)$.  It is also easy to see that $\mathbf{vr}$ is continuous at singular points.
We know that a point $x_0$ is singular if and only if $\mathbf{vr}(x_0)=0$.  Clearly, for every sequence $x_i \to x_0$, we must have
$\displaystyle \lim_{i \to \infty} \mathbf{vr}(x_i)=0$.
Otherwise, we have a sequence $x_i$ converging to $x_0$ and
$\displaystyle \lim_{i \to \infty} \mathbf{vr}(x_i) \geq \xi>0$.  However, we note that $x_0 \in B(x_i, \tilde{K}^{-1} \xi)$ for
large $i$. Therefore, $x_0$ is forced to be regular by the improving regularity property of volume radius. Contradiction.

Therefore, discontinuity point must admit positive $\mathbf{vr}$ if it does exist.  Suppose $x_0$ is a discontinuous point of $\mathbf{vr}$.
Then we can find a sequence of points $x_i \in X$ such that
\begin{align*}
 &x_0=\lim_{i \to \infty} x_i, \\
 &0<\mathbf{vr}(x_0)=r_0<\infty, \\
 &\lim_{i \to \infty} \mathbf{vr}(x_i) \neq r_0.
\end{align*}
Clearly,  $\log \mathbf{vr}(x_i)$ are uniformly bounded by Proposition~\ref{prn:SC26_7}.
So we can assume $\mathbf{vr}(x_i)$ converge to a positive number $\bar{r}$. By volume continuity, we clearly have
\begin{align*}
  \omega_{2n}^{-1}r_0^{-2n}|B(x_0,r_0)|=1-\delta_0=\lim_{i \to \infty} \omega_{2n}^{-1}\mathbf{vr}(x_i)^{-2n}|B(x_i,\mathbf{vr}(x_i))|=
  \omega_{2n}^{-1}\bar{r}^{-2n}|B(x_0,\bar{r})|.
\end{align*}
Since $\bar{r} \neq r_0$, we obtain from Proposition~\ref{prn:SC17_3} that $B(x_0,r_0)$ is a ball in a metric cone centered at the vertex. Note that $x_0$ is a regular
point since $\mathbf{vr}(x_0)>0$. Therefore, $B(x_0,r_0)$ is the standard ball in $\C^{n}$ with radius $r_0$.
Consequently, the normalized volume ratio of $B(x_0,r_0)$ is $1$, which contradicts the fact that $\mathbf{vr}(x_0)=r_0$ and the definition of volume radius.
\end{proof}

The volume radius has better property. It satisfies Harnack inequality in the interior of a shortest geodesic.
The H\"older continuity estimate of Colding-Naber (c.f.~\cite{ColdNa}) can be interpreted by volume radius as follows.

  \begin{proposition}[\textbf{Global Harnack inequality of volume radius}]
 For every small constant $c$, there is a constant $\epsilon=\epsilon(n,\kappa,c)$ with the following properties.

 Suppose $(X,g) \in \widetilde{\mathscr{KS}}(n,\kappa)$, $x,y \in X$, $\gamma$ is  shortest, unit speed geodesic connecting $x$ and $y$, with smooth interior parts.
 Suppose $\gamma(0)=x, \gamma(L)=y$, $L \leq r$.
 If $\mathbf{vr}(y)>c r$, then we have
 \begin{align}
  \mathbf{vr}(\gamma(t))> \epsilon r, \quad \forall \; t \in [cL, L].
  \label{eqn:SC26_4}
 \end{align}
 In particular, if $\min\{\mathbf{vr}(x), \mathbf{vr}(y)\}>cr$, then we have
 \begin{align}
    \mathbf{vr}(\gamma(t))>\epsilon r, \quad \forall \; t \in [0,L].
   \label{eqn:SC26_44}
 \end{align}
 \label{prn:SC26_2}
 \end{proposition}

 \begin{proof}
   Clearly, (\ref{eqn:SC26_44}) follows from (\ref{eqn:SC26_4}).  Therefore,  it suffices to prove (\ref{eqn:SC26_4}) only.

   Up to a normalization, we can assume $r=L=1$.
   So $\gamma$ is the shortest geodesic connecting $x,y$ such that $\gamma(0)=x, \gamma(1)=y$.
   By assumption, we have $\mathbf{vr}(y)>c$.
   By local Harnack inequality of volume radius, Proposition~\ref{prn:SC26_7},
   there exists $\bar{\epsilon}=\bar{\epsilon}(n, c)$ such that $\mathbf{vr}>\bar{\epsilon}$ for each $\gamma(t)$
   with $t \in  [1-\bar{\epsilon}, 1]$.  Clearly, in the middle part of $\gamma$, i.e., for every $t \in [\bar{\epsilon}, 1-\bar{\epsilon}]$, we have
   $|\Delta r|< \frac{C}{\bar{\epsilon}}$ for a universal $C=C(n)$, where $r$ is the distance to $\gamma(0)$.
   Because of the segment inequality (Proposition~\ref{prn:HD17_1}) and the parabolic approximation (Lemma~\ref{lma:HE04_4}),
   we can follow the proof of Proposition 3.6 and Theorem 1.1 of~\cite{ColdNa} verbatim.
   Similar to the statement in the  proof of Theorem 1.1 on page 1213 of~\cite{ColdNa}, we can find constants $\bar{s}=\bar{s}(n,c,\bar{\epsilon}), \bar{r}=\bar{r}(n,c,\bar{\epsilon})$
   such that for every $t_1, t_2 \in [\bar{\epsilon},1-\bar{\epsilon}]$ satisfying $|t_1-t_2|<\bar{s}$ and every $r \in (0,\bar{r})$, we have
   \begin{align*}
     1-\frac{\delta_0}{100} \leq \frac{|B(\gamma(t_1), r)|}{|B(\gamma(t_2),r)|} \leq 1+\frac{\delta_0}{100}.
   \end{align*}
   Then it is easy to see that if the volume radius is uniformly bounded below at $t_1$, it must be uniformly bounded below at $t_2$.
   Actually, suppose the volume radius at $\gamma(t_1)$ is greater than $r_1$ for some $r_1 \in (0,\bar{r})$,
   by inequality (\ref{eqn:SC26_9}) in Proposition~\ref{prn:SC26_7}, we have
   $\omega_{2n}^{-1}r^{-2n}|B(\gamma(t_1),r)| \geq 1-\frac{\delta_0}{100}$ for every $r \in [0, \frac{r_1}{\tilde{K}}]$.
   Put this information into the above inequality implies that
   \begin{align*}
     \omega_{2n}^{-1}r^{-2n}|B(\gamma(t_2),r)| \geq \frac{1-\frac{\delta_0}{100}}{1+\frac{\delta_0}{100}}>1-\delta_0,
     \quad \forall \; r \in \left[0, \frac{r_1}{\tilde{K}} \right].
   \end{align*}
   Therefore, the volume radius of $\gamma(t_2)$ is at least $\frac{r_1}{\tilde{K}}$.  From this induction, it is clear that
   \begin{align*}
      \mathbf{vr}(\gamma(t)) \geq
      \tilde{K}^{-\frac{1-\bar{\epsilon}-t}{\bar{s}}} \mathbf{vr}(\gamma(1-\bar{\epsilon}))
      >\bar{\epsilon}  \tilde{K}^{-\frac{1-\bar{\epsilon}-t}{\bar{s}}}.
   \end{align*}
  Let $\epsilon$ be the number on the right hand side of the above inequality when $t=\bar{\epsilon}$.
  Then $\epsilon=\epsilon(\bar{\epsilon}, \tilde{K}, \bar{s})=\epsilon(n,\kappa,c)$
  and we finish the proof of (\ref{eqn:SC26_4}).
 \end{proof}

In general, if $X$ is only a metric space, we even do not know whether $\mathbf{vr}$ is semi-continuous.  The continuity of $\mathbf{vr}$ on $X$ whenever
  $X \in \widetilde{\mathscr{KS}}(n,\kappa)$ makes $\mathbf{vr}$ a convenient tool to study the geometry of $X$.  By Proposition~\ref{prn:SC26_7},
  one can improve regularity on a scale proportional to $\mathbf{vr}$. So it is convenient to decompose the space $X$ according to the function $\mathbf{vr}$.

 \begin{definition}
 Suppose $(X,g) \in \widetilde{\mathscr{KS}}(n,\kappa)$. Define
 \begin{align}
  &\mathcal{F}_{r}(X) \triangleq \left\{ x \in X | \mathbf{vr}(x) \geq r \right\}, \label{eqn:SB25_1}\\
  &\mathcal{D}_{r}(X) \triangleq \left( \mathcal{F}_{r}(X) \right)^{c}=\left\{ x \in X | \mathbf{vr}(x) < r \right\}. \label{eqn:SB25_2}
\end{align}
 We call $\mathcal{F}_{r}(X)$ the $r$-regular part of $X$, $\mathcal{D}_{r}(X)$ the $r$-singular part of $X$.
\label{dfn:SB25_1}
\end{definition}

 From Definition~\ref{dfn:SB25_1}, it is clear that
 \begin{align}
  & \mathcal{R}(X)= \bigcup_{r>0} \mathcal{F}_{r}(X), \label{eqn:SB25_4}\\
  & \mathcal{S}(X)= \bigcap_{r>0} \mathcal{D}_{r}(X).  \label{eqn:SB25_5}
 \end{align}

We observe that the volume radius of each point is related to its distance to singular set by the following property.
 \begin{proposition}[\textbf{$\mathbf{vr}$ bounded from above by distance to $\mathcal{S}$}]
   Suppose $(X,x,g) \in \widetilde{\mathscr{KS}}(n,\kappa)$, $r$ is a positive number.   Then
 \begin{align}
  & \left\{ x | d(x, \mathcal{S}) \geq r \right\} \supset \mathcal{F}_{\tilde{K}r}, \label{eqn:SB25_14} \\
  & \left\{x | d(x,\mathcal{S}) < r \right\} \subset \mathcal{D}_{\tilde{K}r}. \label{eqn:SB25_13}
 \end{align}
 \end{proposition}
\begin{proof}
 Choose an arbitrary point $x \in \mathcal{F}_{\tilde{K}r}$, then $ {\mathbf{vr}}(x) \geq \tilde{K} r.$  It follows from Proposition~\ref{prn:SB25_2} that
 $\mathbf{vr}(y) \geq r>0$ for every point $y \in B(x, r)$.  Therefore, every point in $B(x,r)$ is regular. So $d(x,\mathcal{S})>r$. This proves
 (\ref{eqn:SB25_14}) by the arbitrary choice of $x \in \mathcal{F}_{\tilde{K}r}$.  Taking complement of (\ref{eqn:SB25_14}), we obtain (\ref{eqn:SB25_13}).
\end{proof}

\subsection{Compactness of $\widetilde{\mathscr{KS}}(n, \kappa)$}

As a model space, $\widetilde{\mathscr{KS}}(n,\kappa)$ should have compactness.
However, we need first to obtain a weak compactness, then we improve regularity further to obtain the
genuine compactness.
It is not hard to see the weak compactness theory of  Anderson-Cheeger-Colding-Tian-Naber can be generalized to apply on $\widetilde{\mathscr{KS}}(n,\kappa)$
without fundamental difficulties, almost verbatim.
Actually,  the key of Anderson-Cheeger-Colding-Tian-Naber  theory is that one can approximate the distance function by harmonic function, or heat flow solution,
which have much better regularity for developing integral estimates.  These estimates are justified by the technical preparation in previous subsections.

\begin{proposition}[\textbf{Weak compactness}]
Suppose $(X_i,x_i,g_i) \in \widetilde{\mathscr{KS}}(n,\kappa)$, by taking subsequences if necessary,
we have
\begin{align*}
    (X_i, x_i, g_i) \stackrel{\hat{C}^{\infty}}{\longrightarrow} (\bar{X},\bar{x},\bar{g})
\end{align*}
for some length space $\bar{X}$ which satisfies all the properties of spaces in
 $\widetilde{\mathscr{KS}}(n,\kappa)$ except the 3rd and 4th property, i.e., the weak convexity of $\mathcal{R}$
 and the Minkowski dimension estimate $\mathcal{S}$.
 However, the Hausdorff dimension of $\mathcal{S}$ is not greater than $2n-4$.
\label{prn:HA05_1}
\end{proposition}

\begin{proof}

Note that each space in $\widetilde{\mathscr{KS}}(n,\kappa)$ satisfies volume doubling property.  Therefore, if
there exists a sequence $(X_i,x_i,g_i) \in \widetilde{\mathscr{KS}}(n,\kappa)$, by standard ball packing argument, it is clear that
\begin{align*}
    (X_i, x_i, g_i) \stackrel{G.H.}{\longrightarrow} (\bar{X},\bar{x},\bar{g})
\end{align*}
for some length space $\bar{X}$. Then let us list the properties satisfied by $\bar{X}$.

By Proposition~\ref{prn:HD22_1},  $\bar{X}$ inherits a natural measure from the limit process, which is a measure compatible with
the limit metric structure, as that in~\cite{CC1}. Then  the volume convergence follows, almost  tautologically.
It follows directly from this property and the volume comparison that $\bar{X}$ satisfies
Property 6 in Definition~\ref{dfn:SC27_1}.

 In the limit space $\bar{X}$, we can define regular points as the collection
of points where every tangent space is $\R^{2n}$, singular points as those points which are not regular.
Let $\mathcal{R}(\bar{X})$ and $\mathcal{S}(\bar{X})$ be the regular and singular part of $\bar{X}$ respectively.
We automatically obtain the regular-singular decomposition $\bar{X}=\mathcal{R}(\bar{X}) \cup \mathcal{S}(\bar{X})$.
By a version of Anderson's gap theorem(c.f. Proposition~\ref{prn:SC17_1}) and volume convergence,
a blowup argument shows that each regular point
has a small neighborhood which has a smooth manifold structure.  Clearly, this manifold is Ricci-flat with an attached
limit K\"ahler structure.  So we proved Property 1 and Property 2, except the non-emptiness of $\mathcal{R}$.

Note that  each tangent space of $\bar{X}$ is a volume cone, due to the volume convergence
and Bishop-Gromov volume comparison, which can be established as that in~\cite{CC1}.
The it follows from Proposition~\ref{prn:HD22_2} that every volume cone is actually a metric cone.  Then an induction argument can be applied, like that in~\cite{CC1}, to obtain the stratification of singularities $\mathcal{S}=\mathcal{S}_1\cup \mathcal{S}_2 \cdots \cup \mathcal{S}_{2n}$, where $\mathcal{S}_k$
is the union of singular points whose tangent space can split-off at least $(2n-k)$-straight lines.
In particular, generic points of $\bar{X}$ have tangent spaces $\R^{2n}$. In other words, generic points
are regular, so $\mathcal{R} \neq \emptyset$ and we finish the proof of Property 2.

The K\"ahler condition
guarantees that each tangent cone exactly splits off $\C^k$, by Proposition~\ref{prn:HE08_1}, as done in~\cite{CCT}.
So the stratification of singular set can be improved as
$\mathcal{S}=\mathcal{S}_2 \cup \mathcal{S}_4 \cup \cdots \mathcal{S}_{2n}$.
By Lemma~\ref{lma:HE07_1}, we can apply slice argument as that in \cite{CCT} and \cite{Cheeger03}.
Consequently, Chern-Simons theory implies that codimension 2 singularity cannot appear, due to the fact that a generic slice is a smooth surface with boundary, and the Ricci curvature's restriction on such a surface is zero.
Actually, the smoothness of generic slices follows from the high codimension of the singular set (item 4 of Definition~\ref{dfn:SC27_1}) and the gradient estimates of the harmonic approximation functions(Proposition~\ref{prn:HC29_6}).  
Therefore, $\mathcal{S}=\mathcal{S}_4 \cup \cdots \mathcal{S}_{2n}$,
which means $\dim_{\mathcal{H}} \mathcal{S} \leq 2n-4$.

Let $\bar{y} \in \mathcal{S}(\bar{X})$. Suppose $y_i \in X_i$ satisfies $y_i \to \bar{y}$. Then either there is a uniform $\xi$ such that
every point in each $B(y_i,\xi)$ are regular(but without uniform curvature bound as $i$ increase), or we can choose $y_i$ such that every $y_i$ is singular.  In the first case,
we can use a blowup argument and Anderson's gap theorem to show that the volume density of $\bar{y}$ is strictly less
that $1-2\delta_0$. In the second case, we can use volume comparison and convergence to show
$\mathrm{v}(\bar{y}) \leq 1-2\delta_0$.  So we proved Property 5.

We have checked all the properties of $\bar{X}$ as claimed. We now need to improve the convergence topology from Gromov-Hausdorff topology.
However, this improvement follows from volume convergence and the improving regularity property of volume radius, Corollary~\ref{cly:SL23_1}.
\end{proof}

 From the above argument, it is clear that no new idea is needed beyond the traditional theory,
 when technical lemmas and propositions in the previous sections are available.  
 Actually, weak compactness can be established under even weaker conditions, which will be discussed in our forthcoming work. 
 Based on the weak compactness, we immediately obtain an $\epsilon$-regularity property, as that in~\cite{CCT}.

\begin{proposition}[\textbf{$\epsilon$-regularity}]
There exists an $\epsilon=\epsilon(n,\kappa)$ with the following properties.

 Suppose $X \in \widetilde{\mathscr{KS}}(n,\kappa)$, $x_0 \in X$.
 Suppose
 \begin{align*}
    d_{GH}\left( B(x_0,1),  B((z_0^*,0), 1) \right)<\epsilon
 \end{align*}
  where $(z_0^*,0) \in C(Z_0) \times \R^{2n-3}$ for some metric cone $C(Z_0)$ with vertex $z_0^*$.  Then we have
  \begin{align*}
      \mathbf{vr}(x_0)>\frac{1}{2}.
  \end{align*}
\label{prn:HD20_1}
\end{proposition}

\begin{proof}
 Otherwise, there is a sequence of $\epsilon_i \to 0$ and $x_i \in X_i$ violating the statement.
 By weak compactness of $\widetilde{\mathscr{KS}}(n,\kappa)$,  we can assume $x_i \to x$ and $z_i^* \to z^*$ with the following identity holds.
  \begin{align*}
    d_{GH}\left( (B(x,1),  B((z^*,0), 1))\right)=0.
 \end{align*}
 In particular, the tangent cone at $x$ is exactly the cone $C(Z) \times \R^{2n-3}$, which must be $\C^n$ by the complex rigidity.
 Therefore, $B(x,1)$ is the unit ball in $\C^n$. Thus, the volume convergence implies that for large $i$,
 $2^{2n}\left|B \left(x_i, \frac{1}{2} \right) \right|$ can be very close to $1$. In particular,  $\mathbf{vr}(x_i)>\frac{1}{2}$
 by the definition of volume radius. However, this contradicts our assumption.
\end{proof}

Then we are able to move the integral estimate of~\cite{CN} to $X \in \widetilde{\mathscr{KS}}(n,\kappa)$.

\begin{proposition}[\textbf{Density estimate of regular points}]
 For every $0<p<2$, there is a constant $E=E(n,\kappa,p)$ with the following properties.

   Suppose $(X,x,g) \in \widetilde{\mathscr{KS}}(n,\kappa)$, $r$ is a positive number. Then we have
 \begin{align}
   r^{2p-2n} \int_{B(x,r)} \mathbf{vr}(y)^{-2p} dy \leq E(n,\kappa,p).  \label{eqn:SB25_15}
 \end{align}
 \label{prn:SB25_2}
 \end{proposition}

\begin{proof}
 In light of Proposition~\ref{prn:SC26_7}, it is clear that volume radius and harmonic radius are uniformly equivalent. Therefore, Proposition~\ref{prn:SB25_2} is nothing but a singular version of the Cheeger-Naber estimate(c.f. the second inequality of part 2 of Corollary 1.26 in~\cite{CN}).
As pointed out by Cheeger and Naber, their estimate holds for Gromov-Hausdorff limit for Ricci-flat manifolds.
Actually, going through their proof, it is clear that the smooth structure of the underlying space is not used.
Intuitively, if Bishop-Gromov volume comparison holds, then most geodesic balls are almost volume cones, hence almost metric cones.
However, if a cone is very close to a cone which splits off at least $(2n-3)$-lines,
 then it must be Euclidean space by the $\epsilon$-regularity property.  This intuition was quantified in~\cite{CN}, by the method they called quantitative calculus,
 which does not depends on smooth structure by its nature.
 We note that the quantitative calculus argument of~\cite{CN} works when we have the following properties.
\begin{itemize}
 \item Bishop-Gromov volume comparison, by Proposition~\ref{prn:HD19_1}.
 \item Weak compactness of $ \widetilde{\mathscr{KS}}(n,\kappa)$,  by Proposition~\ref{prn:HA05_1}.
 \item Volume convergence, by Proposition~\ref{prn:HD22_1}.
 \item Almost volume cone implies almost metric cone, by Proposition~\ref{prn:HD22_2}.
 \item $\epsilon$-regularity, by Proposition~\ref{prn:HD20_1}.
\end{itemize}
Since all these properties hold on $ \widetilde{\mathscr{KS}}(n,\kappa)$, the proof follows that of~\cite{CN} verbatim.
\end{proof}

An immediate consequence of Proposition~\ref{prn:SB25_2} is the following volume estimate of neighborhood of
singular set.
\begin{corollary}[\textbf{Volume estimate of singular neighborhood}]
    Suppose $(X,x_0,g) \in \widetilde{\mathscr{KS}}(n,\kappa)$, $0<\rho<<1$.
    Then for each $0<p<2$,  we have
    \begin{align*}
      \left| \{x| d(x,\mathcal{S})<\rho, x \in B(x_0, 1)\} \right|< C\rho^{2p},
    \end{align*}
    for some $C=C(n,\kappa,p)$.
\label{cly:HD19_2}
\end{corollary}

\begin{proof}
  It follows from Definition~\ref{dfn:SB25_1} that
  \begin{align*}
     (2r)^{-2p}\left|B(x_0,1) \cap \mathcal{F}_{r} \backslash \mathcal{F}_{2r}\right|
     = \int_{B(x_0,1) \cap \mathcal{F}_{r} \backslash \mathcal{F}_{2r}} (2r)^{-2p}
     < \int_{B(x_0,1) \cap \mathcal{F}_{r} \backslash \mathcal{F}_{2r}} \mathbf{vr}^{-2p}<E(n,\kappa,p),
  \end{align*}
  which implies that
  \begin{align*}
    \left|B(x_0,1) \cap \mathcal{D}_{2r} \backslash \mathcal{D}_{r}\right| < 2^{2p} E r^{2p}
    \Rightarrow  |B(x_0,1) \cap \mathcal{D}_{2r}|<\frac{E}{1-4^{-p}} (2r)^{2p}.
  \end{align*}
  By virtue of (\ref{eqn:SB25_13}), we have
  \begin{align*}
     \left|B(x_0,1) \cap \{x |d(x,\mathcal{S})<\rho\} \right| \leq |B(x_0,1) \cap \mathcal{D}_{\tilde{K}\rho}|< \frac{E}{1-4^{-p}} \tilde{K}^{2p} \rho^{2p}<C\rho^{2p}.
  \end{align*}
\end{proof}

Now we are ready to prove the compactness theorem.
 \begin{theorem}[\textbf{Compactness}]
  $\widetilde{\mathscr{KS}}(n,\kappa)$ is compact under the pointed Cheeger-Gromov topology.
 \label{thm:HD19_1}
 \end{theorem}
 \begin{proof}
  Suppose $(X_i, x_i,g_i) \in \widetilde{\mathscr{KS}}(n,\kappa)$,  we already know, by Proposition~\ref{prn:HA05_1},
  that $(X_i,x_i,g_i)$ converges to a limit space $(\bar{X},\bar{x},\bar{g})$, which satisfies almost all the properties  of
  spaces in $\widetilde{\mathscr{KS}}(n,\kappa)$, except the weak convexity of $\mathcal{R}$ and the
  Minkowski dimension estimate of $\mathcal{S}$.
  However, fix every two points  $\bar{y},\bar{z} \in \mathcal{R} \subset \bar{X}$, we can find a sequence of points
  $y_i, z_i \in X_i$ such that $y_i \to \bar{y}$ and $z_i \to \bar{z}$. It is clear that
   $\mathbf{vr}(y_i) \to \mathbf{vr}(\bar{y})$ and $\mathbf{vr}(z_i) \to \mathbf{vr}(\bar{z})$.
   It follows from the global Harnack inequality of volume radius, Proposition~\ref{prn:SC26_2}, that each shortest geodesic
   $\gamma_i$ connecting $y_i$ and $z_i$ is uniformly regular.  Consequently, the limit shortest geodesic $\bar{\gamma}$
   connecting $\bar{y}$ and $\bar{z}$ is a smooth shortest geodesic. Therefore, we have actually proved that
   $\mathcal{R}$ is  convex, rather than weakly convex.
   Furthermore, if we repeatedly use the first inequality in Proposition~\ref{prn:SC26_2} and smooth convergence determined by volume radius,
   one can see that a shortest geodesic $\bar{\gamma}$ with smooth interior can be obtained, even if we drop the condition $\bar{y} \in \mathcal{R}$.
   In other words, if $\bar{z} \in \mathcal{R}$, $\bar{y} \in \bar{X}$, then there is a shortest geodesic $\bar{\gamma}$ connecting them, with smooth interior.
   This means that $\mathcal{R}$ is strongly convex.

   By convexity of $\mathcal{R}$,  it is clear that the limit space $\bar{X}$ has Bishop-Gromov volume comparison.
   By virtue of volume convergence and the same argument in Proposition~\ref{prn:HE07_2},
   we see that $\mathbf{vr}$ is a continuous function under the pointed Cheeger-Gromov topology.  In other words, for every point $\bar{z} \in \bar{X}$, and points
   $z_i \in X_i$ satisfying $z_i \to \bar{z}$, we have  $\displaystyle \mathbf{vr}(\bar{z})=\lim_{i \to \infty} \mathbf{vr}(z_i)$.
   For each $r>0$, by density estimate, Proposition~\ref{prn:SB25_2},
   we see that inequality (\ref{eqn:SB25_15}) holds for every $B(x_i,r)$ uniformly.
   Taking limit, by the convergence of volume radius, we obtain (\ref{eqn:SB25_15}) holds on $(\bar{X},\bar{x},\bar{g})$, for
   each $p\in (1.5, 2)$.  Then it follows from Corollary~\ref{cly:HD19_2} and the definition of Minkowski dimension (c.f. Definition~\ref{dfn:HE08_1})
   that  $\dim_{\mathcal{M}} \mathcal{S} \leq 2n-4$.
 \end{proof}

\begin{theorem}[\textbf{Space regularity improvement}]
 Suppose $X \in \widetilde{\mathscr{KS}}(n,\kappa)$, then $\mathcal{R}$ is strongly convex, and
 $\dim_{\mathcal{M}}\mathcal{S} \leq 2n-4$.
 Suppose  $x_0 \in \mathcal{S}$,  $Y$ is a tangent space of $X$ at $x_0$.
 Then $Y$ is a metric cone in $\widetilde{\mathscr{KS}}(n,\kappa)$ with the splitting
 \begin{align*}
   Y=\C^{n-k} \times C(Z)
 \end{align*}
 for some $k \geq 2$, where $C(Z)$ is a metric cone without lines.
\label{thm:HE08_1}
\end{theorem}

\begin{proof}
 The strong convexity of $\mathcal{R}$ and $\dim_{\mathcal{M}}\mathcal{S} \leq 2n-4$ follows from the argument in the proof of Theorem~\ref{thm:HD19_1}.
 Moreover, by Theorem~\ref{thm:HD19_1}, we know each tangent space, as a pointed Gromov-Hausdorff limit,
 must locate in $\widetilde{\mathscr{KS}}(n,\kappa)$.  Since $Y$ is a volume cone, due to volume convergence,
 the splitting of $Y$ follows from Lemma~\ref{lma:HD30_1}.
\end{proof}

 Because of Theorem~\ref{thm:HD19_1} to Theorem~\ref{thm:HE08_1}, it seems reasonable to make the following definition for the simplicity of notations.

\begin{definition}
 A length space $(X^n,g)$ is called a conifold of complex dimension $n$ if the following properties are satisfied.
\begin{enumerate}
  \item  $X$ has a disjoint regular-singular decomposition $X=\mathcal{R} \cup \mathcal{S}$, where $\mathcal{R}$ is the regular part,  $\mathcal{S}$ is the singular part.
   A point is called regular if it has a neighborhood which is isometric to a totally geodesic convex domain of  some smooth Riemannian manifold.  A point is called singular
   if it is not regular.
  \item  The regular part $\mathcal{R}$ is a nonempty, open manifold of real dimension $2n$.
            Moreover, there exists a complex structure $J$ on $\mathcal{R}$ such that $(\mathcal{R}, g, J)$ is a K\"ahler manifold.
  \item     $\mathcal{R}$ is  strongly convex, i.e., for every two points $x \in \mathcal{R}$ and $y \in X$, one can find a shortest geodesic  $\gamma$ connecting $x$, $y$ whose every interior point is in $\mathcal{R}$. In particular, $\mathcal{R}$ is geodesic convex.
  \item $\dim_{\mathcal{M}} \mathcal{S} \leq  2n-4$, where $\mathcal{M}$ means Minkowski dimension.
  \item Every tangent space of $x \in \mathcal{S}$ is a metric cone of Hausdorff dimension $2n$.  Moreover, if $Y$ is a tangent cone of $x$, then the unit ball
  $B(\hat{x},1)$ centered at vertex $\hat{x}$ must satisfy
        \begin{align*}
             |B(\hat{x},1)|_{d\mu} \leq (1-\delta_0) \omega_{2n}
       \end{align*}
       for some uniform positive number $\delta_0=\delta_0(n)$.
       Here $d\mu$ is the $2n$-dimensional Hausdorff measure, $\omega_{2n}$ is the volume of unit ball in $\C^n$.
\end{enumerate}
\label{dfn:HD20_1}
\end{definition}
Roughly speaking, a conifold is a space which is almost a manifold away from a small singular set, where every tangent space is a metric cone.
Note that we abuse notation here since the conifold has different meaning
in the literature of string theory(c.f.~\cite{Green}).
It is easy to see that every K\"ahler orbifold with singularity codimension not less than $4$ is a conifold in our sense.
With this terminology, we see that $\widetilde{\mathscr{KS}}(n,\kappa)$ is nothing but the collection of Calabi-Yau conifold
with Euclidean volume growth, i.e.,
 \begin{align*}
      \lim_{r \to \infty} \frac{|B(x,r)|_{d\mu}}{\omega_{2n}r^{2n}} \geq \kappa, \quad \forall \; x \in X.
 \end{align*}
Then Theorem~\ref{thm:HD19_1} can be interpreted as that the moduli space of non-collapsed Calabi-Yau conifolds is compact, under the pointed Cheeger-Gromov topology. Theorem~\ref{thm:HE08_1} can be understood as that a ``weakly" Calabi-Yau conifold is really a conifold,  due to an intrinsic improving regularity property
originates from the intrinsic Ricci flatness of the underlying space.
The property of the moduli space $\widetilde{\mathscr{KS}}(n,\kappa)$ is quite clear now.

\begin{proof}[Proof of Theorem~\ref{thmin:HE21_1}]
It follows directly from Theorem~\ref{thm:HD19_1}, Theorem~\ref{thm:HE08_1} and Definition~\ref{dfn:HD20_1}.
\end{proof}

Actually, along the route to prove Theorem~\ref{thm:HE08_1}, we shall be able to improve the regularity of the
spaces in  $\widetilde{\mathscr{KS}}(n,\kappa)$ even further.  For example,  we believe the following statement is true.

\begin{conjecture}
 At every point $x_0$ of a Calabi-Yau conifold $X \in \widetilde{\mathscr{KS}}(n,\kappa)$, the tangent space is unique.
\label{cje:HE09_1}
\end{conjecture}

The above problem is only interesting when $n>2$ and away from generic singular point.
Note that if $X$ is a limit space of a sequence of Ricci flat manifolds, then the uniqueness of tangent cone is a well known open problem, in the classical theory of Cheeger-Colding-Tian.
Clearly, similar questions can be asked for general K\"ahler Einstein conifold.
It is not hard to see that a compact K\"ahler Einstein conifold is a projective variety.
Due to its independent interest, we shall discuss this issue in  another separate paper.

\subsection{Space-time structure of $\widetilde{\mathscr{KS}}(n)$}
\label{subsec:reduced}

Every space $X \in \widetilde{\mathscr{KS}}(n)$ can be regarded as a trivial Ricci flow solution.
Therefore, Perelman's celebrated work \cite{Pe1}  can find its role in the study of $X$.
Let us briefly recall some fundamental functionals defined for the Ricci flow by Perelman.

Suppose $\{(X^m, g(t)), -T \leq t \leq 0\}$ is a Ricci flow solution on a smooth complete Riemannian manifold
$X$ of real dimension $m$.   Suppose $x,y \in X$. Suppose $\boldsymbol{\gamma}$ is a space-time curve
parameterized by $\tau=-t$ such that
\begin{align*}
   \boldsymbol{\gamma}(0)=(x,0), \quad  \boldsymbol{\gamma}(\bar{\tau})=(y,-\bar{\tau}).
\end{align*}
Let $\gamma$ be the space-projection curve of $\boldsymbol{\gamma}$.
In other words, we have
\begin{align*}
    \boldsymbol{\gamma}(\tau)=(\gamma(\tau), -\tau).
\end{align*}
By the way, for the simplicity of notations, we always use bold symbol of a Greek character to denote a space-time curve. The corresponding space projection will be denoted by the normal Greek character.
Following Perelman, the Lagrangian of the space-time curve $\boldsymbol{\gamma}$ is defined as
\begin{align}
     \mathcal{L}(\boldsymbol{\gamma})=\int_0^{\bar{\tau}} \sqrt{\tau} \left(R+|\dot{\gamma}|^2 \right)_{g(-\tau)} d\tau.
\label{eqn:MA22_1}     
\end{align}
Among all such $\boldsymbol{\gamma}$'s that connected $(x,0)$, $(y,-\bar{\tau})$ and parameterized by $\tau$,
there is at least one smooth curve $\boldsymbol{\alpha}$ which minimizes the Lagrangian.
This curve is called a shortest reduced geodesic. The reduced distance between $(x,0)$ and $(y,-\bar{\tau})$ is
defined as
\begin{align}
  l((x,0),(y,-\bar{\tau}))=\frac{\mathcal{L}(\boldsymbol{\alpha})}{2\sqrt{\bar{\tau}}}.
\label{eqn:MA22_2}  
\end{align}
Let $V=\dot{\alpha}$. Then $V$ satisfies the equation
\begin{align}
  \nabla_V V +\frac{V}{2\tau} + 2Ric(V, \cdot) + \frac{\nabla R}{2}=0,
\label{eqn:MA22_3}  
\end{align}
which is called the reduced geodesic equation.  It is easy to check that $\dot{\alpha}=V=\nabla l$.
The reduced volume is defined as
\begin{align}
   \mathcal{V}((x,0), \bar{\tau})=\int_{X} (4\pi \bar{\tau})^{-\frac{m}{2}} e^{-l} dv.
\label{eqn:MA22_4}   
\end{align}
It is proved by Perelman that $(4\pi \tau)^{-\frac{m}{2}} e^{-l}dv$, the reduced volume element, is
monotonically non-increasing along each reduced geodesic emanating from $(x,0)$.

Suppose the Ricci flow solution mentioned above is static, i.e., $Ric \equiv 0$.   Then it is easy to check that
\begin{align}
\begin{cases}
&\mathcal{L}(\boldsymbol{\alpha})=\frac{d^2(x,y)}{2\sqrt{\bar{\tau}}}, \\
&l((x,0),(y,-\bar{\tau}))=\frac{d^2(x,y)}{4\bar{\tau}}, \\
&\nabla_V V +\frac{V}{2\tau}=0, \\
&|\dot{\alpha}|^2=|V|^2=|\nabla l|^2=\tau l, \\
&\mathcal{V}((x,0), \bar{\tau})=\int_{X} (4\pi \bar{\tau})^{-\frac{m}{2}} e^{-\frac{d^2}{4\bar{\tau}}} dv.
\end{cases}
\label{eqn:SL25_6}
\end{align}

Now we assume $X \in \widetilde{\mathscr{KS}}(n)$.
By a trivial extension in an extra time direction, we obtain a static, eternal singular K\"ahler Ricci flow solution.
Since distance structure is already known, we can define reduced distance, reduced volume, etc, following the equation
(\ref{eqn:SL25_6}).  Clearly, this definition coincides with the original one when $X$ is smooth.

The following theorem is important to bridge the Cheeger-Colding's structure theory to the Ricci flow theory.

\begin{theorem}[\textbf{Volume ratio and reduced volume}]
 Suppose $X \in \widetilde{\mathscr{KS}}(n)$, $x \in X$.  Let $X \times (-\infty, 0]$ have the obvious
 static space-time structure.  Then we have
 \begin{align}
     &\mathrm{avr}(X)=\lim_{\tau \to \infty} \mathcal{V}((x,0), \tau).       \label{eqn:SL25_4}\\
     &\mathrm{v}(x)=\lim_{\tau \to 0} \mathcal{V}((x,0), \tau).  \label{eqn:SL25_7}
 \end{align}
\label{thm:SL25_1}
\end{theorem}

\begin{proof}
The proof relies on the volume cone structure at local tangent space, or tangent space at infinity.  So
the proof of (\ref{eqn:SL25_4}) and (\ref{eqn:SL25_7}) are almost the same. For simplicity, we will only
prove (\ref{eqn:SL25_4}) and leave the proof for (\ref{eqn:SL25_7}) to the readers.

Clearly, the real dimension of $X$ is $m=2n$.
For each $\epsilon$ small, we have
  \begin{align*}
     m \omega_m \mathrm{avr}(X)+\epsilon
     > H^{-m+1}|\partial B(x,H)| >m \omega_m \mathrm{avr}(X)-\epsilon,
  \end{align*}
  whenever $H$ is large enough.  Note that
  \begin{align*}
  \mathcal{V}((x,0), H^2)&=(4\pi)^{-\frac{m}{2}}  H^{-m} \int_0^{\infty} |\partial B(x,r)| e^{-\frac{r^2}{4H^2}} dr, \\
    1&=(4\pi)^{-\frac{m}{2}} H^{-m} \int_0^{\infty} m \omega_m r^{m-1} e^{-\frac{r^2}{4H^2}} dr.
  \end{align*}
  So we have
  \begin{align*}
     \mathcal{V}((x,0), H^2)-\mathrm{avr}(X)
    =(4\pi)^{-\frac{m}{2}} H^{-m} \int_0^{\infty} \left\{|\partial B(x,r)| -m\omega_m\mathrm{avr}(X) r^{m-1}\right\}
    e^{-\frac{r^2}{4H^2}} dr.
  \end{align*}
  We can further decompose the last integral as follows.
  \begin{align*}
     &\quad \left| \int_0^{\epsilon H} \left\{|\partial B(x,r)| -m\omega_m\mathrm{avr}(X)r^{m-1} \right\}
     e^{-\frac{r^2}{4H^2}} dr \right|\\
     &\leq m\omega_m \int_0^{\epsilon H} r^{m-1} e^{-\frac{r^2}{4H^2}} dr
     =m\omega_m H^m \int_0^{\epsilon} s^{m-1}e^{-\frac{s^2}{4}} ds, \\
     &\quad \left| \int_{\epsilon H}^{\infty} \left\{|\partial B(x,r)| -m\omega_m\mathrm{avr}(X)r^{m-1} \right\}
     e^{-\frac{r^2}{4H^2}} dr \right|\\
     &\leq \epsilon \int_{\epsilon H}^{\infty} r^{m-1}e^{-\frac{r^2}{4H^2}} dr
     <\epsilon H^m \int_{0}^{\infty} e^{-\frac{s^2}{4}} ds=\epsilon H^m \pi^{\frac{1}{2}}.
      \end{align*}
   Therefore, we have
   \begin{align*}
      \left|\mathcal{V}((x,0), H^2)-\mathrm{avr}(X)\right|
     <(4\pi)^{-\frac{m}{2}} \left\{m\omega_m \int_0^{\epsilon} s^{m-1}e^{-\frac{s^2}{4}} ds + \epsilon \pi^{\frac{1}{2}}\right\}.
   \end{align*}
   Since the above inequality holds for every $H$ large enough, we see that
   \begin{align*}
      \left|\lim_{\tau \to \infty} \mathcal{V}((x,0),\tau)-\mathrm{avr}(X) \right|
     <(4\pi)^{-\frac{m}{2}} \left\{m\omega_m \int_0^{\epsilon} s^{m-1}e^{-\frac{s^2}{4}} ds + \epsilon \pi^{\frac{1}{2}}\right\}.
   \end{align*}
   Let $\epsilon \to 0$, we obtain (\ref{eqn:SL25_4}).
\end{proof}
 Theorem~\ref{thm:SL25_1} says that when we study the asymptotic behavior of $X$, the volume ratio and reduced volume
 play the same role.   Note that volume ratio is monotone along radius direction on a manifold with nonnegative Ricci curvature, which property plays an essential role in Cheeger-Colding's theory.
 Since reduced volume is monotone along Ricci flow, Theorem~\ref{thm:SL25_1} suggests that Cheeger-Colding's theory can be transplanted to the Ricci flow case.

\section{Canonical radius}

In section 2, we established the compactness of the model space $\widetilde{\mathscr{KS}}(n,\kappa)$,
following the route of Anderson-Cheeger-Colding-Tian-Naber.
It is clear that the volume ratio's monotonicity is essential to this route.  However, most  K\"ahler manifolds  do not have this monotonicity.
For example, if we take out a time slice from a K\"ahler Ricci flow solution, there is no obvious reason at all that volume ratio monotonicity holds on it.
Therefore, in order to set up weak compactness for general K\"ahler manifolds, we have to give up the volume ratio monotonicity and search for a new route.
This will be done in this section.

\subsection{Motivation and definition}

 Let us continue the discussion in Section 2.6.
 As a consequence of the weak compactness of $\widetilde{\mathscr{KS}}(n,\kappa)$, we have
density estimate of volume radius, Proposition~\ref{prn:SB25_2}.  
For simplicity of notation,  we fix some $p_0$ very close to $2$, say $p_0=2-\frac{1}{1000n}$. 
Define
 \begin{align}
    \mathbf{E} \triangleq E(n,\kappa,p_0)+200 \omega_{2n} \kappa^{-1}. \label{eqn:SC17_11}
 \end{align}
 Here we adjust the number $E(n,\kappa,p_0)$ to a much larger number, to reserve spaces for later use.
 Then Proposition~\ref{prn:SB25_2}  implies
 \begin{align}
   r^{2p_0-2n} \int_{B(x,r)} \mathbf{vr}(y)^{-2p_0} dy < \mathbf{E}.  \label{eqn:SL23_5}
 \end{align}
 The above inequality contains a lot of information. For example,  it immediately implies that in every unit  ball,
 there exists a fixed sized sub-ball with uniform regularity.

\begin{proposition}[\textbf{Generic regular sub-ball}]
 Suppose $(X,x_0,g) \in \widetilde{\mathscr{KS}}(n,\kappa)$, $r$ is a positive number. Then  we have
  \begin{align}
      \mathcal{F}_{c_b r} \cap B(x_0,r) \neq \emptyset, \label{eqn:SC29_10}
  \end{align}
where
\begin{align}
   c_b \triangleq \left(\frac{\omega_{2n} \kappa}{4 \mathbf{E}} \right)^{\frac{1}{2p_0}}.
\label{eqn:SL23_6}
\end{align}
\label{prn:HE11_1}
\end{proposition}

\begin{proof}
Let $\mathbf{vr}$ achieve maximum value at $y_0$ in the ball closure $\overline{B(x_0,r)}$.
By inequality (\ref{eqn:SL23_5}),  we have
\begin{align*}
   \mathbf{vr}(y_0)^{-2p_0} \leq   \fint_{B(x,r)} \mathbf{vr}(y)^{-2p_0} dy
   \leq (\omega_{2n} \kappa)^{-1} r^{-2n} \int_{B(x,r)} \mathbf{vr}(y)^{-2p_0} dy
   \leq (\omega_{2n} \kappa)^{-1} r^{-2p_0} \mathbf{E}.
\end{align*}
It follows that
\begin{align*}
   \mathbf{vr}(y_0) \geq \left(\frac{\omega_{2n}\kappa}{\mathbf{E}} \right)^{\frac{1}{2p_0}}r
   >c_br.
\end{align*}
By continuity of $\mathbf{vr}$, there must exist a point $z \in B(x_0,r)$ such that  $\mathbf{vr}(z) >c_b r$.
In other words, we have $z \in \mathcal{F}_{cr} \cap B(x_0,r)$.  So (\ref{eqn:SC29_10}) holds.
\end{proof}

Let $\mathbf{E}$ and $c_b$ be the constants defined in (\ref{eqn:SC17_11}) and (\ref{eqn:SL23_6}).
Then we can choose a small constant $\epsilon_b$ such that
\begin{align}
  \epsilon_b \triangleq \epsilon\left(n,\kappa,\frac{c_b}{100} \right)
\label{eqn:SL23_7}
\end{align}
by the dependence in (\ref{eqn:SC26_44}) of Proposition~\ref{prn:SC26_2}.
Combining the estimates in $\widetilde{\mathscr{KS}}(n,\kappa)$, we obtain the following theorem.

 \begin{theorem}[\textbf{A priori estimates in model spaces}]
  Suppose $(X,x_0,g) \in \widetilde{\mathscr{KS}}(n,\kappa)$, $r$ is a positive number.  Then the following estimates
  hold.
   \begin{enumerate}
  \item Strong volume ratio estimate: $\kappa \leq \omega_{2n}^{-1}r^{-2n}|B(x_0,r)| \leq 1$.
  \item Strong regularity estimate: $r^{2+k}|\nabla^k Rm|\leq c_a^{-2}$ in the ball $B(x_0, c_a r)$ for every $0 \leq k \leq 5$ whenever     $\mathbf{vr}(x_0) \geq r$.
  \item Strong density estimate: $\displaystyle r^{2p_0-2n} \int_{B(x_0, r)} \mathbf{vr}(y)^{-2p_0} dy \leq \mathbf{E}$.
  \item Strong connectivity estimate: Every two points $y_1,y_2 \subset B(x_0,r) \cap \mathcal{F}_{\frac{1}{100}c_b r}(X)$
  can be connected by a shortest geodesic $\gamma$ such that
  $\gamma \subset  \mathcal{F}_{\epsilon_b r} (X)$.
\end{enumerate}
\label{thm:SL21_1}
\end{theorem}

We shall show that a weak compactness of $\widetilde{\mathscr{KS}}(n,\kappa)$ can be established using the estimates in Theorem~\ref{thm:SL21_1},
without knowing the volume ratio monotonicity.
For this new route of weak compactness theory, we define a scale called canonical radius with respect to $\widetilde{\mathscr{KS}}(n,\kappa)$.
Under the canonical radius,  rough estimates like that in Theorem~\ref{thm:SL21_1} are satisfied.

In this section, we focus on the study of smooth complete K\"ahler manifold.
Every such a manifold is denoted by $(M^n, g, J)$, where $n$ is the complex dimension.
The Hausdorff dimension, or real dimension of $M$ is  $m=2n$.
We first need to make sense of the rough volume radius, without the volume ratio monotonicity.
\begin{definition}
Denote the set $\left\{ r \left| 0<r<\rho, \omega_{2n}^{-1}r^{-2n}|B(x_0,r)|\geq 1-\delta_0 \right. \right\}$ by
 $I_{x_0}^{(\rho)}$ where $x_0 \in M$, $\rho$ is a positive number.
Clearly, $I_{x_0}^{(\rho)} \neq \emptyset$ since $M$ is smooth. Define
  \begin{align*}
    \mathbf{vr}^{(\rho)}(x_0) \triangleq  \sup I_{x_0}^{(\rho)}.
  \end{align*}
For each pair $0<r \leq \rho$, define
\begin{align*}
  &\mathcal{F}_{r}^{(\rho)}(M) \triangleq \left\{ x \in M | \mathbf{vr}^{(\rho)}(x) \geq r  \right\}, \\
  &\mathcal{D}_{r}^{(\rho)}(M) \triangleq \left\{ x \in M | \mathbf{vr}^{(\rho)}(x) < r  \right\}.
\end{align*}
\label{dfn:SC24_1}
\end{definition}

\begin{definition}
  A subset $\Omega$ of $M$ is called $\epsilon$-regular-connected on the scale $\rho$ if every two points $x,y \in \Omega$ can be connected by a
   rectifiable curve $\gamma \subset \mathcal{F}_{\epsilon}^{(\rho)}$ and $|\gamma| < 2d(x,y)$. 
   For notational simplicity, if the scale is clear in the context, we shall just say
   $\Omega$ is $\epsilon$-regular-connected. 
\label{dfn:SC15_3}
\end{definition}

Inspired by the estimates in Theorem~\ref{thm:SL21_1}, we can define the concept of canonical radius as follows.

\begin{definition}
We say that the canonical radius (with respect to model space $\widetilde{\mathscr{KS}}(n,\kappa)$) of a point $x_0 \in M$ is
not less than $r_0$ if for every  $r < r_0$, we have the following properties.
\begin{enumerate}
  \item Volume ratio estimate: $\kappa \leq \omega_{2n}^{-1}r^{-2n}|B(x_0,r)| \leq \kappa^{-1}$.
  \item Regularity estimate: $r^{2+k}|\nabla^k Rm|\leq 4 c_a^{-2}$ in the ball $B(x_0, \frac{1}{2}c_a r)$ for every $0 \leq k \leq 5$ whenever
        $\omega_{2n}^{-1}r^{-2n}|B(x_0,r)| \geq 1-\delta_0$.
  \item Density estimate: $\displaystyle r^{2p_0-2n} \int_{B(x_0, r)} \mathbf{vr}^{(r)}(y)^{-2p_0} dy \leq 2\mathbf{E}$.
  \item Connectivity estimate: $B(x_0,r) \cap \mathcal{F}_{\frac{1}{50}c_b r}^{(r)}(M)$ is $\frac{1}{2}\epsilon_b r$-regular-connected on the scale $r$.
\end{enumerate}
Then we define canonical radius of $x_0$ to be the supreme of all the $r_0$ with the properties mentioned above.
We denote the canonical radius by $\mathbf{cr}(x_0)$.
For subset $\Omega \subset M$, we define  the canonical radius of $\Omega$ as the infimum of all $\mathbf{cr}(x)$ where $x \in \Omega$.
We denote this canonical radius by $\mathbf{cr}(\Omega)$.
\label{dfn:SC02_1}
\end{definition}

\begin{remark}
 In Definition~\ref{dfn:SC02_1}, the first condition(volume ratio estimate) is used to guarantee the existence of Gromov-Hausdorff limit. The second condition (regularity estimate) is for the purpose of improving regularity.  The third condition (density estimate), together with the second condition(regularity estimate),
 implies that the regular part is almost dense(c.f.~Theorem~\ref{thm:HE11_1}). The fourth condition(connectivity estimate) is defined to assure that the regular part is connected (c.f.~Proposition~\ref{prn:SB27_1}).
\label{rmk:SB16_1}
\end{remark}

Because of the regularity estimate of Definition~\ref{dfn:SC02_1}, it is useful to define the concept of canonical volume radius as follows.

\begin{definition}
Suppose $\rho_0=\mathbf{cr}(x_0)$. Then we define
  \begin{align}
    \mathbf{cvr}(x_0) \triangleq \mathbf{vr}^{(\rho_0)}(x_0).  \label{eqn:SC24_4}
  \end{align}
We call $\mathbf{cvr}(x_0)$ the canonical volume radius of the point $x_0$.
\label{dfn:SB25_2}
\end{definition}

\begin{remark}
 For every compact smooth manifold $M$,  there is an $\eta>0$ such that
 every geodesic ball with radius less than $\eta$ must have normalized volume radius at least $1-\delta_0$.
 Then it is easy to see that $\displaystyle r^{2p_0-2n} \int_{B(x_0, r)} \mathbf{vr}^{(r)}(y)^{-2p_0} dy$ is a continuous
 function with respect to $x_0$ and $r$.    Therefore, if $\rho_0=\mathbf{cr}(x_0)$ is a finite positive number, we have
 \begin{align}
   \displaystyle \rho_0^{2p_0-2n} \int_{B(x_0, \rho_0)} \mathbf{vr}^{(\rho_0)}(y)^{-2p_0} dy \leq 2\mathbf{E}.
 \label{eqn:SL25_1}
 \end{align}
 If $r \leq \mathbf{cr}(M)$, then $\mathbf{vr}^{(r)} \leq \mathbf{cvr}$ as functions.  Therefore, we have
 \begin{align}
      r^{2p_0-2n} \int_{B(x_0, r)} \mathbf{cvr}(y)^{-2p_0} dy \leq
      r^{2p_0-2n} \int_{B(x_0, r)} \mathbf{vr}^{(r)}(y)^{-2p_0} dy \leq 2\mathbf{E}.
 \label{eqn:SL25_2}
 \end{align}
\label{rmk:SL22_1}
\end{remark}

Let $r_0$ be $\mathbf{cvr}(x_0)$. By Definition~\ref{dfn:SB25_2}, it is clear that $r_0 \leq \mathbf{cr}(x_0)$.
If $r_0=\mathbf{cvr}(x_0)<\mathbf{cr}(x_0)$, then we have
  \begin{align}
    &\omega_{2n}^{-1}r_0^{-2n}|B(x_0,r_0)| = 1-\delta_0,  \label{eqn:SC16_2}\\
    &\omega_{2n}^{-1}r^{-2n}|B(x_0,r)| < 1-\delta_0, \quad \forall \; r \in (r_0, \mathbf{cr}(x_0)).  \label{eqn:SC16_3}
  \end{align}
If $r_0=\mathbf{cvr}(x_0)=\mathbf{cr}(x_0)$, then we only have
  \begin{align}
    \omega_{2n}^{-1} r_0^{-2n}|B(x_0,r_0)| \geq 1-\delta_0.   \label{eqn:SC17_1}
  \end{align}
It is possible that equality (\ref{eqn:SC16_2}) does not hold on the scale $r_0$ in this case.

\begin{remark}
The three radii functions, $\mathbf{cr}$,$\mathbf{vr}$ and $\mathbf{cvr}$ are all positive functions on the interior part of $M$.
However, we do not know whether they are continuous in general.
\label{rmk:DA25_1}
\end{remark}

We shall use canonical radius as a tool to study the weak-compactness theory of K\"ahler manifolds.

\subsection{Rough estimates when canonical radius is bounded from below}
We assume $\mathbf{cr}(M) \geq 1$ in the following discussion of this subsection.
Under this condition,  we collect important estimates for the development of weak-compactness.

For simplicity of notation, we denote
\begin{align}
\mathcal{F}_{r} \triangleq \mathcal{F}_{r}^{(\mathbf{cr}(M))},  \quad
\mathcal{D}_{r} \triangleq  \mathcal{D}_{r}^{(\mathbf{cr}(M))}.
\label{eqn:SL27_1}
\end{align}
Note that this definition can be regarded as the generalization of the corresponding definition for metric spaces in $\widetilde{\mathscr{KS}}(n,\kappa)$.
It coincides the original one since $\mathbf{cr}(M)=\infty$ whenever $M \in \widetilde{\mathscr{KS}}(n,\kappa)$.

\begin{proposition}
  For every $0<r\leq \rho_0 \leq 1$, $x_0 \in M$,  we have
\begin{align}
  &\left| B(x_0,\rho_0) \cap \mathcal{D}_r \right| < 4\mathbf{E} \rho_0^{2n-2p_0} r^{2p_0}, \label{eqn:SC02_2} \\
  &\left| B(x_0,\rho_0) \cap \mathcal{F}_r \right| > \left(\kappa \omega_{2n} - 4\mathbf{E}r^{2p_0}\rho_0^{-2p_0} \right) \rho_0^{2n}.
    \label{eqn:SC02_4}
\end{align}
In particular, there exists at least one point $z \in B(x_0,\rho_0)$ such that
\begin{align}
    \mathbf{cvr}(z) > c_b \rho_0,
 \label{eqn:SC02_5}
 \end{align}
where $c_b=\left(\frac{\kappa \omega_{2n}}{4 \mathbf{E}} \right)^{\frac{1}{2p_0}}$.
\label{prn:SC02_2}
\end{proposition}

\begin{proof}
Recall that  $\mathbf{vr}^{(\mathbf{cr}(M))} \geq \mathbf{vr}^{(\rho_0)}$. By density estimate(c.f. Definition~\ref{dfn:SC02_1}), we have
\begin{align*}
  r^{-2p_0}  \left|B(x_0,\rho_0) \cap \mathcal{D}_r \right|
  \leq \int_{B(x_0,\rho_0) \cap  \mathcal{D}_{r}}\left\{ \mathbf{vr}^{(\mathbf{cr}(M))}\right\}^{-2p_0}
  \leq  \int_{B(x_0,\rho_0)}  \left\{ \mathbf{vr}^{(\rho_0)}\right\}^{-2p_0} \leq 2\mathbf{E} \rho_0^{2n-2p_0}.
\end{align*}
Then (\ref{eqn:SC02_2}) follows from above inequality. Recall that $\mathcal{D}_r$ is the set where $\mathbf{vr}^{(\mathbf{cr}(M))}<r$.
Together with the $\kappa$-non-collapsing condition, (\ref{eqn:SC02_2}) yields (\ref{eqn:SC02_4}).
Let $r=c_b \rho_0$, then (\ref{eqn:SC02_2}) implies
\begin{align*}
  \left| B(x_0,\rho_0) \cap \mathcal{F}_{c_b \rho_0} \right| >0.
\end{align*}
In particular, $B(x_0,\rho_0) \cap \mathcal{F}_{c_b \rho_0} \neq \emptyset$.  In other words, we can find a point $z \in B(x_0, \rho_0)$ satisfying
$\mathbf{vr}^{(\mathbf{cr}(M))}>c_b \rho_0$ and consequently inequality (\ref{eqn:SC02_5}).
\end{proof}

\begin{corollary}
  Suppose $x_0 \in M$, $H \geq 1 \geq r$, then we have
\begin{align}
  \left| B(x_0,H) \cap \mathcal{D}_r \right| \leq \left(\frac{2^{2n+2}\left|B(x_0, 2H) \right|}{\kappa \omega_{2n}} \right) r^{2p_0} \mathbf{E}.
\label{eqn:SC06_1}
\end{align}
\label{cly:SC04_2}
\end{corollary}

\begin{proof}
  Try to fill the ball $B(x_0,H)$ with balls $B(y_i,\frac{1}{2})$ such that $y_i \in B(x_0,H)$ until no more such balls can squeeze in. Clearly, we have
  \begin{align*}
      B(x_0,H) \subset \bigcup_{i=1}^{N} B(y_i,1), \quad   \bigcup_{i=1}^{N} B \left(y_i, \frac{1}{2} \right) \subset B \left(x_0,H+\frac{1}{2} \right) \subset B(x_0,2H).
  \end{align*}
  On one hand,  the balls $B\left(y_i,\frac{1}{2}\right)$ are disjoint to each other. So we have
  \begin{align}
                 N\kappa \omega_{2n} \left( \frac{1}{2}\right)^{2n} \leq  \sum_{i=1}^{N} \left|B(y_i,1)\right|
                 \leq \left|B(x_0,2H)\right|, \quad
                 \Rightarrow \quad N \leq \frac{2^{2n} |B(x_0,2H)|}{\kappa \omega_{2n}}.
  \label{eqn:SC06_2}
  \end{align}
  On the other hand,  $B(x_0,H)$ is covered by $\displaystyle \bigcup_{i=1}^{N} B(y_i,1)$. So we have
  \begin{align*}
    \left| B(x_0,H) \cap \mathcal{D}_r \right| \leq \sum_{i=1}^{N} \left|  B(y_i,1) \cap \mathcal{D}_r \right| \leq 4N\mathbf{E} r^{2p_0} \leq \left(\frac{2^{2n+2}\left|B(x_0, 2H) \right|}{\kappa \omega_{2n}} \right) r^{2p_0} \mathbf{E},
  \end{align*}
  where we used (\ref{eqn:SC06_2}) and (\ref{eqn:SC02_2}).
\end{proof}

\begin{proposition}
 For every $r \leq 1$, two points $x,y  \in \mathcal{F}_{r}$ can be connected by a curve $\gamma \subset \mathcal{F}_{\frac{1}{2}\epsilon_b r}$  with length $|\gamma| <3d(x,y)$.
  \label{prn:SB27_1}
\end{proposition}

\begin{proof}
  By rescaling if necessary, we can assume $r=1$. Then $\mathbf{cr}(M) \geq 1$.

  Suppose $x,y \in \mathcal{F}_{1}$. If $d(x,y) \leq 1$, then there is a curve connecting $x,y$ and it satisfies the requirements,  by the $\frac{1}{2}\epsilon_b$-regular connectivity property of the canonical radius. So we assume $H=d(x,y) >1$ without loss of generality.

 Let $\beta$ be a shortest geodesic connecting $x,y$ such that $\beta(0)=x$ and $\beta(H)=y$.
 Let $N$ be an integer locating in $[2H, 2H+1]$. Define
 \begin{align*}
   s_i=\frac{Hi}{N},  \quad x_i=\beta(s_i).
 \end{align*}
 Clearly, $x_0=x, x_N=y$, which are both in $\mathcal{F}_1 \subset \mathcal{F}_{\frac{c_b}{50}}$.
 For each $1 \leq i \leq N-1$, $x_i$ may not locate in $\mathcal{F}_{\frac{c_b}{50}}$. However,
 in the ball $B(x_i, \frac{1}{20})$, there exists a point $x_i'$ such that
 \begin{align*}
   \mathbf{vr}(x_i') \geq \frac{1}{2} c_b \cdot \frac{1}{20}=\frac{c_b}{40}> \frac{c_b}{50}.
 \end{align*}
 Clearly, we have
 \begin{align*}
   d(x_i', x_{i+1}') \leq d(x_i', x_i) + d(x_i, x_{i+1}) + d(x_{i+1}, x_{i+1}') \leq \frac{H}{N} + \frac{1}{10} < \frac{3}{5} \leq 1.
 \end{align*}
 Since $\mathbf{cr}(M) \geq 1$, one can apply $\frac{1}{2}\epsilon_b$-regular connectivity property of the canonical radius
 to find a curve $\beta_i$ connecting $x_i'$ and $x_{i+1}'$ such that
 $\beta_i \subset \mathcal{F}_{\frac{1}{2}\epsilon_b}$. Moreover, we have
 \begin{align*}
   |\beta_i| \leq 2d(x_i', x_{i+1}') \leq 2\left(\frac{H}{N} + \frac{1}{10} \right).
 \end{align*}
 Concatenating all $\beta_i$'s, we obtain a curve $\gamma$ connecting $x=x_0,y=x_{N}$ and $\gamma \subset \mathcal{F}_{\frac{1}{2} \epsilon_b}$.
 Furthermore, we have
 \begin{align*}
   |\gamma|=\sum_{i=0}^{N-1} |\beta_i| \leq 2N \left( \frac{H}{N} + \frac{1}{10} \right)= 2H +\frac{1}{15}N \leq 2H + \frac{2H+1}{5} \leq \frac{12}{5}H + \frac{1}{5}
   < \frac{13}{5}H <3H.
 \end{align*}
\end{proof}

\begin{corollary}
 For every $x \in M$, $0<r \leq 1$, we can find a curve $\gamma$ connecting $\partial B(x, \frac{r}{2})$
 and $\partial B(x,r)$ such that
 \begin{align*}
  \gamma \subset \mathcal{F}_{\frac{1}{2}\epsilon_br},  \quad |\gamma| \leq 2r.
 \end{align*}
 In particular, we have
 \begin{align*}
    \partial B(x, r) \cap \mathcal{F}_{\frac{\epsilon_b}{2}r} \neq \emptyset.
 \end{align*}
 \label{cly:SL24_1}
\end{corollary}

\begin{proof}
 Let $\beta$ be a shortest geodesic connecting $x$ and some point $y \in \partial B(x, \frac{9}{8}r)$.
 Let $z$ be the intersection point of $\beta$ and $\partial B(x, \frac{3}{8}r)$.  Let $y', z'$ be regular points
 around $y,z$, i.e., we require
 \begin{align*}
     y' \in B\left(y,\frac{r}{8} \right) \cap \mathcal{F}_{\frac{c_b}{8}r}, \quad
     z' \in B\left(z,\frac{r}{8} \right) \cap \mathcal{F}_{\frac{c_b}{8}r}.
 \end{align*}
 Clearly, triangle inequality implies that
  \begin{align*}
     d(y',z') \leq \frac{6}{8}r + \frac{1}{8}r + \frac{1}{8}r=r \leq 1.
 \end{align*}
 Since $\mathbf{cr}(M) \geq 1$, by connectivity estimate, there is a curve $\alpha$ connecting $y'$ and $z'$ such that
 \begin{align*}
    |\alpha| \leq 2r,  \quad \alpha \subset \mathcal{F}_{\frac{\epsilon_b}{2}r}.
 \end{align*}
 Note that $z' \in B(x,\frac{r}{2})$ and $y' \in B(x,r)^{c}$.  The connectedness of $M$ guarantees that
 $\alpha$ must have intersection with both $\partial B(x,\frac{r}{2})$ and $\partial B(x,r)$.
 So we can truncate $\alpha$ to obtain a curve $\gamma$ which connects $\partial B(x,\frac{r}{2})$ and $\partial B(x,r)$. Clearly, we have
 \begin{align*}
    \gamma \subset \alpha \subset \mathcal{F}_{\frac{\epsilon_b}{2}r},
    \quad  |\gamma| \leq |\alpha| \leq 2r.
 \end{align*}
\end{proof}

\begin{proposition}
  Suppose $x \in M$, $0<r \leq 1$.  Then for every point $y \in \mathcal{F}_{\frac{1}{2}\epsilon_br} \cap \partial B(x,r)$.
  There is a curve $\gamma$ connecting $x$ and $y$ such that
  \begin{itemize}
  \item $|\gamma|<10r$.
  \item For each nonnegative integer $i$,
    $\gamma \cap B(x,2^{-i}r) \backslash B\left(x, 2^{-i-1}r \right)$ contains a component which connects
           $\partial B(x,2^{-i}r)$ and $\partial B(x, 2^{-i-1}r)$ and is contained in
           $\mathcal{F}_{2^{-i-3}\epsilon_b^2r}$.
  \end{itemize}
\label{prn:SL24_1}
\end{proposition}

\begin{proof}
  Choose $y_i$ be a point on $\partial B(x, 2^{-i}r) \cap \mathcal{F}_{2^{-i-1}\epsilon_b r}$.
  By Proposition~\ref{prn:SB27_1},  for each $i \geq 0$, there is a curve $\gamma_i$ connecting
  $y_i$ and $y_{i+1}$ such that
  \begin{align*}
    |\gamma_i| < 9 \cdot 2^{-i-1} r, \quad   \gamma_i \subset \mathcal{F}_{2^{-i-3}\epsilon_b^2 r}.
  \end{align*}
  Concatenate all the $\gamma_i$'s to obtain $\gamma$. Then $\gamma$ satisfies all the properties.
\end{proof}

For the purpose of improving regularity, we need to study  the behavior of $\mathbf{cvr}$.
Similar to $\mathbf{vr}$ on spaces in $\widetilde{\mathscr{KS}}(n,\kappa)$ (c.f. Proposition~\ref{prn:SC26_7}),
$\mathbf{cvr}$ satisfies a local Harnack inequality.

\begin{proposition}
There is a constant $K=K(n,\kappa)$ with the following properties.

Suppose $x \in X$, $r=\mathbf{cvr}(x)<\frac{1}{K}$, then for every point $y \in B(x, K^{-1}r)$, we have
\begin{align}
   &K^{-1}r  \leq \mathbf{cvr}(y) \leq Kr, \label{eqn:SC24_1}\\
   &\omega_{2n}^{-1}\rho^{-2n} |B(y,\rho)| \geq 1-\frac{1}{100}\delta_0, \quad \forall \; \rho \in (0, K^{-1}r), \label{eqn:SC24_2}\\
   & |Rm|(y) \leq K^{2} r^{-2}, \label{eqn:SC29_6}\\
   & inj(y) \geq K^{-1} r. \label{eqn:SC29_7}
\end{align}
\label{prn:SC24_1}
\end{proposition}

\begin{corollary}
 For every $r \in (0,1]$, $\mathcal{F}_{r}(M)$ is a closed set. Moreover,
 $\mathbf{cvr}$ is an upper-semi-continuous function on $\mathcal{F}_{r}(M)$.
\label{cly:SC24_1}
\end{corollary}

\begin{proof}
  Fix $r \in (0,1]$. Suppose $x_i \in \mathcal{F}_{r}(M)$ converges to a point $x \in M$.
  Let $r_i=\mathbf{cvr}(x_i)$.  We need to show $\mathbf{cvr}(x) \geq r$. Clearly, this follows directly
  if $\displaystyle \lim_{i \to \infty} r_i=\infty$ by Proposition~\ref{prn:SC24_1}. Without loss of generality, we may assume
  $r_i$ is uniformly bounded from above.  Use Proposition~\ref{prn:SC24_1} again, we see that
  $r_i$ is uniformly bounded away from zero.  Let $r$ be a limit of $r_i$. Then we have
  \begin{align}
   |B(x,r)|=\lim_{i \to \infty} |B(x_i,r_i)| \geq \lim_{i \to \infty} (1-\delta_0) \omega_{2n} r_i^{2n}
   = (1-\delta_0) \omega_{2n} r^{2n},
  \end{align}
  which implies $\displaystyle \mathbf{cvr}(x) \geq r= \lim_{i \to \infty} r_i$ by definition of canonical volume radius and the fact that
  $\displaystyle r=\lim_{i \to \infty} r_i\leq 1 \leq \mathbf{cvr}$.
  Consequently, we have $x \in \mathcal{F}_{r}(M)$. Therefore, $\mathcal{F}_{r}(M)$ is a closed set by the arbitrary choice of $\{x_i\}$.
  From the above argument, we have already seen that
  \begin{align}
    \mathbf{cvr}(x) \geq \lim_{i \to \infty} \mathbf{cvr}(x_i),
  \end{align}
  which means that $\mathbf{cvr}$ is an upper-continuous function on $\mathcal{F}_{r}(M)$.
\end{proof}

Clearly, the conclusion in the above corollary is weaker than that in Proposition~\ref{prn:HE07_2},  since here we do not have a rigidity property like Proposition~\ref{prn:SC17_3}.   However,  even if
$\displaystyle \mathbf{cvr}(x)>\lim_{i \to \infty} \mathbf{cvr}(x_i)$, the local Harnack inequality of $\mathbf{cvr}$ guarantees that $\displaystyle \mathbf{cvr}(x) < K \lim_{i \to \infty} \mathbf{cvr}(x_i)$.  So $\mathbf{cvr}$ is better than
general semi-continuous function.   For example, in the decomposition $M=\mathcal{F}_r \cup \mathcal{D}_r$,
every point $y \in \partial \mathcal{F}_r$ satisfies $r \leq \mathbf{cvr}(y) \leq Kr$.  In many situations, it is convenient to just regard $K=1$, i.e., $\mathbf{cvr}$ being continuous, without affecting the effectiveness of the argument.
Furthermore,  up to perturbation, one can even regard $\mathbf{cvr}$ as smooth functions.   Full details of the perturbation can be found in Appendix~\ref{app:B}.

\subsection{K\"ahler manifolds with canonical radius bounded from below}

Similar to the traditional theory, volume convergence is very important.  However, in the current situation, the volume convergence can be proved in a much easier way.

\begin{proposition}[\textbf{Volume convergence}]
 Suppose $(M_i,g_i,J_i)$ is a sequence of  K\"ahler manifolds satisfying $\mathbf{cr}(M_i) \geq r_0$.  Then
 we have \begin{align*}
  (M_i, x_i, g_i) \longright{G.H.} (\bar{M}, \bar{x}, \bar{g}).
\end{align*}
 Moreover, the volume (2n-dimensional Gromov-Hausdorff measure) is continuous under this convergence, i.e., for every fixed $\rho_0>0$, we have
\begin{align*}
    |B(\bar{x}, \rho_0)|= \lim_{i \to \infty} |B(x_i,\rho_0)|.
\end{align*}
\label{prn:SC12_1}
\end{proposition}

\begin{proof}
The existence of the Gromov-Hausdorff limit space follows from the volume doubling property and the standard ball-packing argument.  Fix $r<<\rho_0$, then it follows from the definition of $\mathcal{F}_r$ that the convergence
on $B(x_i, \rho_0) \cap \mathcal{F}_r$ can be improved to $C^4$-topology. Then the volume converges trivially on this part. On the other hand, the volume of $B(x_i, \rho_0) \cap \mathcal{D}_r$ is bounded by $C r^{2p_0}$, which tends to
zero as $r \to 0$.  So the volume convergence of geodesic balls $B(x_i,\rho_0)$ follows from the combination of the two factors mentioned above.   More details are given as follows.

Let  $\left( \bar{M}, \bar{x}, \bar{g} \right)$ be
the limit space.  For each $r \leq r_0$, define
\begin{align}
  &\mathcal{R}_{r} \triangleq  \left\{ \bar{y} \in \bar{M}
  \left| \textrm{There exists} \; y_i \in M_i \; \textrm{such that} \; y_i \to \bar{y} \; \textrm{and} \; \liminf_{i \to \infty} \mathbf{cvr}(y_i) \geq r \right. \right\},  \label{eqn:HA11_3}\\
  &\mathcal{S}_{r} \triangleq \left( \mathcal{R}_{r} \right)^{c},  \label{eqn:HA11_4}\\
  &\mathcal{R}' \triangleq  \bigcup_{0<r\leq r_0} \mathcal{R}_{r},  \label{eqn:GD25_2} \\
  &\mathcal{S}' \triangleq \bigcap_{0<r\leq r_0} \mathcal{S}_{r}.   \label{eqn:GD25_3}
\end{align}
We now show that $\mathcal{S}'$ is a subset of $\bar{M}$ of Minkowski dimension at most $2n-2p_0$.   
Without loss of generality, it suffices to show this dimension for $\mathcal{S}' \cap B(\bar{x}, \rho_0)$.  

For each small $r>0$, we shall construct a covering for the set $\mathcal{S}_r \cap B(\bar{x}, \rho_0)$.   Clearly, the choice
$\cup_{z \in  \mathcal{S}_r \cap B(\bar{x}, \rho_0)} B(z, r)$ is a cover, but with uncoutable many balls.
By Vitali covering lemma, we can find countable many $z_k$'s such that $\cup_{z_k} B(z_k, r)$ is a disjoint union and  
\begin{align}
\mathcal{S}_r \cap B(\bar{x}, \rho_0) \subset \bigcup_{z_k} B(z_k, 5r).     \label{eqn:GD25_1}   
\end{align}
We shall show that this covering is actually a finite covering with number of balls $N$ uniformly bounded by $Cr^{2p_0-2n}$. 
Let $z_k$ be the limit point of $z_{k,i} \in M_i$.  For large $i$, it follows from definition that $\mathbf{cvr}(z_{k,i}) <2r$. 
By Proposition~\ref{prn:SC24_1}, we see that $B(z_{k,i}, 5r) \subset \mathcal{D}_{5Kr}$.  It follows that
\begin{align*}
   \bigcup_{k=1}^{} B(z_{k,i}, 0.5 r) \subset \bigcup_{k} B(z_{k,i}, 5r) \subset   \left\{ B(x_i, 2\rho_0) \cap \mathcal{D}_{5Kr} \right\}. 
\end{align*}
Note that $\bigcup_{k} B(z_{k,i}, 0.5 r)$ is a disjoint union. 
Taking volume on the manifold $M_i$, using the volume ratio's lower bound and Proposition~\ref{prn:SC02_2}, we obtain
\begin{align*}
 N \kappa \omega_{2n} (0.5 r)^{2n} \leq  \sum_{k} |B(z_{k,i}, 0.5r)| \leq |B(x_i, 2\rho_0) \cap \mathcal{D}_{5Kr}| \leq C (5Kr)^{2p_0}.
\end{align*}
It follows that $N \leq Cr^{2p_0-2n}$ for some uniform constant $C$.  
Therefore, the covering we choose in (\ref{eqn:GD25_1}) is a finite covering with the number of balls dominated by $Cr^{2p_0-2n}$. 
Since $\mathcal{S}'$ is a subset of $\mathcal{S}_r$, we obtain a covering of $\mathcal{S}' \cap B(\bar{x},\rho_0)$ by size-$r$ balls with number at most $Cr^{2p_0-2n}$,
where $C$ is independent of $r$. Therefore, we have
\begin{align}
   \dim_{\mathcal{M}} \left\{ \mathcal{S}' \cap B(\bar{x}, \rho_0) \right\} \leq 2n-2p_0. 
\label{eqn:GD25_4}   
\end{align}
In particular, $\mathcal{S}' \cap B(\bar{x},\rho_0)$ has $2n$-Hausdorff measure zero, or volume zero.  This means we can ignore the effect of $\mathcal{S}'$ when we consider 
volume convergence.  On the other hand, away from $\mathcal{S}'$, the volume convergence is obvious.     We therefore obtain the volume convergence property
whenever $B(x_i,\rho_0)$ converges to $B(\bar{x},\rho_0)$. 
\end{proof}

Now we are able to show the weak compactness theorem.

\begin{theorem}[\textbf{Rough weak compactness}]
 Same conditions as Proposition~\ref{prn:SC12_1}.
 Denote $\mathcal{R} \subset \bar{M}$ as the set of regular points, i.e., the points with some small neighborhoods which have $C^4$-Riemannian manifolds structure.
 Denote $\mathcal{S} \subset \bar{M}$ be the set of singular points, i.e., the points which are not regular.
 Then we have the regular-singular decomposition $\bar{M}=\mathcal{R} \cup \mathcal{S}$ with the following properties.
 \begin{itemize}
 \item The regular part $\mathcal{R}$ is an open, path connected $C^4$-Riemannian manifold.
           Furthermore, for every two points $x,y \in \mathcal{R}$, there exists a curve $\gamma$ connecting $x,y$ satisfying
           \begin{align}
           \gamma \subset \mathcal{R}, \quad  |\gamma| \leq 3d(x,y).       \label{eqn:HE11_1}
           \end{align}
 \item The singular part $\mathcal{S}$ satisfies the Minkowski dimension estimate
 \begin{align}
   \dim_{\mathcal{M}} \mathcal{S} \leq 2n-2p_0.
 \label{eqn:SB13_4}
 \end{align}
 \end{itemize}
 \label{thm:HE11_1}
\end{theorem}

\begin{proof}
Let  $\left( \bar{M}, \bar{x}, \bar{g} \right)$ be
the limit space.  For each $r \leq r_0$, define $\mathcal{R}_r$, $\mathcal{S}_r$ as in (\ref{eqn:HA11_3}) and (\ref{eqn:HA11_4}). 
Define $\mathcal{R}', \mathcal{S}'$ as in (\ref{eqn:GD25_2}) and (\ref{eqn:GD25_3}). 
Recall that the regular set $\mathcal{R} \subset \bar{M}$ is defined as the collection of points which have small neighborhoods with manifolds structure.  
We shall show that $\mathcal{R}'$ is nothing but $\mathcal{R}$, i.e., $\displaystyle \mathcal{R} = \bigcup_{0<r\leq r_0} \mathcal{R}_{r}$.

Actually, by regularity estimate property of canonical radius, for every fixed $r \in (0,r_0)$, every point
$\bar{y} \in \mathcal{R}_{r}$, we see that the convergence to $B(\bar{y}, \frac{1}{3}c_a r)$ can be improved to be in the $C^{4}$-topology.  Clearly, $B(\bar{y}, \frac{1}{3}c_a r)$ has a  manifold structure. So $\mathcal{R}_r \subset \mathcal{R}$.  Let $r \to 0$, we have $\displaystyle \bigcup_{0<r\leq r_0} \mathcal{R}_{r} \subset \mathcal{R}$.
On the other hand, suppose $\bar{y} \in \mathcal{R}$. Then there is a ball $B(\bar{y},r)$ with a  manifold structure.
By shrinking $r$ if necessary, we can assume that the volume ratio of this ball is very close to the Euclidean one.
Note that the volume ($2n$-dimensional Hausdorff measure) converges when $(M_i, x_i, g_i)$ converges to $(\bar{M}, \bar{x}, \bar{g})$.
Suppose $y_i \to \bar{y}, y_i \in M_i$. Then we have $\omega_{2n}^{-1}r^{-2n}|B(y_i,r)|>1-\delta_0$ for large $i$. By definition, this means that
$\mathbf{cvr}(y_i,0) \geq r$. It follows from the regularity estimates that
$\displaystyle \bar{y} \in \mathcal{R}_{r} \subset \mathcal{R}'$.
By the arbitrary choice of $\bar{y}$, we obtain $\mathcal{R} \subset \mathcal{R}'$.
So we finish the proof of
\begin{align*}
  \mathcal{R}=\mathcal{R}'=\bigcup_{0<r\leq r_0} \mathcal{R}_{r}.
\end{align*}
Combining the above equation with the definitions in (\ref{eqn:GD25_2}) and (\ref{eqn:GD25_3}), we have $\mathcal{S}=\mathcal{S}'$.   Therefore, (\ref{eqn:SB13_4}) follows from $\dim_{\mathcal{M}} \mathcal{S}' \leq 2n-2p_0$, 
which can be proved following (\ref{eqn:GD25_4}).  Alternatively, we can prove (\ref{eqn:SB13_4}) as follows. 

Fix $r<r_0$. Let $\rho_0=r_0$ and take limit of (\ref{eqn:SC02_2}), we obtain
\begin{align}
   |B(\bar{y}, r_0) \cap  \mathcal{S}_{r}| \leq 4\mathbf{E}r_0^{2n-2p_0} r^{2p_0}, \label{eqn:B10_1}
\end{align}
for every $\bar{y} \in \bar{M}$.  Suppose $y \in \mathcal{R}_r \subset \bar{M}$.
The regularity estimate property of canonical radius yields that every point in $B(y, \frac{1}{4}c_a r)$ is regular.
So $d(y,\mathcal{S})\geq \frac{1}{4}c_a r$.
It follows that
\begin{align*}
 \mathcal{R}_{r} \subset \left\{x \in \bar{M} \left| d(x, \mathcal{S}) \geq \frac{1}{4}c_a r  \right.\right\}
 \Leftrightarrow \mathcal{S}_{r} \supset \left\{x \in \bar{M} \left| d(x, \mathcal{S}) <\frac{1}{4}c_a r \right. \right\}.
\end{align*}
Therefore, we have
\begin{align}
  \{x \in \bar{M} | d(x, \mathcal{S}) <r\} \subset \mathcal{S}_{4c_a^{-1} r},  \label{eqn:SB13_1}
\end{align}
whenever $r$ is very small. Combining (\ref{eqn:B10_1}) and (\ref{eqn:SB13_1}) yields
\begin{align*}
  \left| B(\bar{y},r_0) \cap \left\{x \in \bar{M} | d(x, \mathcal{S}) <r \right\} \right| \leq 4^{2p_0+1}\mathbf{E}c_a^{-2p_0} r_0^{2n-2p_0} r^{2p_0}=C r^{2p_0}.
\end{align*}
Since the above inequality holds for every small $r$ and every $\bar{y} \in \bar{M}$, it yields (\ref{eqn:SB13_4}) directly.

 It follows from definition that $\mathcal{R}$ is an open $C^4$-manifold.  The path connectedness of $\mathcal{R}$ follows from (\ref{eqn:HE11_1}).
 Now we proceed to show (\ref{eqn:HE11_1}).  Fix $x,y \in \mathcal{R}$, let $r=\sup\{\rho | x \in \mathcal{R}_{\rho}, y \in \mathcal{R}_{\rho}\}$.
 Since $r>0$, we can choose sequence $x_i, y_i \in \mathcal{F}_{\frac{r}{2}}(M_i)$ such that  $x_i \to x$, $y_i \to y$.
 Let $\gamma_i$ be a curve connecting $x_i, y_i$ constructed by the method described in Proposition~\ref{prn:SB27_1}.
 Clearly, $\gamma_i \subset \mathcal{F}_{\frac{1}{4}\epsilon_b r}$ and $|\gamma_i|<3d(x_i,y_i)$. Note that the convergence of
 $\mathcal{F}_{\frac{1}{4}\epsilon_b r}$ to its limit set is in $C^{4}$-topology.  Consequently, the limit curve $\gamma$ satisfies (\ref{eqn:HE11_1}).
\end{proof}

\begin{remark}
  The definition of regular points in Theorem~\ref{thm:HE11_1} is stronger than the classical one, i.e., a point is regular if and only if every tangent space at this point is
  isometric to $\C^{n}$.  Therefore, some regular points in the classical definition may be singular in our definition.
  Of course, at last, our definition coincides with the classical one after we set up sufficient estimates(c.f. Remark~\ref{rmk:HA07_1}).
\label{rmk:HA01_1}
\end{remark}

\begin{remark}
 Note that in Definition~\ref{dfn:SC02_1}, the definition of canonical radius, no K\"ahler condition is used.  Therefore, the convergence results discussed in this subsection works naturally in the Riemannian setting.
\label{rmk:MA21_1} 
\end{remark}

The properties of the limit space $\bar{M}$ in Theorem~\ref{thm:HE11_1} are not good enough. For example, we do not know if every tangent space is a metric cone, we do not know if $\mathcal{R}$
is convex. In general, one should not expect these to hold. However, if $(M_i, g_i, J_i)$ is a blowup sequence from  given K\"ahler Ricci flow solutions with proper geometric bounds, we shall show that $\bar{M}$ do have the mentioned good properties.

\section{Polarized canonical radius}

In this section, we shall improve the regularity of the limit pace $\bar{M}$ in Theorem~\ref{thm:HE11_1}, under the help of K\"ahler geometry and the Ricci flow.
The Ricci flow has reduced volume and local $W$-functional monotonicity, discovered by Perelman.  These monotonicities will be used to show
that each tangent space is a metric cone, and the regular part $\mathcal{R}$ is weakly convex, under some natural geometric conditions.  However, the weak-compactness we developed in last section only deals with the metric structure. On $\bar{M}$, we cannot see a Ricci flow structure. In order to make use of the intrinsic monotonicity of the Ricci flow, we need a weak compactness of Ricci flows, not just the weak compactness of time slices.  However, along the Ricci flow, the metric at different time slices cannot be compared obviously if no estimate of Ricci curvature is known.  This is one of the fundamental difficulty to develop the weak compactness theory of the Ricci flows.
We overcome this difficulty by taking advantage of the rigidity of K\"ahler geometry.

\subsection{A rough long-time pseudolocality theorem for polarized K\"ahler Ricci flow}

Suppose $\mathcal{LM}=\left\{ (M^n, g(t), J, L, h(t)),  t \ \in I \subset \R \right\}$  is a polarized K\"ahler Ricci flow.
Let $\mathbf{b}$ be the Bergman function with respect to $\omega(t)$ and $h(t)$, i.e.,
\begin{align*}
   \mathbf{b}(x,t) =\log \sum_{k=0}^{N} \norm{S_k}{h(t)}^2,  \quad \int_{M} \langle S_k,  S_l\rangle \omega(t)^n = \delta_{kl},
\end{align*}
where $N=\dim (H^0(L))-1$, $\{S_k\}_{k=0}^{N}$ are holomorphic sections of $L$.
By pulling back the Fubini-Study metric through the natural holomorphic embedding, we have
\begin{align*}
    \tilde{\omega}= \iota^*(\omega_{FS})= \omega+ \sqrt{-1} \partial \bar{\partial} \mathbf{b}.
\end{align*}
Let $\omega_0=\omega(0), \mathbf{b}_0=\mathbf{b}(0)$. Then
\begin{align*}
\omega(t)=\omega_0 + \sqrt{-1}\partial \bar{\partial} \varphi,
\quad
\tilde{\omega}(t)=\omega_0 + \sqrt{-1}\partial \bar{\partial} (\varphi + \mathbf{b}).
\end{align*}
Clearly, $\varphi(0)=0$.

In this section, we focus on polarized K\"ahler Ricci flow $\mathcal{LM}$ satisfying the following estimate
 \begin{align}
    \norm{\dot{\varphi}}{C^0(M)} + \norm{\mathbf{b}}{C^0(M)} + \norm{R}{C^0(M)} +|\lambda| +C_S(M) \leq B
 \label{eqn:SK23_1}
 \end{align}
 for every time $t \in I$.  Let $\mathbf{b}^{(k)}$ be the Bergman function of the line bundle $L^{k}$ with the naturally induced metric.
Then a standard argument implies that
 \begin{align}
    \norm{\dot{\varphi}}{C^0(M)} + \norm{\mathbf{b}^{(k)}}{C^0(M)} + \norm{R}{C^0(M)} +|\lambda| +C_S(M) \leq B^{(k)}
 \label{eqn:SK23_2}
 \end{align}
 for a constant $B^{(k)}$ depending on $B$ and $k$. Define
 \begin{align*}
    &\tilde{\omega}^{(k)} \triangleq \frac{1}{k} \left( \iota^{(k)} \right)^* (\omega_{FS}), \\
    &F^{(k)} \triangleq  \Lambda_{\omega_t} \tilde{\omega}_{0}^{(k)}=n- \Delta \left(\varphi-\mathbf{b}_0^{(k)} \right).
 \end{align*}
 In this section, the existence time of the polarized K\"ahler Ricci flow is always infinity, i.e., $I=[0,\infty)$.

\begin{lemma}[\textbf{Integral bound of trace}]
  Suppose $\mathcal{LM}$ is a polarized K\"ahler Ricci flow satisfying (\ref{eqn:SK23_1}).
  Suppose $u$ is a positive, backward heat equation solution, i.e.,
  \begin{align*}
     \square^* u=(-\partial_{t}-\Delta +R-\lambda n) u=0,
  \end{align*}
  and $\int_{M} u dv \equiv 1$. Then for every $t_0>0$, we have
  \begin{align}
    \int_0^{t_0} \int_{M} F^{(k)}u dv dt \leq \left(n+2B^{(k)} \right)t_0 + 2B^{(k)}.
    \label{eqn:SK07_2}
  \end{align}
  \label{lma:I29_1}
\end{lemma}

\begin{proof}
For simplicity of notation, we only give a proof of the case $k=1$ and denote $F=F^{(1)}$. Note that $\mathbf{b}=\mathbf{b}^{(1)}, B=B^{(1)}$.
The proof of general $k$ follows verbatim.

   Direct calculation shows that
   \begin{align*}
     \int_0^{t_0} \int_{M} Fu dv &=nt_0- \int_{0}^{t_0} \int_{M} \{\Delta (\varphi-\mathbf{b}_0)\} u dv\\
     &=nt_0-\int_0^{t_0}\int_{M} (\varphi-\mathbf{b}_0) (\Delta u) dv\\
     &=nt_0+ \int_{0}^{t_0} \int_{M} (\varphi-\mathbf{b}_0) \left( \dot{u}-Ru +\lambda n u \right) dv\\
     &=nt_0 + \int_{0}^{t_0}   \left[ \frac{d}{dt}\left( \int_{M} (\varphi-\mathbf{b}_0) u dv \right) - \int_{M}\dot{\varphi} u dv\right]dt\\
     &=nt_0+ \left. \int_{M} (\varphi-\mathbf{b}_0) u dv \right|_0^{t_0} - \int_0^{t_0} \int_{M} \dot{\varphi} u dv dt.
   \end{align*}
   Note $|\varphi| \leq Bt_0$ at time $t_0$,  then (\ref{eqn:SK07_2}) follows from the above inequality and (\ref{eqn:SK23_2}).
\end{proof}

We shall proceed to improve the integral estimate (\ref{eqn:SK07_2}) of $F^{(k)}$ to pointwise estimate, under local geometry bounds.
Before we go into details, let us first fix some notations.
Suppose $\mathcal{LM}$ is a polarized K\"ahler Ricci flow solution satisfying (\ref{eqn:SK23_1}),  $x_0 \in M$.
In this subsection, we shall always assume
\begin{align}
    \Omega \triangleq B_{g(0)}(x_0,r_0),  \quad \Omega' \triangleq B_{g(0)}(x_0, (1-\delta)r_0),
    \quad  \Omega'' \triangleq B_{g(0)}(x_0, (1-2\delta)r_0).
\label{eqn:SL26_3}
\end{align}
See Figure~\ref{figure:fourballs} for intuition. Then we define
\begin{align}
  w_0 \triangleq \phi \left(\frac{2(d-1+2\delta)}{\delta} \right),
\label{eqn:SL26_4}
\end{align}
where $d=d_{g(0)}(x_0, \cdot)$, $\phi$ is a cutoff function, which equals one on $(-\infty, 1]$, decreases to $0$ on $(1,2)$. Moreover, $(\phi')^2 \leq 10\phi$.
Note that such $\phi$ exists by considering the behavior of $e^{-\frac{1}{s}}$ around $s=0$.
Clearly, $w_0$ satisfies
\begin{align}
\begin{cases}
   &|\nabla w_0|^2 \leq \frac{40}{\delta^2}w_0, \\
   &w_0 \equiv 1, \quad \textrm{on} \quad \Omega'', \\
   &w_0 \equiv 0, \quad \textrm{on} \quad (\Omega')^c.
\end{cases}
\label{eqn:SL26_5}
\end{align}

\begin{lemma}[\textbf{Pointwise bound of trace}]
Suppose $\mathcal{LM}$ is a polarized K\"ahler Ricci flow satisfying (\ref{eqn:SK23_1}), $x_0 \in M$, $\Omega'$ is defined by (\ref{eqn:SL26_3}).
Suppose $\frac{1}{2} \omega_0 \leq \tilde{\omega}_0^{(k)} \leq 2 \omega_0$ on $\Omega'$.
Let $w$ be a solution of heat equation $\square w=\left( \frac{\partial}{\partial t} -\Delta \right)w=0$, initiating from a cutoff function $w_0$
satisfying (\ref{eqn:SL26_5}). Then for every $t_0>0$ and $y_0 \in M$, we have
  \begin{align}
      F^{(k)}(y_0,t_0)  w(y_0, t_0) \leq C   \label{eqn:SK23_3}
  \end{align}
where $C=C(B,k,\delta,t_0)$.
  \label{lma:J02_1}
\end{lemma}

\begin{proof}
For simplicity of notation, we only give a proof for the case $k=1$ and denote $F=F^{(1)}$, $B=B^{(1)}$, $H=\frac{40}{\delta^2}$.
The proof of general $k$ follows verbatim.

Note that $0 \leq w_0 \leq 1$, since $w$ is the heat solution, it follows from maximum principle that $0 \leq w \leq 1$.
On the other hand,  according to the choice of $w_0$, we have $|\nabla w|^2 -Hw \leq 0$ at the initial time.  Direct calculation implies that
\begin{align*}
  \square \left\{ e^{\lambda t}\left(|\nabla w|^2 -Hw\right) \right\}
  =-e^{\lambda t}\left\{|\nabla \nabla w|^2 +|\nabla \bar{\nabla} w|^2 +Hw\right\} \leq 0.
\end{align*}
Therefore, $|\nabla w|^2 -Hw \leq 0$ is preserved along the flow by maximum principle.
In other words, we always have
\begin{align*}
  w|\nabla \log w|^2 \leq H, \quad 0 \leq w \leq 1
\end{align*}
on the space-time $M \times [0,\infty)$.
In light of parabolic Schwarz lemma(c.f.~\cite{SongWeinnotes} and references therein), we obtain
\begin{align*}
 \square \log F \leq B F-\lambda.
\end{align*}
Note that
\begin{align*}
  \omega(t)=\omega_0+\sqrt{-1} \partial \bar{\partial} \varphi=\tilde{\omega}_0+\sqrt{-1} \partial \bar{\partial} (\varphi-\mathbf{b}_0)
  =\tilde{\omega}_0+\sqrt{-1} \partial \bar{\partial}\tilde{\varphi},
\end{align*}
where we denote $\varphi-\mathbf{b}_0$ by $\tilde{\varphi}$ for simplicity of notation. It is obvious that $\dot{\tilde{\varphi}}=\dot{\varphi}$.
Direct calculation shows that
\begin{align*}
  &\square \tilde{\varphi}=\dot{\varphi}-\Delta \tilde{\varphi}= F-n+\dot{\varphi}, \\
  &\square (\log F - B\tilde{\varphi})\leq B(n-\dot{\varphi})-\lambda \leq B(n+\norm{\dot{\varphi}}{C^0(\mathcal{M})})+|\lambda|
  \leq C.
\end{align*}
Let $u$ be the solution of $\square^*u=0$, starting from a $\delta$-function from $(y_0, t_0)$. Then we calculate
\begin{align*}
  \frac{d}{dt} \int_{M} Fe^{-B\tilde{\varphi}} wu dv
  &=\int_{M} \square(Fe^{-B\tilde{\varphi}}w) u dv - \int_{M} Fe^{-B\tilde{\varphi}}w \square^*u dv\\
     &=\int_{M} \square(Fe^{-B\tilde{\varphi}}w) u dv\\
     &=\int_{M} Fe^{-B\tilde{\varphi}}w \left\{\square \log (Fe^{-B\tilde{\varphi}}w)-|\nabla \log (Fe^{-B\tilde{\varphi}}w)|^2 \right\} udv\\
     &\leq \int_{M} Fe^{-B\tilde{\varphi}}w \left\{\square \log (Fe^{-B\tilde{\varphi}}w) \right\} udv\\
     &=\int_{M} Fe^{-B\tilde{\varphi}}w \left\{\square \log (Fe^{-B\tilde{\varphi}}) + \square \log w \right\} udv\\
     &=\int_{M} Fe^{-B\tilde{\varphi}}w \left\{\square \log (Fe^{-B\tilde{\varphi}}) + |\nabla \log w|^2 \right\} udv\\
     &\leq C \int_{M} Fe^{-B\tilde{\varphi}} w u dv + H \int_{M} Fe^{-B\tilde{\varphi}} u dv.
\end{align*}
It follows that
\begin{align*}
  \frac{d}{dt} \left\{e^{-Ct} \int_{M} Fe^{-B\tilde{\varphi}} w u dv \right\} \leq H e^{-Ct} \int_{M} Fu dv.
\end{align*}
Integrate the above inequality and apply Lemma~\ref{lma:I29_1}, we have
\begin{align*}
  e^{-Ct_0} F(y_0,t_0)w(y_0, t_0) e^{-B\tilde{\varphi}(y_0, t_0)} &\leq \left. \int_{M} Fe^{-B\tilde{\varphi}}wu dv \right|_{t=0} + H \int_0^{t_0} \int_{M} Fu dv dt\\
                       &\leq C \left. \int_{\Omega'} Fu dv \right|_{t=0}+C\\
                       &\leq C \left. \int_{\Omega'} u dv \right|_{t=0} + C \leq C.
\end{align*}
Therefore,  (\ref{eqn:SK23_3}) follows directly from the above inequality.
\end{proof}

\begin{figure}
 \begin{center}
 \psfrag{t0}[c][c]{$t=0$}
 \psfrag{t1}[c][c]{$t=t_0$}
 \psfrag{x0}[c][c]{$x_0$}
 \psfrag{A1}[c][c]{$(M,x_0,g(t_0))$}
 \psfrag{A2}[c][c]{$(M,x_0,g(0))$}
 \psfrag{B1}[c][c]{$\Omega=\textrm{blue}$}
 \psfrag{B2}[c][c]{$\Omega'=\textrm{green}$}
 \psfrag{B3}[c][c]{$\Omega''=\textrm{red}$}
 \psfrag{B4}[c][c]{$B_{g(t_0)}(x_0,r)=\textrm{black}$}
 \includegraphics[width=0.5 \columnwidth]{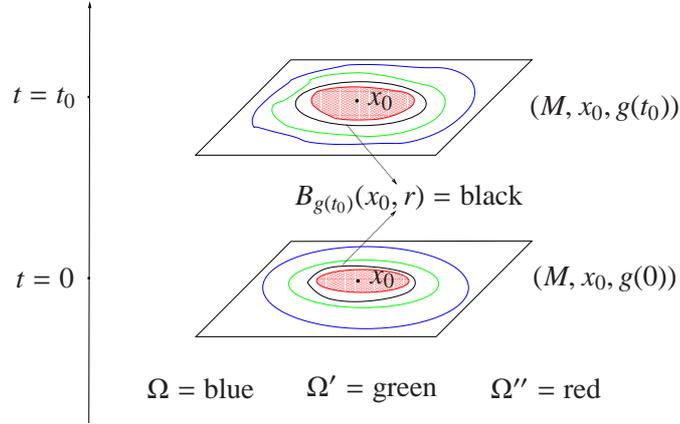}
 \caption{Different domains}
  \label{figure:fourballs}
 \end{center}
 \end{figure}

\begin{lemma}[\textbf{Lower bound of heat solution}]
Suppose $\mathcal{LM}$ is a polarized K\"ahler Ricci flow satisfying (\ref{eqn:SK23_1}), $x_0 \in M$,
notations fixed by (\ref{eqn:SL26_3}) and (\ref{eqn:SL26_4}).

Suppose $\Omega'' \subset B_{g(t)}(x_0, r)$ for some $t>0$ and $r>0$. Then in the geodesic ball $B_{g(t)}(x_0, r)$, we have
\begin{align*}
  w(y,t)>c
\end{align*}
for some constant $c=c(n,B,k,\delta,r_0,r,t)$.
\label{lma:SK23_1}
\end{lemma}

\begin{proof}
By the construction of $w_0$ and maximum principle, it is clear that $0 <w \leq 1$ when $t>0$.
Let $P$ be the heat kernel function, then we can write
\begin{align}
    w(x_0,t)=\int_M P(x_0,t; y,0) w_0(y) dv_y \geq \int_{\Omega''} P(x_0,t;y,0)w_0(y)dv_y.
\label{eqn:SK23_4}
\end{align}
In light of the Sobolev constant bound and scalar curvature bound, one has the on-diagonal bound
\begin{align*}
           \frac{1}{C} t^{-n} \leq P(x,t;x,0) \leq C t^{-n},
\end{align*}
which combined with the gradient estimate of heat equation(c.f. Theorem 3.3 of~\cite{Zhq2}) implies that
\begin{align*}
   P(x,t;y,0) \geq \frac{1}{C} t^{-n}
\end{align*}
where $C=C(B,d_{g(t)}(x,y))$.  Plugging this estimate into (\ref{eqn:SK23_4}) implies that
\begin{align*}
  w(x_0,t) \geq \frac{|\Omega''|}{Ct^n}.
\end{align*}
Note that $C_S$ bound forces $|\Omega''|$ is bounded from below.  Since $0<w\leq 1$,
then (\ref{eqn:SK23_4}) follows from the above inequality and the gradient estimate of heat equation.
\end{proof}

The following two lemmas show that K\"ahler geometry is much more rigid than Riemannian geometry.

\begin{lemma}[\textbf{Fubini-Study approximation}]
Suppose $\mathcal{LM}$ is a polarized K\"ahler Ricci flow satisfying (\ref{eqn:SK23_1}), $x_0 \in M$,
notations fixed by (\ref{eqn:SL26_3}) and (\ref{eqn:SL26_4}).

Suppose $|Rm| \leq r_0^{-2}$ in $\Omega$ at time $t=0$.  Then there exists an integer $k=k(B,r_0, \delta)$ such that
\begin{align}
   \frac12 \omega_0 \leq \tilde{\omega}_0^{(k)} \leq 2\omega_0      \label{eqn:SL29_3}
\end{align}
on $\Omega'$.
\label{lma:SK25_1}
\end{lemma}

\begin{proof}
 This follows essentially from the peak section method of Tian(c.f.~\cite{Tian90JDG},\cite{Lu}).  We give a proof here for the convenience of the readers.

 Fix arbitrary $x \in \Omega'$, $V \in T_x^{(1,0)}M$ with unit norm.  In order to prove (\ref{eqn:SL29_3}), it suffices to show that
 \begin{align}
          \frac{1}{2} \leq \tilde{\omega}(V,JV) \leq 2.   \label{eqn:SL29_1}
 \end{align}
 Around $x$,   we can always choose a  normal coordinate (K-coordinate, c.f.~\cite{Lu}) chart around $x$ such that
\begin{align*}
  V=\frac{\partial}{\partial z_1},  \quad
  g_{i\bar{j}}(x)=\delta_{i\bar{j}}, \quad
  \frac{\partial^{p_1+p_2+\cdots +p_n}}{\partial z_1^{p_1} \partial z_2^{p_2} \cdots \partial z_n^{p_n}} g_{i\bar{j}}(x)=0
\end{align*}
for any nonnegative integers $p_1,p_2, \cdots, p_n$ with $p=p_1+p_2+p_3+\cdots p_n>0$.  Moreover,
there exists a
local holomorphic frame $e_L$ of $L$ around $x$ such that the local representation $a$ of the Hermitian metric $h$ has the
properties
\begin{align*}
  a(x)=1, \quad    \frac{\partial^{p_1+p_2+\cdots +p_n}}{\partial z_1^{p_1}\partial z_2^{p_2} \cdots \partial z_n^{p_n}}a(x)=0
\end{align*}
for  any nonnegative integers $p_1,p_2, \cdots, p_n$ with $p=p_1+p_2+p_3+\cdots p_n>0$.

Suppose $\{S_0^{k}, \cdots S_{N_k}^{k}\}$ is an orthonormal basis of $H^0(M, L^k)$,
where $N_k=\dim_{\C} H^0(M,L^k)-1$.  Around $x$, we can write
   \begin{align*}
      S_0^{k}=f_0^{k} e_L, \cdots, S_{N_k}^{k}=f_{N_k}^{k} e_L.
   \end{align*}
Rotating basis if necessary(c.f.~\cite{Tian90JDG}), we can assume
   \begin{align*}
     &f_i^{k}(x)=0, \quad \forall \; i \geq 1, \\
     &\frac{\partial f_i^k}{\partial z_j}(x)=0, \quad \forall \; i \geq j+1.
   \end{align*}
 Recall that
   \begin{align*}
       \tilde{\omega}^{(k)}&=\omega_0 + \frac{1}{k} \sqrt{-1} \partial \bar{\partial} \log \sum_{j=0}^{N_k} \norm{S_j^{k}}{}^2
        =\frac{1}{k} \sqrt{-1} \partial \bar{\partial} \log \sum_{j=0}^{N_k} |f_j^k|^2.
   \end{align*}
 So we have
    \begin{align}
     \tilde{\omega}^{(k)}(V,JV)
     =\frac{1}{k} \frac{\partial^2 \log \sum_{j=0}^{N_k} |f_j^k|^2}{\partial z_1 \bar{\partial} z_1}
     =\frac{|\frac{\partial f_1^k}{\partial z_1}|^2}{k|f_0^k|^2}.
   \label{eqn:SL29_2}
   \end{align}
   Because of (\ref{eqn:SL29_1}) and (\ref{eqn:SL29_2}), the problem boils down to a precise estimate of
   $\frac{\partial f_1^k}{\partial z_1}$ and $f_0^k$.

 As pointed out by Tian in \cite{Tian90JDG}, the peak section method is local in nature.
 The global information of the underlying manifold is only used in the step of H\"{o}rmander's estimate. However, in our case, we have
 \begin{align*}
     &\sqrt{-1} \partial \bar{\partial} \dot{\varphi}+Ric = \lambda g,  \quad  |\dot{\varphi}| + |\lambda| \leq B.
 \end{align*}
 Due to the uniformly bounded geometry (up to $C^2$-norm of $g$) inside $\Omega'$ and the uniform bound of
 $\sqrt{-1} \partial \bar{\partial} \dot{\varphi}+Ric$ on the whole manifold $M$,  Lemma 1.2 of~\cite{Tian90JDG}
 follows directly and can be written as follows.

  \textit{ For an $n$-tuple of integers $(p_1,p_2, \cdots, p_n) \in \Z_{+}^n$ and an integer $p'>p=p_1+p_2+\cdots+p_n$, there
   exists an $k_0=k_0(n,B,r_0,\delta)$ such that for $k>k_0$, there is a unit norm holomorphic section $S \in H^0(M, L^k)$ satisfying}
   \begin{align*}
       \int_{M \backslash \{|z|^2 \leq \frac{(\log k)^2}{k}\}} \norm{S}{}^2 dv \leq \frac{1}{k^{2p'}}.
   \end{align*}
  Then the same argument as in~\cite{Tian90JDG} implies that(c.f. Lemma 3.2 of~\cite{Tian90JDG})
  \begin{align}
     &\left|f_0^k(x)-\sqrt{\frac{(n+k)!}{k!}} \left\{1+\frac{1}{2(k+n+1)!} (R(x)-n^2-n) \right\} \right|<\frac{C}{k^2}, \label{eqn:MC18_1}\\
     &\left|\frac{\partial f_1^k}{\partial z_1}(x)-\sqrt{\frac{(n+k+1)!}{k!}}\left\{ 1+\frac{1}{2(k+n+1)}(R(x)-n^2-3n-2)\right\} \right| <\frac{C}{k^2},  \label{eqn:MC18_2}
  \end{align}
  for some $C=C(n,B,r_0,\delta)$.  Here $R$ is the complex scalar curvature.  Plugging the above estimate into (\ref{eqn:SL29_2}),
  we obtain (\ref{eqn:SL29_1}), whenever $k$ is larger than a big constant, which depends only on $n,B,r_0,\delta$.
\end{proof}

\begin{lemma}[\textbf{Liouville type theorem}]
 Every complete K\"ahler Ricci flat metric $\tilde{g}$ on $\C^n$ must be a Euclidean metric if there is a constant $C$ such that
\begin{align}
   \frac{1}{C} \delta_{i\bar{j}}(z) \leq \tilde{g}_{i\bar{j}}(z) \leq C \delta_{i\bar{j}}(z), \quad \forall \; z \in \C^n.
\label{eqn:SL26_1}
\end{align}
\label{lma:HE11_1}
\end{lemma}

\begin{proof}
 The original proof of this lemma goes back to the famous paper of E. Calabi~\cite{Ca58} and Pogorelov~\cite{Po78} on real Monge Amp\`{e}re equation. For complex
Monge Amp\`{e}re equation, this is initially due to  Riebesehl-Schulz~\cite{RS84} where higher derivatives are used heavily.
We say a few words here for the convenience of the readers,
using the Schauder estimate of Evans-Krylov.

 Actually,  it is not difficult to see that the problem boils down to the study of a
 global pluri-subharmonic function $u$ in $\C^n$ such that
\begin{align}
\begin{cases}
  &\det \left({\partial^2 u \over {\partial z_i\partial \bar z_j}}\right) = 1, \\
  &C^{-1} (\delta_{i\bar j}) <   \left({\partial^2 u \over {\partial z_i\partial \bar z_j}}\right) < C (\delta_{i\bar j}).
\end{cases}
\label{eqn:SL26_2}
\end{align}
In order to show the metric $\tilde{g}$ is Euclidean, it suffices to show that $u$ is a global quadratic polynomial.
Without loss of generality, we may assume that  $u(0) = D u(0) = 0$.
For every positive integer $k$, we can define a function $u^{(k)}$ in the unit ball by
\begin{align*}
u^{(k)} (z) = {u(kz)\over {k^2}}.
\end{align*}
Clearly, $u^{(k)}$ satisfies (\ref{eqn:SL26_2}). Note that $\norm{u^{(k)}}{C^2}$ is uniformly bounded,
in the unit ball $B(0,1)$.  By standard Evans-Krylov theorem, there exists a
uniform constant $C$ such that
\begin{align*}
[D^2 u^{(k)}]_{C^\alpha(B(0,\frac{1}{2}))} \leq C
\end{align*}
for every $k$.  Putting back the scaling factor,  the above inequality is equivalent to
\begin{align*}
[D^2 u]_{C^\alpha(B(0, \frac{k}{2}))} \leq C k^{-\alpha},\qquad \forall \quad  k=1,2,\cdots.
\end{align*}
Let $k\rightarrow \infty$, we have $[D^2 u]_{C^\alpha(\C^n)} = 0$. Therefore,  $D^2 u$ is a constant matrix,
$u$ is a quadratic polynomial. So we finish the proof.
\end{proof}

\begin{proposition}[\textbf{Ball containing relationship implies regularity improvement}]
Suppose $\mathcal{LM}$ is a polarized K\"ahler Ricci flow satisfying (\ref{eqn:SK23_1}), $x_0 \in M$,
notations fixed by (\ref{eqn:SL26_3}) and (\ref{eqn:SL26_4}).
Suppose $|Rm| \leq r_0^{-2}$ in $\Omega$ at time $t=0$.
Moreover, we assume
 \begin{align}
       \Omega'' \subset B_{g(t)}(x_0,r) \subset \Omega'        \label{eqn:MC09_1}
 \end{align}
 for every $0\leq t \leq t_0$.
 Then the following estimates hold.
 \begin{itemize}
 \item In the geodesic ball $B_{g(t)}(x_0,r)$,  we have
 \begin{align}
     \frac{1}{C} \omega_0 \leq  \omega_t \leq C \omega_0
 \label{eqn:SK23_6}
 \end{align}
 for some constant $C=C(n,B,k,\delta,r_0,r,t)$.
 \item  In the geodesic ball $B_{g(t)}(x_0,r-\xi)$,  we have
 \begin{align}
    |Rm|(x,t) \xi^2 \leq C  \label{eqn:SK23_5}
 \end{align}
 for each small $\xi$ and some constant $C=C(n,B,k,\delta,r_0,r,t_0)$.
 \end{itemize}
\label{prn:HE11_2}
\end{proposition}

\begin{proof}
Note that Perelman's strong version of pseudolocality theorem, i.e., Theorem 10.3 of~\cite{Pe1}, can be modified and applied here.
In fact, the almost Euclidean volume ratio condition in that theorem can be replaced by $\kappa$-noncollapsing condition.
Since one has injectivity radius estimate when curvature and volume ratio bounds are available,
thanks to the work of Cheeger, Gromov and Taylor, in~\cite{ChGrTa}.
By shrinking the ball to some fixed smaller size, one can get back the condition of almost Euclidean volume ratio.
Up to a covering argument, we can apply this strong version pseudolocality theorem to show that
$|Rm|$ is uniformly bounded on $\Omega' \times [0,\eta]$ for some positive $\eta=\eta(n,\kappa,\delta)$.
Then (\ref{eqn:SK23_6}) and (\ref{eqn:SK23_5}) follows trivially. For this reason, we can assume $t_0>\eta$.

We first prove estimate (\ref{eqn:SK23_6}).   Due to Fubini-Study metrics' approximation, Lemma~\ref{lma:SK25_1}, it is clear that one can regard
$\omega_0$ and $\tilde{\omega}_0^{(k)}$ as the same metric on $\Omega'$.  Therefore, it follows from the combination of Lemma~\ref{lma:J02_1}
and Lemma~\ref{lma:SK23_1} that $F^{(k)}$ is bounded from above, which implies $\Lambda_{\omega_t} \omega_0 \leq C$.
Recall that the volume element $\omega_0^n$ and $\omega_t^n$ are uniformly equivalent, due to the uniform bound of $|R|+|\lambda|$ and the evolution equation
\begin{align*}
   \frac{\partial}{\partial t} \log \omega_t^n= n\lambda-R.
\end{align*}
Consequently, (\ref{eqn:SK23_6}) follows. We remind the readers that condition (\ref{eqn:MC09_1}) is used in the above discussion. 

Then we proceed to prove inequality (\ref{eqn:SK23_5}). Fix $L$ very large.
If (\ref{eqn:SK23_5}) does not hold uniformly, then we can find some space-time point $(y_0, s_0)$ such that  $y_0 \in B_{g(s_0)}(x_0, r-\xi)$ and $Q_0 \triangleq |Rm|(y_0, s_0)>100L^2 \xi^{-2}$ is very large.
Set $\rho_0 \triangleq d_{g(s_0)}(y_0, x_0)$.  On one hand, $\rho_0<r-\xi$ by the choice of $(y_0, s_0)$. On the other hand,  $s_0>\eta$ for some uniform $\eta$ due to the application of Perelman's pseudo-locality,
as discussed above. Search whether there is a point $(x,t)$ satisfying
\begin{align*}
 |Rm|(x,t)>4Q_0, \quad x \in B_{g(t)} \left(x_0, \rho_0 + LQ_0^{-\frac{1}{2}} \right), \; t \in \left[t_0-Q_0^{-1}, t_0 \right]. 
\end{align*}
If there exists such a point, we denote it by $(y_1, s_1)$ and continue the above searching.
We can find $(y_k, s_k)$ by induction.  Actually, if $(y_{k-1}, s_{k-1})$ is defined, then we denote $|Rm|(y_{k-1}, s_{k-1})$ by $Q_{k-1}$, denote $d_{g(s_{k-1})}(x_0, y_{k-1})$ by $\rho_{k-1}$ and search point $(x,t)$ satisfying
\begin{align*}
 |Rm|(x,t)>4Q_{k-1}, \quad x \in B_{g(t)} \left(x_0, \rho_{k-1}+ LQ_{k-1}^{-\frac{1}{2}} \right), \; t \in \left[t_{k-1}-Q_{k-1}^{-1}, t_{k-1} \right]. 
\end{align*}
If there is no such point, we stop the process. 
Otherwise, we denote such a point by $(y_k, s_k)$ and continue the process.  Clearly, we have
\begin{align*}
& Q_k=4^{k}Q_0 >100L^2 \xi^{-2}, \\
&\rho_k \leq \rho_0 + L \left(Q_0^{-\frac{1}{2}} + \cdots Q_{k-1}^{-\frac{1}{2}} \right)<\rho_0 + 4LQ_0^{-\frac{1}{2}}< r-0.5\xi, \\
&|s_0-s_k|=s_0-s_k \leq Q_0^{-1} + Q_{1}^{-1} + \cdots Q_{k-1}^{-1}< 2Q_0^{-1}<\frac{\xi^2}{50 L^2}<<\eta.
\end{align*}
Since the process happens in a compact space-time domain with bounded geometry, it must stop after finite steps.   Let $k$ be the last $(y_k, s_k)$. 
We denote it by $(y,s)$ and set $Q=|Rm|(y,s)$ and $\rho=d_{g(s)}(y, x_0)$.  Then we have
\begin{align}
\begin{cases}
&Q>100 L^2 \xi^{-2}, \\
&\rho< r-0.5 \xi,\\
&s> 0.5 \eta, \\
&|Rm|(x,t)<4Q,   \quad \forall \; x \in B_{g(t)}(x_0, \rho+LQ^{-\frac{1}{2}}),   \quad t \in \left[s-Q^{-1}, s \right]. 
\end{cases}
\label{eqn:MC09_2}
\end{align}
By its choice, we have $d_{g(s)}(x_0,y)=\rho$.  We observe that $y$ will stay in $B_{g(t)}(x_0, \rho+2Q^{-\frac{1}{2}})$ whenever $t \in [s-\frac{1}{5nQ}, s]$. 
This is an application of  Lemma 8.3 of Perelman~\cite{Pe1}, or section 17 of Hamilton~\cite{Ha95f}.   
Actually, let $\theta_0$ be the largest positive number such that $y$ fails to locate in $B_{g(t-\theta_0 Q^{-1})}(x_0, \rho+ 2Q^{-\frac{1}{2}})$. 
Then for each $t \in [s-\theta_0 Q^{-1}, s]$, triangle inequality implies that $B_{g(t)}(y, Q^{-\frac{1}{2}}) \subset B_{g(t)}(x_0, \rho+3Q^{-\frac{1}{2}})$. Conseqeuently, we have
\begin{align*}
 &|Rm|(x,t) \leq 4Q,  \quad \forall \; x \in B_{g(t)}(y, Q^{-\frac{1}{2}}), \\
 &|Rm|(x_0,t) \leq 4Q, \quad \forall \; x \in B_{g(t)}(x_0, Q^{-\frac{1}{2}}). 
\end{align*}
It follows from Lemma 8.3 (b) of Perelman~\cite{Pe1} that
\begin{align*}
  \frac{d}{dt} d(x_0,y) \geq -10 n Q^{\frac{1}{2}} \quad \Rightarrow \quad d_{g(s)}(x_0,y)-d_{g(s-\theta_0 Q^{-1})}(x_0,y) \geq -10 n Q^{\frac{1}{2}} \cdot \theta_0 Q^{-1}. 
\end{align*}
According to the choice of $\theta_0$, the left hand side of the second inquality is $-2Q^{-\frac{1}{2}}$.  It follows that $\theta_0 \geq \frac{1}{5n}$. 
Now we know that $y$ stays in $B_{g(t)}(x_0, \rho+2Q^{-\frac{1}{2}})$ for each $t \in [s-\frac{1}{5nQ}, s]$. 
In view of (\ref{eqn:MC09_2}) and the fact $L>>1$, the triangle inequality implies that
\begin{align*}
 |Rm|(x,t) < 4Q, \quad \forall x \in B_{g(t)}(y, 0.5LQ^{-\frac{1}{2}}), \quad t \in \left[s-\frac{1}{5nQ}, s \right]. 
\end{align*}
Let $\tilde{g}(t)=Qg(Q^{-1} t + s)$. We have
\begin{align*}
\begin{cases}
&|\widetilde{Rm}|(y,0)=1,\\
&|\widetilde{Rm}|(x,t) <4, \quad \forall x \in B_{\tilde{g}(t)}(y, 0.5 L),  \quad t \in \left[-\frac{1}{5n}, 0 \right].  
\end{cases}
\end{align*}
Note that $\left[-\frac{1}{5n}, 0 \right]$ is a fixed time period. The application of Perelman's pseudo-locality guarantees the existence of such a time period(c.f. (\ref{eqn:MC09_2})). 
Now let $L \to \infty$, we can use the compactness theorem of Hamilton~\cite{Ha95} to obtain a limit Ricci flow solution, which is non-flat, K\"ahler Ricci-flat and non-collapsed on all scales. 
We remark that the discussion above is nothing but repeating the argument of Claim 1 and Claim 2 in the proof of Perelman's pseudo-locality theorem, i.e, Theorem 10.1 in~\cite{Pe1}. 
Similar argument was also used in the distance estimate of the work of Tian and the second named author~\cite{TW2}. 

Note that $B_{g(s)}(y, 0.5LQ^{-\frac{1}{2}}) \subset B_{g(s)}(x_0, r-0.5\xi) \subset B_{g(s)}(x_0,r) \subset \Omega'=B_{g(0)}(x_0, 1-\delta)$.
Therefore, by the same scale blowup at $(y,0)$, we obtain nothing but $\C^n$. Recall we have (\ref{eqn:SK23_6}), so we obtain a nontrivial K\"ahler Ricci flat metric
$\tilde{g}_{i\bar{j}}$ on $\C^n$ such that (\ref{eqn:SL26_1}) holds for some $C$.
This contradicts Lemma~\ref{lma:HE11_1}.
\end{proof}

The rough estimate (\ref{eqn:SK23_6}) and (\ref{eqn:SK23_5}) can be improved when $|R|+|\lambda|$ is very small.
When curvature tensor is bounded in the space-time, one can estimate the Ricci curvature in terms of scalar curvature.
Let  $|R|+|\lambda|$ tend to zero, we see that the Ricci curvature tends to zero at the space-time where $|Rm|$ is bounded.   By adjusting $\xi$ if necessary, we obtain that in the limit,  $B_{g(t)}(x_0,(1-\xi)r)$ is isometric to
$B_{g(0)}(x_0, (1-\xi)r)$ for every $0<t<t_0$.   By  adjusting $\xi$ and applying Perelman's pseudolocality theorem, we see the convergence at time $t=t_0$ is also smooth since curvature derivatives are all bounded  in
the ball $B_{g(t_0)}(x_0,(1-\xi)r)$ at time $t_0$.

\begin{proposition}[\textbf{Volume element derivative small implies ball containing relationship}]
  For every $r_0, T$ and small $\xi$, there exists an $\epsilon$ with the following property.

  Suppose $\mathcal{LM}$ is a polarized K\"ahler Ricci flow satisfying (\ref{eqn:SK23_1}), $x_0 \in M$, notations fixed by (\ref{eqn:SL26_3}) and (\ref{eqn:SL26_4}).  Suppose $|Rm| \leq r_0^{-2}$ in $\Omega$ at time $t=0$.
If $\displaystyle \sup_{\mathcal{M}} (|R| + |\lambda|) <\epsilon$, then for every $t \in [0,T]$ we have
  \begin{align*}
    &\Omega'' \subset B_{g(t)}\left(x_0, \left(1-\frac{3}{2}\delta \right) r_0 \right) \subset \Omega', \\
    &\left(1- \xi \right)\omega(0) \leq  \omega(t)  \leq \left(1 + \xi \right) \omega(0),  \quad \textrm{in} \;  \Omega''.
  \end{align*}
  \label{prn:HE11_3}
\end{proposition}

\begin{proof}
   If the statement was wrong, we can find a tuple $(n,B,\delta,r_0,T)$ and $\epsilon_i \to 0$ such that the property does not hold for every $\epsilon_i \to 0$.
   Without loss of generality,  we can assume $r_0=1$.

   For each $\epsilon_i$, let $t_i \in [0,T]$ be the critical time of a flow $g_i(t)$ such that the properties hold on $[0,t_i]$.
   In other words, for every $t \in [0, t_i]$, we have
    \begin{align}
    &\Omega_i'' \subset B_{g_i(t)}\left(x_i,  1-\frac{3}{2}\delta  \right) \subset \Omega_i',  \label{eqn:GE08_2}\\
    &\left(1- \xi \right)\omega_i(0) \leq  \omega_i(t)  \leq  \left(1 + \xi \right) \omega_i(0),  \quad \textrm{in} \;  \Omega_i''.  \label{eqn:GE08_3}
    \end{align}
    However, for each time $t>t_i$, at least one of the above relations fails to hold.  Related to (\ref{eqn:SL26_3}), here we set
    \begin{align*}
    \Omega_i \triangleq B_{g_i(0)}(x_i, 1),  \quad \Omega_i' \triangleq B_{g_i(0)}(x_i, 1-\delta),
    \quad  \Omega_i'' \triangleq B_{g_i(0)}(x_i, 1-2\delta).
    \end{align*}
   We shall show that $t_i$ cannot locate in $[0, T]$ for large $i$ and therefore obtain a contradiction.   
   
   Note that $|Rm|_{g_i(0)} \leq 1$ at time $t=0$ in the ball $B_{g_i(0)}(x_i , 1)$.
   By the strong version of Perelman's pseudolocality theorem, i.e., Theorem 10.3 of~\cite{Pe1}, one can find a uniform small constant $\eta$ such that
   \begin{align}
       |Rm|_{g_i}(x,t) \leq \frac{\xi}{100 n^2} \eta^{-2}, \quad   \forall \; x \in B_{g_i(0)}(x_i, 1-\eta), \;  t \in [0, \eta^2]. 
   \label{eqn:GE08_1}    
   \end{align}
   The existence of $\eta$ can be obtained by a contradiction blowup argument.
   Since metrics evolve by $-Ric+\lambda g$, it follows from (\ref{eqn:GE08_1}) and the choice of $t_i$ that $\eta^2 \leq t_i \leq T$.      
   Recall that  we have the relationship (\ref{eqn:GE08_2}) by the choice of $t_i$. 
   Therefore,  Proposition~\ref{prn:HE11_2} can be applied to obtain a unfiorm $C$, independent of $i$, such that 
   \begin{align}
       \frac{1}{C} g_i(0) \leq  g_i(t_i) \leq Cg_i(0)
   \label{eqn:GE08_4}    
   \end{align}
   in the ball $B_{g_i(t_i)}(x_i, 1-\frac{3\delta}{2})$.  Furthermore, the inequality (\ref{eqn:SK23_5}) in Proposition~\ref{prn:HE11_2} yields that
   \begin{align*}
     |Rm|_{g_i}(x,t) \leq \frac{C}{\psi^2}, \quad  x \in  B_{g(t)} \left(x_i, 1-\frac{3}{2} \delta -\psi \right), \; 0 \leq t \leq t_i, 
   \end{align*}
   where $\psi$ is a small constant $\psi <<\delta$, to be determined.   Note that we are in a setting where each geodesic ball's volume ratio is bounded from below, due to the bounds in (\ref{eqn:SK20_1}).
   Consequently, injectivity radius has a lower bound(c.f.~\cite{ChGrTa}), by shrinking the ball if necessary. 
   Therefore, we can apply Theorem 3.2 of~\cite{Wa2} to obtain 
   \begin{align}
     \sup_{\eta^2 \leq t \leq t_i,  d_{g_i(t)}(x, x_i) \leq 1-\frac{3}{2} \delta -2\psi}  |Ric|_{g_i}(x, t)  \to 0, \quad \textrm{as} \; i \to \infty,
   \label{eqn:GE08_5}  
   \end{align}
   where $\eta$ is the constant in (\ref{eqn:GE08_1}).   Alternatively, one can apply Lemma 2.1 of~\cite{CW5} to obtain the above estimate, with the fact that geodesic balls at different times can be compared due to
   the Riemannian curvature bound and the evolution equation of the Ricci flow:  the metrics evolve by $-Ric+\lambda g$.
   Since $t_i$ is uniformly bounded by $T$,  the above equation implies (up to a maximum principle type argument of the first violating time if necessary)
   that
   \begin{align}
       B_{g_i(\eta^2)} \left( x_i, 1-\frac{3}{2}\delta -5\psi \right) \subset B_{g_i(t)} \left( x_i, 1-\frac{3}{2}\delta -4\psi \right) \subset B_{g_i(\eta^2)} \left( x_i, 1-\frac{3}{2}\delta -3\psi \right),  \quad \forall \; t \in [\eta^2, t_i).
   \label{eqn:GE08_6}    
   \end{align}
   Combining the above relationship with (\ref{eqn:GE08_5}), we obtain that
   \begin{align}
      \sup_{\eta^2 \leq t \leq t_i,  d_{g_i(\eta^2)}(x, x_i) \leq 1-\frac{3}{2} \delta-3\psi }  |Ric|_{g_i}(x, t)  \to 0, \quad \textrm{as} \; i \to \infty. 
   \label{eqn:GE08_7}   
   \end{align}
   By (\ref{eqn:GE08_1}) and $|R|+|\lambda| \to 0$, we see the metric at $g_i(0)$ and $g_i(\eta^2)$ are almost isometric to each other on the ball $B_{g_i(0)}(x_i, 1-\frac{3}{2}\delta)$. 
   Consequently, we have
   \begin{align*}
      \Omega_i''=B_{g_i(0)}(x_i, 1-2\delta) \subset B_{g_i(\eta^2)} \left( x_i, 1-\frac{7}{4}\delta \right) \subset B_{g_i(t_i)} \left( x_i, 1-\frac{7}{4}\delta +\psi \right) \Subset  B_{g_i(t_i)} \left( x_i, 1-\frac{3}{2}\delta \right),
   \end{align*}
   where $\Subset$ means ``compactly contained".   We claim that we also have
    \begin{align*}
       B_{g_i(t_i)} \left( x_i, 1-\frac{3}{2}\delta \right) \Subset \Omega_i'. 
   \end{align*}
   For otherwise, by the choice of $t_i$, the boundary of $B_{g_i(t_i)} \left( x_i, 1-\frac{3}{2}\delta \right)$ touches the boundary of $\Omega_i'$ at time $t_i$.
   Therefore, we can find a point $y_i$ satisfying
   \begin{align*}
     d_{g_i(t_i)}(x_i, y_i)=1-\frac{3}{2}\delta, \quad d_{g_i(0)}(x_i, y_i)=1-\delta. 
   \end{align*}
   Let $\gamma_i$ be a shortest unit-speed geodesic connecting $x_i$ and $y_i$, with respect to the metric $g_i(t_i)$. 
   Let $\gamma_i(0)=x_i$ and $\gamma_i(1-\frac{3}{2}\delta)=y_i$.  By previous estimates, we see that 
   \begin{align*}
     \gamma_i \left(1-\frac{3}{2}\delta -100\psi \right) \subset B_{g_i(0)} \left(x_i, 1-\frac{3}{2}\delta -50\psi \right). 
   \end{align*}
   Let $\alpha_i$ be the part of $\gamma_i$, connecting $x_i=\gamma_i(0)$ and $\gamma_i(1-\frac{3}{2}\delta -100\psi )$.
   Let  $\beta_i$ be the remainded part of $\gamma_i$, i.e., the part connecting $\gamma_i(1-\frac{3}{2}\delta -100\psi )$ and $y_i=\gamma_i(1-\frac{3}{2}\delta)$.  
   Using $|\cdot|$ to denote the length of curves.   It is clear that $|\beta_i|_{g_i(t_i)} =100\psi$. 
   Note that $\alpha_i$ locates in $B_{g_i(t_i)}(x_i, 1-\frac{3}{2}\delta -100\psi)$. 
   It follows from (\ref{eqn:GE08_6}), (\ref{eqn:GE08_7}) and  (\ref{eqn:GE08_1}) that
   \begin{align*}
       \sup_{\alpha_i \times [\eta^2, t_i]} |Ric|(x, t) \to 0, \quad \textrm{as} \; i \to \infty;  \quad   \sup_{\alpha_i \times [0, \eta^2]} |Rm|(x, t) \leq \frac{\xi}{100n^2} \eta^{-2}. 
   \end{align*}
   Together with $|R|+|\lambda| \to 0$ as $i \to \infty$, we can compare the length of $\alpha_i$ at time $t=t_i$ and $t=0$. 
   \begin{align*}
        |\alpha_i|_{g_i(t_i)}=1-\frac{3}{2}\delta -100\psi,   \quad |\alpha_i|_{g_i(0)} \leq 1-\frac{3}{2}\delta. 
   \end{align*}
   However,  since $d_{g_i(0)}(x_i, y_i)=1-\delta$, we have
   \begin{align*}
      1-\delta \leq |\gamma_i|_{g_i(0)}=|\alpha_i|_{g_i(0)} + |\beta_i|_{g_i(0)} \leq |\beta_i|_{g_i(0)}  + 1-\frac{3}{2}\delta. 
   \end{align*}
   It follows that $ |\beta_i|_{g_i(0)} \geq \frac{1}{2}\delta$.  Recall that $|\beta_i|_{g_i(t_i)} =100\psi$. 
   Therefore, by mean value theorem, we must have
   \begin{align*}
       \frac{\sqrt{\langle V, V\rangle_{g_i(0)}}}{ \sqrt{\langle V, V \rangle_{g_i(t_i)}}}  \geq \frac{\frac{1}{2}\delta}{100\psi}=\frac{\delta}{200\psi}. 
   \end{align*}
   at some point $z_i \in \beta_i$,  where $V$ is the unit tangent vector (with respect to $g_i(t_i)$)  of $\beta_i$ at $z_i$.  
   Since $z_i \in B_{g_i(t_i)}(x_i, 1-\frac{3}{2}\delta)$,  one can apply (\ref{eqn:GE08_4}) to
   bound the left hand side of the above inequality by $\sqrt{C}$, where $C$ is the constant in (\ref{eqn:GE08_4}).
   It follows that $ C \geq \frac{\delta^2}{40000\psi^2}$, which is impossible if we choose $\psi$ small enough.   Therefore, for $i$ large, we must have
    \begin{align*}
      \Omega_i '' \Subset  B_{g_i(t_i)} \left( x_i, 1-\frac{3}{2}\delta \right) \Subset \Omega_i'. 
   \end{align*}
   Then we can apply (\ref{eqn:GE08_1}), (\ref{eqn:GE08_5}) and the fact that $|R|+|\lambda| \to 0$ to obtain that
   \begin{align*}
     \left(1- \frac{\xi}{100} \right)\omega_i(0)  \leq \omega_i(t)  \leq \left(1 + \frac{\xi}{100} \right) \omega_i(0),  \quad \textrm{in} \;  \Omega_i''= B_{g_i(0)}(x_i, 1-2\delta),
   \end{align*}
   whenever $i$ large enough.   This means that for large $i$, we have both (\ref{eqn:GE08_2}) and (\ref{eqn:GE08_3}) hold for a short while beyond the time $t_i$.
   This contradicts to the choice of time $t_i$. 
\end{proof}

Combine Proposition~\ref{prn:HE11_2} and Proposition~\ref{prn:HE11_3}, by further applying the argument in Proposition~\ref{prn:HE11_2},
the following theorem is clear now.

 \begin{theorem}[\textbf{Rough long-time pseudolocality theorem for polarized K\"ahler Ricci flow}]
 For every group of numbers $\delta, \xi, r_0,T$, there exists an $\epsilon=\epsilon(n,B,\delta, \xi, r_0,T)$
 with the following properties.

  Suppose $\mathcal{LM}$ is a polarized K\"ahler Ricci flow satisfying (\ref{eqn:SK23_1}), $x_0 \in M$.
  Suppose $|Rm| \leq r_0^{-2}$ in $\Omega$ at time $t=0$, where $\Omega=B_{g(0)}(x_0,r_0)$.
  If $\displaystyle \sup_{\mathcal{M}} (|R| + |\lambda|) <\epsilon$, then for every $t \in [0,T]$ we have
 \begin{align}
   & B_{g(t)}(x_0, (1-2\delta)r_0) \subset \Omega, \\
   & |Rm|(\cdot, t) \leq 2r_0^{-2}, \quad \textrm{in} \; B_{g(t)}(x_0, (1-2\delta)r_0), \\
   & \left(1-\xi \right) g(0) \leq g(t) \leq \left(1 +\xi \right) g(0),  \quad \textrm{in}
   \; B_{g(t)}(x_0, (1-2\delta)r_0).
 \end{align}
 \label{thm:K8_1}
\end{theorem}

\subsection{Motivation and definition}

In previous subsection, we see that the assumption (\ref{eqn:SK23_1}) helps a lot to relate different time slices of
the K\"ahler Ricci flow solution.  However, why is this assumption reasonable? This question will be answered in this
subsection.

\begin{proposition}[\textbf{Weak continuity of Bergman function}]
 There is a big integer constant $k_0=k_0(n,A)$ and small constant $\epsilon=\epsilon(n,A)$
  with the following property.

 Suppose $(M,g,J,L,h)$ is a polarized K\"ahler manifold, taken out from a polarized K\"ahler Ricci flow
 in $\mathscr{K}(n,A)$ as a central time slice.  In particular, we have
  \begin{align}
     Osc_{M} \dot{\varphi} + C_S(M) + |\lambda| \leq B,
 \label{eqn:SK27_1}
 \end{align}
 where $B=B(n,A)$.
 If $\mathbf{cr}(M) \geq 1$,  then
 \begin{align}
    \sup_{1 \leq k \leq k_0} \mathbf{b}^{(k)}(x) > -k_0
 \label{eqn:SK27_2}
 \end{align}
 whenever $d_{PGH}((M,x,g), (\tilde{M}, \tilde{x}, \tilde{g}))<\epsilon$ for some space
 $(\tilde{M}, \tilde{x}, \tilde{g}) \in \widetilde{\mathscr{KS}}(n,\kappa)$.
\label{prn:SK27_1}
\end{proposition}

\begin{proof}

 Note that the model moduli  space $\widetilde{\mathscr{KS}}(n,\kappa)$ has compactness
 under the pointed Gromov-Hausdorff topology.  This compactness will be essentially used in
 the following argument.

 Suppose the statement was wrong, then there is a sequence of polarized K\"ahler manifolds,
 whose underlying K\"ahler manifolds converge to
 $(\bar{M},\bar{x},\bar{g}) \in \widetilde{\mathscr{KS}}(n,\kappa)$,
 satisfying the estimate (\ref{eqn:SK27_1}) and violating (\ref{eqn:SK27_2}) for $k_i \to \infty$.
 To be explicit, we have
  \begin{align}
    (M_i,x_i,g_i) \longright{G.H.} (\bar{M},\bar{x},\bar{g});  \qquad
    \sup_{1 \leq j \leq k_i} \mathbf{b}^{(j)}(x_i) \to -\infty, \quad k_i \to \infty. \label{eqn:K8_4}
  \end{align}  
 Then we shall use the argument of  the proof of Theorem 3.2 of~\cite{DS} by Donaldson-Sun to find positive integer $q=q(\bar{x})$, and real numbers $r=r(\bar{x})$, $C=C(\bar{x})$
  such that
  \begin{align}
    \inf_{y \in B(x_i, r)} \mathbf{b}^{(q)}(y) \geq -C.
  \label{eqn:GE09_1}  
  \end{align}
  Note that the proof of Theorem 3.2~\cite{DS} is based on a blowup argument.
  The essential ingredients there are the convergence theory,  the H\"{o}mander's estimate, and the fact that each tangent space in the limit space is a good metric cone. 
  By ``good" we mean the singular set of the metric cone has Hausdorff codimension strictly greater than $2$. 
  It is important to observe that  whether the limit space $\bar{M}$ is compact or not does not affect the argument.
  Basically, this is because of the local property of the H\"{o}mander's estimate. 
  Actually, no matter whether $\bar{M}$ is compact or not,  every tangent space of a point on $\bar{M}$ must be non-compact. 
  The contradiction is obtained from the convergence to the good tangent metric cone.   
  With the argument of Theorem 3.2 of~\cite{DS} in mind,   we  now check the conditions available to us in the current case.  
  Firstly,  the canonical radius assumption makes sure that the topology of the convergence can be improved to the
  $\hat{C}^4$-Cheeger-Gromov topology.
  Secondly,  by the uniform bound of Sobolev constant and $\norm{\dot{\varphi}}{C^0}$,  the general H\"ormander's estimate(c.f. section 3 of~\cite{CW4} and section 5 of~\cite{Wa1} for this particular case) can be applied. 
  Thirdly, we know each tangent space at $\bar{x}$ is a  good metric cone, by Theorem~\ref{thm:HE08_1}  since $(\bar{M},\bar{x},\bar{g}) \in \widetilde{\mathscr{KS}}(n,\kappa)$. 
  Therefore, we can use a contradiction blowup argument, like that in Theorem 3.2 of~\cite{DS}, to obtain (\ref{eqn:GE09_1}). 
  Consequently, we have
  \begin{align*}
    \mathbf{b}^{(q)}(x_i) \geq -C, \quad \Rightarrow \quad \sup_{j \leq k_i} \mathbf{b}^{(j)} (x_i) \geq -C,
  \end{align*}
  which contradicts (\ref{eqn:K8_4}), the assumption.
\end{proof}

Proposition~\ref{prn:SK27_1} means that the Bergman function has a weak continuity under the Cheeger-Gromov
convergence if the limit space is the model space.   Inspired by this property, we can define the polarized canonical radius as follows.

\begin{definition}
 Suppose $\left( M,g,J,L,h\right)$ is a polarized K\"ahler manifold satisfying (\ref{eqn:SK27_1}), $x \in M$. We say the polarized canonical radius of $x$ is not less than $1$ if
 \begin{itemize}
   \item $\mathbf{cr}(x) \geq 1$.
   \item $\displaystyle \sup_{1 \leq j \leq 2k_0}\mathbf{b}^{(j)}(x) \geq -2k_0$.
 \end{itemize}
 For every $r=\frac{1}{j}, j \in \Z^{+}$, we say the polarized canonical radius of $x$ is not less than $r$ if the rescaled polarized manifold $\left( M, j^2 g, J, L^j, h^j\right)$
 has polarized canonical radius at least $1$ at the point $x$.
 Fix $x$, let $\mathbf{pcr}(x)$ be the supreme of all the $r$ with the above property and call it
 as the polarized canonical radius of $x$.
 \label{dfn:SK27_1}
\end{definition}

We can define the polarized canonical radius of a manifold as the infimum of the polarized canonical radii of all points in that manifold.
Similarly, we can define the polarized canonical radius of time slices of a flow.
Note that from the above definition, $\mathbf{pcr}$ is always the reciprocal of a positive integer.
It could not be zero because of (\ref{eqn:MC18_1}) in the proof of Lemma~\ref{lma:SK25_1} and the fact that every compact smooth manifold has bounded geometry and positive $\mathbf{cr}$. 
Under this terminology, the continuity of Bergman function implies the following corollary.

\begin{corollary}[\textbf{Weak equivalence of $\mathbf{cr}$ and $\mathbf{pcr}$}]
  There is a small constant $\epsilon=\epsilon(n,B,\kappa)$ with the following property.

 Suppose $(M,g,J,L,h)$ is a polarized K\"ahler manifold satisfying (\ref{eqn:SK27_1}) and
 $\mathbf{cr}(M) \geq 1$. Then
 \begin{align}
   \mathbf{pcr}(x) \geq 1
    \label{eqn:SK27_4}
 \end{align}
 whenever $d_{PGH}((M,x,g), (\tilde{M}, \tilde{x}, \tilde{g}))<\epsilon$ for some space
 $(\tilde{M}, \tilde{x}, \tilde{g}) \in \widetilde{\mathscr{KS}}(n,\kappa)$.
\label{cly:SK27_1}
\end{corollary}

\subsection{K\"ahler Ricci flow with lower bound of polarized canonical radius}

Suppose the polarized canonical radius is uniformly bounded from below, then the convergence theory is much better
than that in section 3.   This is basically because of the rough long-time pseudolocality theorem, Theorem~\ref{thm:K8_1}.

\begin{proposition}[\textbf{Improving regularity in forward time direction}]
For every $r_0>0$, $r \in (0,r_0)$ and $T_0>0$, there is an $\epsilon=\epsilon(n,A,r_0,r,T_0)$ with the following properties.

If $\mathcal{LM}$ is a polarized K\"ahler Ricci flow satisfying (\ref{eqn:SK20_1}) and
\begin{align}
   \mathbf{pcr}(\mathcal{M}^{t}) \geq r_0, \quad \forall \;  t \in [0,T_0],
\label{eqn:SL04_1}
\end{align}
then  \begin{align}
    \mathcal{F}_{r}(M, 0) \subset \bigcap_{0\leq t \leq T_0}  \mathcal{F}_{\frac{r}{K}}(M, t)   \label{eqn:SC19_3}
 \end{align}
 whenever $\displaystyle \sup_{\mathcal{M}} (|R|+|\lambda|)<\epsilon$.
 Here $K$ is the constant in Proposition~\ref{prn:SC24_1}.
 \label{prn:SC09_3}
\end{proposition}

\begin{proof}
  It follows directly from Theorem~\ref{thm:K8_1}, the long time pseudolocality theorem for polarized K\"ahler Ricci flow with partial-$C^0$-estimate.
\end{proof}

\begin{proposition}[\textbf{Improving regularity in backward time direction}]
For every $r_0>0$, $r \in (0,r_0)$ and $T_0>0$, there is an $\epsilon=\epsilon(n,A,r_0,r,T_0)$ with the following properties.

If $\mathcal{LM}$ is a polarized K\"ahler Ricci flow satisfying (\ref{eqn:SK20_1}) and (\ref{eqn:SL04_1}), then  \begin{align}
    \bigcup_{0 \leq t \leq T_0}  \mathcal{F}_{r}(M, t)   \subset \mathcal{F}_{\frac{r}{K}}(M,0) \label{eqn:SK27_6}
 \end{align}
 whenever $\displaystyle \sup_{\mathcal{M}} (|R|+|\lambda|)<\epsilon$.
 \label{prn:SK27_3}
\end{proposition}

\begin{proof}
At time $0$, $\mathcal{F}_r(M,0) \subset \mathcal{F}_{\frac{r}{K}}(M,0)$ trivially. Suppose $t_0>0$ is the first time such that
(\ref{eqn:SK27_6}) start to fail.  It suffices to show that $t_0 >T_0$ whenever $\epsilon$ is small enough. Otherwise, at time $t_0 \in (0,T_0]$,
we can find a point $x_0 \in  (\partial \mathcal{F}_{\frac{r}{K}}(M,0)) \cap (\partial \mathcal{F}_{r}(M,t_0))$. In other words, we have
\begin{align*}
  \mathbf{cvr}(x_0,0)=\frac{r}{K}, \quad \mathbf{cvr}(x_0,t_0)=r.
\end{align*}
In particular, we have
\begin{align}
   \left|B_{g(0)}\left(x_0,\frac{r}{K} \right) \right|_0 = \left(1-\delta_0 \right) \omega_{2n} \left(\frac{r}{K} \right)^{2n}.  \label{eqn:SL16_1}
\end{align}

 Let $\xi$ be a small number which will be fixed later.   Let $\Omega_{\xi}(x_0,t_0)$ be the subset of unit sphere of tangent space of $T_{x_0}(M, g(t_0))$ such that every geodesic (under metric $g(t_0)$) emanating from $x_0$ along the direction in $\Omega_{\xi}(x_0,t_0)$ does not hit points in $\mathcal{D}_{\xi}(M,0)$ before
 distance $\frac{r}{K}$.    By canonical radius assumption, $|Rm|_{g(t_0)}$ is uniformly bounded in $B_{g(t_0)}(x_0,\frac{r}{K})$(See Figure~\ref{figure:backpseudo} for intuition).
 By long-time pseudolocality theorem (c.f. Proposition~\ref{prn:SK27_3}), $B_{g(t_0)}(x_0,\frac{r}{K^3})$ has empty intersection with $\mathcal{D}_{\xi}(M,0)$
 when $\xi<<\frac{r}{K^3}$. Note that every geodesic (emanating from $x_0$) entering $\mathcal{D}_{\xi}(M,0)$ must hit $\partial \mathcal{D}_{\xi}(M,0)$ first,
 where $\mathbf{cvr}(\cdot, 0)=\xi$.  So every point in $\partial \mathcal{D}_{\xi}(M,0)$ will be uniformly regular at time $t_0$,  in light of the long-time pseudolocality.
 At time $t_0$, observing from $x_0$, the set which stays behind $\partial \mathcal{D}_{\xi}(M,0)$ must have small measure.
 Since $B_{g(t_0)}(x_0,\frac{r}{K})$ has uniformly bounded curvature,
 it is clear that $\Omega_{\xi}(x_0,t_0)$ is an almost full measure subset of $S^{2n-1}$. Actually, we have
 \begin{align*}
    |\Omega_{\xi}(x_0,t_0)| \geq 2n\omega_{2n} \cdot \left(1-C \xi^{2p_0} \right)
 \end{align*}
 whenever $\epsilon$ is sufficiently small.  On the other hand, we see that every geodesic (under metric $g(t_0)$) emanating from $\Omega_{\xi}(x_0,t_0)$ is almost geodesic at time $t=0$ (under metric $g(0)$), when $\epsilon$ small enough.    Therefore,
$|B_{g(0)}(x_0,\frac{r}{K})|_{0}$ is almost not less than $|B_{g(t_0)}(x_0,\frac{r}{K})|_{t_0}$.  Note that the volume ratio of
$B_{g(t_0)}(x_0,\frac{r}{K})$ is at least $(1-\frac{\delta_0}{100})\omega_{2n}$.
Suppose we choose $\xi$ small (according to $\delta_0$) and $\epsilon$ very small (based on $\xi,\delta_0,A,T_0$), we obtain
\begin{align*}
   \left|B_{g(0)}\left(x_0,\frac{r}{K} \right) \right|_0 \geq \left(1-\frac{\delta_0}{2} \right) \omega_{2n} \left(\frac{r}{K} \right)^{2n},
\end{align*}
which contradicts (\ref{eqn:SL16_1}).
\end{proof}

\begin{figure}
 \begin{center}
 \psfrag{0}[c][c]{$0$}
 \psfrag{t}[c][c]{$t$}
 \psfrag{x0}[c][c]{$x_0$}
 \psfrag{M1}[c][c]{$(M,x_0,g(t_0))$}
 \psfrag{M2}[c][c]{$(M,x_0,g(0))$}
 \psfrag{B1}[c][c]{$yellow: B_{g(t_0)}(x_0,r)$}
 \psfrag{B2}[c][c]{$blue: B_{g(t_0)}(x_0,\frac{r}{K})$}
 \psfrag{B3}[c][c]{$red: B_{g(t_0)}(x_0,\frac{r}{K^3})$}
 \psfrag{B4}[c][c]{$green: B_{g(0)}(x_0,r)$}
 \psfrag{B5}[c][c]{$black: \mathcal{D}_{\xi}(M,0)$}
 \includegraphics[width=0.5 \columnwidth]{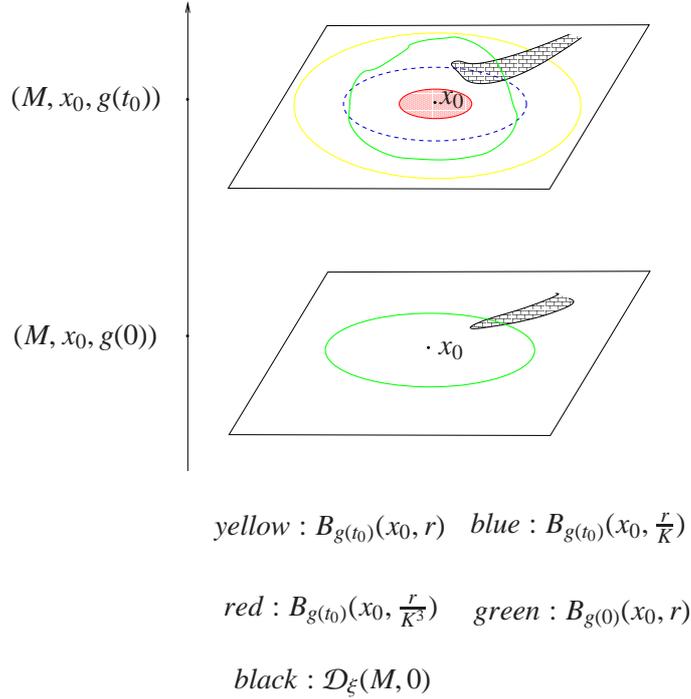}
 \caption{Find a geodesic ball with almost Euclidean volume ratio}
  \label{figure:backpseudo}
 \end{center}
 \end{figure}

\begin{definition}
 Let $\mathscr{K}(n,A)$ be the collection of polarized K\"ahler Ricci flows satisfying (\ref{eqn:SK20_1}).
 For every $r \in (0,1]$, define
    \begin{align*}
    \mathscr{K}(n,A;r)
    \triangleq \left\{\mathcal{LM} \left| \mathcal{LM} \in \mathscr{K}(n,A), \mathbf{pcr}(M \times [-1,1]) \geq r \right. \right\}.
  \end{align*}
   \label{dfn:SC23_1}
\end{definition}

Clearly,  $\mathscr{K}(n,A) \supset \mathscr{K}(n,A;r_1) \supset \mathscr{K}(n,A;r_2)$ whenever $1\geq r_2 >r_1>0$.
Since every polarized K\"ahler Ricci flow $\mathcal{LM} \in \mathscr{K}(n,A)$ has a smooth compact underlying manifold,  we see $\mathcal{LM} \in \mathscr{K}(n,A;r)$
for some very small $r$, which depends on $\mathcal{LM}$. Therefore, it is clear that
\begin{align*}
          \bigcup_{0<r<1}  \mathscr{K}(n,A;r) = \mathscr{K}(n,A).
\end{align*}
Fix $r>0$, we shall first make clear the structure of $\mathscr{K}(n,A;r)$ under the help of polarized canonical radius.
Then we show that the canonical radius can actually been bounded a priori.  In other words, there exists a uniform small constant $\hslash$ (Planck scale)
such that
\begin{align*}
  \mathscr{K}(n,A) = \mathscr{K}(n,A;\hslash),
\end{align*}
which will be proved in Theorem~\ref{thm:SC28_1}.

\begin{proposition}[\textbf{Limit space-time with static regular part}]
 Suppose $\mathcal{LM}_i \in \mathscr{K}(n,A)$ satisfies the following properties.
 \begin{itemize}
 \item  $\mathbf{pcr}(M_i \times [-T_i,T_i]) \geq r_0$ for each $i$.
 \item  $\displaystyle \lim_{i \to \infty} \sup_{\mathcal{M}_i} (|R|+|\lambda|)=0$.
 \end{itemize}
 Suppose $x_i \in M_i$ and
 $\displaystyle \lim_{i \to \infty} \mathbf{cvr}(x_i,0)>0$, then
 \begin{align}
   (M_i,x_i, g_i(0)) \longright{\hat{C}^{\infty}}  (\bar{M},\bar{x}, \bar{g}).   \label{eqn:MC05_10}
 \end{align}
 Moreover,  we have
 \begin{align}
   (M_i,x_i, g_i(t)) \longright{\hat{C}^{\infty}}  (\bar{M},\bar{x}, \bar{g})    \label{eqn:MC05_11}
 \end{align}
 for every $t \in (-\bar{T}, \bar{T})$, where $\displaystyle \bar{T}=\lim_{i \to \infty} T_i>0$.
 In particular, the limit space does not depend on time.
\label{prn:SL04_1}
\end{proposition}

\begin{proof}
  It follows from the combination of Proposition~\ref{prn:SC09_3} and Proposition~\ref{prn:SK27_3} that the limit
  space does not depend on time.   From the definition of canonical radius, the convergence locate in
  $\hat{C}^4$-topology for each time. However, this can be improved to $\hat{C}^{\infty}$-topology.
  Actually, if $\bar{y}$ is a regular point of $\bar{M}$(c.f. the definition in Theorem~\ref{thm:HE11_1}), then we can find
  $y_i \in M_i$ such that $y_i \to \bar{y}$ and $\mathbf{cvr}(y_i,0) \geq \eta$ uniformly for some $\eta>0$, in light of equation (\ref{eqn:HA11_3}) and the proof of Theorem~\ref{thm:HE11_1}. 
  It follows from Proposition~\ref{prn:SK27_3} that $\inf_{t \in [-1,0]} \mathbf{cvr}(y_i, t) \geq K^{-1} \eta$ for all large $i$. 
  By second property, or regularity estimate of canonical radius (c.f. Definition~\ref{dfn:SC02_1}), we know 
  \begin{align*}
      |Rm|(z,t) \leq CK^{-4} \eta^{-2}, \quad \forall \; z \in B_{g_i(t)}(y_i, K^{-2}\eta), \; t \in [-1, 0]. 
  \end{align*}
  Note that $R \to 0$, which implies $|Ric| \to 0$ when we have $|Rm|$-bound in a bigger ball(c.f. the $|Ric| \leq \sqrt{|Rm||R|}$-type estimate in~\cite{Wa2}).   In particular, we have
  $B_{g_i(0)}(y_i, 0.1 K^{-2} \eta) \subset B_{g_i(t)}(y_i, K^{-2} \eta)$ for all $t \in [-0.5, 0]$.  Hence, we obtain
  \begin{align}
     |Rm|(z,t) \leq CK^{-4} \eta^{-2}, \quad \forall \; z \in B_{g_i(0)}(y_i, 0.1 K^{-2}\eta), \; t \in [-0.5, 0].   \label{eqn:MC05_5}
  \end{align}
  Then we can apply Shi's estimate to obtain that $|\nabla^k Rm| \leq C_k$ on $B_{g_i(0)}(y_i, 0.01 K^{-2}\eta)$ for each positive integer $k$. 
  This is enough to set up a uniform sized harmonic coordinate chart around $y_i$ (with respect to metric $g_i(0)$) and all the metric tensor and its derivatives are uniformly bounded(c.f. Hamilton~\cite{Ha95}). 
  Clearly, the convergence around $\bar{y}$ happens in the $C^{\infty}$-topology.  Since $\bar{y}$ is an arbitrary regular point, we see that the convergence to $\bar{M}$ is in $\hat{C}^{\infty}$-Cheeger-Gromov topology.
  
  Note that we currently do not know whether $\bar{M}$ locates in the model space $\widetilde{\mathscr{KS}}(n,\kappa)$.   However, we do know that
  $\bar{M}=\mathcal{R}(\bar{M}) \cup \mathcal{S}(\bar{M})$. The regular part is a smooth Ricci-flat manifold, due to the smooth convergence and $|Ric| \to 0$ on regular part. 
  The singular part satisfies the Minkowski dimension bound(c.f. (\ref{eqn:SB13_4}) in Theorem~\ref{thm:HE11_1}): 
  \begin{align}
   \dim_{\mathcal{M}} \mathcal{S} \leq 2n-2p_0< 2n-4+\frac{2}{1000n}<2n-4+\frac{2}{2n-1},   \label{eqn:MC05_1}
 \end{align}
 where we used the choice of  $p_0$, which is discussed above the equation (\ref{eqn:SC17_11}).  
 Recall that each $(M_i, g_i(0))$ has uniform Sobolev constant (c.f. equation (\ref{eqn:SK20_1})) and uniform volume doubling condition(c.f. Q. Zhang~\cite{Zhq3} and Chen-Wang~\cite{CW5}).
 Therefore, there is a uniform local $L^2$-Poincar\'e constant.  All these estimates bypass to the limit space $(\bar{M}, \bar{x}, \bar{g})$.
 There exists a good heat semigroup theory on $\bar{M}$.  By the high codimension(c.f. inequality (\ref{eqn:MC05_1})) of $\mathcal{S}(\bar{M})$, we know that for every bounded heat solution $u$ on $\bar{M}$, 
 the gradient function $|\nabla u|\in N_{loc}^{1,2}(\bar{M})$ and it is a heat subsolution(c.f. Lemma~\ref{lma:HH18_2} and Remark~\ref{rmk:GD11_1}).
 Consequently, the technique used in the proof of the Cheeger-Gromoll splitting, i.e., Claim~\ref{clm:GD11_1} in Lemma~\ref{lma:HE04_1}, 
 can be applied here. Basically, every function $f \in N_0^{1,2}(\hat{Y})$ has a better version $\displaystyle \tilde{f}=\lim_{t \to 0^{+}} e^{-\Delta t}(f)$, 
 whose point-wise gradient(even on $\mathcal{S}(\bar{M})$, understood as weak upper gradient, c.f. Definition~\ref{dfn:HD18_1}, or Cheeger~\cite{Cheeger99}) can be bounded by the $L^{\infty}$-norm of the gradient. 
 We remark that such technique should be standard in the study of general metric measure spaces. 
  For example, it was used in the discussion in section 4.1.3 of Gigli's work~\cite{Gigli13}.  
  Actually,  the Lipshitz function approximation Lemma, i.e., Lemma 10.7 of Cheeger's work in 1990's~\cite{Cheeger99},  can be regarded as an predecessor of the technique mentioned above. 
 For the convenience of the readers who are not familiar with the singular space, we also provide a detailed alternative proof as follows, which only uses the Ricci flow, heat equation, Proposition~\ref{prn:SK27_3}
 and the canonical radius assumption.
 
 \begin{claim}[\textbf{Good version of Lipshitz function}]
  Every bounded function $f \in N_0^{1,2}(\bar{M})$ with finite $\norm{\nabla f}{L^{\infty}(\bar{M})}$  and Lipschitz on $\mathcal{R}(\bar{M})$
  has a good version $\tilde{f}$ such that 
  \begin{align}
       &f(x)=\tilde{f}(x), \quad \forall \; x \in \mathcal{R}(\bar{M}), \label{eqn:MC05_8} \\
       &\sup_{\bar{M}} |\nabla \tilde{f}| \leq \norm{\nabla f}{L^{\infty}(\bar{M})},  \label{eqn:MC05_9}
  \end{align}
  where the inequality (\ref{eqn:MC05_9}) can be understood as
  \begin{align*}
     |\tilde{f}(x)-\tilde{f}(y)| \leq \norm{\nabla f}{L^{\infty}(\bar{M})} \cdot d(x, y), \quad \forall \; x, y \in \bar{M}. 
  \end{align*}
 \label{clm:MC05_1} 
 \end{claim}

 This is a flow property, so we assume $\lambda=0$ without loss of generality. 
 For simplicity of notation, we also assume that  $\norm{\nabla f}{L^{\infty}(\bar{M})}=1$ and the support of $f$ is contained in $B(\bar{x}, 1)$.
 Note that these assumptions can always be achieved up to rescaling argument.
  Let $\chi_{\epsilon}=\phi(\frac{d(x, \mathcal{S})}{\epsilon})$ be the cutoff function where $\phi$ is a smooth cutoff function such that
 $\phi \equiv 1$ on $(2,\infty)$ and $\phi \equiv 0$ on $(-\infty, 1)$ and $|\nabla \phi| \leq 2$, $\phi' \leq 0$. 
 Then $\chi_{\epsilon} f$ is a Lipschitz function with compact support. 
 By the smooth convergence away from singularity(c.f. Proposition~\ref{prn:SK27_3} and the discussion around inequality (\ref{eqn:MC05_5})),
 we can regard $\chi_{\epsilon} f$ as a Lipshitz function on $(M_i, g_i(-\delta))$, denoted by $f_{\epsilon, i}$, where $\delta=\epsilon^{\frac{1}{n}}$.
 Starting from $f_{\epsilon,i}$, we solve the heat equation until time $t=0$
 and obtain a function $h_{\epsilon,i}=f_{\epsilon,i}(0)$, together with the metric evolving by the Ricci flow. Then we have
 \begin{align*}
   h_{\epsilon, i}(x)=\int_{M_i} w(x,y,-\delta) f_{\epsilon, i}(y) dv_y, \quad \forall \; x \in M_i,
 \end{align*}
 where $w$ is the fundamental solution of $\square^{*} w=(\partial_{\tau}-\Delta +R) w=0$.   Recall that $\int_{M_i} w dv \equiv 1$ and $|f_{\epsilon, i}| \leq C$ uniformly, we have
 \begin{align}
   |h_{\epsilon, i}|(x)= \left| \int_{M_i} w(x,y,-\delta) f_{\epsilon, i}(y) dv_y \right| \leq \sup_{M_i} |f_{\epsilon, i}| \int_{M_i} w(x,y, -\delta) dv_y = \sup_{M_i} |f_{\epsilon, i}| \leq C.
 \label{eqn:MC05_4}  
 \end{align}
 Direct calculation shows that
 \begin{align*}
  \square |\nabla f_{\epsilon, i}|^2=\left( \partial_t -\Delta \right) |\nabla f_{\epsilon,i}|^2=-2|\nabla \nabla f_{\epsilon, i}|^2 \leq 0. 
 \end{align*}
 It follows that
 \begin{align*}
   |\nabla h_{\epsilon,i}|^2(x) - \int_{M_i} w(x,y,-\delta) |\nabla f_{\epsilon, i}|^2 (y) dv_y=-2 \int_{-\delta}^{0} \int_{M_i} w(x,y,t) |\nabla \nabla f_{\epsilon, i}|^2 dv_y dt \leq 0. 
 \end{align*}
 Consequently, we have
 \begin{align}
   |\nabla h_{\epsilon,i}|^2(x) &\leq  \int_{M_i} w(x,y,-\delta) |\nabla f_{\epsilon, i}|^2 (y) dv_y \notag\\
     &=\int_{\Omega_i \backslash A_i}  w(x,y,-\delta) |\nabla f_{\epsilon, i}|^2 (y) dv_y + \int_{A_i} w(x,y,-\delta) |\nabla f_{\epsilon, i}|^2 (y) dv_y,  \label{eqn:MC05_2}
 \end{align}
 where $A_i$ is the set where the pull back of $\chi_{\epsilon}$ achieves values in $(0,1)$,  $\Omega_i$ is the support of the pull back function $f_i$.  Note that $|\nabla f_{\epsilon, i}|(x) \leq 1+\xi$ for arbitrary small, but fixed $\xi$, 
 whenever $i$ is large enough and $x \in \Omega_i \backslash A_i$.  On $A_i$, we have $|\nabla f_{\epsilon, i}| \leq C \epsilon^{-1}$ for some universal constant $C$.  Note that $\Omega_i \subset B_{g_i(0)}(x_i, 1)$, the canonical assumption then
 implies the density estimate $ |A_i| \leq C \epsilon^{2p_0}$.   Recall that we have the heat kernel(hence conjugate heat kernel estimate) estimate $w(x,y,-\delta) \leq C \delta^{-n}$ for some universal constant $C$. 
 Plugging these inequalities into (\ref{eqn:MC05_2}), we obtain
  \begin{align*}
   |\nabla h_{\epsilon,i}|^2(x) &\leq  \int_{\Omega_i \backslash A_i}  w(x,y,-\delta) |\nabla f_{\epsilon, i}|^2 (y) dv_y + \int_{A_i} w(x,y,-\delta) |\nabla f_{\epsilon, i}|^2 (y) dv_y\\
     &\leq (1+\xi) \int_{\Omega_i \backslash A_i}  w(x,y,-\delta) dv_y + C\epsilon^{-2} \cdot C \delta^{-n} \cdot |A_i|\\
     &\leq (1+\xi) \int_{M_i} w(x,y,-\delta) dv_y  + C \epsilon^{2p_0-2} \delta^{-n}\\
     &\leq (1+\xi) + C \epsilon^{2p_0-2} \delta^{-n} \leq 1+\xi + C \epsilon^{2-\frac{1}{500n}} \delta^{-n},
 \end{align*}
 where we used the fact $\epsilon<1$ and inequality (\ref{eqn:MC05_1}) in the last step. 
 Recall that $\delta=\epsilon^{\frac{1}{n}}$ and let $\xi=\epsilon^{1-\frac{1}{500n}}<<\sqrt{\epsilon}$, we then have
 \begin{align}
    |\nabla h_{\epsilon,i}|^2(x) \leq1+ \sqrt{\epsilon}.    \label{eqn:MC05_3}
 \end{align}
  Moreover, if $\bar{z}$ is a regular point of $\bar{M}$, i.e.,  $\bar{z} \in \mathcal{R}(\bar{M})$.  Let $z_i \in M_i$ and $z_i \to \bar{z}$. Then we have
 \begin{align*}
      &\quad \left| h_{\epsilon, i}(z_i)   -f_{\epsilon, i}(z_i) \right| \\
      &=\left| \int_{M_i} w(z_i,y,-\delta) \left\{ f_{\epsilon, i}(y) - f_{\epsilon, i}(z_i) \right\} dv_y \right|\\
      &\leq  \int_{B_{g_i(0)}(z_i, \delta^{\frac{1}{4}})} w(z_i,y,-\delta) \left| f_{\epsilon, i}(y) - f_{\epsilon, i}(z_i) \right| dv_y +  \int_{M_i \backslash B_{g_i(0)}(z_i, \delta^{\frac{1}{4}})} w(z_i,y,-\delta) \left| f_{\epsilon, i}(y) - f_{\epsilon, i}(z_i) \right| dv_y. 
 \end{align*}
 Note that $\bar{z}$ is regular, we can assume that the regularity scale(for example, $\mathbf{cvr}$) of each $z_i$ is much larger than $\delta=\epsilon^{\frac{1}{n}}$, if we choose $\epsilon$ small enough. 
 Clearly, $B_{g_i(0)}(z_i, \delta^{\frac{1}{4}}) \cap A_i =\emptyset$, which implies the Lipschitz constant of $f_{\epsilon, i}$ on $B_{g_i(0)}(z_i, \delta^{\frac{1}{4}})$ is uniformly bounded by $C$. 
 Recall that $\int_{M_i} w dv \equiv 1$.  Therefore, we have
 \begin{align}
  \left| h_{\epsilon, i}(z_i)   -f_{\epsilon, i}(z_i) \right| \leq C \delta^{\frac{1}{4}} + C \int_{M_i \backslash B_{g_i(0)}(z_i, \delta^{\frac{1}{4}})} w(z_i,y,-\delta) dv_y.  \label{eqn:MC05_6}
 \end{align}
 The last term of the above inequality is a small term which can be absorbed in $C\delta^{\frac{1}{4}}$.
 Actually, let $\psi$ be a cutoff function such that $\psi \equiv 0$ on $B_{g_i(0)}(z_i, 0.5 \delta^{\frac{1}{4}})$,  $\psi \equiv 1$ on $M_i \backslash B_{g_i(0)}(z_i, \delta^{\frac{1}{4}})$.
 Moreover, $|\nabla \psi|^2 + |\Delta \psi| \leq C \delta^{-\frac{1}{2}}$. This can be done since $\delta^{\frac{1}{4}}$ is much less than the regularity scale of $z_i$. 
 Now we extend $\psi$ to be a function on space-time by letting $\psi(x,t)=\psi(x)$.   Due to Proposition~\ref{prn:SK27_3} and the discussion before(c.f. inequality (\ref{eqn:MC05_5})), we obtain
 $|\nabla \psi|^2 + |\Delta \psi| \leq C \delta^{-\frac{1}{2}}$ on $M_i \times [-\delta, 0]$.  Consequently, we obtain
 \begin{align*}
   \frac{d}{dt} \int_{M_i} \psi(y, t) w(z_i,y,t) dv_y=\int_{M_i} ( w \square \psi - \psi \square^* w ) dv_y=- \int_{M_i} w \Delta \psi dv_y.
 \end{align*}
 As $w$ converges to the $\delta$-function at $z_i$ as $t$ approaches $0$, $\psi(z_i, 0)=0$, we have
 \begin{align*}
   0-  \int_{M_i} \psi(y, -\delta) w(z_i,y,-\delta) dv_y =-\int_{-\delta}^{0} \int_{M_i} w \Delta \psi dv_y dt \geq -C \delta^{-\frac{1}{2}} \int_{-\delta}^{0} \int_{M_i} w dv_y dt =-C \delta^{\frac{1}{2}},
 \end{align*}
 which implies that
 \begin{align*}
  \int_{M_i \backslash B_{g_i(0)}(z_i, \delta^{\frac{1}{4}})} \psi(y, -\delta) w(z_i,y,-\delta) dv_y \leq  \int_{M_i} \psi(y, -\delta) w(z_i,y,-\delta) dv_y \leq C \delta^{\frac{1}{2}}. 
 \end{align*}
 Plugging the above inequality into (\ref{eqn:MC05_6}), and noticing that $\delta=\epsilon^{\frac{1}{n}}$, we obtain
 \begin{align}
   \left| h_{\epsilon, i}(z_i)   -f_{\epsilon, i}(z_i) \right| \leq C \delta^{\frac{1}{4}} \leq C \epsilon^{\frac{1}{4n}}.    \label{eqn:MC05_7}
 \end{align}
 It follows from the combination of (\ref{eqn:MC05_4}), (\ref{eqn:MC05_3}) and (\ref{eqn:MC05_7}) that there is a limit function $h_{\epsilon}$ on $\bar{M}$. 
 Let $\epsilon=2^{-i} \to 0$, up to a diagonal sequence argument, we can assume that  $h_{2^{-i}, i}$ converges to a limit function $h$, which satisfies
 \begin{align*}
    &\sup_{\bar{M}} |h| \leq C,  \quad  \sup_{\bar{M}} |\nabla h| \leq 1= \norm{\nabla f}{L^{\infty}(\bar{M})}, \\
    & h(x)=f(x), \quad \forall \; x \in \mathcal{R}(\bar{M}). 
 \end{align*}
 In particular, $h$ is a good version of $f$.  We finish the proof of Claim~\ref{clm:MC05_1}.

 Based on Claim~\ref{clm:MC05_1}, the proof of (\ref{eqn:MC05_11}) follows from the standard technique used in the proof of Lemma~\ref{lma:HE04_1}. 
 Actually, for each $t \neq 0$, we already know that $(M_i, x_i, g_i(t))$ converges in the pointed Gromov-Hausdorff topology to some $(\bar{M}', \bar{x}', \bar{g}')$.
 We only need to show that $\bar{M}'$ is isometric to $\bar{M}$.  
 By Proposition~\ref{prn:SK27_3} and the fact $|R|+|\lambda| \to 0$, we know that there is a natural identification map between $\mathcal{R}(\bar{M})$ and $\mathcal{R}(\bar{M}')$, which contain a common point $\bar{x}$.  
 In the following discussion, we shall show that this identification map can be extended to an isometry between $\bar{M}$ and $\bar{M}'$. 
 
 Let $\bar{y}, \bar{z}$ be two regular points of $\bar{M}$. Clearly, $\bar{y}$ and $\bar{z}$ can also be regarded as regular points on $\bar{M}'$. We omit the identification map for the simplicity of notations. 
 Suppose $d_{\bar{g}}(\bar{y}, \bar{z})=D>0$. We can construct a function $\chi$ on $\bar{M}'$ as follows
 \begin{align}
    \chi(x) =  
  \begin{cases}    
    \max\{ D-d_{\bar{g}}(x, \bar{y}), 0\}, \quad \; &\textrm{if} \quad x \in \mathcal{R}(\bar{M}'), \\
   0,  \quad &\textrm{if} \quad x \in \mathcal{S}(\bar{M}'). 
  \end{cases}    
  \label{eqn:MC07_1}
 \end{align} 
 Fix point $x \in \mathcal{R}(\bar{M}') \backslash B_{\bar{g}'}(\bar{y}, 3D)$,  every smooth curve connecting $x$ and $\bar{y}$ has length as least $3D$.
 In light of  inequality (\ref{eqn:HE11_1}) in Theorem~\ref{thm:HE11_1}(applying to both $\bar{g}$ and $\bar{g}'$), we know that 
 \begin{align*}
   \min\{d_{\bar{g}}(x, \bar{y}),  d_{\bar{g}'}(x, \bar{y}) \} \geq D,
 \end{align*} 
 which implies that $\chi(x)=0$ by definition equation (\ref{eqn:MC07_1}).  We remark that inequality (\ref{eqn:HE11_1}) together with the high codimension of $\mathcal{S}$ implies that
 $\chi \in N_0^{1,2}(\bar{M}')$ and we have
 \begin{align*}
 \norm{\nabla \chi}{L^{\infty}(\bar{M}')}=\norm{\nabla \chi}{L^{\infty}(\mathcal{R}(\bar{M}'))}=\norm{\nabla \chi}{L^{\infty}(\mathcal{R}(\bar{M}))}=1. 
 \end{align*}
 Then we can apply Claim~\ref{clm:MC05_1} to obtain a good version $\tilde{\chi}$ of $\chi$. In particular, we have
 \begin{align}
 d_{\bar{g}}(\bar{y}, \bar{z})=D=\left| \chi(\bar{y}) - \chi(\bar{z}) \right|=\left| \tilde{\chi}(\bar{y}) - \tilde{\chi}(\bar{z}) \right| \leq d_{\bar{g}'}(\bar{y}, \bar{z}) \norm{\nabla \chi}{L^{\infty}(\bar{M}')} \leq d_{\bar{g}'}(\bar{y}, \bar{z}). 
 \label{eqn:MC07_2}
 \end{align}
 Similarly, by reversing the role of $\bar{g}'$ and $\bar{g}$ when we choose the test function, we obtain that
 \begin{align}
   d_{\bar{g}'}(\bar{y}, \bar{z}) \leq d_{\bar{g}}(\bar{y}, \bar{z}). 
 \label{eqn:MC07_3}  
 \end{align}
 By the arbitrary choice of $\bar{y}, \bar{z}$, we know the identity map between $\mathcal{R}(\bar{M})$ and $\mathcal{R}(\bar{M}')$ is an isometry map by (\ref{eqn:MC07_2}) and (\ref{eqn:MC07_3}). 
 Since $\mathcal{R}(\bar{M})$ is dense in $\bar{M}$,  $\mathcal{R}(\bar{M}')$ is dense in $\bar{M}'$, we obtain $\bar{M}$ and $\bar{M}'$ are isometric to each other by taking metric completion. 
 Consequently, (\ref{eqn:MC05_11}) follows from (\ref{eqn:MC05_10}). 
\end{proof}

In Proposition~\ref{prn:SL04_1},
we show that the limit flow exists and is static in the regular part, whenever we have $|R|+|\lambda| \to 0$.
It is possible that the limit points in the singular part $\mathcal{S}$ are moving as time evolves.  However, this possibility
will be ruled out finally(c.f. Proposition~\ref{prn:HE15_1}).

\subsubsection{Tangent structure of the limit space}

In this subsection, we shall show that the tangent space of each point in the limit space has a metric cone structure, provided polarized canonical radius is uniformly bounded below.   Basically, the cone structure is induced from
the localized $W$-functional's monotonicity.   Up to a parabolic rescaling, we can assume $\lambda=0$ without loss
of generality.

\begin{proposition}[\textbf{Local $W$-functional}]
 Suppose $\mathcal{LM}_i \in \mathscr{K}(n,A;r_0)$ and 
 $\sup_{\mathcal{M}_i} (|R|+|\lambda|) \to 0$.

 Let $u_i$ be the fundamental solution of the
 backward heat equation $\left[-\frac{\partial}{\partial t} - \triangle + R \right] u_i =0$ based at the space-time point $(x_i, 0)$. Then $u_i$ converges to
 a limit positive solution $\bar{u}$ on $\mathcal{R} \times (-1, 0]$, i.e.,
 \begin{align*}
   \left[-\frac{\partial}{\partial t} - \Delta + R \right] \bar{u}=0.
 \end{align*}
 Moreover, we have
\begin{align}
     \iint_{\mathcal{R} \times (-1, 0]}  2|t| \left|Ric+\nabla \nabla \bar{f} +\frac{\bar{g}}{2t} \right|^2 \bar{u} dv_{\bar{g}} dt \leq C,
 \label{eqn:SC16_4}
\end{align}
 where $C=C(n,A)$, $\bar{u}=(4\pi |t|)^{-n} e^{-\bar{f}}$.
\label{prn:SC16_1}
\end{proposition}

\begin{proof}
 This is a flow property and has nothing to do with polarization. So we can assume $\lambda=0$ for simplicity. 
 
  Fix $r>0$.  Choose a point $\bar{y} \in \mathcal{R}_{r}$ and a time $\bar{t}<0$.   Without loss of generality, we assume that
  there is a sequence of points $(y_i, t_i)$ converging to $(\bar{y}, \bar{t})$.  Note that $d_{g_i(0)}(y_i, x_i)$ is uniformly bounded. It is not hard to see that $u_i$ is uniformly bounded around $(y_i, t_i)$.
  Actually, let $w_i$ be the heat equation $\square w_i=(\partial_t -\Delta) w_i=0$, starting from a $\delta$-function at $(y_i, t_i)$.  Then by the heat kernel estimate of Cao-Zhang(c.f.~\cite{CaoZhang}), we obtain the 
  on-diagonal bound $\frac{1}{C} |t_i|^{-n}<w_i(y_i, 0)<C|t_i|^{-n}$ for some uniform constant $C$. Then the gradient estimate of Cao-Hamilton-Zhang(c.f.~\cite{Zhq2},~\cite{CaHa}) and the fact $d_{g_i(0)}(y_i, x_i)<C$ implies that $|\log w_{i}(x_i, 0)|$ 
  is uniformly  bounded.  Note that $w_i(x_i, 0)=u_i(y_i, t_i)$ since the integral $\int_M u_i v_i d\mu$ does not depend on time.   Therefore, we have
  \begin{align}
   \frac{1}{C} \leq  u_i(y_i, t_i)=w_i(x_i, 0) \leq C,    \label{eqn:MC13_3}
  \end{align}
  where $C$ depends on $|t_i|$ and $d_{g_i(0)}(y_i, x_i)$. It clearly works uniformly for a fixed-sized space-time neighborhood of $(y_i, t_i)$, where curvatures are uniformly bounded.  
  Then standard regularity argument from heat equation shows that all derivatives of $u_i$ are uniformly bounded around $(y_i, t_i)$. Therefore, there is a limit positive solution $\bar{u}$ around $(\bar{y}, \bar{t})$.   By the arbitrary choice of $r,\bar{y}, \bar{t}$. It is clear
  that there is a smooth heat solution $\bar{u}$ defined on $\mathcal{R} \times (-1, 0)$.

  By Perelman's calculation, for each flow $g_i$, we have
  \begin{align}
      \int_{-1}^{0}\int_{M_i} 2|t| \left|Ric_{g_i}+\nabla \nabla f_i +\frac{g_i}{2t} \right|^2 u_i dv_{g_i} dt=-\mu(M_i, g_i(t_i), 1) \leq C,
  \label{eqn:SK17_1}
  \end{align}
  since Sobolev constant is uniformly bounded. By passing to limit, (\ref{eqn:SC16_4}) follows.
\end{proof}

\begin{theorem}[\textbf{Tangent cone structure}]
 Suppose $\mathcal{LM}_i$ is a sequence of polarized K\"ahler Ricci flow solutions in $\mathscr{K}(n,A;r_0)$,
 $x_i \in M_i$. Let $(\bar{M}, \bar{x}, \bar{g})$ be the limit space of
 $(M_i,x_i,g_i(0))$, $\bar{y}$ be an arbitrary point of $\bar{M}$.
Then every tangent space of $\bar{y}$ is an irreducible metric cone.
\label{thm:SC09_1}
\end{theorem}

\begin{proof}
 Suppose $\hat{Y}$ is a tangent space of $\bar{M}$ at the point $\bar{y}$, i.e., there are scales $r_k\to 0$ such that
 \begin{align}
     (\hat{Y}, \hat{y}, \hat{g})=\lim_{k \to \infty} (\bar{M}, \bar{y}, \bar{g}_k)   \label{eqn:MC12_1}
 \end{align}
 where $\bar{g}_k= r_k^{-2}\bar{g}$.  By taking subsequence if necessary, we can assume $ (\hat{Y}, \hat{y}, \hat{g})$
 as the limit space of $(M_{i_k}, y_{i_k}, \tilde{g}_{i_k})$ where $\tilde{g}_{i_k}=r_i^{-2}g_{i_k}(0)$.
 Denote the regular part of $\hat{Y}$ by $\mathcal{R}(\hat{Y})$.
 Then on the space-time
 $\mathcal{R}(\hat{Y}) \times (-\infty, 0]$, there is
 a smooth limit backward heat solution $\hat{u}$. 
 Recall that $\hat{u}$ is positive by Proposition~\ref{prn:SC16_1}. 
 For every compact subset $K \subset \mathcal{R}(\hat{Y})$ and positive number $H$, it follows from Cheeger-Gromov convergence and
 the estimate (\ref{eqn:SK17_1}) that
 \begin{align*}
     \iint_{K \times [-H, 0]}   2|t| \left| Ric + \nabla \nabla \hat{f} + \frac{\hat{g}}{2t}\right|^2 \hat{u} dv dt =0.
 \end{align*}
 Note the scaling invariance of $\hat{u} dv$ and $|t| \left| Ric + \nabla \nabla \hat{f} + \frac{\hat{g}}{2t}\right|^2 dt$. Actually, if the above equality fails for some $K$ and $H$, then by definition of tangent space and
 the integral accumulation,  we shall obtain the left hand side of (\ref{eqn:SK17_1}) is infinity and obtain a contradiction. 
 Then by  the arbitrary choice of $K$ and $H$, we arrive
 \begin{align*}
     \iint_{\mathcal{R}(\hat{Y}) \times (-\infty, 0]}   2|t| \left| Ric + \nabla \nabla \hat{f} + \frac{\hat{g}}{2t}\right|^2 \hat{u} dv dt =0.
 \end{align*}
 Note that $\mathcal{R}(\hat{Y})$ is Ricci flat. So there is a smooth function $\hat{f}$ defined on $\mathcal{R}(\hat{Y}) \times (-\infty, 0]$ such that
   \begin{align}
   \nabla \nabla \hat{f} + \frac{\hat{g}}{2t} \equiv 0.
   \label{eqn:SL13_1}
   \end{align}
  The above equation means that $\nabla \hat{f}$ is a conformal Killing vector field, when restricted on each time slice $t<0$.
  It follows from the work of Cheeger-Colding(c.f.~\cite{CCWarp}) that there is a local cone structure around each regular point. 
  We shall show that a global cone structure can be obtained due to the high co-dimension of the singular set $\mathcal{S}$ 
  and the Killing property arised from (\ref{eqn:SL13_1}).  
  The basic techniques we shall use in our proof is very similar to that in the proof of Lemma~\ref{lma:HE04_1}, Lemma~\ref{lma:HE04_2} and Proposition~\ref{prn:SL04_1}.

 Let's first list the excellent properties of $\hat{f}$. Recall that $\hat{f}$ satisfies the following differential equation on $\mathcal{R} \times (-\infty, 0)$ from the limit process. 
 \begin{align}
     \hat{f}_{t}=-\Delta \hat{f} + |\nabla \hat{f}|^2-R -\frac{n}{t}=|\nabla \hat{f}|^2.      \label{eqn:MC12_6}
 \end{align}
 On the other hand,  it follows from (\ref{eqn:SL13_1}) that 
 \begin{align*}
   \nabla \left(t|\nabla \hat{f}|^2 +\hat{f} \right)=2t Hess_{\hat{f}}(\nabla \hat{f}, \cdot) + \nabla \hat{f}= -\nabla \hat{f} + \nabla \hat{f} \equiv 0.
 \end{align*}
 So we have $t|\nabla \hat{f}|^2 +\hat{f}=C(t)$, whose time derivatives calculation yields that
 \begin{align*}
   C'(t)&=|\nabla \hat{f}|^2 + 2t \left \langle \nabla \hat{f}, \nabla \hat{f}_t \right \rangle + \hat{f}_t
            =2|\nabla \hat{f}|^2 +2t\left\langle \nabla \hat{f}, \nabla |\nabla \hat{f}|^2 \right \rangle\\
          &=2|\nabla \hat{f}|^2 +4t Hess_{\hat{f}} \left(\nabla \hat{f}, \nabla \hat{f} \right)=2|\nabla \hat{f}|^2-2|\nabla \hat{f}|^2=0,    
 \end{align*}
 where we repeatedly used (\ref{eqn:MC12_6}) and (\ref{eqn:SL13_1}).   Therefore, $t|\nabla \hat{f}|^2 +\hat{f} \equiv C$ on $\mathcal{R} \times (-\infty, 0)$. 
 Replacing $\hat{f}$ by $\hat{f}+C$ if necessary, we can assume that $t|\nabla \hat{f}|^2 + \hat{f} \equiv 0$,  which implies that
 \begin{align}
   \left(t\hat{f} \right)_{t}=\hat{f}+ t \hat{f}_t=\hat{f}+t|\nabla \hat{f}|^2 \equiv 0.     \label{eqn:MC12_5}
 \end{align}
 Consequently, we have
 \begin{align}
    &\hat{f}(x, t)= \frac{-1}{t} \hat{f}(x, -1),  \quad \forall \; x \in \mathcal{R}(\hat{Y}).  \label{eqn:MC12_7}\\
    & \left|\nabla \sqrt{\hat{f}(x, t)} \right|=\frac{1}{2|t|}.      \label{eqn:MC15_2}
 \end{align}
We remark that the above discussion is nothing but the application of general property of gradient shrinking solitons(c.f.  Chapter 4 of Chow-Lu-Ni~\cite{CLN}), in the special case that $Ric \equiv 0$. 

Intuitively, a space which is both Ricci-flat and is a gradient shrinking soliton must be a metric cone.
This can be easily proved if the underlying space is smooth.   
In our current situation,  due to high codimension of $\mathcal{S}$, the cone structure can be established using the technique developed in section 2.  
Suppose $\hat{Y}$ is a metric cone based at $\hat{y}$, then we should have
\begin{align}
 \hat{f}=\frac{d^2}{4|t|},  \label{eqn:MC17_7}
\end{align}
where $d$ is the distance to the origin. This will be confirmed in the following discussion.
The cone structure of $\hat{Y}$ will be established together with equality (\ref{eqn:MC17_7}). 
The basic idea to prove (\ref{eqn:MC17_7}) is to compare the level sets of $\hat{f}$ with geodesic balls, with more and more preciseness. 
Note that similar ideas to estimate distance will be essentially used in section 5.3(c.f. Lemma~\ref{lma:HA15_1}).
We remark that our proof could be much simpler if we use Lemma~\ref{lma:SK27_4}, which is independent(c.f. Remark~\ref{rmk:MC07_1}).
For example, the application of Lemma~\ref{lma:SK27_4} directly implies that $\hat{f}$ must achieve minimum only at base point $\hat{y}$(see step 3 below), 
 since $\hat{f}$ is a strictly convex function in regular part $\cal R$ and can be regarded as a continuous function on $\hat{Y}$(c.f. setp 1 below). 
 Here we want to give a self-contained proof,  using only the good property of $\hat{f}$ to improve the regularity of $\hat{Y}$.

We divide the proof of (\ref{eqn:MC17_7}) into four steps.

 \textit{Step 1. $\hat{f}$ is a nonngative, continuous, proper function which achieves minimum value $0$ at $\hat{y}$. }

 Let us focus our attention on time slice $t=-1$ for a while.  Denote $\hat{f}(x, -1)$ by $\hat{f}(x)$ for simplicity of notation.
 It is not hard to observe that $\hat{f}(x)$ is weakly proper. In other words, we have
 \begin{align}
   \lim_{\mathcal{R}(\hat{Y}) \ni x \to \infty} \hat{f}(x)=\infty.    \label{eqn:MC17_4}
 \end{align}
 For otherwise, we can find a sequence of points $z_i \in \mathcal{R}(\hat{Y})$ such that
 $d(z_i, \hat{y}) \to \infty$ and $\hat{f}(z_i) \leq D$ for some positive number $D$.  
 Note that $\hat{f}$ is uniformly bounded from below in the ball $B(z_i, 1)$.
 Actually, for every smooth point $x \in B(z_i, 1)$, we can find a smooth curve $\gamma$ connecting $x$ to $z_i$ such that $|\gamma| \leq 3 d(x,z_i)$.
 This is an application of inequality (\ref{eqn:HE11_1}) in Theorem~\ref{thm:HE11_1}.  Note that the caonical radius is very large in the current situation. 
 Parametrize $\gamma$ by arc length and let $\gamma(0)=z_i$ and $\gamma(L)=x$. Then $|\gamma|=L \leq 3$. 
 Along the curve $\gamma$,  by (\ref{eqn:MC15_2}), we have
 \begin{align*}
     \frac{d}{ds} \sqrt{\hat{f}}(\gamma(s))=\left \langle \nabla \sqrt{\hat{f}}, \dot{\gamma}(s) \right \rangle \leq \left|\nabla \sqrt{\hat{f}} \right| = \frac{1}{2}.
 \end{align*}
 Integration of the above inequality implies that
 \begin{align}
  \hat{f}(x)=\hat{f}(\gamma(L)) \leq \left( \frac{1}{2}L + \hat{f}(\gamma(0)) \right)^2= \left( \frac{1}{2}L + \hat{f}(z_i) \right)^2 \leq (1.5+D)^2 \leq 2(1+D)^2.
 \label{eqn:MC13_7} 
 \end{align}
 The above inequality holds for every regular point $x \in B(z_i,1)$.  In particular,  we know $\int_{B(z_i, 1)} e^{-\hat{f}} dv$ is uniformly bounded from below by some $C^{-1}$.
 Consequently, we have
 \begin{align*}
    \int_{B(z_i,1)} \hat{u} dv = (4\pi)^{-n} \int_{B(z_i,1)} e^{-\hat{f}} dv \geq \frac{1}{C},
 \end{align*} 
 for some uniform constant $C$ depending on $\kappa$ and $D$.   Up to reselecting a subsequence if necessary, we can assume that all $B(z_i,1)$ are disjoint to each other. Then we have
 \begin{align*}
  C \geq  \sum_{i=1}^{\infty} \int_{B(z_i, 1)} \hat{u} dv \geq \infty,
 \end{align*}
 which is impossible. This contradiction establishes the proof of (\ref{eqn:MC17_4}). 
 Note that in the above disccusion, we already know that  the function $\hat{f}$ is bounded on $B \cap \mathcal{R}(\hat{Y})$ for each fixed geodesic ball $B$, by the application of the proof of (\ref{eqn:MC13_7}). 
 Consequently, we have uniform gradient estimate of $\hat{f}$ in $B \cap \mathcal{R}(\hat{Y})$ by (\ref{eqn:MC12_5}), since $t=-1$. 
 The locally Lipschitz condition guarantees that $\hat{f}$ can be extended as a continuous function on whole $\hat{Y}$. 
 Actually, let $\bar{z}$ be a singular point on $\hat{Y}$.  Suppose $a_k$ and $b_k$ are two sequences of regular points in $\mathcal{R}(\hat{Y})$ converging to $\bar{z}$. 
 Clearly, $d(a_k, b_k) \to 0$. By inequality (\ref{eqn:HE11_1}) in Theorem~\ref{thm:HE11_1}, we can find a smooth curve $\gamma_k \subset \mathcal{R}(\hat{Y})$ connecting
 $a_k, b_k$ such that $|\gamma_k|<3d(a_k, b_k) \to 0$.   The bound of $|\nabla \hat{f}|$ then implies that $|\hat{f}(a_k) -\hat{f}(b_k)| \to 0$.  So we can define
 $\displaystyle \hat{f}(\bar{z}) \triangleq \lim_{y \to \bar{z},  y \in \mathcal{R}(\hat{Y})} \hat{f}(y)$ without ambiguity(c.f. Proposition~\ref{prn:HD16_1} for similar discussion). 
 Therefore, from now on we can regard $\hat{f}$ as a continuous function on $\hat{Y}$, rather than only on $\mathcal{R}(\hat{Y})$. 
 Clearly, the previous discussion implies that $\hat{f}$ is proper. Namely, we have
 \begin{align}
    \lim_{\hat{Y} \ni x \to \infty} \hat{f}(x)=\infty.  \label{eqn:MC17_5}
 \end{align}
 Consequently, the minimum value of $\hat{f}$ can be achieved at some point $\hat{z}$.
 The above discussion can be trivially extended for the function $\hat{f}(\cdot, t)$ for each $t \in (-\infty, 0)$.
 So we know $\hat{f}(\cdot, t)$ is a continuous proper function, which achieves minimum value at $\hat{z}$ also, by (\ref{eqn:MC12_7}). 
 Furthermore, it is also clear that (\ref{eqn:MC12_7}) and the first part of (\ref{eqn:MC12_5}) can be extended to hold on whole $\hat{Y} \times (-\infty, 0)$.  
 Then we observe that
 \begin{align}
  \hat{f}(\hat{y}, t)=\min_{x \in \hat{Y}} \hat{f}(x, t) =0, \quad \forall \; t \in (-\infty, 0).   \label{eqn:MC17_2}
 \end{align}
 Actually, following the discussion around inequality (\ref{eqn:MC13_3}), we can use the on-diagonal estimate of Cao-Zhang and the gradient estimate of Cao-Hamilton-Zhang
 to obtain that 
 \begin{align*}
  (4\pi |t|)^{-n} e^{-\hat{f}(x,t)-C}=\hat{u}(x, t) \geq \frac{1}{C}|t|^{-n}, \quad \forall \; x \in B\left(\hat{y}, \sqrt{|t|} \right) \cap \mathcal{R}(\hat{Y}),  
 \end{align*}
 where we used the fact that we adjusted $\hat{f}$ globally by adding a constant to obtain (\ref{eqn:MC12_5}).  By the continuity of $\hat{f}$,  the above inequality implies that
 \begin{align*}
    \frac{\hat{f}(\hat{y}, -1)}{|t|}=\hat{f}(\hat{y}, t) \leq C, \quad \Rightarrow \quad  \hat{f}(\hat{y}, -1) \leq C|t|,  \quad \forall \; t \in (-\infty, 0).   
 \end{align*}
 This forces that $\hat{f}(\hat{y}, -1)=0$. Recall that $\hat{f}$ is a nonnegative function by (\ref{eqn:MC12_5}), so we obtain $\displaystyle \min_{x \in \hat{Y}} \hat{f}(x, -1)=0$. 
 Then (\ref{eqn:MC17_2}) follows from the extended version of (\ref{eqn:MC12_7}).  So we finish Step 1.

 \textit{Step 2. Unit level set of $\hat{f}$ is comparable with unite geodesic ball centered at $\hat{y}$.}
 
 For each nonnegative number $a$, we define $\Omega_a \triangleq \{x \in \hat{Y} | \hat{f}(x, -1)  \leq a^2 \}$. 
 According to this definition, we immediately know that $\hat{y} \in \Omega_0$.  Furthermore, by (\ref{eqn:MC12_7}), it is clear that
 \begin{align*}
   \Omega_a = \{x \in \hat{Y} | \hat{f}(x, t)  \leq |t|^{-1}a^2 \}, \quad \forall \; t \in (-\infty, 0). 
 \end{align*}
 Note that $\Omega_1$ is bounded by the properness of $\hat{f}(\cdot)=\hat{f}(\cdot, -1)$.
 For simplicity, we assume that $ \Omega_1 \subset B(\hat{y}, 0.5H)$ for some $H>0$. 
 On the other hand, applying the gradient estimate of $\sqrt{\hat{f}}$, i.e., (\ref{eqn:MC15_2}),  and the smooth curve length estimate (\ref{eqn:HE11_1}), we have
 \begin{align*}
   \sqrt{\hat{f}(x)} \leq \sqrt{\hat{f}(\hat{y})} + \frac{1}{2} \cdot 4 \cdot H \leq 2H,  \quad \forall \; x \in B(\hat{y}, H)
 \end{align*}
 which means that $B(\hat{y}, H) \subset \Omega_{2H}$.  Let $D=2H$, we have the following relationships in short:
 \begin{align}
   \Omega_1 \subset B(\hat{y}, 0.25D) \subset B(\hat{y}, 0.5D) \subset \Omega_{D}.  \label{eqn:MC16_5}
 \end{align}
 Euqation (\ref{eqn:MC16_5}) can be regarded as the first step to improve (\ref{eqn:MC17_5}) and (\ref{eqn:MC17_2}). 
 In order to obtain the estimates of general level sets of $\hat{f}$, we need to use the conformal Killing equation (\ref{eqn:SL13_1}). 
 We observe that the space-time vector field $(-\nabla \hat{f}, \frac{\partial}{\partial t})=(-\frac{\hat{r}}{2} \frac{\partial}{\partial \hat{r}}, \frac{\partial}{\partial t})=(-0.5 \hat{r}\partial_{\hat{r}}, \partial_t)$, as the ``lift"
 of the conformal Killing vector field $-\nabla \hat{f}$(c.f. (\ref{eqn:SL13_1})), has many excellent properties.  First, direct calculation(c.f. (\ref{eqn:MC12_6})) shows that
 \begin{align}
   \frac{d}{dt} \hat{f}=\hat{f}_t -|\nabla \hat{f}|^2 \equiv 0     \label{eqn:MC16_3}
 \end{align}
 along the integral curve of this space-time vector field.  Second, it follows from (\ref{eqn:SL13_1}) that
 \begin{align}
    \mathcal{L}_{(-\nabla \hat{f}, \frac{\partial}{\partial t})}  \left\{ |t|\hat{g} \right\}=0.   \label{eqn:MC16_4}
 \end{align}
 Now we can regard $\hat{Y} \times (-\infty, 0)$ as a Riemannian manifold, equipped with metric $|t|\hat{g}(t) + dt^2$(c.f. section 6 of Perelman~\cite{Pe1}). 
 Then $(-\nabla \hat{f}, \frac{\partial}{\partial t})$ is really a Killing vector field. 
 
 \textit{Step 3. $\hat{f}$ and $d(\hat{y}, \cdot)$ have the same unique minimum value point $\hat{y}$.}
 
 In other words, the infimum of $\hat{f}$ must be $0$ and it is only achieved at base point $\hat{y}$. 
 We shall use Killing vector field to generate quasi-isometric diffeomorphisms. 
 Then an application of the technique, i.e., bounding distance by choosing good Lipshitz functions,  used in the proof  Lemma~\ref{lma:HE04_1}, Lemma~\ref{lma:HE04_2} and Proposition~\ref{prn:SL04_1} will imply the diameter bound for general
 level sets $\Omega_a$.  For small $a$, we shall show that $\diam \Omega_a$ is also small. Then $\Omega_0$ has diameter $0$ and consists of only one point $\hat{y}$, which is of course the unique minimum point of $d(\hat{y}, \cdot)$. 
 Actually, if one only want to show $\Omega_0 =\{\hat{y}\}$, then there is a shortcut by using the uniform convexity of $\hat{f}$ on $\mathcal{R}(\hat{Y})$ (i.e. equation (\ref{eqn:SL13_1})), the homogeneity of $\hat{f}$ 
 in time direction(i.e., equation (\ref{eqn:MC12_7})), the fact $\int_{\hat{Y}} e^{-\hat{f}} dv<C$ and the application of Lemma~\ref{lma:SK27_4}. 
 We leave the details to interested readers.   In the following paragraph, we shall show $\Omega_0=\{\hat{y}\}$ together with the construction of the cone structure. 
 
 Killing vector property together with high codimension of $\mathcal{S}$ implies metric product rigidity, as we have done in Lemma~\ref{lma:HE04_1}, Lemma~\ref{lma:HE04_2}. 
 We repeat the dicussion here again for the convenience of the readers. 
 Fix each positive integer $k$. We claim that there is a bounded closed set $E_k \subset \hat{Y}$ satisfying $\dim_{\mathcal{M}}(E_k)<2n-2$.
 Furthermore, for each $t \in [-2^{-k}, -2^{-k}]$, we have a family of smooth diffeomorphism $\varphi_{k,t}$
 from $\Omega_{D} \backslash E_{k}$ to $\Omega_{\sqrt{2^k|t|}D} \backslash E_{k}$ with
 \begin{align}
   \varphi_k^*(\hat{g})(z)=2^k|t| \hat{g}(z), \quad \forall \; t \in [-2^{-k}, -2^{-k}],  \; z \in \Omega_{\sqrt{2^k|t|}D} \backslash E_{k}.   \label{eqn:MC17_6}
 \end{align}
 The set $E_k$ can be constructed similarly as the set $E_k$ in the proof of Claim~\ref{clm:MA20_1}.
 Now the Killing vector field $\nabla b^{+}$ is replaced by the space-time ``Killing"(c.f. (\ref{eqn:MC16_4})) vector field $(-\nabla \hat{f}, \partial_t)$. 
 Let's describe more details about the construction of $E_k$. Actually, fixing a small positive number $\xi$, we define the set $E_{k,\xi}^{-}$ to be
 \begin{align}
     \{x \in \Omega_D | \textrm{flow line of} \; (-\nabla \hat{f}, \partial_t) \; \textrm{passing through} \; (x,-2^{-k}) \; \textrm{hits} \; \mathcal{D}_{\xi} \; \textrm{at some} \; t \in (-2^{k}, -2^{-k})\}.
 \label{eqn:MC17_1}    
 \end{align}
 The minus sign in $E_{k,\xi}^{-}$ indicates that we are flowing backward along the space-time integral curve of $(-\nabla \hat{f}, \partial_t)$, since $-2^{-k}>t$ for each $t \in (-2^{k}, -2^{-k})$. 
 Note that the intersection point to $\mathcal{D}_{\xi}$ locates in a uiformly bounded set.
 This can be simply proved as follows.
 Let $(y, -\tau)$ be the first point on $\mathcal{D}_{\xi}$.  
 By (\ref{eqn:MC16_3}) and (\ref{eqn:MC12_7}), we have
 \begin{align*}
     \frac{\hat{f}(y, -1)}{\tau}=\hat{f}(y, -\tau)=\hat{f}(x, -2^{-k})=2^k \hat{f}(x, -1) \leq2^k \cdot D^2, \quad \Rightarrow \quad \hat{f}(y, -1) \leq 2^k \tau D^2 \leq 4^k D^2, 
 \end{align*}
 which means that $y \in \Omega_{2^k D}$, a uniformly bounded set by the properness of $\hat{f}$. 
 By high Minkowski codimension of $\mathcal{S}$ and the application of the Killing condiiton (\ref{eqn:MC16_4}), similar argument for (\ref{eqn:MA20_3}) in Claim~\ref{clm:MA20_1} implies that
 \begin{align*}
    |E_{k,\xi}^{-}| \leq C \xi^{2p_0-1-\epsilon}, 
 \end{align*}
 where $p_0$ is the constant appeared in (\ref{eqn:MC05_1}), i.e., $\dim_{\mathcal{M}} \mathcal{S}<2n-2p_0$, $C$ may depends on $\epsilon$ also. 
 Let $\xi_i \to 0$ and define $E_k^{-}=\cap_{i=1}^{\infty} E_{k, \xi}^{-}$.  We obtain a measure-zero closed set $E_{k}^{-}$.   Moreover, same as (\ref{eqn:MC14_0}) in the proof of Claim~\ref{clm:MA20_1}, 
 the $\xi$-neighborhood of $E_{k}^{-}$ is contained in $E_{k, C\xi}^{-}$
 for some uniform constant $C$.  Then the above volume estimate implies that $\dim_{\mathcal{M}} E_{k}^{-} \leq 2n-2p_0+1<2n-2$. 
 Now we reverse the direction.  Similar to the definition of $E_{k,\xi}^{-}$ in (\ref{eqn:MC17_1}), we can define $E_{k,\xi}^{+}$ as follows:
 \begin{align*}
     \{x \in \Omega_{2^k D} | \textrm{flow line of} \; (-\nabla \hat{f}, \partial_t) \; \textrm{passing through} \; (x,-2^{k}) \; \textrm{hits} \; \mathcal{D}_{\xi} \; \textrm{at some} \; t \in (-2^{k}, -2^{-k})\}. 
 \end{align*}
 Clearly, the plus sign in $E_{k,\xi}^{+}$ indicates that we are flowing forward along the space-time integral curve of $(-\nabla \hat{f}, \partial_t)$, since $t>-2^{k}$ for each $t \in (-2^k, -2^{-k})$. 
  Suppose we start from $(x, -2^{k})$ outside $\mathcal{D}_{\xi}$ and the flow line of $(-\nabla \hat{f}, \partial_t)$ enters $\mathcal{D}_{\xi}$ at some $(y, -\tau)$. 
 We know $\hat{f}(y, -\tau)=\hat{f}(x, -2^{k})$ since the flow preserves $\hat{f}$-value.  Then we have
 \begin{align*}
  \hat{f}(y, -1)=\tau \hat{f}(y, -\tau)=\tau \hat{f}(x, -2^k)=2^{-k}\tau \hat{f}(x,  -1) \leq \hat{f}(x, -1) \leq 4^k D^2. 
 \end{align*} 
 Consequently, $y \in \Omega_{2^k D}$.  Therefore, the forward flow is also restricted in a bounded domain when we start from a point $(x, -2^{k})$ satisfying $x \in \Omega_{2^k D}$.  
 Applying high codimension of $\mathcal{S}$ and Killing condition again, we know $E_{k,\xi}^{+}$ has volume bounded by $C \xi^{2p_0-1-\epsilon}$.   Let $\xi_i \to 0$ and
 set $E_{k}^{+}$ to be $\cap_{i=1}^{\infty} E_{k, \xi}^{+}$. We know $E_{k}^{+}$ is a bounded closed set satisfying $\dim_{\mathcal{M}} E_{k}^{+}<2n-2$. 
 Now we define
 \begin{align}
    E_k \triangleq E_{k}^{+} \cup E_{k}^{-}.    \label{eqn:MC17_3}
 \end{align}
 Then each $E_k$ is a closed bounded set satisfying $\dim_{\mathcal{M}} E_k<2n-2$.
 According to their definitions and the above discussion, we know that there is a family of diffeomorphism $\varphi_{k,t}$, parametrized by $t \in [-2^{k}, -2^{-k}]$,  from $\Omega_{D} \backslash E_k$ to $\Omega_{\sqrt{2^k|t|} D} \backslash E_k$, generated by the integral curve of $(-\nabla \hat{f}, \partial_t)$. 
 It is clear that (\ref{eqn:MC17_6})  follows from the integration of (\ref{eqn:MC16_4}). 
 The above argument is almost the same as that in the proof of Claim~\ref{clm:MA20_1} in Lemma~\ref{lma:HE04_1}. 
 In particular, the argument for the proof of equation (\ref{eqn:MC14_4}) is more or less repeated here.
 We remind the readers that weak convexity of $\mathcal{R}$ is not used in the proof of equation (\ref{eqn:MC14_4}). Only the high codimension of $\mathcal{S}$ and the Killing vector properties are used.

 Now we are ready to use the existence of the diffeomrophism(c.f. discussion around (\ref{eqn:MC17_6})) $\varphi_{k, -2^{k}}: \Omega_D \backslash E_k \to \Omega_{2^kD} \backslash E_k$ to relate the estimate of general $\Omega_a$ to (\ref{eqn:MC16_5}). 
 We are particularly interested in the sets $\Omega_a$ for small $a$'s.  Without loss of generality, let $a=2^{-k}$. Fix some points $x,y \in \Omega_{2^{-k}} \backslash E_k$. Denote $\rho=d(x,y)$. 
 Similar to (\ref{eqn:MC07_1}) in the proof of Proposition~\ref{prn:SL04_1},  we choose a function
  \begin{align}
  \tilde{\chi} \triangleq \max \{\rho-d(\cdot, x), 0\}. 
 \label{eqn:MC12_3}
 \end{align}
 Note that $x,y \in \Omega_{2^{-k}} \subset \Omega_1 \subset B(\hat{y}, 0.25D)$ by (\ref{eqn:MC16_5}), which forces that $\rho=d(x,y)<0.25D$.
 Also by (\ref{eqn:MC16_5}), we know that $\tilde{\chi}$ is supported in $B(x,\rho) \subset B(x, 0.25D) \subset B(\hat{y}, 0.5D) \subset \Omega_{D}$. 
 Let $\varphi$ be the diffeomorphism generated by integrating $(-\nabla \hat{f}, \partial_t)$ from $t=-2^{-k}$ to $t=-2^{k}$. 
 In other words, $\varphi=\varphi_{k, -2^{-k}}$.   Using $\varphi$, we can push forward the function $\tilde{\chi}$ to obtain
 \begin{align*}
    \varphi_*(\tilde{\chi})(z) \triangleq \tilde{\chi} (\varphi^{-1} (z)),  \quad \forall \; z \in \Omega_{2^kD} \backslash E_k.  
 \end{align*}
 Clearly, $\varphi_*(\tilde{\chi})$ is supported on $\Omega_{2^k D} \backslash E_k$ with $\norm{\nabla \varphi_*(\tilde{\chi})}{L^{\infty}(\hat{Y})} \leq 2^{-k} \norm{\nabla\tilde{\chi}}{L^{\infty}(\hat{Y})}=2^{-k}$, in light of
 (\ref{eqn:MC17_6}) and $t=-2^{k}$. 
 By the high codimension of $E_k$, we know that $\varphi_*(\tilde{\chi})$ is an $N_0^{1,2}$-function, 
 which has a good version such that $\sup_{\hat{Y}} |\nabla \varphi_*(\tilde{\chi})| \leq \norm{\nabla \varphi_*(\tilde{\chi})}{L^{\infty}(\hat{Y})}$, due to the high codimension of $\mathcal{S}$(c.f. Claim~\ref{clm:MC05_1}). 
 For simplicity of notation, we still denote the new version of $\tilde{\chi}$ by $\tilde{\chi}$. Note that the values of $\tilde{\chi}(x)$ and $\tilde{\chi}(y)$ are independent of the different versions, since $x,y$ are away from $E_k$. 
 Recall that $x, y \in \Omega_{2^{-k}}$. Integration of (\ref{eqn:MC16_3}) implies that 
 \begin{align*}
   \hat{f}(x, -2^{k})=2^{-k} \hat{f}(x, -1)=4^{-k} \hat{f}(x, 2^{-k}) \leq 4^k \cdot 4^{-k} =1.
 \end{align*}
 Therefore, $\varphi(x) \in \Omega_1$. Similarly, we also know $\varphi(y) \in \Omega_1$. 
 Combining the previous inequalities and use (\ref{eqn:MC16_5}) again, we obtain that
 \begin{align*}
   0.5D \geq  d(\varphi(x), \varphi(y)) \geq    \frac{|\varphi_*(\tilde{\chi})( \varphi(x))-\varphi_*(\tilde{\chi})(\varphi(y))|}{\sup_{\hat{Y}} |\nabla \varphi_*(\tilde{\chi})|}
   \geq  \frac{|\varphi_*(\tilde{\chi})(\varphi(x))-\varphi_*(\tilde{\chi})(\varphi(y))|}{\norm{\nabla \varphi_*(\tilde{\chi})}{L^{\infty}(\hat{Y})}} \geq \frac{|\tilde{\chi}(x)-\tilde{\chi}(y)|}{2^{-k}}.
 \end{align*}
 Recall that $\tilde{\chi}(x)=\rho$ and $\tilde{\chi}(y)=0$ by (\ref{eqn:MC12_3}).
 It follows from the above inequality that 
 \begin{align}
   \rho=d(x,y) \leq 0.5D \cdot 2^{-k}=2^{-1-k} D,  \label{eqn:MC12_4}
 \end{align}
 which is independent of the choice of $x,y \in \Omega_{2^{-k}} \backslash E_k$.  Recall that $\Omega_{2^{-k}} \backslash E_k$ is dense in $\Omega_{2^{-k}}$.  So we have 
 \begin{align*}
 \diam \Omega_{2^{-k}}=\diam \{\Omega_{2^{-k}} \backslash E_k\} \leq 2^{-1-k}D.
 \end{align*}
 Consequently,  $\displaystyle \lim_{k \to \infty} \diam(\Omega_{2^{-k}})=0$. 
 Since $\Omega_0=\bigcap_{1 \leq k <\infty} \overline{\Omega_{2^{-k}}}$, we know that $\Omega_0$ consists of only one point $\{\hat{y}\}$.

\textit{Step 4. The level sets of $\hat{f}$ coincide the geodesic balls centered at $\hat{y}$.} 

Define
 \begin{align}
   \hat{r}(x) \triangleq \sqrt{4\hat{f}(x, -1)}=\sqrt{4\hat{f}(x)}, \quad d(x) \triangleq d(x, \hat{y}).   \label{eqn:MC16_1}
 \end{align}
 Recall that in the standard Euclidean case, $\hat{f}=\frac{d^2}{4}$ and $\hat{r}=d$.
 Our destination (\ref{eqn:MC17_7}) is equivalent to the equation $\hat{r}-d \equiv 0$.
 Clearly, we have $  |\nabla \hat{r}|=\frac{|\nabla \hat{f}|}{\sqrt{\hat{f}}}=1$. 
 Recall that (c.f. (\ref{eqn:MC17_3})) each $E_k$ is a bounded closed set with $\dim_{\mathcal{M}} E_k<2n-2$.  Let $E=\cup_{k=1}^{\infty} E_k$.
 Then it is clear that $E$ is measure-zero and $\hat{Y} \backslash E$ is dense in $\hat{Y}$.
 Note that $\hat{Y} \backslash E$ has a cone structure, as every point $x \in \hat{Y} \backslash E$ can be flowed to $\hat{y}$ along the integral curve of $\nabla \hat{f}=\frac{1}{2}\hat{r}\partial_{\hat{r}}$ 
 without hitting singularities(c.f. Section 1 of~\cite{CCWarp}). 
 Let $x \in \mathcal{R}(\hat{Y})$ and $a=\hat{r}(x)>0$, we can find $x_k \in \mathcal{R}(\hat{Y}) \backslash E$ approaching $y$.
 Every point $x_k$ can be flowed to a point nearby $\hat{y}$.  So we obtain 
 \begin{align}
  d(x)=d(x, \hat{y}) \leq \lim_{k \to \infty} d(x_k, \hat{y}) \leq\lim_{k \to \infty} \hat{r}(x_k) \leq \hat{r}(x).   \label{eqn:MC16_7}
 \end{align}
 On the other hand, we can construct a function $\chi$ as
 \begin{align*}
   \chi(x) \triangleq \max\{a-\hat{r}(x), 0\},
 \end{align*}
 which is supported on  a bounded set $\Omega_{0.5 a}$.  Clearly, $\chi$ is Lipshitz. By the high codimeision of $\mathcal{S}$, by replacing $\chi$ with a new version if necessary, we can assume
 $\sup_{\hat{Y}} |\nabla \chi| \leq \norm{\nabla \chi}{L^{\infty}(\hat{Y})} \leq 1$.  Note the values at $x$ adn $y$ does not depend on the choice of versions since they are regular points.   
 Therefore, we have
 \begin{align*}
   d(x, y) \geq \frac{|\chi(x)-\chi(y)|}{\sup_{\hat{Y}} |\nabla \chi|} \geq = \frac{|\chi(x)-\chi(y)|}{\norm{\nabla \chi}{L^{\infty}(\hat{Y})}}=|\chi(x)-\chi(y)|=|\chi(y)|\geq a-|\hat{r}(y)|,   
 \end{align*} 
 for every $y \in \mathcal{R}(\hat{Y})$.  Let $y$ approach $\hat{y}$ in $\mathcal{R}(\hat{Y}) \backslash E$, we obtain $d(x) \geq a=\hat{r}(x)$, which together with (\ref{eqn:MC16_7}) yields that
 \begin{align}
   d(x)=\hat{r}(x)  \label{eqn:MC16_8}
 \end{align}
 for arbitrary $x \in \mathcal{R}(\hat{Y}) \backslash \{\hat{y}\}$.   Since both $d$ and $\hat{r}$ are uniformly Lipshitz, the equation (\ref{eqn:MC16_8}) holds for every $y \in \hat{Y}$ by continuity and density reason.
 In particular, the relationship (\ref{eqn:MC16_5}) can be improved to the following one:
 \begin{align*}
    \Omega_a =B(\hat{y}, 2a), \quad \forall \; a \geq 0. 
 \end{align*}
 This confirms our expectation.  Clearly,  (\ref{eqn:MC17_7}) follows from the combination of (\ref{eqn:MC16_1}) and the extended version of (\ref{eqn:MC16_8}). 
 The proof of (\ref{eqn:MC17_7}) is complete. 
  
 From the discussion in Step 4 of the proof of (\ref{eqn:MC17_7}), we already know that $\hat{Y} \backslash \{E \cup \{\hat{y}\}\}$ has a local cone structure, which induces the global cone structure of $\hat{Y}$ by taking completion. 
 In view of (\ref{eqn:HE11_1}) in Proposition~\ref{thm:HE11_1}, we know $\mathcal{R}(\hat{Y})$ is path connected. Therefore, the cone $\hat{Y}$ is irreducible, i.e., $\hat{Y} \backslash \{\hat{y}\}$ is path connected.
 Therefore, we obtain the global cone structure  from the local cone structure,  due to the high co-dimension of the singular set $\mathcal{S}$ and the Killing property arised from (\ref{eqn:SL13_1}), as we claimed.   
\end{proof}

\subsubsection{Improved estimates in $\mathscr{K}(n,A;r_0)$}

In this subsection, we shall improve the limit space structure by the fact that every tangent space is a metric cone.
For simplicity, we assume $r_0=1$ if we do not mention otherwise. 

\begin{proposition}[\textbf{Improvement of codimension estimate of $\mathcal{S}$}]
 Suppose $\mathcal{LM}_i$ is a sequence of polarized K\"ahler Ricci flow solutions in $\mathscr{K}(n,A;1)$,
 $x_i \in M_i$. Let $(\bar{M}, \bar{x}, \bar{g})$ be the limit space of  $(M_i,x_i,g_i(0))$.
 Let $\mathcal{S}$ be the singular part of $\bar{M}$.  Then
 \begin{align}
     \dim_{\mathcal{M}} \mathcal{S} \leq 2n-2p_0, \quad
     \dim_{\mathcal{H}}\mathcal{S} \leq 2n-4.  \label{eqn:SL15_1}
 \end{align}
\label{prn:SL15_1}
\end{proposition}

\begin{proof}
  The Minkowski dimension estimate follows from Theorem~\ref{thm:HE11_1}.
  Recall that we are in a situation where canonical radius is uniformly bounded from below. 
  Therefore, there is a gap between local behavior of singular point and regular point. 
  In particular, if one tangent space is Euclidean space, then the base point has a neighborhood with smooth manifold structure.
  This follows from the volume convergence(c.f. Proposition~\ref{prn:SC12_1}) and the \textit{regularity estimate} in the definition of canonical radius(c.f. Definition~\ref{dfn:SC02_1}).
  One can find the detailed argument in the proof of Proposition~\ref{prn:HA08_1}, where only polarized canonical radius lower bound is used.
  Note that each iterated tangent space(away from vertex) is also a tangent space, and henceforce a tangent cone with more splitting directions. 
  Consequently, we can use induction to show that every tangent cone's singularity has an integer Hausdorff dimension(c.f.~\cite{CC1}).  
  However, the Minkowski dimension of singularity is
  at most $2n-2p_0$. This forces that every tangent cone's singularity has Hausdorff dimension $2n-4$ at most,
  which in turn implies $\dim_{\mathcal{H}} \mathcal{S}\leq 2n-4$.
\end{proof}

After we set up the tangent cone structure, we can improve Proposition~\ref{prn:SB27_1}.

\begin{proposition}[\textbf{Improvement of regular curve estimate}]
Same conditions as in Proposition~\ref{prn:SL15_1}.

For every two points $x,y \in \mathcal{R}$ and every small positive number $\epsilon>0$, there exists a rectifiable curve connecting $x,y$ such that
\begin{itemize}
 \item $\gamma$ locates in $\mathcal{R}$.
 \item $|\gamma| \leq (1+\epsilon)d(x,y)$.
\end{itemize}
\label{prn:SC17_4}
\end{proposition}

\begin{proof}
  The proof is very similar to the proof of Proposition~\ref{prn:SB27_1}. The basic idea is to use the tangent cone structure, i.e., Theorem~\ref{thm:SC09_1} to improve Proposition~\ref{prn:SB27_1}. 

 \textit{First,  every point in $\bar{M}$ has a cone-like neighborhood.} 
 
 To be more precise,   fix $\epsilon>0$,  for every point $z \in \bar{M}$, there is a radius $r_z$, depending on $z$ and $\epsilon$,  with the following property:
  
  \textit{For every point $v \in B(z, r_z)$, one can find a curve $\alpha$ such that} 
  \begin{itemize}
  \item Initial point of $\alpha$ locates in $B(z, \epsilon d(v,z))$,  end point of $\alpha$ locates in $B(v, \epsilon d(v,z))$.
  \item $\alpha \subset \mathcal{R}$, \; $|\alpha|<(1+\epsilon) d(v,z)$.
  \end{itemize}
The existence of $r_z$ can be obtained by application of Theorem~\ref{thm:SC09_1} and a contradiction blowup argument.
Actually, if for some $z$ such $r_z$ does not exist, we can find $v_i \to z$ such that corresponding $\alpha_i$ does not exist.  
Blowup by $d^{-2}(v_i, z)$,  we obtain a tangent cone $M_{\infty}$ with vertex $z_{\infty}$ and a point $v_{\infty}$ on the unit sphere of the cone.   
By the density of regular part in the tangent cone, we have a regular point $\tilde{v}_{\infty} \in B(v_{\infty}, 0.5 \epsilon)$. 
The cone structure guarantees that the shortest geodesic connecting $\tilde{v}_{\infty}$ to $z_{\infty}$, which we denote by $\overline{z_{\infty} \tilde{v}_{\infty}}$, has regular interior(c.f. (\ref{eqn:MC16_8})).  
Denote the intersection of $\overline{z_{\infty} \tilde{v}_{\infty}}$ and $M_{\infty} \backslash B(z_{\infty}, 0.5 \epsilon)$ by $\alpha_{\infty}$.  
Then  $\alpha_{\infty}$ is a compact curve and locates in the regular part of $M_{\infty}$. 
By the uniform convergence around $\alpha_{\infty}$, we obtain a curve $\alpha_i$ with the desired property before we arrive limit.
Contradiction. 

\textit{Second,  we can find a good covering of each shortest geodesic by cone-like neighborhoods.}
 
Fix any two points $x,y \in \mathcal{R}$. 
Let $\beta$ be a shortest geodesic connecting $x,y$. 
Since $\bigcup_{z \in \beta} B(z, \frac{1}{4} r_z)$ is a cover of a compact curve $\beta$, we can find a finite covering. 
Starting from this finite covering, by deleting redundant extra balls from $x$ to $y$(e.g., using the ``greedy algorithm"), we obtain a covering $\cup_{i=1}^N B(z_i, \frac{1}{4}r_{z_i})$ with the following properties. 
\begin{itemize}
\item $z_i$'s are ordered by their distance to $x$. 
\item  Each point on $\beta$ locates in at most two balls.  If a point on $\beta$ is contained in two balls, then these two balls must be ``adjacent". In other words, if $z \in \beta \cap B(z_k, \frac{1}{4}r_k) \cap B(z_l, \frac{1}{4}r_l)$, then $|k-l|=1$. 
\item  Every  pair of ``adjacent" balls have nonempty intersection, i.e., if $|k-l|=1$, then $B(z_k, \frac{1}{4}r_k) \cap B(z_l, \frac{1}{4}r_l) \neq \emptyset$. 
\end{itemize}
 
\textit{Third, based on the good covering, one can construct approximation curve. }

Now we have a covering of $\beta$ by $\cup_{k=0}^{N} B(z_k, \frac{1}{4}r_k)$ with the property mentioned in the second step. 
Without loss of generalirty, we further assume $z_0=x, z_N=y$.   For each $0 \leq k \leq N-1$,  let $\beta_k$ be the part of $\beta$ connecting $z_{k}$ and $z_{k+1}$, 
let $d_k$ be the length of $\beta_k$. Then we have
\begin{align*}
  d_k=d(z_k, z_{k+1}) < \frac{1}{4}r_k + \frac{1}{4}r_{k+1} \leq \frac{1}{2} \max \{ r_k, r_{k+1}\}. 
\end{align*}
Hence either $z_{k+1}$ locates in the cone-like neighborhood of $z_k$, or $z_k$ locates in the cone-like neighborhood of $z_{k+1}$.  
No matter what case happens, we can find an approximation curve $\alpha_k \subset \mathcal{R}$, whose two ends locate in the $\epsilon d_k$ neighborhood of $z_{k}$ and $z_{k+1}$, satisfying
$|\alpha_k|<(1+\epsilon) d_k$.   According to this choice, the end point of $\alpha_{k-1}$ and the initial point of $\alpha_k$ have distance bounded by $\epsilon (d_k+d_{k-1})$, whenever $1 \leq k \leq N-1$. 
So they can be connected by a curve $\gamma_k \subset \mathcal{R}$ with $|\gamma_k| \leq 3\epsilon (d_k+d_{k-1})$, due to Proposition~\ref{prn:SB27_1}. 
For the boundary case, it is not hard to see that $z_0=x$ can be connected to the initial point of $\alpha_0$ by $\gamma_0 \subset \mathcal{R}$ and $|\gamma_0|<3\epsilon d_0$.
Similarly, $z_N=y$ can be connected to the end point of $\alpha_{N-1}$ by $\gamma_N \subset \mathcal{R}$ and $|\gamma_N|<3\epsilon d_{N-1}$.   
Concatenating all the curves $\alpha_k$ and $\gamma_k$, we obtain a curve $\gamma \subset \mathcal{R}$ connecting $x,y$ and satisfying(c.f. Figure~\ref{figure:gcovering} for the case $N=2$)
  \begin{align*}
   |\gamma| &=\sum_{k=0}^{N-1} |\alpha_k| + \sum_{k=0}^{N} |\gamma_k| \leq   \left( \sum_{k=0}^{N-1} (1+\epsilon) d_k \right) +\left( 3\epsilon d_0 +\sum_{k=0}^{N-2} 3\epsilon (d_k+d_{k+1}) + 3\epsilon d_{N-1} \right) \\
     &=(1+\epsilon) \sum_{k=0}^{N-1} d_k  + 6 \epsilon \sum_{k=0}^{N-1} d_k =(1+7\epsilon) |\beta|=(1+7 \epsilon) d(x,y). 
  \end{align*}
  Replacing $\epsilon$ by $0.1 \epsilon$ at the beginning, we then find a curve $\gamma$ satisfying the requirement.  
  \end{proof}

\begin{figure}
 \begin{center}
 \psfrag{B0}[c][c]{$B(z_0,\frac14 r_0)$}
 \psfrag{B1}[c][c]{$B(z_1,\frac14 r_1)$}
 \psfrag{B2}[c][c]{$B(z_2,\frac14 r_2)$}
 \psfrag{z0}[c][c]{$z_0=x$}
 \psfrag{z1}[c][c]{$z_1$}
 \psfrag{z2}[c][c]{$z_2=y$}
 \psfrag{g0}[c][c]{$\color{blue}{\gamma_0}$}
 \psfrag{g1}[c][c]{$\color{blue}{\gamma_1}$}
 \psfrag{g2}[c][c]{$\color{blue}{\gamma_2}$}
 \psfrag{a0}[c][c]{$\color{red}{\alpha_0}$}
 \psfrag{a1}[c][c]{$\color{red}{\alpha_1}$}
 \includegraphics[width=0.5 \columnwidth]{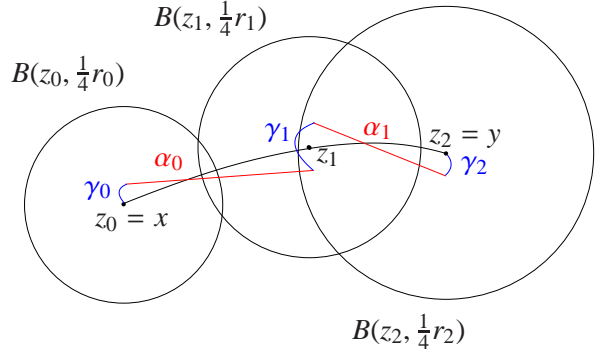}
 \caption{Construction of approximation curve $\gamma$}
 \label{figure:gcovering}
 \end{center}
\end{figure}

\begin{lemma}[\textbf{Rough estimate of reduced distance}]
There is an $\epsilon=\epsilon(n,A)$ with the following properties.

Suppose $\mathcal{LM} \in \mathscr{K}(n,A;1)$,
$x,y\in M$ and $r=d_0(x,y)<1$.  Suppose $y \in \mathcal{F}_{\frac{\epsilon_b}{2} r}(M,0)$.  Then we have
\begin{align}
   l((x,0),(y,-r^2))< 100        \label{eqn:SL25_3}
\end{align}
whenever $\sup_{\mathcal{M}} (|R|+|\lambda|)<\epsilon$.
\label{lma:SL13_1}
\end{lemma}

\begin{proof}
 Let $y_0=y$.
 According to the construction in Proposition~\ref{prn:SL24_1}, there exists a point
 $y_1 \in \partial B_{g(0)}(x,\frac{r}{2}) \cap \mathcal{F}_{\frac{\epsilon_b r}{4}}(M,0)$ and a curve
 $\gamma_1 \subset \mathcal{F}_{\frac{\epsilon_b^2}{8} r}(M,0)$ connecting $y_0,y_1$, with length less than
 $\frac{9}{2}r$.
 Suppose $|R|+|\lambda|$ is small enough, then
 $\displaystyle \gamma_1 \subset \bigcap_{-r^2 \leq t \leq 0} \mathcal{F}_{\frac{\epsilon_b^2 r}{16}}(M,t)$.
 So $\gamma_1$ can be lifted as a space-time curve  connecting $(y_1, -\frac{r^2}{4})$
 and $(y_0, -r^2)$. Reparameterizing $\gamma_1$ by $\tau$, after a proper adjustment, we have
  \begin{align*}
     \int_{\frac{r^2}{4}}^{r^2}  \sqrt{\tau}  |\dot{\gamma}_1|_{g(-\tau)}^2 d\tau < 100 r.
 \end{align*}
 Following the same procedure, we can find $\gamma_2$ connecting $y_1$ to
 $y_2 \in \partial B_{g(0)}(x,\frac{r}{4}) \cap \mathcal{F}_{\frac{\epsilon_b r}{8}}(M,0)$ with
 $\displaystyle \gamma_2 \subset \bigcap_{-\frac{r^2}{4} \leq t \leq 0} \mathcal{F}_{\frac{\epsilon_b^2 r}{32}}(M,t)$.
 By a proper reparameterization of $\tau$, we can regard $\gamma_2$ as a space-time curve connecting
 $(y_1,-\frac{r^2}{4})$ and $(y_2, -\frac{r^2}{16})$, and it satisfies the estimate
 \begin{align*}
      \int_{\frac{r^2}{16}}^{\frac{r^2}{4}}  \sqrt{\tau} |\dot{\gamma}_2|_{g(-\tau)}^2 d\tau < 100 \cdot \frac{r}{2}.
 \end{align*}
 Note that there is no need to choose a new $\epsilon$ because of the rescaling property of $|R|+|{\lambda}|$. 
 Repeating this process, we can find curve $\gamma_k$ connecting $(y_k, -\frac{r^2}{4^k})$
 and $(y_{k+1}, -\frac{r^2}{4^{k+1}})$.  Concatenating all $\gamma_k$'s together, we obtain a space-time curve
 $\gamma$ connecting $(x,0)$ and $(y, -r^2)$ such that
 \begin{align*}
    \int_0^{r^2}  \sqrt{\tau} |\dot{\gamma}|_{g(-\tau)}^2 d\tau < 100 \sum_{k=0}^{\infty} \frac{r}{2^k}=200r.
 \end{align*}
 It follows that
 \begin{align*}
   l((x,0), (y,-r^2)) < \frac{200r}{2\sqrt{r^2}}=100.
 \end{align*}

\end{proof}

\begin{lemma}[\textbf{Most shortest reduced geodesics avoid high curvature part}]
For every group of numbers $0<\xi<\eta<1<H$, there is a big constant $C=C(n,A,\eta,H)$
and a small constant $\epsilon=\epsilon(n,A,H,\eta,\xi)$ with the following properties.

Suppose $\mathcal{LM} \in \mathscr{K}(n,A;1)$, $x\in \mathcal{F}_{\eta}(M,0)$.
Let $\Omega_{\xi}$ be the collection of points $z \in M$ such that there exists a
shortest reduced geodesic $\boldsymbol{\beta}$ connecting $(x,0)$ and $(z,-1)$ satisfying
\begin{align}
   \beta \cap \mathcal{D}_{\xi}(M,0) \neq \emptyset.
\label{eqn:SL24_4}
\end{align}
Then
\begin{align}
|B_{g(0)}(x,H) \cap \mathcal{F}_{\eta}(M,0) \cap \Omega_{\xi}| < C \xi^{2p_0-1}
\label{eqn:SL24_5}
\end{align}
whenever $\sup_{\mathcal{M}} (|R|+|\lambda|)<\epsilon$.
\label{lma:SL14_1}
\end{lemma}

\begin{proof}
This is a flow property, we assume $\lambda=0$ without loss of generality.

From the argument in Lemma~\ref{lma:SL13_1}, it is not hard to obtain the following bound
\begin{align}
  l((x,0),(z,-1)) < C, \quad \forall \; z \in B_{g(0)}(x,H) \cap \mathcal{F}_{\eta}(M,0), \label{eqn:SL24_6}
\end{align}
where $C=C(\eta,H)$. 
Suppose  $z \in B_{g(0)}(x,H) \cap \mathcal{F}_{\eta}(M,0)$,  $\boldsymbol{\beta}$ is a shortest reduced geodesic connecting $(x,0)$ and $(z,-1)$.  Let $\beta$ be the corresponding space curve.
Note that $x$ and $z$, the two end points of $\beta$,  locate outside of $ \mathcal{D}_{\xi}(M,0)$.
Therefore if $z \in \Omega_{\xi}$, then (\ref{eqn:SL24_4}) is satisfied.  In other words, the shortest reduced geodesic connecting $(x, 0)$ and $(z, -1)$ cannot avoid the ``high curvature" part $\mathcal{D}_{\xi}(M,0)$. 
By continuity(c.f. Appendix~\ref{app:B}), we have
\begin{align*}
    \beta \cap \partial \mathcal{F}_{\xi}(M,0) =\beta \cap \partial \mathcal{D}_{\xi}(M,0) \neq \emptyset.
\end{align*}
Let $\tau_a$ be the first time $\boldsymbol{\beta}$ escape from $\mathcal{F}_{K^{-1}\eta}$, $\tau_b$ be the last time such that $\boldsymbol{\beta}(\tau)$ re-enter $\mathcal{F}_{K^{-1}\eta}$, when we move along backward time direction. 
Here $K$ is the constant defined in Proposition~\ref{prn:SC24_1}. To be more precise, we define
\begin{align*}
  &\tau_a \triangleq \sup \left\{ \tau \left| \beta(s) \in \mathcal{F}_{K^{-1}\eta}, \quad \forall \; s \in (0,\tau) \right. \right\}, \\
  &\tau_b \triangleq \inf \left\{ \tau \left| \beta(s) \in \mathcal{F}_{K^{-1}\eta}, \quad \forall \; s \in (\tau, 1) \right. \right\}.
\end{align*}
By the choice of $x$ and $z$, it is clear that $0<\tau_a<\tau_b<1$.   We can further estimate $\tau_a$ and $\tau_b$ uniformly. 
Actually,  since $l$ is achieved by $\boldsymbol{\beta}$ and is bounded by $C$,  it follows from the definition of $l$(c.f. equation (\ref{eqn:MA22_1}) and (\ref{eqn:MA22_2})) that
\begin{align}
  \int_0^1 \sqrt{\tau} \left( R + |\dot{\beta}|^2\right)_{g(-\tau)} d\tau<C. 
\label{eqn:MA22_5}  
\end{align}
Note that $\beta(\tau) \in \mathcal{F}_{\eta}(M, 0)$ whenever $\tau \in (0, \tau_a) \cup (\tau_b, 1)$.  In view of Proposition~\ref{prn:SL04_1},
we have the metric equivalence $0.5 g(x, 0)<g(x, -\tau)<2g(x, 0)$
for all $x \in \mathcal{F}_{K^{-1}\eta}(M, 0)$ and $\tau \in (0,1)$.  Recalling that $|R|$ is uniformly small.  Then (\ref{eqn:MA22_5}) implies that
\begin{align*}
  \int_0^{\tau_a} \sqrt{\tau} |\dot{\beta}|_{g(0)}^2 d\tau <C, \quad \int_{\tau_b}^1 \sqrt{\tau} |\dot{\beta}|_{g(0)}^2 d\tau <C. 
\end{align*} 
It follows from Proposition~\ref{prn:SC24_1} and the above inequality that
\begin{align}
  &\frac{\eta}{C}<d_{g(0)}(x, \beta(\tau_a)) \leq   \int_0^{\tau_a} |\dot{\beta}|_{g(0)} d\tau< \left( \int_0^{\tau_a} \sqrt{\tau} |\dot{\beta}|_{g(0)}^2 d\tau \right)^{\frac{1}{2}}  \left( \int_0^{\tau_a}  \frac{1}{\sqrt{\tau}} d\tau \right)^{\frac{1}{2}}<C\tau_a^{\frac{1}{4}}, \label{eqn:MA22_7}\\
  &\frac{\eta}{C}<d_{g(0)}(\beta(\tau_b), z) \leq   \int_{\tau_b}^{1} |\dot{\beta}|_{g(0)} d\tau< \left( \int_{\tau_b}^{1} \sqrt{\tau} |\dot{\beta}|_{g(0)}^2 d\tau \right)^{\frac{1}{2}}  \left( \int_{\tau_b}^{1}  \frac{1}{\sqrt{\tau}} d\tau \right)^{\frac{1}{2}}
  <C\sqrt{1-\sqrt{\tau_b}}, 
\label{eqn:MA22_8}  
\end{align}
where $C=C(\eta,H,K)$.  Consequently, we have 
\begin{align*}
  \tau_a>\frac{\eta^4}{C}, \quad 1-\tau_b=\left(1-\sqrt{\tau_b} \right) \left(1+\sqrt{\tau_b} \right) \geq 1-\sqrt{\tau_b} \geq \frac{\eta^2}{C}. 
\end{align*}
This means that $[\tau_a, \tau_b] \subset \left[\frac{\eta^4}{C}, 1-\frac{\eta^2}{C} \right]$. 
Define $\bar{\tau}$ as 
\begin{align}
  \bar{\tau} \triangleq \max\{ \tau| \beta(\tau) \in \mathcal{D}_{\xi}(M, 0)\}.   \label{eqn:MC23_1}
\end{align}
Clearly, we have  $\beta(\bar{\tau}) \in \partial \mathcal{D}_{\xi}(M, 0)=\partial \mathcal{F}_{\xi}(M, 0)$. 
Since $\xi<K^{-1}\eta$, we have $\bar{\tau} \in [\tau_a, \tau_b]$ for continuity reason.    Consequently, we know 
\begin{align}
   \frac{\eta^4}{C}< \bar{\tau} < 1-\frac{\eta^2}{C}, 
\label{eqn:MA22_6}   
\end{align}
for some $C=C(\eta, H, K)$, whenever $\xi<K^{-1} \eta$ and $|R|$ very small. 

Beyond the estimate of $\bar{\tau}$, there are more estimates around $\boldsymbol{\beta}(\bar{\tau})$. 
In light of the choice of $\bar{\tau}$, we have $\beta(\tau) \in \mathcal{F}_{\xi}(M,0)$ for each $\tau \in [\bar{\tau}, 1]$. 
By (\ref{eqn:SL24_6}),  we have uniform rough bound of the reduced distance from $(x, 0)$ to $(z, -1)$.
Noting that $R$ may be negative and $|R|$ is very small,  we have
\begin{align*}
  \int_{\bar{\tau}}^{1} \sqrt{\tau} \left( R + |\dot{\beta}|^2\right)_{g(-\tau)} d\tau <1+ \int_0^1 \sqrt{\tau} \left( R + |\dot{\beta}|^2\right)_{g(-\tau)} d\tau<C(\eta,H).
\end{align*}
Following the route of (\ref{eqn:MA22_7}), noting that metrics $g(0)$, $g(-\tau)$ and $g(-1)$ are all uniformly equivalent on $\beta(\tau)$ whenever $\tau \in [\bar{\tau}, 1]$, we have
\begin{align*}
 d_{g(0)}(z, \beta(\bar{\tau})) \leq   \int_{\bar{\tau}}^{1} |\dot{\beta}|_{g(0)} d\tau \leq   2\int_{\bar{\tau}}^{1} |\dot{\beta}|_{g(-1)} d\tau < 2\left( \int_{\bar{\tau}}^{1} \sqrt{\tau} |\dot{\beta}|_{g(-1)}^2 d\tau \right)^{\frac{1}{2}}  \left( \int_{\bar{\tau}}^{1}  \frac{1}{\sqrt{\tau}} d\tau \right)^{\frac{1}{2}}
  < C.
\end{align*} 
Note that $d_{g(0)}(z,x)<H$.  Triangle inequality then implies that
\begin{align}
 d_{g(0)}(\beta(\bar{\tau}), x)<F
 \label{eqn:MA22_9}  
\end{align}
for some $F$ independent of $\xi$ when $|R|+|\lambda|$ small enough.

The purpose of this paragraph is to estimate $\dot{\beta}(\bar{\tau})$. 
Recalling the reduced geodesic equation (\ref{eqn:MA22_3}):
\begin{align*}
  \nabla_V V +\frac{V}{2\tau} + 2Ric(V, \cdot) + \frac{\nabla R}{2}=0,
\end{align*}
where $V=\dot{\beta}$.  It follows that along the reduced geodesic $\boldsymbol{\beta}$, we have
\begin{align*}
  & \frac{d}{d\tau} |\dot{\beta}|^2=2\langle \nabla_{\dot{\beta}} \dot{\beta}, \dot{\beta} \rangle +2Ric(\dot{\beta}, \dot{\beta})=-\frac{|\dot{\beta}|^2}{\tau} -2Ric(\dot{\beta}, \dot{\beta}) -  \langle \nabla R, \dot{\beta} \rangle,\\
  & \frac{d}{d\tau} \left\{\tau |\dot{\beta}|^2 \right\}=-\tau \left\{ 2Ric(\dot{\beta}, \dot{\beta}) +\langle \nabla R, \dot{\beta} \rangle \right\}. 
\end{align*}
Note that $\beta(\tau) \in \mathcal{F}_{\xi}(M,0)$ for each $\tau \in [\bar{\tau}, 1]$.  By Proposition~\ref{prn:SC09_3}, we can assume $\beta(\tau) \in \mathcal{F}_{K^{-1}\xi}(M,-\tau)$. 
It follows that $|Ric|$ and $|\nabla R|$ uniformly small whenever $|R|$ globally very small.   Therefore, the above equation implies
\begin{align}
 \left| \frac{d}{d\tau} \left( \tau |\dot{\beta}|^2 +1\right) \right| <\theta \left( \tau |\dot{\beta}|^2 + 1\right) , \quad \forall \; \tau \in (\bar{\tau}, 1)
\label{eqn:MA22_10} 
\end{align}
for some small constant $\theta$ depending on $\xi$ and $\sup_{\mathcal{M}} |R|$.  Moreover, $\theta \to 0$ if $\sup_{\mathcal{M}} |R| \to 0$ and $\xi$ is fixed. 
Integrating (\ref{eqn:MA22_10}) and using (\ref{eqn:MA22_6}), we obtain
\begin{align*}
   \tau |\dot{\beta}|^2+1> e^{-\theta}  \left( \bar{\tau} |\dot{\beta}|_{g(-\bar{\tau})}^2 +1\right). 
\end{align*}
It follows that
\begin{align*}
    \int_{\bar{\tau}}^{1} \sqrt{\tau}|\dot{\beta}|_{g(-\tau)}^2  d\tau &= \int_{\bar{\tau}}^{1} \frac{1}{\sqrt{\tau}} \tau |\dot{\beta}|_{g(-\tau)}^2 d\tau
     \geq \int_{\bar{\tau}}^{1} \left\{ e^{-\theta} \left( \bar{\tau} |\dot{\beta}|_{g(-\bar{\tau})}^2 +1\right)-1 \right\} d\tau\\
     &=e^{-\theta}  |\dot{\beta}|_{g(-\bar{\tau})}^2 \bar{\tau} (1-\bar{\tau}) + (e^{-\theta}-1)(1-\bar{\tau}).
\end{align*}
In view of (\ref{eqn:MA22_5}) and the fact that $|R|$ is very small, we know the left hand side of the above inequality is bounded above by $C=C(\eta, H)$.  Since $\theta$ is very small, 
$\bar{\tau} \in \left[ \frac{\eta^4}{C}, 1-\frac{\eta^2}{C}  \right]$ by (\ref{eqn:MA22_6}),  the above inequality yields that
\begin{align}
   \left|\dot{\beta}(\bar{\tau}) \right|_{g(-\bar{\tau})} < C,
\label{eqn:MA22_11}   
\end{align} 
where $C=C(\eta, H, K)$ is independent of $\xi$.  Note that $\boldsymbol{\beta}(\tau)=(\beta(\tau), -\tau)$, the space-time tangent vector of $\boldsymbol{\beta}$ is $(\dot{\beta}, -1)$. 
Intuitively, (\ref{eqn:MA22_11}) can be understood that the ``angle" between the space-time tangent and the space tangent form a positive ``angle" which is uniformly bounded below.

Note that the reduced volume element $(4\pi \tau)^{-n}e^{-l} dv$ is decreasing along $\boldsymbol{\beta}$.
Up to a perturbation,  $\partial \mathcal{F}_{\xi}(M,0)$ can be regarded(c.f. Corollary~\ref{cor:GI26_1}) as a smooth hypersurface in $M$ satisfying
\begin{align}
\left| \partial \mathcal{F}_{\xi}(M,0) \cap B_{g(0)}(x,F) \right|_{\mathcal{H}^{2n-1}}   \leq C\xi^{2p_0-1},
\label{eqn:MA21_1}  
\end{align}
for some $C=C(n,A,F)$, $F=F(\eta, H)$ is the constant in (\ref{eqn:MA22_9}). 
Consequently, $\partial \mathcal{F}_{\xi}(M,0) \times [-1, 0]$ can be regarded as a hypersurface in the space-time.
Recall that  $\Omega_{\xi}$ is the collection of points $z \in M$ such that there exists a
shortest reduced geodesic $\boldsymbol{\beta}$ connecting $(x,0)$ and $(z,-1)$ satisfying (\ref{eqn:SL24_4}).
By reduced geodesic theory(c.f. Section 7 of~\cite{Pe1} and the corresponding sections in~\cite{KL} for more details),
the following results are known.
\begin{itemize}
\item[(a).]  For every $z \in M$, $(z, -1)$ can be connected to $(x, 0)$ by a shortest reduced geodesic.
\item[(b).]  For every $z \in M \backslash E$, $(z, -1)$ can be connected to $(x, 0)$ by a unique shortest reduced geodesic, where $E$ is a measure-zero set and is called the $\mathcal{L}$-cut-locus. 
\end{itemize}
Therefore, we can define a projection map $\varphi$ as follows.
\begin{align}
  \varphi:   B_{g(0)}(x,H) \cap \mathcal{F}_{\eta}(M,0) \cap  \left\{ \Omega_{\xi}  \backslash E \right\} &\mapsto \partial \mathcal{F}_{\xi}(M,0) \times [-1, 0],  \notag\\
                    z &\mapsto  \boldsymbol{\beta}(\bar{\tau}). 
\label{eqn:MC23_2}                    
\end{align}
For simplicity, let $\Omega=B_{g(0)}(x,H) \cap \mathcal{F}_{\eta}(M,0) \cap  \left\{ \Omega_{\xi}  \backslash E \right\}$. 
The reduced distance bound (\ref{eqn:SL24_6}) and  the entering-time bound (\ref{eqn:MA22_6}) implies that the reduced volume element  $(4\pi \tau)^{-n}e^{-l} dv$ along $\boldsymbol{\beta}$ is uniformly equivalent to $dv$,
whenever $\tau \in [\bar{\tau}, 1]$.  Since $(4\pi \tau)^{-n}e^{-l} dv$ is monotone along $\boldsymbol{\beta}$, we can regard $dv$ as almost monotone, up to multiplying a uniform constant $C$.    
Therefore, we have
\begin{align*}
 |\Omega|_{\mathcal{H}^{2n}}=\int_{\Omega}1 dv \leq C \int_{\Omega} e^{-l(z)} dv_z \leq \int_{\varphi(\Omega)} \bar{\tau}^{-n} e^{-l(y)} dv_y \leq C \int_{\varphi(\Omega)} dv_y, 
\end{align*}
where $y=\varphi(x)$.  Note that inequality  (\ref{eqn:MA22_11}) can be regarded as an ``angle" bound, since $\dot{\boldsymbol{\beta}}=(\dot{\beta}, -1)$. 
The uniform bound of $|\dot{\beta}|$ guarantees that  $dv_y \leq C |d\sigma_y \wedge dt|$ where $d \sigma_y$ is the ``area" element of $\partial \mathcal{F}_{\xi}$. 
Then we have
\begin{align*}
 |\Omega|_{\mathcal{H}^{2n}}&\leq C \int_{\varphi(\Omega)} |d\sigma_y \wedge dt| \leq C\int_{\left\{\partial \mathcal{F}_{\xi}(M,0) \cap B_{g(0)}(x, F) \right\} \times [-1,0]}   |d\sigma_y \wedge dt|\\
  &=C  \left| \left\{\partial \mathcal{F}_{\xi}(M,0) \cap B_{g(0)}(x, F) \right\} \times [-1, 0] \right|_{\mathcal{H}^{2n}}\\
  &\leq C \left| \partial \mathcal{F}_{\xi}(M,0) \cap B_{g(0)}(x,F) \right|_{\mathcal{H}^{2n-1}}, 
\end{align*} 
where we used the almost product structure of $\partial \mathcal{F}_{\xi} \times [-1, 0]$ in the last step. 
Note that $C=C(\eta, H, K)=C(n,A,\eta,H)$ since $K$ is determined by $n,A$(c.f. Proposition~\ref{prn:SC24_1}). 
Recall that $\Omega=B_{g(0)}(x,H) \cap \mathcal{F}_{\eta}(M,0) \cap  \left\{ \Omega_{\xi}  \backslash E \right\}$. 
Plugging (\ref{eqn:MA21_1}) into the above inequality, we obtain (\ref{eqn:SL24_5}). 
\end{proof}

Note that in Lemma~\ref{lma:SL14_1}, for every point $z \in  \left\{ B_{g(0)}(x,H) \cap \mathcal{F}_{\eta}(M,0) \right\}  \backslash \{\Omega_{\xi} \cup E\}$, there is 
a unique shortest reduced geodesic connecting $(z,-1)$ to $(x, 0)$ and avoiding $\mathcal{D}_{\xi}(M,0)$. 
If $z \in  \left\{ B_{g(0)}(x,H) \cap \mathcal{F}_{\eta}(M,0) \cap E \right\}  \backslash \Omega_{\xi}$, then every shortest reduced geodesic connecting 
$(z, -1)$ to $(x, 0)$ avoids  $\mathcal{D}_{\xi}(M,0)$.  However, we may not have uniqueness. 

Now we pass Lemma~\ref{lma:SL14_1} to Cheeger-Gromov limit and have the following property.

\begin{lemma}[\textbf{Rough weak convexity by reduced geodesics}]
Suppose $\mathcal{LM}_i \in \mathscr{K}(n,A;1)$ satisfies
\begin{align}
  \lim_{i \to \infty} \left( \frac{1}{T_i} + \frac{1}{\Vol(M_i)} + \sup_{\mathcal{M}_i} (|R|+|\lambda|) \right)=0.  \label{eqn:SL06_1}
\end{align}
Suppose $x_i \in M_i$.
Let $(\bar{M}, \bar{x}, \bar{g})$ be the limit space of $(M_i, x_i, g_i(0))$, $\mathcal{R}$ be the regular part of $\bar{M}$ and $\bar{x} \in \mathcal{R}$.
Suppose $\bar{t}<0$ is a fixed number. Then every $(\bar{z}, \bar{t})$ can be connected to $(\bar{x}, 0)$ by a smooth reduced geodesic, whenever $\bar{z}$ is away from a closed measure-zero set.
\label{lma:SK27_4}
\end{lemma}

\begin{proof}
Without loss of generality, let $\bar{t}=-1$ and $\lambda=0$. We use $\bar{g}(0)$ as the default metric on the limit space. 

Since $\bar{x} \in \mathcal{R}$, it locates in $\mathcal{R}_{\eta_0}$ for some $\eta_0 \in (0,1)$, where we used the notation defined in equation (\ref{eqn:HA11_3}). 
Fix $\eta \in (0, \eta_0)$. Let $E_{\eta,\xi}$ be the closure of the limit set of  $B_{g_i(0)}\left(x_i, \eta^{-1} \right) \cap \mathcal{F}_{\eta}(M_i, 0) \cap \Omega_{\xi}(M_i)$, which 
we denote by $E_{\eta, \xi}'(M_i)$ for simplicity.   Suppose $\bar{z} \in E_{\eta,\xi}$ is the limit of some sequence $z_i \in E_{\eta,\xi}'(M_i)$. 
Then it is easy to see that $\bar{z} \in \overline{B(\bar{x}, \eta^{-1})} \cap \mathcal{R}_{\eta}(\bar{M})$. 
For each $i$, there is a shortest reduced geodesic $\boldsymbol{\beta}_i$ connecting $(x_i, 0)$ to $(z_i, -1)$ and passing through $\mathcal{D}_{\xi}(M_i, 0)$.
Let $\boldsymbol{\beta}$ be the limit  of $\boldsymbol{\beta}_i$. 
Note that $\boldsymbol{\beta}$ may pass through singularity. The largest $\tau$ such that $\boldsymbol{\beta}(\tau)$ comes out of $\mathcal{S}_{\xi}$(c.f. equation (\ref{eqn:HA11_4}) for notations)
is denoted by $\bar{\tau}$(c.f. equation (\ref{eqn:MC23_1})).
By (\ref{eqn:MA22_6}), i.e.,  $ \frac{\eta^4}{C}< \bar{\tau} < 1-\frac{\eta^2}{C}$, we know  $\bar{\tau}$ is uniformly bounded away from $0$ and $1$.  Moreover, $d(\bar{x}, \beta(\bar{\tau}))$ is uniformly
bounded by some constant $F$(c.f. inequality (\ref{eqn:MA22_9})), the value $|\dot{\beta}(\bar{\tau})|$ is uniformly bounded by inequality (\ref{eqn:MA22_11}). 
By taking limit on $\bar{M}$, we see that for every point $\bar{z}$(no matter whether it is a limit of points in $E_{\eta,\xi}'(M_i)$), we can find a shortest reduced geodesic $\boldsymbol{\beta}$ connecting $(\bar{x}, 0)$ and $(\bar{z}, -1)$, with $\bar{\tau}$ satisfying (\ref{eqn:MA22_6})
and $\beta(\bar{\tau})$ locating in $B(\bar{x}, F)$ for some uniform constant $F$, and $|\dot{\beta}(\bar{\tau})|$ uniformly bounded by $C$. Note that both $C$ and $F$ are independent of $\xi$.

As a closure, $E_{\eta, \xi}$ is clearly a closed set. Note that $E_{\eta, \xi} \subset \overline{B(\bar{x}, \eta^{-1}) \cap \mathcal{R}_{\eta}} \subset B(\bar{x}, 2\eta^{-1}) \cap \overset{\circ}{\mathcal{R}}_{0.5\eta}$, which is an open smooth manifold. 
Therefore, $E_{\eta, \xi}$ is measurable.

Suppose $\bar{z}_a, \bar{z}_b$ are two points in $\bar{E}_{\eta, \xi}$. Tracing their origin and use the shortest property, it is clear that $\boldsymbol{\beta}_a$ and $\boldsymbol{\beta}_b$ have no intersection
except $(\bar{x}, 0)$, where $\boldsymbol{\beta}_a$ is a shortest reduced geodesic connecting $(\bar{x}, 0)$ to $(\bar{z}_a, -1)$, $\boldsymbol{\beta}_b$ is a shortest reduced geodesic connecting $(\bar{x}, 0)$ to $(\bar{z}_b, -1)$.
Similar to (\ref{eqn:MC23_2}) in the proof of Lemma~\ref{lma:SL14_1}, we now define a mutli-value projection map $\tilde{\varphi}$ from $E_{\eta,\xi}$ to $\partial \mathcal{R}_{\xi}$ as follows:
\begin{align*}
  \tilde{\varphi}:  E_{\eta,\xi} &\mapsto \partial \mathcal{R}_{\xi} \times [-1, 0], \\
  z &\mapsto \{\boldsymbol{\beta}(z), \; \boldsymbol{\beta} \; \textrm{is a shortest reduced geodesic connecting} \; (z, -1) \; \textrm{to} \; (\bar{x}, 0) \; \textrm{with} \; \beta \cap \mathcal{S}_{\xi} \neq \emptyset \}. 
\end{align*}
Following the argument at the end of the proof of Lemma~\ref{lma:SL14_1},  we have
\begin{align*}
 |E_{\eta,\xi}|_{\mathcal{H}^{2n}} = \int_{E_{\eta,\xi}} 1 dv \leq C \int_{E_{\eta,\xi}} e^{-l(z)} dv_{z} \leq \int_{\tilde{\varphi}(E_{\eta,\xi})} \bar{\tau}^{-n} e^{-l(y)} dv_y,
\end{align*}
where $(y, -\bar{\tau})=\boldsymbol{\beta}(\bar{\tau})$ for some $\boldsymbol{\beta}$ connecting $(\bar{x}, 0)$ to $(z, -1)$ satisfying $\beta \cap \mathcal{S}_{\xi} \neq \emptyset$.
Note that the last inequality holds even if $\tilde{\varphi}$ is multi-valued.  
Starting from the above step, the remainder argument exactly follows from the proof of  (\ref{eqn:SL24_5}). Consequently, we have
\begin{align}
 |E_{\eta,\xi}| \leq C\xi^{2p_0-1}     \label{eqn:MA23_2}
\end{align}
for some $C$ independent of $\xi$.  
Note that  $E_{\eta_1,\xi_1} \subset E_{\eta_2, \xi_2}$ whenever $0<\xi_1<\xi_2$, $0<\eta_2 \leq \eta_1$. 
Then we define
\begin{align}
  E_{\eta} \triangleq \bigcap_{\xi \in (0, \eta)} E_{\eta, \xi}.  
\label{eqn:MA23_1}  
\end{align}
In light of (\ref{eqn:MA23_2}), we see that $E_{\eta}$ is a closed subset of $\overline{B(\bar{x}, \eta^{-1}) \cap \mathcal{R}_{\eta}(\bar{M})}$ with measure zero.  
Suppose 
\begin{align*}
   \bar{z} \in \left\{ B(\bar{x}, \eta^{-1}) \cap \mathcal{R}_{\eta}(\bar{M}) \right\} \backslash E_{\eta}=\bigcup_{\xi \in (0, \eta)} \left\{  \left\{B(\bar{x}, \eta^{-1}) \cap \mathcal{R}_{\eta}(\bar{M}) \right\} \backslash E_{\eta, \xi} \right\}, 
\end{align*}
then $\bar{z} \in  B(\bar{x}, \eta^{-1}) \cap \mathcal{R}_{\eta}(\bar{M}) \backslash E_{\eta, \xi}$ for some $\xi \in (0,\eta)$. 
By the smooth flow convergence on $\mathcal{F}_{\xi}(M_i, 0) \times [-1, 0]$(c.f. Proposition~\ref{prn:SL04_1}) and the definition of $E_{\eta, \xi}$, we obtain that
$(\bar{z}, -1)$ can be connected to $(\bar{x}, 0)$ by some shortest smooth reduced geodesic contained in $\mathcal{R}_{\xi}(\bar{M}) \times [-1,0]$. 
Moreover, every smooth shortest reduced geodesic connecting $(\bar{x}, 0)$ and $(\bar{z}, -1)$ are uniformly $\xi$-regular. 
To be more precise, every point $\bar{z} \in  \left\{ B(\bar{x}, \eta^{-1}) \cap \mathcal{R}_{\eta}(\bar{M}) \right\} \backslash E_{\eta}$ satisfies the following property:

\textit{$(\bar{z}, -1)$ can be connected to $(\bar{x}, 0)$ by a shortest smooth reduced geodesic $\boldsymbol{\beta}$. In other words, for every other smooth reduced geodesic $\boldsymbol{\gamma}$ with the same ends, we have
$\mathcal{L}(\boldsymbol{\gamma}) \geq \mathcal{L}(\boldsymbol{\beta})$.}

Now we define
\begin{align}
  E \triangleq \bigcup_{k\in \{1, 2, \cdots \}} E_{2^{-k}\eta_0} \backslash  \left\{ B(\bar{x}, 2^{k-2} \eta_0^{-1})  \cap \overset{\circ}{\mathcal{R}}_{2^{-k+2} \eta_0}\right\}. 
\label{eqn:MC25_1}  
\end{align}
The $\eta_0$ above is some fixed positive number.  
According to this definition,  every regular point locates in finitely many closed sets $E_{2^{-k}\eta_0} \backslash  \left\{ B(\bar{x}, 2^{k-2} \eta_0^{-1})  \cap \overset{\circ}{\mathcal{R}}_{2^{k-2} \eta_0}\right\}$. 
The reason we choose to define $E$ in the way of (\ref{eqn:MC25_1}) is to obtain the closedness of $E \cup \mathcal{S}$. Note that if we simply define $E$ to be the union of all $E_{2^{-k}\eta_0}$, then $E \cup \mathcal{S}$ may not be closed set.
It is possible to obtain points in $E_{2^{-k} \eta_0}$ converging to a regular point. However, from the discussion in the above paragraph, it is clear that for every regular point, one can find a small closed ball  regular neighborhood $\bar{B}$ where 
every point (with time $t=-1$) can be connected to $(\bar{x}, 0)$ away from a closed set $E_{\bar{B}}=E_{\eta} \cap \bar{B}$, where $\eta$ depends on $\bar{B}$.
Taking a countable, locally finite cover of $\mathcal{R}$ by such $\bar{B}$'s and let $E'$ be the union of such $E_{\bar{B}}$.
Then $E'$ is measure zero and relatively closed in $\mathcal{R}$.  The choice of $E$ in (\ref{eqn:MC25_1}) follows the same idea, with the covering of $\mathcal{R}$ being written down explicitly. 

It follows from (\ref{eqn:MC25_1}) that $E$ is the union of countably many measure-zero sets.
Consequently, $E$ is measure-zero. 
Fix arbitrary $\bar{z} \in \mathcal{R} \backslash E$. 
Because $\bar{z} \in \mathcal{R}$, we see that $\bar{z} \in B(\bar{x}, \eta^{-1}) \cap \mathcal{R}_{\eta}(\bar{M})$ for some $\eta>0$. 
Accordingly, we can find $k_0$ very large such that $\bar{z} \in B(\bar{x}, 2^{k_0}\eta_0^{-1}) \cap \mathcal{R}_{2^{-k_0}\eta_0}(\bar{M})$.  
Now using $\bar{z} \notin E$ and the decomposition of $E$ in (\ref{eqn:MA23_1}), we have
\begin{align*}
 \bar{z} \notin    E_{2^{-k_0}\eta_0} \backslash  \left\{ B(\bar{x}, 2^{k_0-2} \eta_0^{-1})  \cap \overset{\circ}{\mathcal{R}}_{2^{-k_0+2} \eta_0}\right\} \quad
 &\Leftrightarrow \quad \bar{z} \in   \left\{ B(\bar{x}, 2^{k_0-2} \eta_0^{-1})  \cap \overset{\circ}{\mathcal{R}}_{2^{-k_0+2} \eta_0}\right\} \backslash E_{2^{-k_0}\eta_0}, \\
 &\Rightarrow \quad \bar{z} \in \left\{ B(\bar{x}, 2^{k_0} \eta_0^{-1})  \cap \mathcal{R}_{2^{-k_0} \eta_0}\right\} \backslash E_{2^{-k_0}\eta_0}. 
\end{align*}
Then it follows from our discussion in the previous paragraph that $(\bar{z}, -1)$ can be connected to $(\bar{x}, 0)$ by a shortest smooth reduced geodesic in $\mathcal{R}(\bar{M}) \times [-1, 0]$.

It is not hard to see that $E \cup \mathcal{S}$ is a closed set, which will be proved in this paragraph. 
Suppose $z_i$ is a sequence of points in $E$. Without loss of genearlity, we can assume 
\begin{align*}
z_i \in E_{2^{-k}\eta_0} \backslash  \left\{ B(\bar{x}, 2^{k-2} \eta_0^{-1})  \cap \mathcal{R}_{2^{k-2} \eta_0}\right\},
\end{align*}
where $k=k(i)$. 
Let $z$ be a limit point of $z_i$.  There are two possibilities(by taking subsequence if necessary):
\begin{itemize}
\item $z \in \mathcal{S}$.
\item $z \in \mathcal{R}$. Then $z \in \mathcal{R}_{2\eta} \cap B(\bar{x}, 0.5 \eta^{-1})$ for some $\eta>0$.  Therefore, we can assume $z_i \in \mathcal{R}_{\eta} \cap B(\bar{x}, \eta^{-1})$ for large $i$. 
This forces that $k(i)$ is uniformly bounded.  By taking subsequence if necessary, we can assume that $z_i \in E_{2^{-k}\eta_0} \backslash  \left\{ B(\bar{x}, 2^{k-2} \eta_0^{-1})  \cap \overset{\circ}{\mathcal{R}}_{2^{k-2} \eta_0}\right\}$ for a fixed $k$.
By closedness of each $E_{\eta}$, we see that  $z \in E_{2^{-k}\eta_0} \backslash  \left\{ B(\bar{x}, 2^{k-2} \eta_0^{-1})  \cap \overset{\circ}{\mathcal{R}}_{2^{k-2} \eta_0}\right\} \subset E$. 
\end{itemize}
Therefore, we conclude that $z \in E \cup \mathcal{S}$.   Note that $\mathcal{S}$ is a closed set and has measuzre($2n$-Hausdorff measure) zero.
Then we obtain $E \cup \mathcal{S}$ is a closed measure-zero set.  

Clearly, away from the closed measure-zero set $E \cup \mathcal{S}$, every point $\bar{z} \in \bar{M}$ satisfies the following property:
\textit{$(\bar{z}, -1)$ can be connected to $(\bar{x}, 0)$ by a shortest smooth reduced geodesic.}

\end{proof}

\begin{remark}
Note that the devlopment from Lemma~\ref{lma:SL13_1} to Lemma~\ref{lma:SK27_4} is parallel, or independent to the development 
from Proposition~\ref{prn:SC16_1} to Prposition~\ref{prn:SC17_4}.
Our key observation is that the limit space has weakly convex regular part, which essentialy arises from
the weak convexity of $\mathcal{R} \times [-1, 0]$ in terms of reduced geodesics.  For the convenience of the readers who are not familiar with singular space theory, we also provide an alternative proof of Proposition~\ref{prn:SC17_4} 
in Appendix~\ref{app:C}, based on the discussion from Lemma~\ref{lma:SL13_1} to Lemma~\ref{lma:SK27_4}. 
Then Proposition~\ref{prn:SC17_4}  follows directly from Proposition~\ref{prn:MC04_1}.
\label{rmk:MC07_1}
\end{remark}

By natural projection to the time slice $t=0$, we obtain the following property.
\begin{proposition}[\textbf{Weak convexity by Riemannian geodesics}]
Same conditions as in Lemma~\ref{lma:SK27_4}.
Then away from a measure-zero set, every point in $\mathcal{R}$ can be connected to $\bar{x}$ with a unique smooth shortest geodesic. Consequently, $\mathcal{R}$ is weakly convex.
\label{prn:SC30_1}
\end{proposition}
\begin{proof} 
 Fix $\bar{x} \in \mathcal{R}$ and let $E$ be the measure-zero set constructed in the proof of Lemma~\ref{lma:SK27_4}. 
 Therefore, $(\bar{y}, -1)$ can be connected to $(\bar{x}, 0)$ by a smooth shortest reduced geodesic $\boldsymbol{\beta}$, with space projection curve $\beta$, whenever $\bar{y} \in \mathcal{R} \backslash E$.
 For our purpose of weak convexity, it suffices to show that each $\beta$ is a smooth shortest geodesic connecting $\bar{x}$ and $\bar{y}$. 
 Actually,  it follows from reduced geodesic equation on Ricci-flat manifold (c.f. equations (\ref{eqn:SL25_6})) that $\displaystyle  \mathcal{L}(\boldsymbol{\beta})=\frac{1}{2}|\beta|^2$, where $|\beta|$ is the length of $\beta$.     
 Since both $\bar{x}$ and $\bar{y}$ are regular, for each small $\epsilon>0$, we can find a smooth geodesic $\gamma$ such that $|\gamma|<d_0(\bar{x}, \bar{y})+\epsilon$, by Proposition~\ref{prn:SC17_4}. 
 Because the limit space-time is static, we can lift $\gamma$ to be a space-time curve $\boldsymbol{\gamma}$ such that $\mathcal{L}(\boldsymbol{\gamma})=\frac{|\gamma|^2}{2}$.
 Using the shortest property of $\boldsymbol{\beta}$ and the construction of $\gamma$, we have
      \begin{align*}
        \frac{|\beta|^2}{2}=\mathcal{L}(\boldsymbol{\beta}) \leq \mathcal{L}(\boldsymbol{\gamma})=\frac{|\gamma|^2}{2}< \frac{(d_0(\bar{x},\bar{y}) +\epsilon)^2}{2}, 
        \quad \Rightarrow \quad |\beta|<d_0(\bar{x}, \bar{y}) + \epsilon. 
      \end{align*} 
   Since $\epsilon$ can be chosen arbitrarily small,  we have $|\beta| \leq d_0(\bar{x}, \bar{y})$, which means $|\beta|=d_0(\bar{x}, \bar{y})$ and $\beta$ is a shortest Riemannian geodesic. 
   
   By adjusting $E$ to a bigger measure zero set $E'$ if necessary, we obtain the uniqueness of geodesics from $\bar{y}$ to $\bar{x}$ for each $\bar{y} \in \mathcal{R} \backslash E'$. 
   This follows from standard Riemannian geometry argument since $E' \backslash E \subset \mathcal{R}$. 
\end{proof}

By the correspondence between smooth Riemannian geodesic and smooth reduced geodesic(c.f. the discussion in Section~\ref{subsec:reduced}), it is clear (from the proof of Proposition~\ref{prn:SC30_1}) now that most smooth reduced geodesics obtained in Lemma~\ref{lma:SK27_4} are shortest among all smooth reduced geodesics.
Furthermore, the rough estimate in Lemma~\ref{lma:SL13_1} can be improved as the following proposition.

\begin{proposition}[\textbf{Continuity of reducecd distance}]
Same conditions as in Lemma~\ref{lma:SK27_4}.
Suppose $(y_i,t_i) \in \mathcal{M}_i$ converges to $(\bar{y}, \bar{t})$, which is regular and $\bar{t}<0$.
Then we have
\begin{align}
  \lim_{i \to \infty} l((x_i,0), (y_i,t_i))=\frac{d_0^2(\bar{x}, \bar{y})}{4|\bar{t}|}=l((\bar{x},0), (\bar{y},\bar{t}))
 \label{eqn:SK27_9}
\end{align}
where $l$ is Perelman's reduced distance. Therefore, reduced distance is continuous function under Cheeger-Gromov topology whenever $\bar{y}$ is regular.
\label{prn:SL14_1}
\end{proposition}

\begin{proof}

Without loss of generality, we assume $t_i \equiv -1$, $d_0(x_i,y_i) \equiv 1$.

We first show
\begin{align}
  \lim_{i \to \infty} l((x_i,0), (y_i,t_i)) \leq \frac{1}{4}.
  \label{eqn:SL14_5}
\end{align}
If $x_i$ are uniformly regular, then there is a limit smooth geodesic connecting $\bar{x}$ and $\bar{y}$, which can
be lifted to a smooth reduced geodesic connecting $(\bar{x},0)$ and $(\bar{y},-1)$ with reduced length $\frac{1}{4}$.
Then (\ref{eqn:SL14_5}) follows trivially.  So we focus on the case when $\bar{x}$ is a singular point.
Choose a smooth point $\bar{z}$ very close to $\bar{x}$, say $\delta$-away from $\bar{x}$ under metric $\bar{g}(0)$.
From Lemma~\ref{lma:SL13_1},  the reduced length from $(x_i, 0)$ to $(z_i, -\delta^2)$ is uniformly less than $100$.   So we have  space-time curves $\boldsymbol{\alpha}_i$ connecting these two points such that
\begin{align*}
   \int_0^{\delta^2} \sqrt{\tau} |\dot{\alpha}_i|^2 d\tau< 200\delta.
\end{align*}
Note that $(\bar{z},-\delta^2)$ and $(\bar{y}, -1)$ can be connected by a space-time curve $\boldsymbol{\beta}$ such that
\begin{align*}
       \int_{\delta^2}^{1} \sqrt{\tau} |\dot{\beta}|^2 d\tau< \frac{1}{2} + 100\delta
\end{align*}
if $\delta$ is small enough.
So for large $i$, we have space-time curve $\boldsymbol{\beta}_i$ connecting $(z_i,-\delta^2)$ and $(y_i, -1)$ such that
\begin{align*}
  \int_{\delta^2}^{1} \sqrt{\tau} |\dot{\beta_i}|^2 d\tau < \frac{1}{2} + 200\delta.
\end{align*}
Concatenating $\boldsymbol{\alpha}_i$ and $\boldsymbol{\beta}_i$ to obtain $\boldsymbol{\gamma}_i$ such that
\begin{align*}
  \int_{\delta^2}^{1} \sqrt{\tau} |\dot{\gamma_i}|^2 d\tau < \frac{1}{2} + 400\delta,
\end{align*}
which implies $l((x_i,0),(y_i,-1)) < \frac{1}{4} + 200\delta$ for large $i$.   Thus (\ref{eqn:SL14_5}) follows by letting $i \to \infty$ and $\delta \to 0$. 

Then we  show the equality holds. Otherwise, there exists a small $\epsilon$ such that
\begin{align*}
     \lim_{i \to \infty}  l((x_i,0), (y_i, -1))  < \frac{1}{4} -\epsilon.
\end{align*}
Note that $(y_i,-1)$ is uniformly regular. So we can find small $\delta$ such that
\begin{align*}
  l((x_i,0), (z, -1-\delta^2)) < \frac{1}{4}-\frac{1}{2}\epsilon, \quad \forall \;
  z \in B_{g(-1-\delta^2)}(y_i, \epsilon \delta).
\end{align*}
By Lemma~\ref{lma:SK27_4}, we obtain a point $(\bar{z}, -1-\delta^2)$, which can be connected to
$(\bar{x},0)$ by a smooth reduced geodesic, with reduced length smaller than $\frac{1}{4}-\frac{1}{2}\epsilon$.
Projecting this reduced geodesic to time zero slice, we obtain a curve connecting $\bar{x}$ and $\bar{z}$ with
\begin{align*}
 d_0^2(\bar{x}, \bar{y}) < 4(1+\delta^2) \cdot \left(\frac{1}{4}-\frac{1}{2}\epsilon \right)=(1+\delta^2)(1- 2\epsilon)<1-\epsilon
\end{align*}
if we choose $\delta$ sufficiently small.  This is impossible since $d_0(\bar{x},\bar{y})=1$.   Therefore, we have
\begin{align*}
   \lim_{i \to \infty}  l((x_i,0), (y_i, -1))  = \frac{1}{4}.
\end{align*}
 \end{proof}
Since singular set has measure zero, it is clear that
\begin{align}
   \mathcal{V}((\bar{x},0), |\bar{t}|) \leq \lim_{i \to \infty} \mathcal{V}((x_i,0),|\bar{t}|),
\label{eqn:SK27_10}
\end{align}
where the  ``$\lim$" of the right hand side of the above inequality should be understood as  ``$\limsup$". 
We shall improve the above inequality as equality.\\

\begin{lemma}[\textbf{Major part of reduced volume}]
For every positive $\eta$ and $H$, there exists an $\epsilon=\epsilon(n,A,\eta,H)$ with the following properties.

Suppose $\mathcal{LM} \in \mathscr{K}(n,A)$, $x \in \mathcal{F}_{\eta}(M,0)$.  Then we have
\begin{align}
   \left|\mathcal{V}((x,0),1)-(4\pi)^{-n}\int_{B_{g(0)}(x,H)} e^{-l}dv \right| \leq 2a(H),  \label{eqn:SL18_1}
\end{align}
whenever $\sup_{\mathcal{M}} (|R|+|\lambda|)<\epsilon$.
Here $a$ is a positive function defined as
\begin{align}
a(H) \triangleq (4\pi)^{-n} \int_{\{|\vec{w}|>\frac{H}{100}\} \subset \R^{2n}}  e^{-\frac{|\vec{w}|^2}{4}} dw. \label{eqn:SL18_2}
\end{align}
\label{lma:SL18_1}
\end{lemma}
\begin{proof}
The line bundle structure is not used in the following proof. So up to a parabolic rescaling if necessary, we can assume $\lambda=0$.

For every $y \in M$, there is at least one shortest reduced geodesic $\boldsymbol{\gamma}$ connecting $(x,0)$ and $(y,-1)$.
By standard ODE theory, the limit $\displaystyle \lim_{\tau \to 0} \sqrt{\tau} \gamma'(\tau)$ is unique as a vector in $T_{x} M$,
which is called the reduced tangent vector of $\boldsymbol{\gamma}$.
Away from a measure-zero set, every $(y,-1)$ can be connected to $(x,0)$ by a
unique shortest reduced geodesic. For simplicity for our argument, we may assume this measure-zero set is empty, since measure-zero
set does not affect integral at all.
So there is a natural injective map from $M$ to $T_xM$, by mapping $y$ to the corresponding
reduced tangent vector $\vec{w}$.
We define
\begin{align*}
  \Omega(H) \triangleq \{ y \in M | |\vec{w}|>H \}.
\end{align*}
It follows from the monotonicity of reduced element along reduced geodesic that
\begin{align*}
  \int_{\Omega(H)} (4\pi)^{-n} e^{-l} dv \leq \int_{\{|\vec{w}| > H\} \subset \R^{2n}} (4\pi)^{-n} e^{-\frac{|\vec{w}|^2}{4}} dw.
\end{align*}

Choose $\xi<\eta$, with size to be determined.
Suppose $\boldsymbol{\gamma}$ is a reduced geodesic connecting $(x,0)$ to $(y,-1)$ for some $y \in M$.
It is clear that $\gamma(0)$ is in the interior part of $\mathcal{F}_{\xi}(M,0)$. Let $\tau$ to be the first time such that $\gamma(\tau)$ touches
the boundary of $\mathcal{F}_{\xi}(M,0)$.  Then we see that $\boldsymbol{\gamma}([0,\tau])$ locates in a space-time domain
with uniformly bounded geometry, Ricci curvature very small.  In particular, the reduced distance between $(x,0)$ and $\boldsymbol{\gamma}(\tau)$
is comparable to the length of $\vec{w}$, which is the reduced tangent vector of $\boldsymbol{\gamma}$ at $(x,0)$.
If $|\vec{w}|<H$, then we see that
\begin{align*}
   \frac{H^2}{4} > \frac{|\vec{w}|^2}{4} \sim \frac{d_{g(0)}^2(x, \gamma(\tau))}{4\tau}> \frac{c_a^2 \eta^2}{100\tau},
   \quad \Rightarrow \quad \tau > \frac{c_a^2 \eta^2}{25H^2}.
\end{align*}
Note that $\boldsymbol{\gamma}([0,\tau])$ is in a space-time region where Ricci curvature is almost flat, geometry is uniformly bounded.
So the lower bound of $\tau$ and the upper bounded of $|\vec{w}|$ imply an upper bound of $d_{g(0)}(x,\gamma(\tau))$.
Say $d_{g(0)}(x,\gamma(\tau))<H'$.

Around $\boldsymbol{\gamma}$, there is a natural projection (induced by reduced geodesic) from the space-time hypersurface
$\partial \mathcal{F}_{\xi}(M,0) \times [-1, -\frac{c_a^2 \eta^2}{25H^2}]$,
to the time slice $M \times \{-1\}$. At point $\boldsymbol{\gamma}(\tau)$, $\boldsymbol{\gamma}$ has space-time tangent vector $(\gamma', -1)$,
with $\tau|\gamma'(\tau)|^2$ is almost less than $\frac{H^2}{4}$.  Together with the lower bound of $\tau$, we obtain an upper bound of $|\gamma'(\tau)|$.
Up to a constant depending on $H,\eta$, the volume element of $\partial \mathcal{F}_{\xi}(M,0) \times [-1,-\frac{c_a^2 \eta^2}{25H^2}]$
is comparable to the reduced volume element
$(4\pi \tau)^{-n} e^{-l}$ of $M$, around the point $\boldsymbol{\gamma}(\tau)$.
Note that the reduced volume element is monotone along each reduced geodesic. This implies that the projection map mentioned above
``almost" decreases weighted hypersurface volume element, if we equip $\{B(x,H') \cap \partial \mathcal{F}_{\xi}(M,0)\} \times [-1, -\frac{\eta^2}{4H^2}]$ with the natural
weighted volume element $e^{-l}|d \sigma \wedge dt|$.  Let $\Omega_{\xi}'$ be the collection of all $y$'s such that $(y,-1)$ cannot be connected to $(x,0)$ by a shortest reduced geodesic $\boldsymbol{\gamma}$ which locates completely in $\mathcal{F}_{\xi}(M,0) \times [-1,0]$.
Then we have
\begin{align*}
   \int_{\Omega_{\xi}'} e^{-l}(4\pi \tau)^{-n} dv &\leq C \int_{\frac{c_a^2 \eta^2}{25H^2}}^1 \int_{B(x,H') \cap \partial \mathcal{F}_{\xi}(M,0)} e^{-l} d\sigma d\tau
        \leq C \int_{B(x,H') \cap \partial \mathcal{F}_{\xi}(M,0)} d\sigma
        \leq C \xi^{2p_0-1}, 
\end{align*}
where $C=C(n,H,H',\eta)=C(n,H,\eta)$.  By choosing $\xi$ small enough, we have
\begin{align}
  \int_{\Omega_{\xi}'} e^{-l}(4\pi \tau)^{-n} dv \leq (4\pi)^n a(H).    \label{eqn:SL18_3}
\end{align}
Note that
\begin{align*}
  \Omega_{100 H} \cap B_{g(0)}(x,H) \subset \Omega_{\xi}', \qquad
  M \backslash (\Omega_{\xi}' \cup B_{g(0)}(x,H)) \subset   \Omega_{\frac{H}{100}}.
\end{align*}
Therefore, recalling the definition of reduced volume (\ref{eqn:MA22_4}),  we have
\begin{align*}
    &\quad (4\pi)^n \mathcal{V}((x,0),1)
     =\int_{M} e^{-l}dv
    =\int_{M \backslash (\Omega_{\xi}' \cup B_{g(0)}(x,H))} e^{-l}dv + \int_{\Omega_{\xi}'} e^{-l}dv  + \int_{B_{g(0)}(x,H) \backslash \Omega_{\xi}'} e^{-l}dv\\
    &\leq \int_{|\vec{w}|>\frac{H}{100}}  e^{-\frac{|\vec{w}|^2}{4}} dw + \int_{\Omega_{\xi}'} e^{-l}dv  + \int_{B_{g(0)}(x,H)} e^{-l}dv
     \leq \int_{|\vec{w}|>\frac{H}{100}}  e^{-\frac{|\vec{w}|^2}{4}} dw + C \xi^{2p_0-1} + \int_{B_{g(0)}(x,H)} e^{-l}dv\\
    &\leq 2 (4\pi)^n a(H) + \int_{B_{g(0)}(x,H)} e^{-l}dv.
\end{align*}
Then (\ref{eqn:SL18_1}) follows from the above inequality directly.
\end{proof}

Lemma~\ref{lma:SL18_1} is related to Corollary 6.82 of~\cite{MT}. 

\begin{lemma}[\textbf{Uniform continuity of reduced volume}]
Suppose $\mathcal{M}=\{(M, g(t)), -\tau \leq t \leq 0\}$ is an unnormalized K\"ahler Ricci flow solution.
Suppose $x,y$ are two points in $M$, $d=d_{g(0)}(x,y)$. Then we have
\begin{align}
   |\mathcal{V}((x,0),\tau) - \mathcal{V}((y,0),\tau)|<(4n+1) (e^{\frac{d}{2}}-1).
\label{eqn:SL15_2}
\end{align}
In particular, the reduced volume changes uniformly continuously with respect to the base point.
\label{lma:SL14_2}
\end{lemma}
\begin{proof}
Recall the definition of reduced volume (\ref{eqn:MA22_4}):
\begin{align*}
   \mathcal{V}((x,0),\tau)= (4\pi \tau)^{-n} \int_M e^{-l} dv.
\end{align*}
Let $x$ move along a unit speed Riemannian geodesic $\alpha$, with respect to the metric  $g(0)$.
Let $x=\alpha(0)$, $s$ be parameter of $\alpha$, $\vec{u}=\alpha'$.
For simplicity of notation, we denote $\mathcal{V}((\alpha(s),0),\tau)$ by $\mathcal{V}_s$.
It can be calculated directly the first variation of $l$ is $\langle\vec{u}, \vec{w} \rangle$ where $\vec{w}$ is the tangent vector of the reduced geodesic at time $t=0$.
Therefore, we have
\begin{align*}
  \left|\frac{d}{ds} \mathcal{V}((\alpha(s),0),\tau) \right|&=\left| (4\pi \tau)^{-n} \int_M \langle \vec{u}, \vec{w} \rangle e^{-l}dv \right|
    \leq  (4\pi \tau)^{-n} \int_M  \frac{1+|\vec{w}|^2}{2} e^{-l}dv\\
     &=\frac{1}{2}\mathcal{V} + \frac{1}{2} \int_{\R^{2n}} |\vec{w}|^2e^{-\frac{|\vec{w}|^2}{4}} J dw
     \leq \frac{1}{2}\mathcal{V} + \frac{(4\pi)^{-n}}{2} \int_{\R^{2n}} |\vec{w}|^2e^{-\frac{|\vec{w}|^2}{4}}dw,
\end{align*}
where $J$ is the Jacobian determinant of the reduced exponential map, which is always not greater than $1$, due to Perelman's argument in Section 7 of \cite{Pe1}.
Plugging the identity
\begin{align*}
(4\pi)^{-n}  \int_{\R^{2n}} |\vec{w}|^2 e^{-\frac{|\vec{w}|^2}{4}} dw = 4n
\end{align*}
into the above inequality implies $ \left|\frac{d}{ds} \mathcal{V} \right| \leq \frac{1}{2}\mathcal{V} +2n$, 
which can be integrated as
\begin{align*}
      (-\mathcal{V}_0+4n) (1-e^{-\frac{s}{2}})
       \leq  \mathcal{V}_s -\mathcal{V}_0
       \leq (\mathcal{V}_0 + 4n) (e^{\frac{s}{2}}-1).
\end{align*}
Note that $0<\mathcal{V}_0\leq 1, s>0$. So we obtain
\begin{align*}
   |\mathcal{V}_s -\mathcal{V}_0| \leq (4n+1) (e^{\frac{s}{2}}-1),
\end{align*}
which yields (\ref{eqn:SL15_2}) by letting $s=d$.
\end{proof}

The above argument clearly works for every Riemannian Ricci flow.

Note that the reduced volume is continuous for geodesic balls of each fixed scale under the Cheeger-Gromov convergence.
Combining this continuity together with the estimate in Lemma~\ref{lma:SL18_1} and Lemma~\ref{lma:SL14_2},
we can improve (\ref{eqn:SK27_10}) as an equality.

\begin{proposition}[\textbf{Continuity of reduced volume}]
Same conditions as in Lemma~\ref{lma:SK27_4}, $\bar{t}<0$ is a finite number.
Then we have
\begin{align}
   \mathcal{V}((\bar{x},0), |\bar{t}|) = \lim_{i \to \infty} \mathcal{V}((x_i,0),|\bar{t}|).
\label{eqn:SL09_1}
\end{align}
Therefore, reduced volume is a continuous function under the Cheeger-Gromov convergence.
\label{prn:SL14_2}
\end{proposition}

Then we can study the gap property of the singularities.

\begin{proposition}[\textbf{Gap of local volume density}]
Same conditions as in Theorem~\ref{thm:SC09_1}.

Suppose $\bar{y} \in \mathcal{S}(\bar{M})$, then we have
\begin{align}
  \mathrm{v}(\bar{y})=\lim_{r \to 0} \omega_{2n}^{-1} r^{-2n}|B(\bar{y},r)| \leq 1-2\delta_0.
  \label{eqn:SC30_4}
\end{align}
\label{prn:SC17_7}
\end{proposition}

\begin{proof}
  Due to the tangent cone structure(c.f. Theorem~\ref{thm:SL25_1}), we have
  \begin{align}
      \mathrm{v}(\bar{y})=\lim_{r \to 0} \omega_{2n}^{-1} r^{-2n}|B(\bar{y},r)|=\lim_{r \to 0} \mathcal{V}((\bar{y},0), r^2).
  \label{eqn:SL14_6}
  \end{align}
  Let $y_i \to \bar{y}$ under the metric $g_i(0)$.  By rearranging points and taking subsequences if necessary, we can assume $y_i$ has the 
  ``local minimum" canonical volume radius $\rho_i$.  
  
  The rearrangement is a standard point-picking technique.
  In fact, since $\bar{y}$ is a singular point, it is clear that $r_i=\mathbf{cvr}(y_i, 0) \to 0$. 
  Since everything is done at time slice $t=0$, we shall drop the time in the following argument. 
  Fix $L \geq 1$ and $i$, we search if $y_i$ is the point such that
  \begin{align*}
     \mathbf{cvr}(y) <0.5 \mathbf{cvr}(y_i), \quad \forall \; y \in B(y_i, Lr_i).
  \end{align*}
  If so, we stop. Otherwise, we can find a point $z \in B(y_i, Lr_i)$ such that $\mathbf{cvr}(z)<0.5  \mathbf{cvr}(y_i)$. 
  Denote such $z$ by $y_i^{(1)}$ and set $r_i^{(1)}=\mathbf{cvr} \left(y_i^{(1)} \right)$.   We then repeat the previous process for $y_i^{(1)}$ and $r_i^{(1)}$. 
  To search points in the ball $B\left(y_i^{(1)}, Lr_i^{(1)} \right)$ with $\mathbf{cvr}<0.5 r_i^{(1)}$. If no such points exist, we stop. Otherwise, we find such a point
  and denote it by $y_i^{(2)}$ and set $r_i^{(2)}=\mathbf{cvr} \left(y_i^{(2)} \right)$. Note this process happens in a compact set since 
  \begin{align*}
    d\left( y_i^{(k)}, y_i\right)<L\left( r_i + r_i^{(1)} + \cdots + r_i^{(k)} \right)<2Lr_i. 
  \end{align*} 
  Each $\mathcal{LM}_i$ is smooth. Therefore, the process above must stop at some finite step $k$.  Denote $z_i=y_i^{(k)}$ and $\rho_i=\mathbf{cvr}(z_i)$.
  Then we have
  \begin{align*}
     \mathbf{cvr}(y)>0.5  \rho_i,  \quad \forall \; y \in B(z_i, L\rho_i). 
  \end{align*}
  Note that $L\rho_i \to 0$ as $i \to \infty$. Therefore,  the limit of $z_i$ and the limit of $y_i$ are the same point $\bar{y}$. 
  Then we let $L \to \infty$ and take diagonal sequence, we obtain $z_i$ such that 
   \begin{align*}
     \mathbf{cvr}(y)>0.5  \rho_i,  \quad \forall \; y \in B(z_i, 2^i \rho_i);  \quad \quad \lim_{i \to \infty} z_i=\bar{y}.  
  \end{align*}
  Therefore, we can regard $z_i$ as the rearrangement of $y_i$, with the property that each $z_i$ achieve the ``local minimum" of $\mathbf{cvr}$.

  By rescaling $\rho_i$ to $1$, we obtain new Ricci flows $\tilde{g}_i$.
  Taking limit of $(M_i,y_i,\tilde{g}_i(0))$, we have  a complete, Ricci flat
  eternal Ricci flow solution. 
  It is not hard to see the limit space is not Euclidean. For otherwise, each geodesic ball's volume ratio, under metric $\tilde{g}_{\infty}(0)$, is exactly the Euclidean volume ratio $\omega_{2n}$.
  Following from the volume convergence and the definition of the canonical volume radius, it is clear that the canonical volume radius of the rescaled flow is strictly greater than $1$ which contradicts to our assumption. 
  So it has normalized asymptotic volume ratio less than $1-2\delta_0$, according to Anderson's gap theorem. Then the infinity tangent cone
  structure implies the asymptotic reduced volume is the same as the asymptotic reduced volume ratio. So it is
  at most  $1-2\delta_0$.    Therefore,  there exists a big constant $H$ such that
  \begin{align*}
    \mathcal{V}_{\tilde{g}_i}((y_i,0),H) < 1-2\delta_0.
  \end{align*}
  Note that $H\rho_i^2<r$ for each fixed $r$ and the corresponding large $i$.  Recall the scaling invariant property of reduced volume, we can apply the reduced volume monotonicity to obtain
 \begin{align*}
     \mathcal{V}_{g_i}((y_i,0),r^2) \leq \mathcal{V}_{\tilde{g}_i}((y_i,0), \rho_i^{-2} r^2)
     \leq  \mathcal{V}_{\tilde{g}_i}((y_i,0),H) < 1-2\delta_0.
 \end{align*}
 The continuity of reduced volume (Proposition~\ref{prn:SL14_2}) then implies that
 \begin{align*}
   \mathcal{V}((\bar{y},0),r^2) \leq 1-2\delta_0
 \end{align*}
 for each $r>0$, which in turn yields
 \begin{align}
     \lim_{r \to 0} \mathcal{V}((\bar{y},0),r^2) \leq 1-2\delta_0.
 \label{eqn:SL14_7}
 \end{align}
 Then (\ref{eqn:SC30_4}) follows from the combination of (\ref{eqn:SL14_6}) and (\ref{eqn:SL14_7}).
\end{proof}

\begin{theorem}[\textbf{Metric structure of a blowup limit}]
Suppose $\mathcal{LM}_i \in \mathscr{K}(n,A;1)$ satisfies (\ref{eqn:SL06_1}), $x_i \in M_i$.
Let $(\bar{M}, \bar{x}, \bar{g})$ be the limit space of $(M_i, x_i, g_i(0))$. Then  $\bar{M} \in \widetilde{\mathscr{KS}}(n,\kappa)$.
\label{thm:SC04_1}
\end{theorem}

\begin{proof}
  We only need to check $\bar{M}$ satisfies all the 6 properties required in the definition of $\widetilde{\mathscr{KS}}(n,\kappa)$. In fact,
  the 1st property is implied by Theorem~\ref{thm:HE11_1}.
  The 2nd property follows from the fact that $\mathcal{R}$ is scalar flat and satisfies K\"ahler Ricci flow equation.
  The 3rd property, weak convexity of $\mathcal{R}$ is shown in Proposition~\ref{prn:SC30_1}.
  The 4th property, codimension estimate of singularity follows from Proposition~\ref{prn:SL15_1}.
  The 5th property, gap estimate, follows from Proposition~\ref{prn:SC17_7}.
  The 6th property, asymptotic volume ratio estimate can be obtained by the condition $Vol(M_i) \to \infty$, Sobolev constant uniformly bounded, and the volume convergence, Proposition~\ref{prn:SC12_1}.
  So we have checked all the properties needed to define $\widetilde{\mathscr{KS}}(n,\kappa)$ are satisfied by $\bar{M}$.
  In other words, $\bar{M} \in \widetilde{\mathscr{KS}}(n,\kappa)$.
\end{proof}

 Since $\bar{M} \in \widetilde{\mathscr{KS}}(n,\kappa)$, it is clear that $\mathbf{cr}(\bar{x})=\infty$.  Therefore, we have $\mathbf{vr}(\bar{x})=\mathbf{cvr}(\bar{x})$
 by definition.

\begin{proposition}
Same conditions as in Theorem~\ref{thm:SC04_1}.  Let $\displaystyle \bar{r}=\lim_{i \to \infty} \mathbf{cr}(x_i)$.
Then we have
\begin{align}
  \min\{\bar{r}, \mathbf{vr}(\bar{x})\} = \lim_{i \to \infty} \mathbf{cvr}(x_i).
\label{eqn:SC06_9}
\end{align}
\label{prn:SC06_5}
\end{proposition}

\begin{proof}
 We divide the proof in three cases according to the value of $\min\{\bar{r}, \mathbf{vr}(\bar{x})\}$. \\

 \textit{Case 1. $\min\{\bar{r}, \mathbf{vr}(\bar{x})\}=0$.}

 Otherwise, there exists a positive number $\rho_0$ such that $\displaystyle \lim_{i \to \infty} \mathbf{cvr}(x_i) \geq \rho_0$. Therefore,
 $\bar{x} \in \mathcal{R}_{\rho_0} \subset \mathcal{R}$, which in turn implies that $\mathbf{vr}(\bar{x})>0$.
 Consequently,  we have $\min\{\bar{r}, \mathbf{vr}(\bar{x})\}>0$. Contradiction.

 \textit{Case 2. $\min\{\bar{r}, \mathbf{vr}(\bar{x})\}=\infty$.}

 In this case, $\bar{r}=\infty$.
 By the gap theorem in the space $\widetilde{\mathscr{KS}}(n,\kappa)$, we see that $\bar{M}$ is the Euclidean space $\C^{n}$.
 Therefore, for each $H>0$, we have $\omega_{2n}^{-1}H^{-2n}|B(x_i,H)|$ converges to $1$, the normalized volume ratio of $\C^{n}$.
 Since $\displaystyle \bar{r}=\lim_{i \to \infty} \mathbf{cr}(x_i)=\infty$, this means that $\mathbf{cvr}(x_i) \geq H$ for large $i$ by the volume convergence.
 Since $H$ is chosen arbitrarily, we obtain $\displaystyle \lim_{i \to \infty} \mathbf{cvr}(x_i)=\infty$.\\

 So the remainder case is that $\min\{\bar{r}, \mathbf{vr}(\bar{x})\}$ is a finite positive number.
 Two more subcases can be divided.\\

 \textit{Case 3(a). $\min\{\bar{r}, \mathbf{vr}(\bar{x})\}<\bar{r}$.}

 Let $H=\mathbf{vr}(\bar{x})$, a finite number in this case.   Clearly, $\bar{x}$ is a regular point and the normalized volume ratio of the ball $B(\bar{x},H)$
 is $1-\delta_0$. Clearly, $B(\bar{x},H)$ cannot be a isometric to a Euclidean ball.
 Therefore, by the rigidity of $\widetilde{\mathscr{KS}}(n,\kappa)$(c.f. Proposition~\ref{prn:SC17_3}), we see that
 \begin{align*}
   &\omega_{2n}^{-1}r^{-2n}|B(\bar{x},r)|>1-\delta_0, \quad \forall \; r \in (0,H), \\
   &\omega_{2n}^{-1}r^{-2n}|B(\bar{x},r)|<1-\delta_0, \quad \forall \; r \in (H,\bar{r}).
 \end{align*}
 Then the volume convergence implies that $\displaystyle \lim_{i \to \infty} \mathbf{cvr}(x_i)=H$. \\

 \textit{Case 3(b). $\min\{\bar{r}, \mathbf{vr}(\bar{x})\}=\bar{r}$.}

 In this case, we see that the normalized volume ratio of $B(\bar{x},\bar{r})$ is at least $1-\delta_0$. Also, we see that $\bar{x}$  is a regular point.
 Same argument as in the previous case, we see that
 \begin{align*}
   \omega_{2n}^{-1}r^{-2n}|B(\bar{x},r)|>1-\delta_0, \quad \forall \; r \in (0,\bar{r}).
 \end{align*}
 Therefore, for every fixed $r \in (0, \bar{r})$, the volume convergence implies that $\displaystyle \lim_{i \to \infty} \mathbf{cvr}(x_i) \geq r$. Consequently,
 we have $\displaystyle \lim_{i \to \infty} \mathbf{cvr}(x_i) \geq \bar{r}$ by the arbitrariness of $r$.  On the other hand, the definition of $\mathbf{cvr}(x_i)$
 implies that
 \begin{align*}
      \lim_{i \to \infty} \mathbf{cvr}(x_i) \leq \lim_{i \to \infty} \mathbf{cr}(x_i)=\bar{r}.
 \end{align*}
 Therefore, we obtain $\displaystyle \lim_{i \to \infty} \mathbf{cvr}(x_i) =\bar{r}$.
\end{proof}

\begin{corollary}
Same conditions as in Theorem~\ref{thm:SC04_1}. Then for each $r \in (0,1)$, we have
\begin{align}
    \mathcal{F}_{r}(\bar{M})=\mathcal{R}_{r}(\bar{M}).
\label{eqn:SC06_11}
\end{align}
In particular, for each $0<r<1<H<\infty$, we have
\begin{align*}
  B(x_i,H) \cap \mathcal{F}_{r}(M_i)  \longright{G.H.} B(\bar{x},H) \cap \mathcal{F}_{r}(\bar{M}).
\end{align*}
Moreover, this convergence can be improved to take place in $C^{\infty}$-topology, i.e.,
\begin{align}
  B(x_i,H) \cap \mathcal{F}_{r}(M_i)  \longright{C^{\infty}} B(\bar{x},H) \cap \mathcal{F}_{r}(\bar{M}). \label{eqn:SC06_12}
\end{align}
\label{cly:SC06_1}
\end{corollary}

\begin{corollary}
Same conditions as in Theorem~\ref{thm:SC04_1}, $0<H\leq 3$.
Then we have
\begin{align}
  \lim_{i \to \infty} \int_{B(x_i,H)} \mathbf{vr}^{(1)}(y)^{-2p_0}dy \leq  H^{2n-2p_0}\mathbf{E}.
 \label{eqn:SC06_13}
\end{align}
\label{cly:SC06_2}
\end{corollary}

\begin{proof}
 Fix two positive scales $r_1,r_2$ such that $0<r_2<r_1<1$.
 \begin{align}
    \int_{B(x_i,H) \cap \mathcal{F}_{r_1}} \mathbf{vr}^{(1)}(y)^{-2p_0} dy\leq r_1^{-2p_0} |B(x_i,H) \cap \mathcal{F}_{r_1}| \leq r_1^{-2p_0}|B(x_i,H)|.
 \label{eqn:SC17_6}
 \end{align}
 Fix arbitrary $r \in (0,1)$, then we have
 \begin{align*}
    \lim_{i \to \infty} \int_{B(x_i,H) \cap (\mathcal{F}_{r_2} \backslash \mathcal{F}_{r_1})} \mathbf{vr}^{(1)}(y)^{-2p_0}  dy
    =\int_{B(\bar{x},H) \cap (\mathcal{F}_{r_2} \cap \mathcal{F}_{r_1})} \mathbf{vr}^{(1)}(y)^{-2p_0}  dy.
 \end{align*}
 Note that
 \begin{align*}
    &\qquad \int_{B(\bar{x},H) \cap (\mathcal{F}_{r_2} \cap \mathcal{F}_{r_1})} \mathbf{vr}^{(1)}(y)^{-2p_0}  dy\\
    &\leq \int_{B(\bar{x},H) \cap (\mathcal{F}_{r_2} \cap \mathcal{F}_{r_1})} \min\{\mathbf{vr}, 1\}^{-2p_0}  dy
    <\int_{B(\bar{x},H) \cap (\mathcal{F}_{r_2} \cap \mathcal{F}_{r_1})} \left\{1+\mathbf{vr}(y)^{-2p_0} \right\} dy\\
    &<\int_{B(\bar{x},H)} \left\{1+\mathbf{vr}(y)^{-2p_0} \right\} dy
      <|B(\bar{x},H)| + H^{2n-2p_0}E(n,\kappa,p_0).
 \end{align*}
 It follows that
 \begin{align}
    \lim_{i \to \infty} \int_{B(x_i,H) \cap (\mathcal{F}_{r_2} \backslash \mathcal{F}_{r_1})} \mathbf{vr}^{(1)}(y)^{-2p_0}  dy
    \leq |B(\bar{x},H)| + H^{2n-2p_0}E(n,\kappa,p_0).
  \label{eqn:SC17_7}
 \end{align}
 Note that $\mathcal{S} \cap \overline{B(\bar{x},H)}$ is a compact set with Hausdorff dimension at most
 $2n-4$, which is strictly less than $2n-2p_0$. By the definition of Hausdorff dimension, for every small number $\xi$,
 we can find finite cover
 $\cup_{j=1}^{N_{\xi}} B(\bar{y}_j, \rho_j)$  of $\mathcal{S} \cap \overline{B(\bar{x},H)}$ , such that $\displaystyle  \sum_{j=1}^{N_{\xi}} |\rho_{j}|^{2n-2p_0} < \xi$. 
 By the finiteness of this cover, we can choose an $r_2$ very small such that $\cup_{j=1}^{N_{\xi}} B(\bar{y}_j, \rho_j)$
 is a cover of $\mathcal{D}_{r_2} \cap \overline{B(\bar{x},H)}$.
 Therefore, for large $i$, we have a finite cover $\cup_{j=1}^{N_{\xi}} B(y_{i,j}, \rho_j)$ of the set
 $\mathcal{D}_{r_2}(M_i) \cap \overline{B(x_i,H)}$ such that $\displaystyle \sum_{j=1}^{N_i} |\rho_{i,j}|^{2n-2p_0} < \xi$. 
 Combining this with the canonical radius density estimate, we have
 \begin{align}
    \int_{B(x_i,H) \cap \mathcal{D}_{r_2}} \mathbf{vr}^{(1)}(y)^{-2p_0}dy
    \leq \sum_{j=1}^{N_i} \int_{B(y_{i,j},\rho_{i,j})} \mathbf{vr}^{(\rho_{i,j})}(y)^{-2p_0} dy
    \leq 2\mathbf{E} \sum_{j=1}^{N_i} |\rho_{i,j}|^{2n-2p_0} < 2\mathbf{E}\xi.
 \label{eqn:SC17_8}
 \end{align}
 Putting (\ref{eqn:SC17_6}), (\ref{eqn:SC17_7}) and (\ref{eqn:SC17_8}) together, we have
 \begin{align*}
   &\qquad \int_{B(x_i,H)} \mathbf{vr}^{(1)}(y)^{-2p_0} dy\\
   &\leq \int_{B(x_i,H) \cap \mathcal{F}_{r_1}} \mathbf{vr}^{(1)}(y)^{-2p_0}dy
    +\int_{B(x_i,H) \cap (\mathcal{F}_{r_2} \backslash \mathcal{F}_{r_1})} \mathbf{vr}^{(1)}(y)^{-2p_0}dy
    +\int_{B(x_i,H) \cap \mathcal{D}_{r_2}} \mathbf{vr}^{(1)}(y)^{-2p_0} dy\\
   &\leq r_1^{-2p_0}|B(x_i,H)|+ |B(\bar{x},H)| +H^{2n-2p_0} E(n,\kappa,p_0) +2\mathbf{E}\xi.
 \end{align*}
 Taking limit on both sides and then letting $\xi \to 0, r_1 \to 1$, we have
 \begin{align*}
  \lim_{i \to \infty} \int_{B(x_i,H)} \mathbf{vr}^{(1)}(y)^{-2p_0}  &\leq 2|B(\bar{x},H)| + H^{2n-2p_0} E(n,\kappa,p_0)
    \leq \left( 2 \omega_{2n} H^{2p_0} + E(n,\kappa,p_0) \right) H^{2n-2p_0}\\
    &\leq  \left( 2 \cdot 9^{p_0}\omega_{2n} +  E(n,\kappa,p_0) \right) H^{2n-2p_0},
 \end{align*}
 where we used the fact that $H\leq 3$ in the last step. Then (\ref{eqn:SC06_13}) follows from the definition of $\mathbf{E}$.
\end{proof}

\begin{proposition}
Same conditions as in Theorem~\ref{thm:SC04_1}, $1 \leq H <\infty$. Then we have
 \begin{align}
    \lim_{i \to \infty} \sup_{1 \leq \rho \leq H} \omega_{2n}^{-1}\rho^{-2n}|B(x_i,\rho)| <\kappa^{-1}, \label{eqn:SC06_4}
 \end{align}
where $g_i(0)$ is the default metric.   In particular, for every large $i$, the volume ratio estimate holds
on $(M_i,x_i,g_i(0))$ for every scale $\rho \in (0,H]$.
\label{prn:SC06_1}
\end{proposition}
\begin{proof}
  We argue by contradiction. If (\ref{eqn:SC06_4}) were false, by taking subsequence if necessary, one can assume that there exists $\rho_i \in [1,H]$ such that
  $ \omega_{2n}^{-1}\rho_i^{-2n}|B(x_i,\rho_i)| >\kappa^{-1}$. 
  Let $\bar{\rho}$ be the limit of $\rho_i$, then by the volume continuity in the Cheeger-Gromov convergence, we see that
  \begin{align}
     \omega_{2n}^{-1}\bar{\rho}^{-2n}|B(\bar{x},\bar{\rho})| \geq \kappa^{-1}.
  \label{eqn:SC06_5}
  \end{align}
  However, since $\bar{M} \in \widetilde{\mathscr{KS}}(n,\kappa)$, we know $ \omega_{2n}^{-1}\bar{\rho}^{-2n}|B(\bar{x},\bar{\rho})| \leq 1$,
  which contradicts (\ref{eqn:SC06_5}).
\end{proof}

\begin{proposition}
Same conditions as in Theorem~\ref{thm:SC04_1}, $0<H<\infty$. For every large $i$, the regularity estimate holds
on $(M_i,x_i,g_i(0))$ for every scale $\rho \in (0,H]$.
\label{prn:SC06_2}
\end{proposition}
\begin{proof}
  If the statement were false, then by taking subsequence if necessary, we can assume there exists $\rho_i \in (0,H]$ such that the regularity estimates fail on
  the scale $\rho_i$, i.e., the following two inequalities hold simultaneously.
  \begin{align}
     &\omega_{2n}^{-1}\rho_i^{-2n}|B(x_i,\rho_i)|>1-\delta_0,  \label{eqn:SC06_6}\\
     &\max_{0 \leq k \leq 5} \left\{\rho_i^{2+k} \sup_{B(x_i,\frac{1}{2}c_a \rho_i)} |\nabla^k Rm| \right\} > 4c_a^{-2}.   \label{eqn:SC06_7}
  \end{align}
  Clearly, $\rho_i \in [1,H]$ by the fact $\mathbf{cr}(x_i,0) \geq 1$. Let $\bar{\rho}$ be the limit of $\rho_i$. Then we have
\begin{align}
  \omega_{2n}^{-1}\bar{\rho}^{-2n}|B(\bar{x},\bar{\rho})| \geq 1-\delta_0. \label{eqn:SC06_8}
\end{align}
 Since $\bar{M} \in \widetilde{\mathscr{KS}}(n,\kappa)$, (\ref{eqn:SC06_8}) implies
\begin{align*}
 \max_{0 \leq k \leq 5} \left\{\bar{\rho}^{2+k} \sup_{B(\bar{x}, c_a \bar{\rho})} |\nabla^k Rm| \right\} < c_a^{-2},
\end{align*}
which contradicts (\ref{eqn:SC06_7}) in light of the smooth convergence(c.f. Proposition~\ref{prn:SL04_1}).
\end{proof}

\begin{proposition}
Same conditions as in Theorem~\ref{thm:SC04_1}, $1 \leq H \leq 2$. Then we have
\begin{align}
    \lim_{i \to \infty} \sup_{1 \leq \rho \leq H} \rho^{2p_0-2n} \int_{B(x_i,\rho)} \mathbf{vr}^{(\rho)}(y)^{-2p_0}dy \leq \frac{3}{2}\mathbf{E}. \label{eqn:SC06_16}
\end{align}
In particular, for every large $i$, the density estimate holds on $(M_i,x_i,g_i(0))$ for every scale $\rho \in (0,H]$.
\label{prn:SC06_3}
\end{proposition}

\begin{proof}
Since $\mathbf{vr}^{(\rho)} \geq \mathbf{vr}^{(1)}$ whenever $\rho \geq 1$, in order to show (\ref{eqn:SC06_16}), it suffices to show
\begin{align}
    \lim_{i \to \infty} \sup_{1 \leq \rho
    \leq H} \rho^{2p_0-2n} \int_{B(x_i,\rho)} \mathbf{vr}^{(1)}(y)^{-2p_0}dy \leq \frac{3}{2}\mathbf{E}.
    \label{eqn:SL23_9}
\end{align}
We argue by contradiction. If (\ref{eqn:SL23_9}) were false, by taking subsequence if necessary, one can assume that there exists $\rho_i \in [1,H]$ such that
  \begin{align*}
     \rho_i^{2p_0-2n} \int_{B(x_i,\rho_i)} \mathbf{vr}^{(1)}(y)^{-2p_0}dy > \frac{3}{2}\mathbf{E}, \quad
     \Rightarrow \quad \int_{B(x_i,\rho_i)} \mathbf{vr}^{(1)}(y)^{-2p_0}dy \geq  \frac{3}{2}\mathbf{E} \rho_i^{2n-2p_0}.
  \end{align*}
  Let $\bar{\rho}$ be the limit of $\rho_i$. Fix $\epsilon$ arbitrary small positive number, then we have
  \begin{align}
     \int_{B(x_i,\bar{\rho}+\epsilon)} \mathbf{vr}^{(1)}(y)^{-2p_0}dy> \frac{3}{2}\mathbf{E} \rho_i^{2n-2p_0}> \frac{5}{4} \mathbf{E} (\bar{\rho}+\epsilon)^{2n-2p_0}
  \label{eqn:SC07_1}
  \end{align}
  for large $i$.  Note that $\bar{\rho}+\epsilon<3$, so (\ref{eqn:SC07_1}) contradicts (\ref{eqn:SC06_13}).
\end{proof}

\begin{proposition}
Same conditions as in Theorem~\ref{thm:SC04_1}, $0<H\leq 2$.
Then for every large $i$, the connectivity estimate holds on $(M_i,x_i,g_i(0))$ for every scale $\rho \in (0,H]$.
\label{prn:SC06_4}
\end{proposition}
\begin{proof}
  By the canonical radius assumption, we know the connectivity estimate holds for every scale $\rho \in (0,1]$.

  If the statement were false, then by taking subsequence if necessary, we can assume that for each $i$,
  there is a scale $\rho_i \in [1,H]$ such that the connectivity estimate fails on the scale $\rho_i$.
  In other words, $\mathcal{F}_{\frac{1}{50}c_b\rho_i} \cap B(x_i,\rho_i)$ is not $\frac{1}{2}\epsilon_b \rho_i$-regular-connected.
  So there exist points $y_i,z_i \in \mathcal{F}_{\frac{1}{50}c_b\rho_i} \cap B(x_i,\rho_i)$ which cannot be connected by a curve
  $\gamma \subset \mathcal{F}_{\frac{1}{2}\epsilon_b \rho_i}$ satisfying $|\gamma|\leq 2d(y_i,z_i)$.
  By the canonical radius assumption, it is clear that $\rho_i \in [1,H]$, $d(y_i,z_i) \in [1,2H]$.
  Let $\bar{\rho}$ be the limit of $\rho_i$, $\bar{y}$ and $\bar{z}$ be the limit of $y_i$ and $z_i$ respectively.
  Clearly, we have $\bar{y}, \bar{z} \in \mathcal{R}_{\frac{1}{50}c_b \bar{\rho}} \subset \mathcal{F}_{\frac{1}{100}c_b \bar{\rho}}(\bar{M})$.
  Since $\bar{M} \in \widetilde{\mathscr{KS}}(n,\kappa)$, we can find a shortest geodesic
  $\bar{\gamma}$ connecting $\bar{y}$ and $\bar{z}$ such that $\bar{\gamma} \subset \mathcal{F}_{\epsilon_b \bar{\rho}}$.
  Note that  the limit set of  $\mathcal{F}_{\frac{1}{50}c_b \rho_i}\cap B(x_i,\rho_i)$ falls into $\mathcal{F}_{\frac{1}{100}c_b \bar{\rho}}$.
  Moreover, this convergence takes place in the smooth topology(c.f.~Corollary~\ref{cly:SC06_1}).
  So by deforming $\bar{\gamma}$ if necessary, we can construct a curve $\gamma_i$ which locates in $\mathcal{F}_{\frac{1}{2}\epsilon_b \rho_i}$
  and $|\gamma_i| <\frac{3}{2}d(\bar{y},\bar{z})<3d(y_i,z_i)$.  The existence of such a curve contradicts the choice of the points
  $y_i$ and $z_i$.
\end{proof}

Combining Proposition~\ref{prn:SC06_1} to~\ref{prn:SC06_4}, we obtain a weak-semi-continuity of canonical radius.

\begin{theorem}[\textbf{Weak continuity of canonical radius}]
Same conditions as in Theorem~\ref{thm:SC04_1}. Then we have
$\displaystyle  \lim_{i \to \infty} \mathbf{cr}(\mathcal{M}_i^{0})=\infty$. 
\label{thm:SC03_1}
\end{theorem}

\begin{proof}
If the statement were wrong, then we can find a sequence of polarized K\"ahler Ricci flow solutions
$\mathcal{LM}_i \in \mathscr{K}(n,A;1)$ satisfying (\ref{eqn:SL06_1}) and
 \begin{align}
    \lim_{i \to \infty} \mathbf{cr}(\mathcal{M}_i^{0})=H<\infty.
 \label{eqn:SC06_17}
 \end{align}
 For each $\mathcal{M}_i$, we can find a point $x_i$ such that $\mathbf{cr}(x_i,0) \leq \frac{3}{2} \mathbf{cr}(\mathcal{M}_i^{0})$ by definition. So we have
 \begin{align}
    \lim_{i \to \infty} \mathbf{cr}(x_i,0) \leq \frac{3}{2}H<\infty.  \label{eqn:SC04_3}
 \end{align}
 Note that $T_i \to \infty$, so the first property holds trivially on the scale $2H$ for large $i$.
 By Propositions listed before, we see that there exists an $N=N(H)$ such that for every $i>N$, we have volume ratio estimate, regularity estimate,
 density estimate and connectivity estimate hold on each scale $\rho \in (0,2H]$. Therefore, by definition, we obtain that
 $\displaystyle \lim_{i \to \infty} \mathbf{cr}(x_i,0) \geq 2H$, which contradicts (\ref{eqn:SC04_3}).
\end{proof}

\begin{corollary}[\textbf{Weak continuity of canonical volume radius}]
  Same conditions as in Theorem~\ref{thm:SC04_1}. Then we have
  $\displaystyle   \mathbf{vr}(\bar{x})=\lim_{i \to \infty} \mathbf{cvr}(x_i)$. 
\label{cly:SL27_1}
\end{corollary}
\begin{proof}
 It follows from the combination of Proposition~\ref{prn:SC06_5} and Theorem~\ref{thm:SC03_1}.
\end{proof}

\begin{theorem}[\textbf{Weak continuity of polarized canonical radius}]
Suppose $\mathcal{LM}_i \in \mathscr{K}(n,A;0.5)$ satisfies (\ref{eqn:SL06_1}).
Then $\mathbf{pcr}(\mathcal{M}_i^{0}) \geq 1$ for $i$ large enough.
\label{thm:SK27_2}
\end{theorem}
\begin{proof}
 It follows from the combination of Theorem~\ref{thm:SC04_1}, Theorem~\ref{thm:SC03_1} (for the case $r_0=0.5$) and Corollary~\ref{cly:SK27_1}.
\end{proof}

\subsection{A priori bound of polarized canonical radius}
We shall use a maximum principle type argument to show that the polarized canonical radius cannot be too small.
The technique used in the following proof is inspired by the proof of Theorem 12.1 of~\cite{Pe1}.

\begin{proposition}[\textbf{A priori bound of $\textbf{pcr}$}]
There is a uniform integer constant $j_0=j_0(n,A)$ with the following property.

Suppose $\mathcal{LM} \in \mathscr{K}(n,A)$, then
 \begin{align}
    \mathbf{pcr}(\mathcal{M}^{t}) \geq \frac{1}{j_0}
 \label{eqn:SB27_4}
 \end{align}
for every $t \in [-1, 1]$.
  \label{prn:SC28_2}
\end{proposition}

\begin{proof}
Suppose for some positive integer $j_0$, (\ref{eqn:SB27_4}) fails at time $t_0 \in [-1,1]$.
Then we check whether  $\mathbf{pcr}(\mathcal{M}^t) \geq \frac{1}{2j_0}$ on the interval $[t_0-\frac{1}{2j_0}, t_0+\frac{1}{2j_0}]$.
If so, stop. Otherwise, choose $t_1$ to be such a time and continue to check if $\mathbf{pcr}(\mathcal{M}^t) \geq \frac{1}{4j_0}$
on the interval $[t_1-\frac{1}{4j_0}, t_1+\frac{1}{4j_0}]$.  In each step, we shrink the scale to one half of the scale in the previous step.
Note this process will never escape the time interval $[-2,2]$ since
\begin{align*}
  |t_k-t_0|<\frac{1}{j_0} \left( \frac{1}{2} + \frac{1}{4} + \cdot +\frac{1}{2^k}\right)<\frac{1}{j_0}<1, \quad |t_k|<|t_0|+1 \leq 2.
\end{align*}
By compactness of the underlying manifold,
it is clear that the process stops after finite steps.  So we can find $t_k$ such that  $\frac{1}{2^{k+1}j_0} \leq \mathbf{pcr}(\mathcal{M}^{t_k}) < \frac{1}{2^kj_0}$ and
$\mathbf{pcr}(\mathcal{M}^{t}) \geq \frac{1}{2^{k+1}j_0}$ for every
$t \in [t_k-\frac{1}{2^{k+1} j_0}, t_k + \frac{1}{2^{k+1} j_0}]$.
Translate the flow and rescale by constant $4^{k}j_0^2$,
we obtain a new polarized K\"ahler Ricci flow $\widetilde{\mathcal{LM}} \in \mathscr{K}(n,A)$ such that
\begin{align}
\begin{cases}
   &\mathbf{pcr}(\widetilde{\mathcal{M}}^0) <1, \\
   &\mathbf{pcr}(\widetilde{\mathcal{M}}^{t}) \geq \frac{1}{2}, \quad \forall \; t \in [-2^{k-1}j_0,2^{k-1}j_0], \\
   &|R|+|\lambda|<\frac{A}{4^{k}j_0^2}<\frac{A}{j_0^2}, \quad \textrm{on} \;  \widetilde{\mathcal{M}}, \\
   &\frac{1}{T}+\frac{1}{\Vol(M)}<\frac{1}{2^{k-1}j_0} + \frac{A}{j_0^2}, \quad \textrm{on} \;  \widetilde{\mathcal{M}}.
\end{cases}
\label{eqn:SL04_3}
\end{align}
In other words, $\widetilde{\mathcal{LM}} \in \mathscr{K}(n,A;0.5)$ and $|R|+|\lambda|+\frac{1}{T}+\frac{1}{\Vol(M)}$ very small.

Now we return to the main proof.  If the statement fails, after adjusting, translating and rescaling,  we can find a sequence of polarized
K\"ahler Ricci flow $\widetilde{\mathcal{LM}}_i \in \mathscr{K}(n,A;0.5)$ satisfying
\begin{align*}
\begin{cases}
   &\mathbf{pcr}\left( \mathcal{M}_i^{0} \right) <1, \\
   &\frac{1}{T_i}+\frac{1}{\Vol(M_i)}+\sup_{\widetilde{\mathcal{M}}_i} (|R|+|\lambda|) \to 0,
\end{cases}
\end{align*}
which contradicts Theorem~\ref{thm:SK27_2}. 
\end{proof}

 Let $\hslash=\frac{1}{j_0}$. Then we have the following fact.

\begin{theorem}[\textbf{Homogeneity on small scales}]
For some small positive number $\hslash=\hslash(n,A)$, we have
 \begin{align}
   \mathscr{K}(n,A)=\mathscr{K}\left(n,A;\hslash \right).  \label{eqn:SL04_4}
 \end{align}
\label{thm:SC28_1}
\end{theorem}

\section{Structure of polarized K\"ahler Ricci flows in $\mathscr{K}(n,A)$}

Because of Theorem~\ref{thm:SC28_1}, $\mathscr{K}(n,A)=\mathscr{K}\left(n,A;\hslash \right)$.
We do have a uniform lower bound for polarized canonical radius.

\subsection{Local metric structure, flow structure, and line bundle structure}
The purpose of this subsection is to set up estimates related to the local metric structure, flow structure and
line bundle structure of every flow in $\mathscr{K}(n,A)$. In particular, we shall prove
Theorem~\ref{thmin:SC24_1} and Theorem~\ref{thmin:HC08_1}.

\begin{proposition}[\textbf{K\"ahler tangent cone}]
Suppose $\mathcal{LM}_i \in \mathscr{K}(n,A)$ is a sequence of polarized K\"ahler Ricci flows.
Let $(\bar{M}, \bar{x}, \bar{g})$ be the limit space of $(M_i,x_i, g_i(0))$.
Then for each $\bar{y} \in \bar{M}$, every tangent space of $\bar{M}$ at $\bar{y}$ is an irreducible metric cone.
Moreover, this metric cone can be extended as an eternal, possibly singular Ricci flow solution.
  \label{prn:HA07_2}
\end{proposition}

\begin{proof}
 It follows from Theorem~\ref{thm:SC28_1} and Theorem~\ref{thm:SC09_1} that every tangent space is an irreducible metric cone.  From the proof of Theorem~\ref{thm:SC09_1}, it is clear that the tangent cone can be extended as
 an eternal, static Ricci flow solution.
\end{proof}

\begin{proposition}[\textbf{Regularity equivalence}]
Same conditions as in Proposition~\ref{prn:HA07_2}, $\bar{y} \in \bar{M}$. Then the following statements are equivalent.
\begin{enumerate}
 \item One tangent space of $\bar{y}$ is $\C^n$.
 \item Every tangent space of $\bar{y}$ is $\C^n$.
 \item $\bar{y}$ has a neighborhood with $C^{4}$-manifold structure.
 \item $\bar{y}$ has a neighborhood with $C^{\infty}$-manifold structure.
 \item $\bar{y}$ has a neighborhood with $C^{\omega}$-manifold (real analytic manifold) structure.
\end{enumerate}
\label{prn:HA08_1}
\end{proposition}

\begin{proof}
   It is obvious that $5 \Rightarrow 4 \Rightarrow 3 \Rightarrow 2 \Rightarrow 1$.  So it suffices to show $1 \Rightarrow 5$
   to close the circle.   Suppose $\bar{y}$ has a tangent space which is isometric to $\C^n$. So we can find a sequence   $r_k \to 0$ such that
   \begin{align*}
     (\bar{M}, \bar{y}, r_k^{-2}\bar{g})  \stackrel{G.H.}{\longrightarrow} (\C^n, 0, g_{Euc}).
   \end{align*}
  So for large $k$, the unit ball $B_{r_k^{-2}\bar{g}}(\bar{y},1)$ has volume ratio almost the Euclidean one.
  Fix such a large $k$, we see that $B_{\bar{g}}(\bar{y}, r_k)$ has almost Euclidean volume ratio.  It follows from volume convergence that $\mathbf{cvr}(y_i,0) \geq r_k$ for large $i$, where $y_i \in M_i$ and $y_i \to \bar{y}$ as
  $(M_i,x_i,g_i(0))$ converges to $(\bar{M}, \bar{x}, \bar{g})$.  By the regularity improving property of
  canonical volume radius, there is a uniform small constant $c$ such that $B(y_i,cr_k)$ is diffeomorphic to the
  same radius Euclidean ball in $\C^n$ and the metrics on $B(y_i, c r_k)$ is $C^2$-close to the Euclidean metric.
  Then one can apply the backward pseudolocality(c.f. Theorem~\ref{thm:SL27_2}) to obtain higher order derivative estimate for the metrics.  Therefore, $B(y_i, \frac{1}{2}c r_k)$ will converge in smooth topology to a limit smooth geodesic ball $B(\bar{y},  \frac{1}{2}cr_k)$.   Moreover,  it is clear that geometry is uniformly bounded in a space-time neighborhood containing
  $B(y_i, \frac{1}{2} cr_k) \times [-c^2 r_k^2, 0]$, by shrinking $c$ if necessary. So we obtain a limit K\"ahler Ricci flow
  solution on $B(\bar{y}, \frac{1}{4}c r_k) \times [-\frac{1}{4}c^2 r_k^2, 0]$.
  It follows from the result of Kotschwar(c.f.~\cite{Kotsch}), that $B(\bar{y}, \frac{1}{4}cr_k)$ is actually an analytic manifold, which is the desired neighborhood of  $\bar{y}$.
  So we finish the proof of $1 \Rightarrow 5$ and close the circle.
   \end{proof}

\begin{remark}
By Proposition~\ref{prn:HA08_1}, our initial non-classical definition of regularity is proved to be the same as the classical
one (c.f.~Remark~\ref{rmk:HA01_1}).
\label{rmk:HA07_1}
\end{remark}

\begin{proposition}[\textbf{Volume density gap}]
  Same conditions as in Proposition~\ref{prn:HA07_2}, $\bar{y} \in \bar{M}$.
  Then $\bar{y}$ is singular if and only if
\begin{align}
   \limsup_{r \to 0}  \frac{|B(\bar{y}, r)|}{\omega_{2n}r^{2n}} \leq 1-2\delta_0.
\label{eqn:HA07_3}
\end{align}
\label{prn:HA07_1}
\end{proposition}

\begin{proof}
 If (\ref{eqn:HA07_3}) holds, then every tangent cone of $\bar{y}$ cannot be $\C^n$, so $\bar{y}$ is singular.
 If $\bar{y}$ is singular, then every tangent space of $\bar{y}$ is an irreducible metric cone in the model space
 $\widetilde{\mathscr{KS}}(n,\kappa)$ with vertex a singular point, it follows from the gap property of
 $\widetilde{\mathscr{KS}}(n,\kappa)$ that asymptotic volume ratio of such a metric cone must be at most $1-2\delta_0$.
 Then (\ref{eqn:HA07_3}) follows from the volume convergence and a scaling argument.
\end{proof}

\begin{proposition}[\textbf{Regular-Singular decomposition}]
Same conditions as in Proposition~\ref{prn:HA07_2}, $\bar{M}$ has the regular-singular decomposition
$\bar{M}=\mathcal{R} \cup \mathcal{S}$. Then the regular part $\mathcal{R}$ admits a natural K\"ahler structure $\bar{J}$.
The singular part $\mathcal{S}$ satisfies the estimate $\dim_{\mathcal{H}} \mathcal{S} \leq 2n-4$.
\label{prn:HA08_2}
\end{proposition}
\begin{proof}
 The existence of $\bar{J}$ on $\mathcal{R}$ follows from smooth convergence,
 due to the backward pseudolocality(c.f. Theorem~\ref{thm:SL27_2}) and Shi's estimate.
 The Hausdorff dimension estimate of $\mathcal{S}$ follows from the
 combination of Proposition~\ref{prn:SL15_1} and Theorem~\ref{thm:SC28_1}.
\end{proof}

Therefore,  Theorem~\ref{thmin:SC24_1} follows from
the combinations from Proposition~\ref{prn:HA07_2} to Proposition~\ref{prn:HA08_2}.
Now we are going to discuss more delicate properties of the moduli space
$\widetilde{\mathscr{K}}(n,A)$.

\begin{proposition}[\textbf{Improve regularity in two time directions}]
There is a small positive constant $c=c(n,A)$ with the following properties.

Suppose $\mathcal{LM} \in \mathscr{K}(n,A)$, $x_0 \in M$. Let $r_0=\min\{\mathbf{cvr}(x_0,0), 1\}$.  Then we have
 \begin{align*}
     r^{2+k}|\nabla^k Rm|(x,t) \leq \frac{C_k}{c^{2+k}},  \quad \forall \;  k \in \Z^{+}, \quad x \in B_{g(0)}(x_0, cr_0),
     \quad t \in [-c^2 r^2, c^2 r^2],
 \end{align*}
where $C_k$ is a constant depending on $n,A$ and $k$.
\label{prn:HA08_3}
\end{proposition}

\begin{proof}
  Otherwise, there exists a fixed positive integer $k_0$ and a sequence of $c_i \to 0$ such that
  \begin{align}
     (c_i r_i)^{2+k_0}|\nabla^{k_0} Rm|(y_i,t_i)  \to \infty
  \label{eqn:HA08_4}
  \end{align}
  for some $y_i \in B_{g_i(0)}(x_i, r_i)$, $t_i \in [-c_i r_i^2, c_i r_i^2]$, where $r_i=\min\{\mathbf{cvr}(x_i,0), 1\}$.

  Let
  $\tilde{g}_i(t)=(c_i r_i)^{-2}g_i((c_ir_i)^2t+t_i)$.
  Then we have $\mathbf{cvr}_{\tilde{g}_i}(y_i, 0)=(c_i r_i)^{-1} \to \infty$.  Note that
  $\mathbf{pcr}_{\tilde{g}_i}(y_i,0) \geq \min\{\hslash (c_i r_i)^{-1}, 1\} \geq 1$.
  It is also clear that for the flows $\tilde{g}_i$, $|R|+|\lambda| \to 0$.
  Therefore, Proposition~\ref{prn:SL04_1} can be applied to obtain
  \begin{align}
    (M_i, y_i, \tilde{g}_i(0)) \stackrel{\hat{C}^{\infty}}{\longrightarrow} (\hat{M},\hat{y},\hat{g}).
  \label{eqn:HA08_5}
  \end{align}
  However, it follows from Theorem~\ref{thm:SC04_1} and Corollary~\ref{cly:SL27_1} that
  \begin{align*}
   (\hat{M},\hat{y},\hat{g}) \in \widetilde{\mathscr{KS}}(n,\kappa),   \quad
   \mathbf{cvr}(\hat{y})=\infty.
  \end{align*}
  In light of the gap property, Proposition~\ref{prn:SC17_1},
  we know that $\hat{M}$ is isometric to $\C^n$.
  So the convergence (\ref{eqn:HA08_5}) can be  rewritten as
  \begin{align*}
   (M_i, y_i, \tilde{g}_i(0)) \stackrel{C^{\infty}}{\longrightarrow} (\C^n, 0, g_{Euc}).
  \end{align*}
  In particular, $|\nabla^{k_0} Rm|_{\tilde{g}_i}(y_i,0) \to 0$, which is the same as
  \begin{align*}
     (c_i r_i)^{2+k_0}|\nabla^{k_0} Rm|(y_i,t_i)  \to 0.
  \end{align*}
  This contradicts the assumption (\ref{eqn:HA08_4}).
\end{proof}

Perelman's pseudolocality theorem says that an almost Euclidean domain cannot become very singular in a short time.
His almost Euclidean condition is explained as isoperimetric constant close to that of the Euclidean one.  In our special setting,
we can reverse this theorem, i.e., an almost Euclidean domain cannot become very singular in the reverse time direction for a short time period.

\begin{theorem}[\textbf{Two-sided pseudolocality}]
There is a small positive constant $\xi=\xi(n,A)$ with the following properties.

  Suppose $\mathcal{LM} \in \mathscr{K}(n,A)$, $x_0 \in M$.
  Let $\Omega=B_{g(0)}(x_0,r)$, $\Omega'=B_{g(0)}(x_0, \frac{r}{2})$ for some $0<r\leq 1$.
  Suppose $\mathbf{I}(\Omega) \geq (1-\delta_0) \mathbf{I}(\C^n)$ at time $t=0$, then
  \begin{align*}
     (\xi r)^{2+k}|\nabla^k Rm|(x,t) \leq C_k,  \quad \forall \;  k \in \Z^{\geq 0}, \quad x \in \Omega',
     \quad t \in [-\xi^2 r^2, \xi^2 r^2],
  \end{align*}
  where $C_k$ is a constant depending on $n,A$ and $k$.
\label{thm:SL27_2}
\end{theorem}
\begin{proof}
  Note that each geodesic ball contained in $\Omega$ has volume ratio at least $(1-\delta_0)\omega_{2n}$.
  Then the theorem follows from directly from Proposition~\ref{prn:HA08_3}.
\end{proof}

After we obtain the bound of geometry, we can go further to study the evolution of potential functions.

\begin{theorem}[\textbf{Two-sided pseudolocality on the potential level}]
Same conditions as in Theorem~\ref{thm:SL27_2}.
Let $\omega_B$ be a smooth metric form in $2\pi c_1(M,J)$ and denote
$\omega_t$ by $\omega_B + \sqrt{-1} \partial \bar{\partial} \varphi(\cdot, t)$.
Suppose $\varphi(x_0,0)=0$ and $Osc_{\Omega} \varphi(\cdot, 0) \leq H$.
Let $\Omega''=B_{g(0)}(x_0,\frac{r}{4})$. Then we have
\begin{align}
    (\xi r)^{-2+k}\norm{\varphi(\cdot, t)}{C^k(\Omega'', \omega_t)} \leq C_k, \quad \forall \; k \in \Z^{\geq 0},
    \quad t \in \left[-\frac{\xi^2}{2} r^2, \frac{\xi^2}{2} r^2 \right],
\label{eqn:HA05_2}
\end{align}
where $C_k$ depends on $k,n,A,\xi$ and $\frac{H}{r^2}$.
\label{thm:HA03_5}
\end{theorem}

\begin{proof}
Up to rescaling, we may assume $\xi r=1$.

Note that $\varphi$ and $\dot{\varphi}$ satisfy the equations
  \begin{align*}
  \begin{cases}
    &\dot{\varphi}=\log \frac{\omega_t^n}{\omega_B^n}+\varphi +\dot{\varphi}(\cdot, 0), \\
    &-\sqrt{-1} \partial \bar{\partial} \dot{\varphi}=Ric-\lambda g.
  \end{cases}
  \end{align*}
  It follows from Theorem~\ref{thm:SL27_2} that  geometry is uniformly bounded in $\Omega' \times [-\xi r^2, \xi r^2]$.
  The trace form of the second equation in the above list is $-\Delta \dot{\varphi}=R-n\lambda$.
  Therefore, the regularity theory
   of Laplacian operator applies and we have uniform bound of $\norm{\dot{\varphi}}{C^k}$ in a neighborhood of
   $\Omega'' \times [-\frac{\xi}{2}r^2, \frac{\xi}{2}r^2]$.
   Up to a normalization, we can rewrite the first  equation as
   \begin{align*}
      \log \frac{(\omega_t-\sqrt{-1}\partial \bar{\partial} \varphi)^n}{\omega_t^n}=\varphi - \dot{\varphi}+\dot{\varphi}(\cdot, 0).
   \end{align*}
   On $\Omega'$, the metric $g(0)$ and $g(t)$ are uniformly equivalent in each $C^k$-topology.  So it is clear that
   $\norm{\dot{\varphi}-\dot{\varphi}(\cdot, 0)}{C^k(\Omega')}$ are uniformly bounded, for each $k$, with respect to
   metric $g(t)$.
   Since all higher derivatives of curvature are uniformly bounded on $\Omega'$, (\ref{eqn:HA05_2})
   follows from standard Monge-Ampere  equation theory and bootstrapping argument.
\end{proof}

\begin{theorem}[\textbf{Improving regularity of potentials}]
 Suppose $\mathcal{LM} \in \mathscr{K}(n,A)$, $\mathbf{cvr}(M,0)=r_0$.
 Let $\omega_B$ be a smooth metric in $[\omega_0]$ such that
 \begin{align}
     \frac{1}{2}\omega_B \leq  \omega_0 \leq 2\omega_B.
 \label{eqn:HA06_1}
 \end{align}
 Let $\omega_0=\omega_B + \sqrt{-1} \partial \bar{\partial} \varphi$.
 Suppose $\int_M \varphi \omega_0^n=0$ and $Osc_M \varphi \leq H$.  Then we have
 \begin{align}
      \norm{\varphi}{C^k(M,\omega_B)} \leq C_k, \quad \forall \; k \in \Z^{\geq 0},
 \label{eqn:HA06_2}
 \end{align}
 where $C_k$ depends on $k,\omega_B, n,A,r_0$ and $H$.
\label{thm:HA06_1}
\end{theorem}

\begin{proof}
Since $\mathbf{cvr}(M,0)=r_0>0$, we see that all the possible $\omega_0$'s form a compact set under the smooth topology.
In other words, $\omega_0$ has uniformly bounded geometry in each regularity level.   Fix a positive integer $k_0 \geq 4$.
Therefore,  around each point $x \in M$, one can find a coordinate chart $\Omega$, with uniform size, such that
\begin{align*}
   \omega_0=\omega_{Euc} + \sqrt{-1} \partial \bar{\partial} f, \quad
   \norm{f}{C^{k_0}(\Omega, \omega_{Euc})} \leq  0.01.
\end{align*}
Note that in $\Omega$, the connection terms of the metric $\omega_0$ are pure derivatives $f_{i\bar{j}l}$,
which are uniformly bounded.  Similarly, all derivatives of connection terms can be expressed as high order
pure derivatives of $f$.  Therefore, up to order $k_0-3$, the derivatives of connections are uniformly bounded.
It is clear that the metric $\omega_0$ and $\omega_{Euc}$ are uniformly equivalent.
By the covariant derivatives' bounds $\norm{\varphi}{C^k(M, \omega_0)} \leq C_k$,
the bounds of connection derivatives yield that
\begin{align}
    \norm{\varphi}{C^k(\Omega, \omega_{Euc})} \leq C_k, \quad \forall \; 0 \leq k \leq k_0-1.
\label{eqn:HA08_3}
\end{align}
In other words, we have uniform bound for every order pure derivatives of $\varphi$, up to order $k_0-1$.
Together with the choice assumption of $\Omega$, we have
\begin{align*}
    \norm{f-\varphi}{C^k(\Omega, \omega_{Euc})} \leq C_k, \quad \forall \; 0 \leq k \leq k_0-1.
\end{align*}
Therefore, the connection derivatives of metric $\omega_B$ in $\Omega$ are uniformly bounded, up to order $k_0-4$.
Consequently, the pure derivative bound (\ref{eqn:HA08_3}) implies
\begin{align*}
   \norm{\varphi}{C^k(\Omega, \omega_B)} \leq C_k, \quad \forall \; 0 \leq k \leq k_0-1,
\end{align*}
since $\omega_B$ is a fixed smooth, compact metric with every level of regularity.  Clearly, the above constant $C_k$ depends on
$k,n,A,r_0,\omega_B$ and $H$.  Recall that the size of $\Omega$ is uniformly bounded from below, $(M, \omega_B)$ is a compact
manifold.   Consequently, a standard covering argument implies (\ref{eqn:HA06_2}) for each $k \leq k_0-1$.
In the end, we free $k_0$ and finish the proof.
\end{proof}

In Ricci-flat theory, a version of Anderson's gap theorem says that regularity can be improved in the center of a ball
if the volume ratio of the unit ball is very close to the Euclidean one.
In our special setting,  this gap theorem has a reduced volume version.

\begin{theorem}[\textbf{Gap of reduced volume}]
There is a  constant $\delta_0' \in (0, \delta_0]$ and a small constant $\eta$ with the following property.

Suppose $\mathcal{LM} \in \mathscr{K}(n,A)$, $x_0 \in M$, $0<r \leq 1$.
If $\mathcal{V}((x_0,0),r^2) \geq 1-\delta_0'$, then we have
\begin{align}
   \mathbf{cvr}(x_0,0) \geq \eta r.
\label{eqn:SL27_2}
\end{align}
\label{thm:SL27_3}
\end{theorem}

\begin{proof}
If $\lambda=0$, reduced volume is monotone.  If $\lambda$ is bounded, then reduced volume is almost monotone.
A simple calculation shows that $\mathcal{V}((x_0,0), \rho^2) \geq 1-\delta_0$ for all $0<\rho \leq r^2$
whenever $\mathcal{V}((x_0,0),r^2) \geq 1-\delta_0'$ for some $0<r\leq 1$.   Therefore, without loss of generality,
we may assume $\lambda=0$ and $\delta_0'=\delta_0$ in the proof.

If the statement was wrong, there exists a sequence of $\eta_i \to 0$, $0<r_i \leq 1$ and $x_i \in M_i$, and corresponding K\"ahler Ricci flows satisfying
\begin{align*}
  \begin{cases}
   &\mathcal{V}((x_i,0),r_i^2) \geq 1-\delta_0, \\
   &\mathbf{cvr}(x_i,0) <\eta_i r_i.
  \end{cases}
\end{align*}
By the monotonicity of reduced volume, we have
\begin{align*}
 \begin{cases}
   &\mathcal{V}((x_i,0),H\eta_i^2 r_i^2) \geq 1-\delta_0, \\
   &\mathbf{cvr}(x_i,0) <\eta_i r_i,
  \end{cases}
\end{align*}
for each fixed $H$ and large $i$.
Let $\tilde{g}_i(t)=(\eta_i r_i)^{-2}g((\eta_i r_i)^2t)$.  It is clear that
\begin{align}
  \mathbf{cvr}_{\tilde{g}_i}(x_i,0)=1.   \label{eqn:HA09_1}
\end{align}
The canonical radius of $\tilde{g}_i$ tends to infinity, $|R|+|\lambda| \to 0$.
Similar to the proof of Proposition~\ref{prn:HA08_3}, we have the convergence:
\begin{align*}
 (M_i, x_i, \tilde{g}_i(0)) \stackrel{\hat{C}^{\infty}}{\longrightarrow} (\hat{M}, \hat{x}, \hat{g}) \in \widetilde{\mathscr{KS}}(n,\kappa).
\end{align*}
The limit space $\hat{M}$ can be extended to a static eternal K\"ahler Ricci flow solution.
Moreover, Proposition~\ref{prn:SL14_2} can be applied here and guarantees the reduced volume convergence.
\begin{align*}
   \mathcal{V}((\hat{x},0), H)=\lim_{i \to \infty} \mathcal{V}_{\tilde{g}_i}((x_i,0), H)
   =\lim_{i \to \infty} \mathcal{V}((x_i,0), H(\eta_i r_i)^2) \geq 1-\delta_0.
\end{align*}
Note that $H$ is arbitrary. By the homogeneity
of reduced volume at infinity, Theorem~\ref{thm:SL25_1}, we see that
\begin{align*}
   \mathrm{avr}(\hat{M})=\lim_{H \to \infty} \mathcal{V}((\bar{x},0),H) \geq 1-\delta_0 \geq 1-\delta_0.
\end{align*}
So Proposition~\ref{prn:SC17_1} applies to force $\hat{M}$ to be isometric to be $\C^n$.
In particular, $\mathbf{vr}(\hat{x})=\infty$.
It follows from Corollary~\ref{cly:SL27_1} that $\displaystyle  \lim_{i \to \infty} \mathbf{cvr}_{\tilde{g}_i}(x_i,0)=\infty$, 
which contradicts (\ref{eqn:HA09_1}).
\end{proof}

According to Theorem~\ref{thm:SL27_3}, one can define a concept of reduced volume radius for the purpose of
improving regularity.  Clearly, other regularity radius can also be defined. However, it seems all of them are equivalent.
For simplicity, we shall not compare all of them, but only prove an example case:
 the equivalence of harmonic radius and canonical volume radius. The proof of other cases are verbatim.

\begin{proposition}[\textbf{Equivalence of regularity radii}]
  Suppose $\mathcal{LM} \in \mathscr{K}(n,A)$, $x \in M$.
  Suppose $\max\{\mathbf{hr}(x,0), \mathbf{cvr}(x,0)\} \leq 1$,
  then we have
  \begin{align*}
       \frac{1}{C} \mathbf{hr}(x,0) \leq    \mathbf{cvr}(x,0) \leq C \mathbf{hr}(x,0)
  \end{align*}
  for some uniform constant $C=C(n,A)$.
\label{prn:HA09_2}
\end{proposition}
\begin{proof}
Clearly, $\mathbf{cvr}(x,0) \leq C \mathbf{hr}(x,0)$ follows from the $C^5$-regularity property of  canonical volume radius.  It suffices to show $\frac{1}{C} \mathbf{hr}(x,0) \leq    \mathbf{cvr}(x,0)$.  However, since $\mathbf{cr}(x,0) \geq \hslash$,
it is clear from definition that
\begin{align*}
  \mathbf{cvr}(x,0) \geq \frac{1}{C} \min\{\mathbf{hr}(x,0), \hslash\}.
\end{align*}
If $\mathbf{hr}(x,0)\leq \hslash$, then we are done. Otherwise, we have $\hslash <\mathbf{hr}(x,0) \leq 1$. It follows that
\begin{align*}
  \mathbf{cvr}(x,0) \geq \frac{1}{C}\hslash \geq \frac{\hslash}{C} \mathbf{hr}(x,0) \geq \frac{1}{C'} \mathbf{hr}(x,0).
\end{align*}
So we finish the proof.
\end{proof}

\begin{theorem}[\textbf{Improved density estimate}]
For arbitrary small $\epsilon$, arbitrary $0 \leq p<2$, there is a constant $\delta=\delta(n,A,p)$ with the following properties.

 Suppose $\mathcal{LM} \in \mathscr{K}(n,A)$, $x \in M$. Then under the metric $g(0)$, we have
 \begin{align}
    \log \frac{\int_{B(x,r)} \mathbf{cvr}^{-2p}dv}{E(n,\kappa,p) r^{2n-2p} } < \epsilon
 \label{eqn:SC13_2}
 \end{align}
whenever $r<\delta$.
 Here the number $E(n,\kappa,p)$ is defined in Proposition~\ref{prn:SB25_2}.
  \label{thm:SC12_2}
\end{theorem}

\begin{proof}
We argue by contradiction.
Note that every  blowup limit is in $\widetilde{\mathscr{KS}}(n,\kappa)$(c.f. Theorem~\ref{thm:SC04_1}).
Then a contradiction can be obtained by the weak continuity of  $\mathbf{cvr}$(c.f. Corollary~\ref{cly:SL27_1})
if the statement of this theorem does not hold.
\end{proof}

Note that $E(n,\kappa,0)=\omega_{2n}$.  So we are led to the volume ratio estimate immediately.

\begin{corollary}[\textbf{Volume-ratio estimate}]
For arbitrary small $\epsilon$, there is a constant $\delta=\delta(n,A)$ with the following properties.

Suppose $\mathcal{LM} \in \mathscr{K}(n,A)$, $x \in M$. Then under the metric $g(0)$, we have
 \begin{align}
    \log \frac{|B(x,r)|}{\omega_{2n} r^{2n}}  < \epsilon
 \label{eqn:SC13_1}
 \end{align}
whenever $r<\delta$.
  \label{cly:SC13_1}
\end{corollary}

In the K\"ahler Ricci flow setting, Corollary~\ref{cly:SC13_1} improves the volume ratio estimates in~\cite{Zhq3}
and~\cite{CW5} (c.f.~Remark 1.1 of \cite{CW5}). Note that the integral (\ref{eqn:SC13_2}) can be used to show that
for every $p \in (0,2)$,  there is a $C=C(n,A,p)$ such that the $r$-neighborhood of $\mathcal{S}$ in a unit ball is
bounded by $C r^{2p}$(c.f. Theorem~\ref{thm:HE11_1}), where $\mathcal{S}$ is the singular part of a limit space.
By the definition of Minkowski dimension(c.f.~Definition~\ref{dfn:HE08_1}),
we can improve Proposition~\ref{prn:HA08_2} as follows.

\begin{corollary}[\textbf{Minkowski dimension of singular set}]
Same conditions as in Proposition~\ref{prn:HA07_2}, $\bar{M}$ has the regular-singular decomposition
$\bar{M}=\mathcal{R} \cup \mathcal{S}$. Then $\dim_{\mathcal{M}} \mathcal{S} \leq 2n-4$.
\label{cly:HE25_2}
\end{corollary}

In \cite{Wa2}, the second author developed an estimate of the type $|Ric| \leq \sqrt{|Rm||R|}$,  where $\sqrt{|Rm|}$ should be understood as the reciprocal of a regular scale. Due to the improving regularity property of canonical volume radius, it induces the estimate
$|Ric| \leq \frac{\sqrt{|R|}}{\mathbf{cvr}}$ pointwisely.
By the uniform bound of scalar curvature and Theorem~\ref{thm:SC12_2}, the following estimate is clear now.

\begin{corollary}[\textbf{Ricci curvature estimate}]
  There is a constant $C=C(n,A,r_0)$ with the following property.

  Suppose $\mathcal{LM} \in \mathscr{K}(n,A)$, $x_0 \in M$, $0<r \leq r_0$, $0<p<2$.
  Then under the metric $g(0)$, we have
  \begin{align}
      r^{2p-2n}\int_{B(x_0,r)}  |Ric|^{2p} dv<C.
  \label{eqn:SL27_3}
  \end{align}
\label{cly:SL27_4}
\end{corollary}

Corollary~\ref{cly:SL27_4} localizes the $L^{2p}$-curvature estimate of \cite{TZZ2} in a weak sense, since
(\ref{eqn:SL27_3}) only holds for $p<2$.   If $n=2$,  (\ref{eqn:SL27_3}) also holds for $p=2$,
since the finiteness of singularity guarantees that one can choose good cutoff functions.
We believe that the same localization result hold for $p=2$ even if $n>2$.\\

We return to the canonical neighborhood theorems in the introduction,
Theorem~\ref{thmin:SC24_1}, Theorem~\ref{thmin:HC08_1} and Theorem~\ref{thmin:HC06_1}.
However, Theorem~\ref{thmin:HC06_1} is not completely local.
Actually,  Theorem~\ref{thm:SL27_2} is enough to show
the local flow structure of $\mathscr{K}(n,A)$ can be approximated by $\mathscr{KS}(n,\kappa)$.
In light of its global properties, the proof of Theorem~\ref{thmin:HC06_1} is harder and is postponed to section 5.5.
On the other hand, Theorem~\ref{thmin:SC24_1} and Theorem~\ref{thmin:HC08_1} are local.
We now close this subsection by proving Theorem~\ref{thmin:SC24_1} and Theorem~\ref{thmin:HC08_1}.

\begin{proof}[Proof of Theorem~\ref{thmin:SC24_1}]
It follows from the combination of Proposition~\ref{prn:HA07_2}, Proposition~\ref{prn:HA08_1}, Proposition~\ref{prn:HA07_1},
and Proposition~\ref{prn:HA08_2}.
\end{proof}

\begin{proof}[Proof of Theorem~\ref{thmin:HC08_1}]
It follows from Theorem~\ref{thm:SC28_1},  Definition~\ref{dfn:SK27_1} and a scaling argument.
\end{proof}

\subsection{Local variety structure}

 We focus on the variety structure of the limit space in this subsection.
 We essentially follow the argument in~\cite{DS}, with slight modification.

 Suppose $\mathcal{LM}_i \in \mathscr{K}(n,A)$, $x_i \in M_i$. Let $(\bar{M}, \bar{x},\bar{g})$
 be a pointed-Gromov-Hausdorff limit of $(M_i,x_i,g_i(0))$.
 Since $\bar{M}$ may be non-compact, the limit line bundle $\bar{L}$ may have infinitely many orthogonal holomorphic sections.
 Therefore, in general, we cannot expect to embed $\bar{M}$ into a projective space of finite dimension
 by the complete linear system of $\bar{L}$. However, when
 we focus our attention to the unit geodesic ball $B(\bar{x},1)$,
 we can choose some holomorphic sections of $\bar{L}$, peaked around $\bar{x}$, to embed $B(\bar{x},1)$ into $\CP^N$ for a
 finite $N$.

 Actually, for every $\epsilon>0$, we can find an $\epsilon$-net of $B(\bar{x},2)$ such that every point in this net has canonical volume radius
 at least $c_0 \epsilon$.    For each point $y$ in this $\epsilon$-net, we have a peak section $s_y$, which is a holomorphic section such that
 $\norm{s(y)}{}$ achieves the maximum among all unit $L^2$-norm holomorphic sections $s \in H^0(\bar{M},\bar{L})$.  By the partial-$C^0$-estimate
 argument(c.f.~\cite{CW4} for the flow case with weak convergence), we can assume that $\norm{s_y}{}^2$ is uniformly bounded below in $B(y,2\epsilon)$.

 On the other hand, by the choice of $y$, $B(y, \eta \epsilon)$ has a smooth manifold structure for some $\eta=\eta(n)$.
 Therefore, we can choose $n$ holomorphic sections of $\bar{L}^k$ such that these sections  are the local deformation of $ z_1, z_2, \cdots, z_n$.
 Here $k$ is a positive integer proportional to $\epsilon^{-2}$. Put these holomorphic sections together with $s_y^k$,
 we obtain $(n+1)$-holomorphic sections of $\bar{L}^k$ based at the point $y$.  Let $y$ run through all points in the $\epsilon$-net and collect all the holomorphic sections based at $y$,  we obtain a set of holomorphic sections $\{s_i\}_{i=0}^N$ of $\bar{L}^{k}$. Let $\{\tilde{s}_i\}_{i=0}^N$
  be the orthonormal basis of $span\{s_0,s_1, \cdots, s_N\}$. We define the Kodaira map $\iota$ as follows.
   \begin{align*}
     \iota:  B(0,2) &\mapsto \CP^N, \\
                x &\mapsto [\tilde{s}_0(x) : \tilde{s}_1(x): \cdots:\tilde{s}_N(x)].
  \end{align*}
  This map is well defined.   In fact,  for every $z \in B(\bar{x},1)$, we can find a point $y$ in the $\epsilon$-net and $z \in B(y,2\epsilon)$, then
  $\norm{s_y}{}^2(z)>0$ by the partial-$C^0$-estimate. It forces that $\tilde{s}_j(z) \neq 0$ for some $j$.
  Since $k$ is proportional to $\epsilon^{-2}$, we can just let $\epsilon=\frac{1}{\sqrt{k}}$ without loss of generality.
  In the following argument,  by saying ``raise the power of line bundle" from $k_1$ to $k_2$,
  we simultaneously means the underlying $\epsilon$-net is
  strengthened from a $\frac{1}{\sqrt{k_1}}$-net to a $\frac{1}{\sqrt{k_2}}$-net.

\begin{lemma}
Suppose $w \in \iota(B(\bar{x},1))$, then $ \iota^{-1}(w) \cap \overline{B(\bar{x},1)}$ is a finite set.
\label{lma:HB09_1}
\end{lemma}

\begin{proof}
Let $y \in \iota^{-1}(w) \cap \overline{B(\bar{x},1)}$. It is clear that $\iota^{-1}(w)$ is contained in a ball centered at $y$ with fixed radius, say $10\epsilon$.
Therefore, $\iota^{-1}(w)$ is a bounded, closed set and therefore compact.  Let $F$ be a connected component of $\iota^{-1}(w)$. Then
$\iota(F)$ is a connected, compact subvariety of $\C^N$, and consequently is a point. Note that $\iota(F)$ is always a connected set no matter how do we raise the power of $\iota$. On the other hand, $\iota(F)$ will contain more than one point if $F$ is not a single point, after we raise power high enough.
These force that $F$ can only be a point. Since $\iota^{-1}(w) \cap \overline{B(\bar{x},1)}$ is compact, it must be union of finite points.
\end{proof}

Denote $\iota(\overline{B(\bar{x},1)})$ by $W$. Then $W$ is a compact set and locally can be extended as an analytic variety.
By dividing $W$ into different components, one can apply induction argument as that in~\cite{DS}.
Following verbatim the argument of Proposition 4.10, Lemma 4.11 of~\cite{DS}, one can show that $\iota$ is an injective, non-degenerate embedding
map on $B(\bar{x},1)$, by raising power of $\bar{L}$ if necessary.  Furthermore, since being normal is a local property, one can
improve Lemma 4.12 of~\cite{DS} as follows.

\begin{lemma}
By raising power if necessary, $W$ is normal at the point $\iota(y)$ for every $y \in B(\bar{x}, \frac{1}{2})$.
\end{lemma}

Under the help of parabolic Schwarz lemma and heat flow localization technique(c.f. Section 4.1 and Proposition~\ref{prn:HB05_1}),
we can parallelly generalize Proposition 4.14 of~\cite{DS} as follows.

\begin{lemma}
 Suppose $y \in B(\bar{x},\frac{1}{2})\cap \mathcal{S}$, then $\iota(y)$ is a singular point of $W$.
\end{lemma}

It follows from the proof of Proposition 4.15 of~\cite{DS} that there always exist a holomorphic form $\Theta$ on $\mathcal{R} \cap B(\bar{x},1)$ such that
$\int_{\mathcal{R} \cap B(\bar{x},1)} \Theta \wedge \bar{\Theta}<\infty$.  This means that every singular point $y \in \iota(B(\bar{x},\frac{1}{2})) \cap W$ is
log-terminal.

Combining all the previous lemmas, we have the following structure theorem.

\begin{theorem}[\textbf{Analytic variety structure}]
Suppose $\mathcal{LM}_i \in \mathscr{K}(n,A)$, $x_i \in M_i$,
$(\bar{M}, \bar{x},\bar{g})$ is a pointed Gromov-Hausdorff limit of $(M_i,x_i,g_i(0))$.
Then $\bar{M}$ is an analytic space with normal, log terminal singularities.
\label{thm:HE19_1}
\end{theorem}

\subsection{Distance estimates}
In this subsection, we shall develop the distance estimate along polarized K\"ahler Ricci flow in terms of the estimates
from line bundle.

\begin{lemma}
 Suppose $(M,L)$ is a polarized K\"ahler manifold satisfying the following conditions
 \begin{itemize}
  \item $|B(x,r)| \geq \kappa \omega_{2n} r^{2n}, \quad \forall x \in M, 0<r<1$.
  \item $|\mathbf{b}|\leq 2c_0$ where $\mathbf{b}$ is the Bergman function.
  \item  $\norm{\nabla S}{} \leq C_1$ for every $L^2$-unit section $S \in H^0(M,L)$.
 \end{itemize}
 For every positive number $a$, define $ \Omega(x,a)$ as the path-connected component containing $x$ of the set  
 \begin{align}
    \left\{z \left| \norm{S}{}^2(z) \geq e^{-2a-2c_0},
   \norm{S}{}^2(x)=e^{\mathbf{b}(x)}, \int_M \norm{S}{}^2dv=1 \right.\right\}.
 \label{eqn:HA15_2}
 \end{align}
 Then we have
 \begin{align}
  B\left(x,r\right) \subset   \Omega(x,a) \subset B(x, \rho)
 \label{eqn:HA15_1}
 \end{align}
 for some $r=r(n,\kappa,c_0,C_1,a)$ and $\rho=\rho(n,\kappa,c_0,C_1,a)$.
\label{lma:HA15_1}
\end{lemma}

\begin{proof}
  Define  $r \triangleq \frac{1-e^{-a}}{C_1 e^{a+c_0}}$. 
  Recall that $\norm{S}{}(x) \geq e^{-c_0}$. By the gradient bound of $S$, it is clear that every point in
  $B(x,r)$ satisfies $\norm{S}{} \geq e^{-a-c_0}$. In other words, we have
  \begin{align*}
      B(x,r) \subset \Omega(x,a).
  \end{align*}
  On the other hand, we can cover $\Omega(x,a)$ by finite balls $B(x_i, 2r)$ such that each $x_i \in \Omega(x,a)$
  and different $B(x_i,r)$'s are disjoint to each other.   Again, the gradient bound of $S$ implies that
  $\norm{S}{} \geq e^{-2a-c_0}$ in each $B(x_i,r)$. Then we have
  \begin{align*}
    N \kappa \omega_{2n} r^{2n} \leq \sum_{i=1}^{N} |B(x_i,r)| \leq |\Omega(x, 2a)| \leq e^{4a+2c_0}.
  \end{align*}
  For every $z \in \Omega(x,a)$, we have
  \begin{align*}
     d(x,z) \leq 4Nr \leq \frac{4e^{4a+2c_0}}{\kappa \omega_{2n}r^{2n-1}}
     =\frac{4e^{4a+2c_0}}{\kappa \omega_{2n}} \cdot \frac{C_1^{2n-1} e^{(2n-1)(a+c_0)}}{(1-e^{-a})^{2n-1}}.
  \end{align*}
  Let $\rho$ be the number on the right hand side of the above inequality. Then it is clear that
  \begin{align*}
     \Omega(x,a) \subset B(x,\rho).
  \end{align*}
  So we finish the proof.
\end{proof}

Lemma~\ref{lma:HA15_1} implies that the level sets of peak holomorphic sections  are comparable to
geodesic balls.  However, the norm of peak holomorphic section has stability under the K\"ahler Ricci flows
in $\mathscr{K}(n,A)$. Therefore, one can compare distances at different time slices in terms the values
of norms of a same holomorphic sections.

\begin{lemma}
 There exists a small constant $\epsilon_0=\epsilon_0(n,A)$ such that
 the following properties are satisfied.

 Suppose $\mathcal{LM} \in \mathscr{K}(n,A)$, then we have
 \begin{align}
     B_{g(t_1)}(x,\epsilon_0) \subset B_{g(t_2)} \left(x,\epsilon_0^{-1} \right) \label{eqn:HA14_1}
 \end{align}
 whenever $t_1, t_2 \in [-1,1]$.
\label{lma:HA14_1}
\end{lemma}

\begin{proof}
Without loss of generality, we only need to show (\ref{eqn:HA14_1}) for time $t_1=0, t_2=1$.
Because of Theorem~\ref{thmin:HC08_1} and Moser iteration, we can assume $|\mathbf{b}| \leq 2c_0$ for some $c_0=c_0(n,A)$.
By Moser iteration technique, we can also assume $\norm{\nabla S}{} \leq C_1$ for every unit $L^2$-norm holomorphic
section of $L$(c.f. Lemma 5.1 of~\cite{Wa1} and Lemma 3.2 of~\cite{CW4}).  Note that $e^{\mathbf{b}(x)}$ is the maximum value of $\norm{S}{}^2$ among all
unit $L^2$-norm holomorphic sections of $L$.  So we can choose $\epsilon$ small enough such that
\begin{align*}
   \norm{S}{}(z) \geq \frac{1}{2}e^{-c_0}, \quad \forall \; z \in B(x,2\epsilon)
\end{align*}
for some unit holomorphic section $S$. Note $\epsilon$ can be chosen uniformly, say $\epsilon=\frac{e^{-c_0}}{4C_1}$.

Fix $S$ and define
\begin{align*}
  \Omega \triangleq \left\{z\left| \norm{S}{0}(z)\geq  \frac{1}{2}e^{-c_0} \right. \right\}, \qquad
  \tilde{\Omega} \triangleq \left\{z\left| \norm{S}{1}(z) \geq \frac{1}{4}e^{-c_0-A} \right. \right\}.
\end{align*}
Without loss of generality, we can assume both $\Omega$ and $\tilde{\Omega}$ are path-connected. Otherwise, just replace them by the corresponding path-connected part containing $z$.
It follows from definition that $B(x,\epsilon) \subset \Omega$.
In view of the volume element evolution equation,
it is also clear that $\Omega \subset \tilde{\Omega}$.

Note that $S$ is a unit section at time $t=0$.
At time $t=1$, its $L^2$-norm locates in $[e^{-2A}, e^{2A}]$. So we have
\begin{align*}
e^{2A} \geq \int_M \norm{S}{1}^2 dv>|\tilde{\Omega}|_{1} \frac{1}{16}e^{-2c_0-2A}, \quad
\Rightarrow \quad |\tilde{\Omega}|_{1}<16 e^{2c_0+4A}.
\end{align*}
Now we can follow the covering argument in the previous lemma to show a diameter bound of $\tilde{\Omega}$
under the metric $g(1)$.   In fact, we can cover $\tilde{\Omega}$ by finite geodesic balls $B(x_i,2\epsilon)$ such that $x_i \in \tilde{\Omega}$
and all different $B(x_i,\epsilon)$'s are disjoint to each other.  Clearly, each geodesic ball
$B(x_i,\epsilon)$
has volume at least $\kappa \omega_{2n} \epsilon^{2n}$, where $\kappa=\kappa(n,A)$.  Let
$N$ be the number of balls, then
\begin{align*}
 N \kappa \omega_{2n} \epsilon^{2n} \leq \sum_{i=1}^{N}|B(x_i,1)| \leq |\tilde{\Omega}| \leq 16e^{2c_0+4A}.
\end{align*}
Therefore, under metric $g(1)$,  we obtain
\begin{align*}
\diam \Omega \leq  \diam \tilde{\Omega} \leq 4N\epsilon \leq \frac{64e^{2c_0+4A}}{\kappa \omega_{2n}\epsilon^{2n-1}}.
\end{align*}
Recall that $B_{g(0)}(x,\epsilon) \subset \Omega \subset \tilde{\Omega}$, $\epsilon=\frac{e^{-c_0}}{4C_1}$.
So under metric $g(1)$, we have
\begin{align*}
   \diam B_{g(0)}(x,\epsilon) \leq \diam \tilde{\Omega} \leq \frac{4^{2n+2}e^{(2n+3)c_0+4A}C_1^{2n-1}}{\kappa \omega_{2n}}.
\end{align*}
Define
\begin{align}
 \epsilon_0 \triangleq
 \min\left\{\frac{e^{-c_0}}{4C_1}, \frac{\kappa \omega_{2n}}{4^{2n+2}e^{(2n+3)c_0+4A}C_1^{2n-1}} \right\}.
\label{eqn:HA15_3}
\end{align}
Note that $\epsilon_0$ depends only on $n,A$.
Then we have $ \diam B_{g(0)}(x,\epsilon_0) \leq \epsilon_0^{-1}$, which implies
\begin{align*}
  B_{g(0)}(x, \epsilon_0) \subset B_{g(1)}(x,\epsilon_0^{-1}).
\end{align*}
So we finish the proof.
\end{proof}

\begin{lemma}
For every $r$ small, there is a $\delta$ with the following property.

 Suppose $\mathcal{LM} \in \mathscr{K}(n,A)$.
 Suppose $|R|+|\lambda|<\delta$ on $M \times [-1,1]$,
 then we have
 \begin{align}
    B_{g(t_1)}(x,\epsilon_0 r) \subset B_{g(t_2)}(x,\epsilon_0^{-1}r)
 \label{eqn:HA15_4}
 \end{align}
 for every $t_1,t_2 \in [-1,1]$.  Here $\epsilon_0$ is the constant in Lemma~\ref{lma:HA14_1}.
\label{lma:HA14_2}
\end{lemma}

\begin{proof}
We proceed by a contradiction argument.

 Again, it suffices to show (\ref{eqn:HA15_4}) for $t_1=0$ and $t_2=1$.  By adjusting $r$ if necessary, we can
 also make a rescaling by integer factor.  Up to rescaling, (\ref{eqn:HA15_4}) is the same as
 \begin{align}
      B_{g(0)}(x,\epsilon_0) \subset B_{g(r^{-2})}(x, \epsilon_0^{-1}).   \label{eqn:HA15_5}
 \end{align}

 Suppose the statement of this lemma was wrong. Then there is an $r_0>0$ and a sequence of
 points $x_i \in M_i$ such that
 \begin{align*}
    B_{g_i(0)}(x_i,\epsilon_0) \not\subset B_{g_i(r_0^{-2})}(x_i,\epsilon_0^{-1}).
 \end{align*}
 However, $|R|+|\lambda| \to 0$ in $C^0$-norm as $i \to \infty$.  So we can take a limit
 \begin{align*}
   (M_i,x_i,g_i(0)) \stackrel{\hat{C}^{\infty}}{\longrightarrow} (\bar{M},\bar{x},\bar{g}).
 \end{align*}
 As usual, we can find a regular point $\bar{z} \in \bar{M}$ near $\bar{x}$.  Let $z_i \in M_i$ and $z_i \to \bar{z}$
 as the above convergence happens.  Then we can extend the above convergence to each time slice.
 \begin{align*}
    (M_i,z_i,g_i(t)) \stackrel{\hat{C}^{\infty}}{\longrightarrow} (\bar{M},\bar{z},\bar{g}), \quad \forall \; t \in [0,r_0^{-2}].
 \end{align*}
 Note that $\bar{x}$ may be a singular point of $\bar{M}$. So in the above convergence,  we only have $x_i$
 converges to $\bar{x}(t)$, which may depends on time $t$.  Lemma~\ref{lma:HA14_1} guarantees that $\bar{x}(t)$
 is not at infinity.

 Note that $Osc_M \dot{\varphi}$ is scaling invariant and consequently uniformly bounded by condition inequality (\ref{eqn:SK20_1}). 
 Therefore $\dot{\varphi}$ converges to a limit bounded function which is harmonic function on $\mathcal{R}(\bar{M})$, the regular part of $\bar{M}$. 
 Such a function must be a constant by Corollary~\ref{cly:HE08_1}.   
 Actually, a bounded harmonic function on $\mathcal{R}(\bar{M})$ will automatically be a bounded Lipschitz function on $\bar{M}$, by Proposition~\ref{prn:HD16_1}. 
 Applying normalization condition, the limit function must be zero on $\bar{M} \times [0,r_0^{-2}]$. Therefore, the limit line bundle $\bar{L}$
 admits a limit metric which does not evolve along time.  Therefore, for a fixed holomorphic section $\bar{S}$ and a
 fixed level value,  the level sets of $\norm{\bar{S}}{}^2$ does not depend on time.

 Choose $S_i$ be the peak section of $L_i$ at $x_i$, with respect to the metrics at time $t=0$.  By the choice of
 $\epsilon_0$, it is clear that $\norm{S_i}{}\geq \frac{1}{2}e^{-c_0}$ on the ball $B(x_i,\epsilon_0)$.
 In other words,  we have
 \begin{align*}
  B(x_i, \epsilon_0) \subset
  \Omega_{i,t} \triangleq \left\{z\left| \norm{S_i}{t}(z)\geq  \frac{1}{2}e^{-c_0} \right. \right\}.
 \end{align*}
 Without loss of generality, we can assume $\Omega_{i,t}$ is path connected.
 Clearly, each $\Omega_{i,t}$ has uniformly bounded diameter, due to Lemma~\ref{lma:HA15_1}.
 Let $\bar{\Omega}$ be the limit set of $\Omega_{i,0}$.  Clearly, $\bar{z} \in \bar{\Omega}$.
 Then the above discussion implies that $\bar{\Omega}$ is actually the limit set of each $\Omega_{i,t}$.

Let $\bar{y}$ be the limit point of $y_i$, which is a point in $B_{g_i(0)}(x_i, \epsilon_0)$ and start to escape
$B(x_i, \epsilon_0^{-1})$ at time $t_i$, which converges to $\bar{t}$. So we obtain
\begin{align}
   \bar{y} \in B_{\bar{g}(0)}(\bar{x},\epsilon_0),
   \quad
   d(\bar{y}, \bar{x}(\bar{t}))=\epsilon_0^{-1}.
\label{eqn:HA15_6}
\end{align}
Since volume element  of the underlying manifold and the line bundle metric are all almost static when time evolves,
it is easy to see that $y_i$ can never escape $\Omega_{i,t}$. So $\bar{y} \in \bar{\Omega}$.
Similarly, we know $\bar{x}(\bar{t}) \in \bar{\Omega}$.  Therefore,  at time $\bar{t}$, we have
\begin{align*}
    d(\bar{y}, \bar{x}(\bar{t})) \leq \diam \bar{\Omega}.
\end{align*}
Note that the argument in the proof of Lemma~\ref{lma:HA15_1} holds for the polarized singular manifold
$(\bar{M}, \bar{L})$, due to the high codimension of $\mathcal{S}(M)$ and the gradient bound of each $S_i$.
Since $\int_{\bar{M}} \norm{\bar{S}}{}^2 dv$ is uniformly bounded from above by $1$,  we can follow the proof of
Lemma~\ref{lma:HA15_1} to show that
\begin{align*}
   \diam(\bar{\Omega}) \leq \rho(n,\kappa,c_0,C_1,\log 2)<\epsilon_0^{-1}
\end{align*}
by the choice of $\epsilon_0$ in (\ref{eqn:HA15_3}).  Consequently, we have $d(\bar{y},\bar{x}(\bar{t}))<\epsilon_0^{-1}$,
which contradicts (\ref{eqn:HA15_6}).
\end{proof}

Based on Lemma~\ref{lma:HA14_2}, we can improve Proposition~\ref{prn:SL04_1}.
Namely, under the condition $|R|+|\lambda| \to 0$, the limit flow is static, even on the singular part.
Clearly,  due to Theorem~\ref{thm:SC28_1}, we do not need the assumption of lower bound of polarized
canonical radius anymore.

\begin{proposition}[\textbf{Static limit space-time}]
 Suppose $\mathcal{LM}_i \in \mathscr{K}(n,A)$ satisfies
 \begin{align*}
      \lim_{i \to \infty} \sup_{\mathcal{M}_i} (|R|+|\lambda|)=0.
 \end{align*}
 Suppose $x_i \in M_i$.  Then
 \begin{align*}
   (M_i,x_i, g_i(0)) \longright{\hat{C}^{\infty}}  (\bar{M},\bar{x}, \bar{g}).
 \end{align*}
 Moreover,  we have
 \begin{align*}
   (M_i,x_i, g_i(t)) \longright{\hat{C}^{\infty}}  (\bar{M},\bar{x}, \bar{g})
 \end{align*}
 for every $t \in (-\bar{T}, \bar{T})$, where $\displaystyle \bar{T}=\lim_{i \to \infty} T_i>0$.
 In other words, the identity maps between different time slices converge to the limit identity map.
\label{prn:HE15_1}
\end{proposition}

As a direct application, we obtain the bubble structure of a given family of polarized K\"ahler Ricci flows.

\begin{theorem}[\textbf{Space-time structure of a bubble}]
Suppose $\mathcal{LM}_i \in \mathscr{K}(n,A)$, $x_i \in M_i$, $t_i \in (-T_i, T_i)$, and $r_i \to 0$.
Suppose $\widetilde{\mathcal{M}}_i$ is the adjusting of  $\mathcal{M}_i$ by shifting time $t_i$ to $0$ and then
rescaling the space-time by the factors $r_i^{-2}$, i.e., $\tilde{g}_i(t)=r_i^{-2}g(r_i^2 t + t_i)$.
Suppose $r_i^{-2} \max \{|t_i-T_i|, |t_i+T_i|\}=\infty$. Then we have
\begin{align*}
  (M_i,x_i, \tilde{g}_i(t)) \longright{\hat{C}^{\infty}} (\hat{M}, \hat{x}, \hat{g})
\end{align*}
for each time $t \in (-\infty, \infty)$ with $\hat{M} \in \widetilde{\mathscr{KS}}(n,\kappa)$.
\label{thm:HE26_1}
\end{theorem}

Theorem~\ref{thm:HE26_1} means that the space-time structure of
$\hat{M} \in  \widetilde{\mathscr{KS}}(n,\kappa)$ is the model for the space-time structures around $(x_i, t_i)$,
up to proper rescaling.  Therefore, Theorem~\ref{thm:HE26_1} is an improvement of Theorem~\ref{thm:SC04_1},
where we only concern the metric structure.

In view of Proposition~\ref{prn:HE15_1},
it is not hard to see that distance is a uniform continuous function of time in $\mathscr{K}(n,A)$.

\begin{theorem}[\textbf{Uniform continuity of distance function}]
Suppose $\mathcal{LM} \in \mathscr{K}(n,A)$, $x,y \in M$.
Suppose $d_{g(0)}(x,y)<1$. Then for every small $\epsilon$, there is a
$\delta=\delta(n,A,\epsilon)$ such that
\begin{align*}
   |d_{g(t)}(x,y)-d_{g(0)}(x,y)|<\epsilon
\end{align*}
whenever $|t|<\delta$.
\label{thm:HE15_1}
\end{theorem}

\begin{proof}
We argue by contradiction.  Suppose the statement was wrong, we can find an $\bar{\epsilon}>0$ and a sequence of
flows violating the statement for time $|t_i| \to 0$.  Around $x_i$,
in the ball $B_{g_i(0)}\left(x_i,\frac{\epsilon_0^2\bar{\epsilon}}{10} \right)$,
we can find $x_i'$ which are uniform regular at time $t=0$, where $\epsilon_0$
is the same constant in Lemma~\ref{lma:HA14_2} and Lemma~\ref{lma:HA14_1}.
By two-sided pseudolocality, Theorem~\ref{thm:SL27_2}, it is clear that $x_i'$ is also
uniform regular at time $t=t_i$.  Similarly, we can choose $y_i'$.
By virtue of triangle inequality and Lemma~\ref{lma:HA14_2}, we obtain
\begin{align*}
 & d_{g_i(0)}(x_i', y_i') -\frac{\epsilon_0^2 \bar{\epsilon}}{5}
 \leq  d_{g_i(0)}(x_i, y_i) \leq d_{g_i(0)}(x_i', y_i') + \frac{\epsilon_0^2 \bar{\epsilon}}{5},\\
 & d_{g_i(t_i)}(x_i', y_i') -\frac{\bar{\epsilon}}{5}
 \leq d_{g_i(t_i)}(x_i, y_i) \leq d_{g_i(t_i)}(x_i', y_i') + \frac{\bar{\epsilon}}{5}.
\end{align*}
By argument similar to that in Proposition~\ref{prn:SC17_4}, it is clear that
\begin{align*}
   \lim_{i \to \infty} d_{g_i(t_i)}(x_i', y_i')= \lim_{i \to \infty} d_{g_i(0)}(x_i', y_i').
\end{align*}
Then it follows that
\begin{align*}
\lim_{i \to \infty} d_{g_i(0)}(x_i,y_i)  -\frac{(1+\epsilon_0^2)}{5}\bar{\epsilon}
\leq
\lim_{i \to \infty} d_{g_i(t_i)}(x_i, y_i) \leq \lim_{i \to \infty} d_{g_i(0)}(x_i,y_i)  + \frac{(1+\epsilon_0^2)}{5} \bar{\epsilon}.
\end{align*}
In particular, for large $i$, we have
\begin{align*}
      \left| d_{g_i(0)}(x_i,y_i) - d_{g_i(t_i)}(x_i, y_i) \right|<\frac{(1+\epsilon_0^2)}{5} \bar{\epsilon}<\bar{\epsilon},
\end{align*}
which contradicts our assumption.
\end{proof}

\subsection{Volume estimate for high curvature neighborhood}
In this subsection, we shall develop the flow version of the volume estimate of
Donaldson and the first author(c.f.~\cite{CD1},~\cite{CD2}, see also~\cite{CN}).

\begin{proposition}[\textbf{K\"ahler cone complex splitting}]
  Same conditions as in Proposition~\ref{prn:HA07_2},  $\bar{y} \in \bar{M}$.  Suppose
  $\hat{Y}$ is a tangent cone of $\bar{M}$ at $\bar{y}$, then there is a fixed nonnegative integer $k$
  such that
  \begin{align}
       \hat{Y}=C(Z) \times \C^{n-k},
  \label{eqn:HA05_1}
  \end{align}
  where $C(Z)$ is a metric cone without straight line.  A point in $\bar{M}$ is regular if and only if one
of the tangent cone is $\C^{n}$.
\label{prn:SL26_1}
\end{proposition}

\begin{proof}
 By definition of tangent cone, one can find a sequence of numbers $r_i \to 0$.   Taking subsequence if necessary,
 let $\tilde{g}_i(t)=r_i^{-2} g_i(r_i^2t)$, then we have
 \begin{align*}
      (M_i, y_i, \tilde{g}_i(0)) \stackrel{\hat{C}^{\infty}}{\longrightarrow} (\hat{Y}, \hat{y}, \hat{g}).
 \end{align*}
 By compactness, we see that $(\hat{Y},\hat{y},\hat{g}) \in \widetilde{\mathscr{KS}}(n,\kappa)$.
 On the other hand, it is a metric cone, which is the tangent space of itself at the origin.
 So $\hat{Y}$ has the decomposition  (\ref{eqn:HA05_1}),
 by Theorem~\ref{thm:HE08_1}  or Lemma~\ref{lma:HD30_1}.
\end{proof}

\begin{proposition}[\textbf{K\"ahler tangent cone rigidity}]
Suppose $\mathcal{LM}_i \in \mathscr{K}(n,A)$.
Suppose $x_i \in M_i$ and $(\bar{M}, \bar{x}, \bar{g})$ is a limit space of $(M_i, x_i, g_i(0))$.
Let $\hat{Y}$ be
a tangent space of $\bar{M}$. Then $\hat{Y}$ satisfies the splitting (\ref{eqn:HA05_1}) for $k=2$ or $k=0$.
\label{prn:HA10_2}
\end{proposition}

\begin{proof}
Clearly, $k=0$ if and only if the base point is regular. So it suffices to show that for every singular tangent space we have
$k=2$.   By Proposition~\ref{prn:SL26_1}, we only need to rule out the case $k \geq 3$. However, this follows from the rigidity of complex structure on the smooth annulus in $\C^k /\Gamma$, where $\Gamma$ is a finite group of holomorphic isometry
of $\C^k$, when $k \geq 3$. Note that $[\omega_i]=c_1(L_i)$, which is an integer class.
Therefore, the proof follows verbatim as that in~\cite{CD2}. Note that Ricci curvature uniformly bounded condition in~\cite{CD2} is basically used to guarantee the Cheeger-Gromov convergence. 
In our case, the convergence can be obtained from Theorem~\ref{thm:SC28_1}. 
\end{proof}

\begin{proposition}[\textbf{Existence of holomorphic slicing}]
 Suppose $\hat{Y} \in \widetilde{\mathscr{KS}}(n,\kappa)$ is a metric cone satisfying the splitting (\ref{eqn:HA05_1}).
 Suppose $\mathcal{LM} \in \mathscr{K}(n,A)$, $x \in M$.
 If $(M,x,g(0))$ is very close to $(\hat{Y}, \hat{y}, \hat{g})$, i.e., the pointed-Gromov-Hausdorff distance
 \begin{align*}
    d_{PGH}((M,x,g(0)), (\hat{Y}, \hat{y}, \hat{g}))<\epsilon
 \end{align*}
 for sufficiently small $\epsilon$, which depends on $n,A,\hat{Y}$, then there exists a holomorphic map
 \begin{align*}
     \Psi=(u_{k+1}, u_{k+2}, \cdots, u_n):  B(x, 10) \mapsto \C^{n-k}
 \end{align*}
 satisfying
 \begin{align}
 & |\nabla \Psi| \leq C(n,A), \label{eqn:HA10_1}\\
 & \sum_{k+1 \leq i,j\leq n}
 \int_{B(x,10)} \left| \delta_{ij} - \langle \nabla u_i, \nabla u_j \rangle\right|dv
 \leq \eta(n,A,\epsilon),  \label{eqn:HA10_2}
 \end{align}
 where $\eta$ is a small number such that $\displaystyle \lim_{\epsilon \to 0} \eta=0$.
\label{prn:HA10_1}
\end{proposition}

\begin{proof}

 It follows from the argument in~\cite{DS} that the constant section $1$ of the trivial bundle over $\hat{Y}$ can be
 ``pulled back" as a non-vanishing holomorphic section of $L$ over $B(x,10)$, up to a finite lifting of power of $L$.
 Therefore, we can regard $L$ as a trivial bundle over $B(x,10)$ without loss of generality.
 Let $S_0$ be the pull-back of the constant $1$ section. In particular, $S_0$ is  a non-vanishing
 holomorphic section on $B(x,10)$.
 On $B(x,10)$,  every holomorphic section $S$ of $L$ can be written as $S=u S_0$ for a holomorphic function $u$ and
 $\norm{S}{h}^2= \snorm{u}{}^2 \norm{S_0}{h}^2$.

 From the splitting (\ref{eqn:HA05_1}),  there exist natural coordinate holomorphic functions
 $\{z_j\}_{j=k+1}^{n}$ on $\hat{Y}$.
 Same as~\cite{DS}, one can apply H\"ormander's estimate
 to construct $\{S_j\}_{j=k+1}^{n}$, which are holomorphic sections of $L$.
 Each $S_j$ can be regarded as an ``approximation" of $z_j$, although they have different base spaces.
 Let $u_j=\frac{S_j}{S_0}$ for each $j \in \{k+1, \cdots, n\}$.
 Then we can define a holomorphic map $\Psi$ from $B(x,10)$ to $\C^{n-k}$ as follows
 \begin{align*}
     \Psi(y) \triangleq  \left( u_{k+1}, u_{k+2}, \cdots, u_n \right).
 \end{align*}
 Note that each $S_i$ is a holomorphic section of $L$ with $L^2$-norm bounded from two sides, according to its
 construction.  It is easy to check that $\norm{\nabla S_i}{h}^2$ satisfies a sub-elliptic equation. So there exists a uniform
 bound $\norm{\nabla S_i}{h}^2 \leq C(n,A)$, which implies (\ref{eqn:HA10_1}) when restricted on $B(x,10)$.
 Moreover, on $B(x,10)$, by smooth convergence, it is not hard to see that $ \langle \nabla u_i, \nabla u_j \rangle$ can pointwisely
 approximate $\delta_{ij}$ away from singularities of $\hat{Y}$, in any accuracy level when $\epsilon \to 0$.
 This approximation together with (\ref{eqn:HA10_1}) yields (\ref{eqn:HA10_2}).
\end{proof}

\begin{theorem}[\textbf{Weak monotonicity of curvature integral}]
There exists a small constant $\epsilon=\epsilon(n,A)$ with the following properties.

 Suppose $\mathcal{LM} \in \mathscr{K}(n,A)$.
 Suppose $x \in M$, $0<r \leq 1$.
 Then under the metric $g(0)$, we have
  \begin{align}
    \sup_{B(x, \frac{1}{2}r)} |Rm| \leq r^{-2}
 \label{eqn:SC13_3}
 \end{align}
 whenever $r^{4-2n}\int_{B(x,r)} |Rm|^2 dv \leq \epsilon$.
 \label{thm:SC12_3}
\end{theorem}

\begin{proof}
 Up to rescaling, we can assume $r=1$ without loss of generality. If the statement was wrong, we can find a sequence of
 points $x_i \in M_i$ such that
 \begin{align*}
   \int_{B(x_i,1)} |Rm|^2 dv \to 0, \quad
  \sup_{B(x_i,\frac{1}{2})} |Rm| \geq 1,
 \end{align*}
 where the default metric is $g_i(0)$, the time zero metric of a flow $g_i$, in the moduli space $\mathscr{K}(n,A)$.
 By the  smooth convergence at places when curvature uniformly bounded, it is clear that the above conditions imply that
 \begin{align*}
   \int_{B(x_i,1)} |Rm|^2 dv \to 0, \quad
  \sup_{B(x_i,\frac{3}{4})} |Rm| \to \infty.
 \end{align*}
 Let $(\bar{M},\bar{x},\bar{g})$ be the limit space of $(M_i,x_i,g_i(0))$. Then $B(\bar{x}, \frac{3}{4})$ contains at least one
 singularity $\bar{y}$.  Without loss of generality, we can assume $\bar{x}$ is a singular point.
 Note that $B(\bar{x}, \frac{1}{4})$ is a flat manifold away from singularities.  So every tangent space
 of $\bar{M}$ at $\bar{x}$ is a flat metric cone.  Let $\hat{Y}$ be one of such a flat metric cone.  By taking subsequence if necessary, we can assume
 \begin{align*}
    (M_i,x_i,\tilde{g}_i(0)) \stackrel{\hat{C}^{\infty}}{\longrightarrow} (\hat{Y}, \hat{x},\hat{g}),
 \end{align*}
 for some flow metrics $\tilde{g}_i$ satisfying $\displaystyle  \tilde{g}_i(t)=r_i^{-2} g_i(r_i^2 t), \; r_i \to 0$. Since $\hat{Y}$ is
 a flat metric cone,  in light of Proposition~\ref{prn:HA10_2}, we have the splitting
 \begin{align*}
    \hat{Y}= (\C^2/\Gamma) \times \C^{n-2}.
 \end{align*}
 Let $(M,x,\tilde{g})$ be one of $(M_i,x_i,\tilde{g}_i(0))$ for some large $i$.  Because of Proposition~\ref{prn:HA10_1}, we can construct a holomorphic map $\Psi:B(x,10) \to \C^{n-2}$
satisfying (\ref{eqn:HA10_1}) and (\ref{eqn:HA10_2}).
Then we can follow the slice argument as in~\cite{CCT} and~\cite{Cheeger03}.
Our argument will be simpler since our slice functions are holomorphic rather than harmonic.

Actually, for generic $\vec{z}=(z_{3}, z_{4}, \cdots, z_n)$ satisfying $|\vec{z}|<0.1$,
we know $\Psi^{-1}(\vec{z}) \cap B(x,5)$ is a complex surface with boundary.
Clearly,  $\Psi^{-1}((S^3/\Gamma) \times \{\vec{z}\})$ is close to $(S^3/\Gamma) \times \vec{z}$,
if we regard $S^3/\Gamma$ as the unit sphere in $\C^2/\Gamma$.  Deform the preimage a little bit if necessary, we can obtain
a $\partial \Omega$ which bounds a complex surface $\Omega$.  By coarea formula and the bound of
$|\nabla \Psi|$, it is clear that for generic $\Omega$ obtained in this way, we have
\begin{align*}
     \int_{\Omega} |Rm|^2 d\sigma\to 0.
\end{align*}
Consider the restriction of $TM$ on $\Omega$. Let $c_2$ be a form representing the second Chern class of the tangent bundle $TM$, obtained from the K\"ahler metric $\tilde{g}(0)$ from the classical way.
Let $\hat{c}_2$ be the corresponding differential character with value in $\R/\Z$.  Since the pointwise norm of $c_2$ is bounded by $|Rm|^2$, it is clear that
\begin{align}
   \hat{c}_2(\partial \Omega)=\int_{\Omega} c_2      \quad (mod \; \Z)  \to 0.  \label{eqn:HA10_3}
\end{align}
On the other hand, since $\partial \Omega$ converges to $S^3/\Gamma$, we have
\begin{align}
    \hat{c}_2(\Omega) \to \frac{1}{|\Gamma|}.  \label{eqn:HA10_4}
\end{align}
Therefore, the combination of (\ref{eqn:HA10_3}) and (\ref{eqn:HA10_4}) forces that $|\Gamma|=1$.  This is impossible
since $|\Gamma| \geq 2$ by our assumption that $\hat{Y}$ is a singular metric cone.
\end{proof}

From now on to the end of this subsection, we use $g(0)$ as the default metric.
Similar to the definition in~\cite{CD1},  for any small $r$, let $\mathcal{Z}_r$ be the $r$-neighborhood of the points where $|Rm| > r^{-2}$.
Recall the definition equation (\ref{eqn:SL27_1}), we denote $\mathcal{F}_r$ as the collection of points
whose canonical volume radii are greater than $r$, $\mathcal{D}_r$ as the component of $\mathcal{F}_r$.
Under these notations, we have the following property.

\begin{proposition}[\textbf{Equivalence of singular neighborhoods}]
Suppose $\mathcal{LM} \in \mathscr{K}(n,A)$, $0<r<\hslash$.
Then at time zero, we have
\begin{align}
      \mathcal{D}_{cr}  \subset \mathcal{Z}_r \subset \mathcal{D}_{\frac{1}{c}r}
\label{eqn:HA11_1}
\end{align}
for some small constant $c=c(n,A)$.
\label{prn:HA11_1}
\end{proposition}

\begin{proof}

  Let us first prove $\mathcal{D}_{cr} \subset \mathcal{Z}_r$. Suppose the statement was wrong, we can find a sequence $c_i \to 0$
  and flows in $\mathscr{K}(n,A)$ such that $\mathcal{D}_{c_ir_i} \not\subset Z_{r_i}$ for some $r_i <\hslash$.
  Choose $x_i \in \mathcal{D}_{c_i r_i} \cap \mathcal{Z}_{r_i}^{c}$.  Let $\rho_i$ be the canonical volume radius of $x_i$.  Rescale
  the flow such that the canonical volume radius at $x_i$ becomes $1$. Take limit, we will obtain a smooth flat space in
  $\widetilde{\mathscr{KS}}(n,\kappa)$, which is nothing but $\C^n$.  Therefore, the canonical volume radii
  of the base points $x_i$ should tend to infinity, which is a contradiction.

 Then we prove $\mathcal{Z}_r \subset \mathcal{D}_{\frac{1}{c}r}$.   Suppose $x \in \mathcal{Z}_r$, then $|Rm|(y) \geq r^{-2}$ for some $y \in B(x,r)$.  By the regularity improving property of canonical volume radius, it is clear
  that $\mathbf{cvr}(x) \leq \frac{2}{c_a} r$. In other words, $x \in \mathcal{D}_{\frac{2}{c_a}r}$.
\end{proof}

 \begin{theorem}[\textbf{Volume estimates of high curvature neighborhood}]
  Suppose $\mathcal{LM} \in \mathscr{K}(n,A)$.
  Under the metric $g(0)$, we have
   \begin{align*}
     |\mathcal{Z}_r| \leq C r^4,
  \end{align*}
  where $C$ depends on $n,A$ and the upper bound of $\int_M |Rm|^2 dv$.
\label{thm:HA11_1}
\end{theorem}

\begin{proof}
Because of Proposition~\ref{prn:HA11_1}, it suffices to show $|\mathcal{D}_{cr}| \leq C r^4$.

In light of Theorem~\ref{thm:SC12_3},  if $r^{4-2n}\int_{B(x,r)} |Rm|^2 dv<\epsilon$ for some
$r<\hslash$, then $x \in \mathcal{F}_{cr}$.  In other words, if $x \in \mathcal{D}_{cr}(M,0)$, then it is forced that
\begin{align*}
    r^{4-2n} \int_{B(x,r)} |Rm|^2 dv \geq \epsilon.
\end{align*}
Let $\bigcup_{i=1}^{N} B(x_i,2r)$ be a finite cover of $\mathcal{D}_{cr}$ such that
\begin{itemize}
\item $x_i \in \mathcal{D}_{cr}$.
\item $B(x_i,r)$ are disjoint to each other.
\end{itemize}
Then we can bound $N$ as follows.
\begin{align*}
 N\epsilon r^{2n-4} \leq  \sum_{i=1}^{N} \int_{B(x_i,r)} |Rm|^2 dv \leq \int_M |Rm|^2 dv \leq H.
\end{align*}
Consequently, we have
\begin{align*}
     |\mathcal{D}_{cr}(M)| \leq \sum_{i=1}^{N} |B(x_i, 2r)| \leq \frac{H}{\epsilon} r^{4-2n} \kappa^{-1} \omega_{2n} (2r)^{2n}
     \leq C r^4.
\end{align*}
Since both $\kappa$ and $\epsilon$ depends only on $n$ and $A$. It is clear that $C=C(n,A,H)$ where $H$
is the upper bound of $\int_M |Rm|^2 dv$.
\end{proof}

 \begin{corollary}[\textbf{Volume estimates of singular neighborhood}]
  Suppose  $\mathcal{LM}_i \in \mathscr{K}(n,A)$.
  Suppose $\int_{M_i} |Rm|^2 dv \leq H$ uniformly under the metric $g_i(0)$. Let $(\bar{M}, \bar{x}, \bar{g})$ be the limit
  space of $(M_i,x_i,g_i(0))$.  Let $\mathcal{S}_r$ be the set defined in (\ref{eqn:HA11_4}), then we have
  \begin{align*}
      |\mathcal{S}_{r}| \leq C r^4
  \end{align*}
  for each small $r$ and some constant $C=C(n,A,H)$.
  In particular, we have  the estimate of Minkowski dimension of the singularity
  \begin{align*}
      \dim_{\mathcal{M}} \mathcal{S} \leq 2n-4.
  \end{align*}
  \label{cly:HA11_1}
\end{corollary}

Following~\cite{CW4},  the space $\bar{M}=\mathcal{R} \cup \mathcal{S}$
is called a metric-normal $Q$-Fano variety if there exists
a homeomorphic map $\varphi: \bar{M} \to Z$ for some $Q$-Fano normal variety $Z$ such that $\varphi|_{\mathcal{R}}$
is a biholomorphic map.  Moreover, $\dim_{\mathcal{M}} \mathcal{S} \leq 2n-4$.

 \begin{theorem}[\textbf{Limit structure}]
  Suppose that $\mathcal{LM}_i  \in \mathscr{K}(n,A)$.
  Under the metric $g_i(0)$, suppose
  \begin{align}
    \Vol(M_i) + \int_{M_i} |Rm|^2 dv \leq H  \label{eqn:HA11_2}
  \end{align}
  for some uniform constant $H$.
  Let $(\bar{M}, \bar{x}, \bar{g})$ be the limit space of $(M_i,x_i,g_i(0))$.
  Then $\bar{M}$ is a compact metric-normal $Q$-Fano variety.
\label{thm:SC13_4}
\end{theorem}
\begin{proof}
 It follows from (\ref{eqn:HA11_2}) and the non-collapsing that $diam(M_i)$ is uniformly bounded.
 So the limit space $\bar{M}$ is compact.   Due to Theorem~\ref{thmin:HC08_1}, the partial $C^0$-estimate,
 one can follow the argument in~\cite{DS} to show that $\bar{M}$ is a $Q$-Fano, normal variety.
 The metric-normal property follows from Corollary~\ref{cly:HA11_1}.
\end{proof}

Based on the estimates developed in this subsection, we can easily prove Corollary~\ref{clyin:SC24_4}
and Corollary~\ref{clyin:SC24_5} in the introduction.

\begin{proof}[Proof of Corollary~\ref{clyin:SC24_4} and Corollary~\ref{clyin:SC24_5}]
It follows from the combination of Theorem~\ref{thm:SC13_4}, Corollary~\ref{cly:HA11_1} of this paper and
main results in~\cite{CW4}. Note that the line bundle metric choice in this paper is
equivalent to that in~\cite{CW4}, due to the bound of $\dot{\varphi}$.
\end{proof}

\subsection{Singular K\"ahler Ricci flow solution}
\label{sec:sflow}

In this subsection, we shall relate the different limit time slices, without the assumption of $|R|+|\lambda| \to 0$.
We shall further improve regularity, by estimates essentially arising from complex analysis of holomorphic sections.

We want to compare $\omega_t$, the K\"ahler Ricci flow metrics, and $\tilde{\omega}_t$, the evolving Bergman metrics.
We first show that $\tilde{\omega}_t$ is very stable when $t$ evolves.

\begin{lemma}
Suppose $G(t)$ is a family of $(N+1)\times (N+1)$ matrices parameterized by $t \in [-1,1]$. Suppose $G(0)=Id$, $\dot{G}(0)=B$.
Let $\lambda_0\leq \lambda_1 \leq \cdots \leq \lambda_N$ be the real eigenvalues of the Hermitian matrix $B+\bar{B}^{\tau}$.
If we regard $G$ as a holomorphic map from $\CP^N$ to $\CP^N$,   then we have
\begin{align}
   (\lambda_0-\lambda_N) \omega_{FS} \leq
  \left. \frac{d}{dt}G(t)^*(\omega_{FS}) \right|_{t=0} \leq (\lambda_N-\lambda_0) \omega_{FS}.
\label{eqn:HB02_1}
\end{align}
\label{lma:HB02_1}
\end{lemma}

\begin{proof}
Let $\{z_i\}_{i=0}^N$ be the homogeneous coordinate of $\CP^N$.   Let $G=G(t)$. Then we have
 \begin{align*}
    &\omega_{FS}=\sqrt{-1} \partial \bar{\partial} \log (|z_0|^2+|z_1|^2+\cdots |z_N|^2)
    =\sqrt{-1} \left\{\frac{\partial z_i \wedge \bar{\partial} \bar{z}_i}{|z|^2}
    +\frac{(z_i \bar{\partial} \bar{z}_i) \wedge (\bar{z}_j \partial z_j)}{|z|^4}\right\}, \\
    &G^*(\omega_{FS})=\sqrt{-1} \partial \bar{\partial} \log (|\tilde{z}_0|^2+|\tilde{z}_1|^2+\cdots |\tilde{z}_N|^2)
    =\sqrt{-1} \left\{\frac{\partial \tilde{z}_i \wedge \bar{\partial} \bar{\tilde{z}}_i}{|\tilde{z}|^2}
    +\frac{(\tilde{z}_i \bar{\partial} \bar{\tilde{z}}_i) \wedge (\bar{\tilde{z}}_j \partial \tilde{z}_j)}{|\tilde{z}|^4}\right\},
 \end{align*}
 where $\tilde{z}_i=G_{ij}z_j$.  Let $\{ w_1, \cdots, w_N\}$ be local coordinate.
 At point $z$, the matrix of $\omega_{FS}$ is
 \begin{align*}
    E_0=J \left(\frac{Id}{|z|^2}-\frac{\bar{z}^{\tau}z}{|z|^4}\right) \bar{J}^{\tau}=J F_0 \bar{J}^{\tau},
 \end{align*}
 where $J$ is an $N \times (N+1)$ matrix which is the Jacobi matrix $\left(\D{z_j}{w_{\alpha}}\right)$.
 The matrix of
 $\omega_{G^*\omega_{FS}}$ is
 \begin{align*}
    E_t=J G\left(\frac{Id}{|\tilde{z}|^2}-\frac{\bar{\tilde{z}}^{\tau} \tilde{z}}{|\tilde{z}|^4}\right) \bar{G}^{\tau} \bar{J}^{\tau}
      =JF_t \bar{J}^{\tau}.
 \end{align*}
 Clearly, we have
 \begin{align*}
   \left. \frac{d}{dt} F_t \right|_{t=0}&=B\left(\frac{Id}{|z|^2}-\frac{\bar{z}^{\tau}z}{|z|^4}\right)
     +\left(\frac{Id}{|z|^2}-\frac{\bar{z}^{\tau}z}{|z|^4}\right) \bar{B}^{\tau}
     -\frac{z(B+\bar{B}^{\tau})\bar{z}^{\tau}}{|z|^4} Id +\frac{2z(B+\bar{B}^{\tau})\bar{z}^{\tau}}{|z|^6} \bar{z}^{\tau}z
    -\frac{\bar{B}^{\tau}\bar{z}^{\tau}z+\bar{z}^{\tau}zB}{|z|^4}\\
    &=\left\{ \frac{B+\bar{B}^{\tau}}{|z|^2} -\frac{(B+\bar{B}^{\tau})\bar{z}^{\tau}z + \bar{z}^{\tau}z(B+\bar{B}^{\tau})}{|z|^4}
   +\frac{z(B+\bar{B}^{\tau})\bar{z}^{\tau}}{|z|^6} \bar{z}^{\tau}z \right\}
   -\frac{z(B+\bar{B}^{\tau})\bar{z}^{\tau}}{|z|^4} F_0\\
    &\triangleq M-\frac{z(B+\bar{B}^{\tau})\bar{z}^{\tau}}{|z|^4} F_0.
 \end{align*}
 It follows that
 \begin{align*}
  \left. \frac{d}{dt} E_t \right|_{t=0}=
  \left.  \frac{d}{dt}
  \left\{J G\left(\frac{Id}{|\tilde{z}|^2}-\frac{\bar{\tilde{z}}^{\tau} \tilde{z}}{|\tilde{z}|^4}\right) \bar{G}^{\tau} \bar{J}^{\tau} \right\}
  \right|_{t=0}=J\left(M-\frac{z(B+\bar{B}^{\tau})\bar{z}^{\tau}}{|z|^2} F_0 \right)\bar{J}^{\tau}.
 \end{align*}
 It is easy to check that
 \begin{align*}
    zM\bar{z}^{\tau}=0,  \quad zF_0\bar{z}^{\tau}=0, \quad
    z\left( M-\frac{z(B+\bar{B}^{\tau})\bar{z}^{\tau}}{|z|^4} F_0\right)\bar{z}^{\tau}=0.
 \end{align*}
 Without loss of generality, we can assume $B+\bar{B}^{\tau}$ is a diagonal matrix
 $\diag(\lambda_0, \lambda_1, \cdots, \lambda_N)$.
 Let $v=(v_0, v_1, \cdots, v_N)$ be a vector in $\C^{N+1}$ satisfying
 \begin{align*}
  z\bar{v}^{\tau}=\bar{v}_0z_0 +\bar{v}_1z_1 + \cdots + \bar{v}_N z_N=0.
 \end{align*}
 Then it is clear that
 \begin{align*}
  vM\bar{v}^{\tau}=\frac{v(B+\bar{B}^{\tau})\bar{v}^{\tau}}{|z|^2}, \quad
  vF_0\bar{v}^{\tau}=\frac{|v|^2}{|z|^2}.
 \end{align*}
Therefore, we have
 \begin{align*}
  &\quad v\left(M-\frac{z(B+\bar{B}^{\tau})\bar{z}^{\tau}}{|z|^2} F_0 \right)\bar{v}^{\tau}\\
  &=\frac{1}{|z|^4}\left\{ (\lambda_0 |v_0|^2+\cdots+\lambda_N |v_N|^2)(|z_0|^2+\cdots+|z_N|^2)
   -(\lambda_0|z_0|^2+\cdots \lambda_N|z_N|^2)(|v_0|^2+\cdots+|v_N|^2)\right\}\\
   &=\frac{1}{|z|^4} \left\{ \left[(\lambda_0-\lambda_0)|z_0|^2 +(\lambda_0-\lambda_1)|z_1|^2+\cdots
         +(\lambda_0-\lambda_N)|z_N|^2 \right]|v_0|^2 \right.\\
    &\quad +  \left[(\lambda_1-\lambda_0)|z_0|^2 +(\lambda_1-\lambda_1)|z_1|^2+\cdots
         +(\lambda_1-\lambda_N)|z_N|^2 \right]|v_1|^2 \\
    &\quad +\cdots\\
    &\quad \left.+\left[(\lambda_N-\lambda_0)|z_0|^2 +(\lambda_N-\lambda_1)|z_1|^2+\cdots
         +(\lambda_N-\lambda_N)|z_N|^2 \right]|v_N|^2 \right\}\\
     &\leq (\lambda_N-\lambda_0) \frac{|v|^2}{|z|^2}.
 \end{align*}
 Similarly, we have
 \begin{align*}
   v\left(M-\frac{z(B+\bar{B}^{\tau})\bar{z}^{\tau}}{|z|^2} F_0 \right)\bar{v}^{\tau} \geq (\lambda_0-\lambda_N)\frac{|v|^2}{|z|^2}.
 \end{align*}
 Note that $z F_0 \bar{v}^{\tau}=0$. Therefore, we can apply the orthogonal decomposition with respect to $F_0$ to obtain that
 for every vector $f=(f_0,f_1,\cdots, f_N) \in \C^{N+1}$, we have
 \begin{align*}
   \left| f\left(M-\frac{z(B+\bar{B}^{\tau})\bar{z}^{\tau}}{|z|^2} F_0 \right)\bar{f}^{\tau} \right|
   \leq (\lambda_N-\lambda_0) \frac{|f|^2}{|z|^2}=(\lambda_N-\lambda_0)fF_0\bar{f}^{\tau}.
 \end{align*}
 Let $u \in T_z^{(1,0)}\CP^N$. Then we have
 \begin{align*}
  \langle u,u \rangle_{\omega_{FS}}&=(uJ) F_0 \overline{(uJ)}^{\tau},\\
  \left| \langle u, u \rangle_{\left.\frac{d}{dt}G^*(\omega_{FS}) \right|_{t=0}} \right|
  &=\left|(uJ) \left(M-\frac{z(B+\bar{B}^{\tau})\bar{z}^{\tau}}{|z|^2} F_0 \right)\overline{uJ}^{\tau} \right|\\
  &\leq (\lambda_N-\lambda_0) (uJ) F_0 \overline{(uJ)}^{\tau}\\
  &\leq (\lambda_N-\lambda_0)\langle u,u \rangle_{\omega_{FS}}.
 \end{align*}
 By the arbitrary choice of $u$, then (\ref{eqn:HB02_1}) follows directly from the above inequality.
 \end{proof}

\begin{lemma}
 Suppose $\mathcal{LM} \in \mathscr{K}(n,A)$.
 Let $\tilde{\omega}_t$ be the pull back of the Fubini-Study metric by orthonormal basis of $L$ with respect to $\omega_t$ and $h_t$. 
 Then we have the evolution inequality of $\tilde{\omega}_t$:
 \begin{align}
     -2A\tilde{\omega}_t \leq  \frac{d}{dt} \tilde{\omega_t} \leq 2A \tilde{\omega}_t.
 \label{eqn:HB03_1}
 \end{align}
\label{lma:HB05_1}
\end{lemma}

\begin{proof}
 Without loss of generality,  it suffices to show (\ref{eqn:HB03_1}) at time $t=0$. \\

  Suppose $\left\{ s_i \right\}_{i=0}^{N}$ is an orthonormal basis at time $0$, $\left\{ \tilde{s}_i \right\}_{i=0}^{N}$
  is an orthonormal basis at time $t$.  They are related by $\tilde{s}_i=s_jG_{ji}$.   Fix $e_L$ a local representation of the line  bundle $L$ around a point $x$ so that locally we have $s_j=z_j e_L$ and $\tilde{s}_j=\tilde{z}_j e_L=z_j G_{ji}e_L$. Then we have
  \begin{align*}
   & \tilde{\omega}_0=\sqrt{-1} \partial \bar{\partial} \log \left(|z_0|^2+|z_1|^2+\cdots +|z_N|^2 \right), \\
   & \tilde{\omega}_t=\sqrt{-1} \partial \bar{\partial} \log \left( |\tilde{z}_0^2|+|\tilde{z}_1|^2+\cdots +|\tilde{z}_N|^2 \right).
  \end{align*}
  Let $\iota$ be the Kodaira embedding map induced by $\left\{ s_i \right\}_{i=0}^{N}$ at time $0$. Then it is clear that
 \begin{align*}
   \tilde{\omega}_0=\iota^* \omega_{FS},  \quad \tilde{\omega}_t=\iota^* (G^*\omega_{FS}).
 \end{align*}
 Therefore, we have
 \begin{align*}
    \left. \frac{d}{dt} \tilde{\omega}_t \right|_{t=0}=\iota^*\left(  \left. \frac{d}{dt}G^*(\omega_{FS}) \right|_{t=0} \right).
 \end{align*}
 So (\ref{eqn:HB03_1}) is reduced to the estimate
 \begin{align}
  -2A \omega_{FS} \leq   \left. \frac{d}{dt}G(t)^*(\omega_{FS}) \right|_{t=0} \leq 2A \omega_{FS}.
  \label{eqn:HB05_5}
 \end{align}
 However, note that
\begin{align*}
  \delta_{ik}=G_{ij}\bar{G}_{kl}  \int_{M}  \langle s_j, s_l\rangle_{h_t} \frac{\omega_t^n}{n!}.
\end{align*}
Taking derivative on both sides at time $0$ and denote $\dot{G}$ by $B$,  we obtain
\begin{align*}
  0=B_{ik}+\bar{B}_{ki} + \int_M (-\dot{\varphi}+n\lambda-R)\langle s_i,s_k\rangle_{h_0} \frac{\omega_0^n}{n!}.
\end{align*}
Therefore, for every $v\in \C^{N+1}$, the following inequality holds.
\begin{align}
  \left| v_i (B_{ij} + \bar{B}_{ji}) \bar{v}_j \right|  =\left|-v_i\bar{v}_j \int_M (-\dot{\varphi}+n\lambda-R)\langle s_i,s_j\rangle_{h_0} \frac{\omega_0^n}{n!} \right|\leq A|v|^2.
  \label{eqn:HB05_1}
\end{align}
In particular, each eigenvalue of the Hermitian matrix $B+\bar{B}^{\tau}$ has absolute value bounded by $A$.  Then (\ref{eqn:HB05_5}) follows from
Lemma~\ref{lma:HB02_1}.
\end{proof}

In view of Lemma~\ref{lma:HB05_1}, the following property is obvious now.

\begin{proposition}[\textbf{Bergman metric equivalence along time}]
 Suppose $\mathcal{LM} \in \mathscr{K}(n,A)$.  Then we have
 \begin{align}
    e^{-2A|t|} \tilde{\omega}_0 \leq \tilde{\omega}_t \leq e^{2A|t|} \tilde{\omega}_0.
 \label{eqn:HB03_2}
 \end{align}
\label{prn:HB05_2}
\end{proposition}

In general, we cannot hope a powerful estimate like (\ref{eqn:HB03_2}) holds for metrics $\omega_t$, since such  an estimate will imply the Ricci curvature is uniformly bounded by $A$.   However,  if we only focus  on points regular enough, then we do have a similar weaker estimate.

\begin{proposition}[\textbf{Flow metric equivalence along time}]
 Suppose $\mathcal{LM} \in \mathscr{K}(n,A)$,  $x \in  \mathcal{F}_r(M,0)$.  Then we have
 \begin{align}
     \frac{1}{C} \omega_0(x) \leq   \omega_t(x) \leq C\omega_0(x)  \label{eqn:HB05_2}
 \end{align}
 for every $t \in [-1,1]$. Here $C$ is a constant depending only on $n,A$ and $r$.
 \label{prn:HB05_1}
\end{proposition}

\begin{proof}
Without loss of generality, we assume $\mathbf{b}$ is uniformly bounded.

  By short time two-sided pseudolocality, Theorem~\ref{thm:SL27_2} and rescaling, it suffices to show (\ref{eqn:HB05_2}) for $t=-1$ and $t=1$.
  At time $0$, it is clear that $\omega_0(x)$ and $\tilde{\omega}_0(x)$ are uniformly equivalent.  The volume form
  $\omega_0^n$ is uniformly equivalent to $\omega_t^n$.  By the stability of $\tilde{\omega}$, inequality (\ref{eqn:HB03_2}), it suffices to prove
  the following two inequalities hold at point $x$.
  \begin{align}
  &  \Lambda_{\omega_1} \tilde{\omega}_0 \leq C,    \label{eqn:HB05_3}\\
  & \Lambda_{\omega_0} \tilde{\omega}_{-1} \leq C.  \label{eqn:HB05_4}
  \end{align}
  We shall prove the above two inequalities separately.

  Let $w_0$ be defined as that before Lemma~\ref{lma:SK23_1}.
  Let $w$ be the solution of $\square w=0$, initiating from $w_0$.   By the heat kernel estimate and the uniform upper
 bound of diameter of $B_{g(0)}(x,r)$ under metric $g(t)$(c.f. Lemma~\ref{lma:HA14_1}), we see that $w(x,1)$ is uniformly bounded away from $0$.
 Then Lemma~\ref{lma:J02_1} applies and we obtain that
 \begin{align*}
  \Lambda_{\omega_1(x)} \tilde{\omega}_0(x) =F(x,1) \leq \frac{C}{w(x,1)} <C.
 \end{align*}
  So we finish the proof of (\ref{eqn:HB05_3}).  The proof of (\ref{eqn:HB05_4}) is similar.  Modular time shifting, the only difference is that we do not know
  whether $x$ is very regular at time $t=-1$, so the construction of initial value of a heat equation may be a problem.
  However, due to Proposition~\ref{prn:SC02_2}, we can always find a point  $y_0 \in \mathcal{F}_{c_b \hslash}(M,-1) \cap B_{g(-1)}(x,\hslash)$. Consider the heat equation $w'$, starting from a cutoff function supported around $y_0$ at time $t=-1$.
 In light of uniform diameter bound of $B_{g(-1)}(y_0,\hslash)$ under the metric $g(0)$, $w'(x,0)$ is uniformly bounded away from $0$.  So we can follow the proof of
 Lemma~\ref{lma:J02_1} to obtain that
  \begin{align*}
       \Lambda_{\omega_0(x)} \tilde{\omega}_{-1}(x) <\frac{C}{w'(x,0)}<C.
 \end{align*}
 Therefore, (\ref{eqn:HB05_4}) is proved.

\end{proof}

Note that due to the two-sided pseudolocality, Theorem~\ref{thm:SL27_2},  we now can use blowup argument, taking for granted that every convergence in regular part takes place in
smooth topology.  Therefore, we can use the blowup argument in the proof of Proposition~\ref{prn:HE11_2},
based on the Liouville type theorem,  Lemma~\ref{lma:HE11_1}.
Then the following corollary follows directly from Proposition~\ref{prn:HB05_1}.

\begin{corollary}[\textbf{Long-time regularity improvement in two time directions}]
 Suppose $\mathcal{LM} \in \mathscr{K}(n,A)$, $r>0$, then
 \begin{align*}
    \mathcal{F}_r(M,0)  \subset \bigcap_{-1 \leq t \leq 1}  \mathcal{F}_{\delta}(M,t)
 \end{align*}
 for some $\delta=\delta(n,A,r)$.
\label{cly:HB05_1}
\end{corollary}

Now we are ready to prove Theorem~\ref{thmin:HC06_1}, the long-time, two-sided pseudolocality theorem.

\begin{proof}[Proof of Theorem~\ref{thmin:HC06_1}]
  It follows from the combination of Corollary~\ref{cly:HB05_1} and Proposition~\ref{prn:HA08_3}.
\end{proof}

Suppose $\mathcal{LM}_i \in \mathscr{K}(n,A)$, $x_i \in M_i$.  Then for each time $t \in [-1,1]$ we have
\begin{align}
   (M_i,x_i,g_i(t)) \stackrel{\hat{C}^{\infty}}{\longrightarrow} (\bar{M}(t),\bar{x}(t),\bar{g}(t)).
\label{eqn:HE15_1}
\end{align}
Let us see how are the two time slice limits $\bar{M}(0)$ and $\bar{M}(1)$ related.
Clearly, by Theorem~\ref{thmin:HC06_1}, the regular parts of $\bar{M}(0)$ and $\bar{M}(1)$ can be identified.
The relations among the singular parts at different time slices are more delicate.
For simplicity of notations, we denote $(\bar{M}(0), \bar{x}(0), \bar{g}(0))$ by $(\bar{M},\bar{x},\bar{g})$,
denote $(\bar{M}(1), \bar{x}(1), \bar{g}(1))$ by $(\bar{M}', \bar{x}', \bar{g}')$.
Let us also assume $\Vol(M_i)$ is uniformly bounded. Then it is clear that both $\bar{M}$
and $\bar{M}'$ are compact by the uniform non-collapsing caused by Sobolev constant bound.
In light of the uniform partial-$C^0$-estimate along the flow,  without loss of generality,
we can assume that the Bergman function $\mathbf{b}$ is uniformly bounded below.  By the fundamental estimates
in~\cite{DS}, we obtain that the map
\begin{align*}
  Id_{0}:   \quad (\bar{M}, \bar{x}, \bar{g}) \to (\bar{M}, \bar{x}, \tilde{\bar{g}})
\end{align*}
is a homeomorphism. Recall that $(\bar{M}, \bar{x}, \tilde{\bar{g}})$ is the limit of $(M_i, x_i, \tilde{g}_i(0))$, where
$\tilde{g}_i$ is the pull-back of Fubini-Study metric.   Similarly, we have another homeomorphism map at time
$t=1$.
\begin{align*}
  Id_{1}:   \quad (\bar{M}', \bar{x}', \bar{g}') \to (\bar{M}', \bar{x}', \tilde{\bar{g}}').
\end{align*}
By Proposition~\ref{prn:HB05_2}, the pulled back Fubini-Study metrics $\tilde{g}_i(t)$ are uniformly equivalent
for $t \in [-1,1]$.  It follows that there is a Lipschitz map $Id_{01}$ between two time slices, for the pulled back
Fubini-Study metrics:
\begin{align*}
   Id_{01}: \quad (\bar{M}, \bar{x}, \tilde{\bar{g}}) \to (\bar{M}', \bar{x}', \tilde{\bar{g}}').
\end{align*}
Combining the previous steps and letting $\Psi=Id_{1}^{-1} \circ Id_{01} \circ Id_{0}$, we obtain that the map
\begin{align*}
   \Psi: (\bar{M}, \bar{x}, \bar{g}) \to (\bar{M}', \bar{x}', \bar{g}')
\end{align*}
is a homeomorphism.
By analyzing each component identity map,
it is  clear that $\Psi|_{\mathcal{R}(\bar{M})}$,  where $\mathcal{R}(\bar{M})$ is the regular part of $\bar{M}$,
maps $\mathcal{R}(\bar{M})$ to $\mathcal{R}(\bar{M}')$, as a biholomorphic map.
Similarly, $\Psi|_{\mathcal{S}(\bar{M})}$ is a homeomorphism to $\mathcal{S}(\bar{M}')$.
Therefore, the variety structure of the $\bar{M}(t)$ does not depend on time.
 We remark that the compactness of $\bar{M}$ is not essentially used here.
 If $\bar{M}$ is noncompact, the above argument go through formally
if we replace the target embedding space $\CP^{N}$ by $\CP^{\infty}$.
This formal argument can be made rigorous by applying delicate localization technique.
However, in our applications, $\bar{M}$ is always compact except it is a bubble, i.e., a blowup limit.
In this situation, we have the extra condition $|R|+|\lambda| \to 0$,
then $\Psi$ can be easily chosen as identity map, due to Proposition~\ref{prn:HE15_1}.

From the above discussion, it is clear that the topology structure and variety structure of $\bar{M}(t)$ does not  depend on time.
So we just denote $\bar{M}(t)$ by $\bar{M}$. Then we can denote the convergence (\ref{eqn:HE15_1}) by
\begin{align*}
   (M_i,x_i,g_i(t)) \stackrel{\hat{C}^{\infty}}{\longrightarrow} (\bar{M},\bar{x}(t),\bar{g}(t))
\end{align*}
for each $t$.
Hence, the limit family of metric spaces can be regarded as a family of evolving metrics on the limit
variety.  Therefore, the above convergence at each time $t$ can be glued together to obtain a global convergence
\begin{align*}
  \left\{ (M_i,x_i,g_i(t)), -T_i < t < T_i\right\}  \stackrel{\hat{C}^{\infty}}{\longrightarrow}  \left\{ (\bar{M}, \bar{x}, \bar{g}(t)), -\bar{T} <t <\bar{T}\right\}
\end{align*}
where $\displaystyle \bar{T}=\lim_{i \to \infty} t_i$.  Clearly, $\bar{g}(t)$ satisfies the K\"ahler Ricci flow equation on the regular part of $\bar{M}$.
Recall that we typically denote the K\"ahler Ricci flow $\left\{ (M_i,x_i,g_i(t)), -T_i < t < T_i\right\}$ by $\mathcal{M}_i$.  Then we obtain the convergence of
K\"ahler Ricci flows (with base points):
\begin{align}
  (\mathcal{M}_i, x_i)  \stackrel{\hat{C}^{\infty}}{\longrightarrow}  (\bar{\mathcal{M}}, \bar{x}).
\label{eqn:HE15_2}
\end{align}
If we further know the underlying space $\bar{M}$ is compact, then the notation can be even simplified as
\begin{align*}
  \mathcal{M}_i  \stackrel{\hat{C}^{\infty}}{\longrightarrow} \bar{\mathcal{M}}.
\end{align*}

\begin{remark}
 The limit flow $\bar{\mathcal{M}}$ can be regarded as an intrinsic K\"ahler Ricci flow on the normal variety $\bar{M}$.
 Actually, it is already clear that $\bar{\mathcal{M}}$ is at least a
 weak super solution of Ricci flow, in the sense of R.J. McCann and  P.M. Topping(\cite{MccTop}).
 From the point of view of K\"ahler geometry, when restricted to the potential level,
 the flow  $\bar{\mathcal{M}}$ coincides with the weak K\"ahler Ricci flow solution defined by
 Song and Tian(\cite{SongTian}), if $\bar{M}$ is compact.
\label{rmk:HE15_1}
\end{remark}

If we also consider the convergence of the line bundle structure, we can obviously generalize the convergence in (\ref{eqn:HE15_2})
as

\begin{align*}
  & \left(\mathcal{LM}_i, x_i \right) \stackrel{\hat{C}^{\infty}}{\longrightarrow} \left(\overline{\mathcal{LM}}, \bar{x}\right),
  \quad \textrm{if $\bar{M}$ is non-compact}.  \\
  & \mathcal{LM}_i \stackrel{\hat{C}^{\infty}}{\longrightarrow} \overline{\mathcal{LM}}, \quad \textrm{if $\bar{M}$ is compact}.
\end{align*}
With these notations, we  can formulate our compactness theorem as follows.

\begin{theorem}[\textbf{Polarized flow limit}]
  Suppose $\mathcal{LM}_i \in \mathscr{K}(n,A)$, $x_i \in M_i$.  Then we have
  \begin{align*}
       \left(\mathcal{LM}_i, x_i \right) \stackrel{\hat{C}^{\infty}}{\longrightarrow} \left(\overline{\mathcal{LM}}, \bar{x}\right),
  \end{align*}
  where $\overline{\mathcal{LM}}$ is a polarized K\"ahler Ricci flow solution on an analytic normal variety $\bar{M}$.
  Moreover, if $\bar{M}$ is compact, then it is a projective normal variety.
\label{thm:HB07_1}
\end{theorem}

Notice that we have already proved Theorem~\ref{thmin:HC06_2} now.

\begin{proof}[Proof of Theorem~\ref{thmin:HC06_2}]
The limit polarized flow on variety follows from the combination of Theorem~\ref{thm:HB07_1} and Theorem~\ref{thm:HE19_1}.
The Minkowski dimension estimate of the singular set follows from Corollary~\ref{cly:HE25_2}.
\end{proof}

The properties of the limit spaces  can be improved if extra conditions are available.

\begin{proposition}[\textbf{KE limit}]
 Suppose $\mathcal{LM}_i \in \mathscr{K}(n,A)$ satisfies
 \begin{align}
 \int_{-T_i}^{T_i}\int_{M_i} |R-n\lambda| dv dt \to 0.
 \label{eqn:HA03_1}
 \end{align}
Then $\overline{\mathcal{LM}}$ is a static, polarized K\"ahler Ricci flow solution.  In other words, $\bar{g}(t) \equiv \bar{g}(0)$
and consequently are K\"ahler Einstein metric.
\label{prn:HB07_1}
\end{proposition}

Suppose $\mathcal{LM} \in \mathscr{K}(n,A)$  and $\lambda>0$.  Then it is clear that $c_1(M)>0$, or $M$ is Fano.
Note that for every Fano manifold, we have a uniform bound $c_1^n(M) \leq C(n)$(c.f.~\cite{Deba}).  This implies that
\begin{align*}
 \frac{1}{A}  \leq \Vol(M)=c_1^n(L)=\lambda^{-n} c_1^n(M) \leq C \lambda^{-n}.
\end{align*}
So $\lambda$ is bounded away from above.   If we assume $\lambda$ is bounded away from zero, then
$\Vol(M)=c_1^{n}(L)$ is uniformly bounded. Consequently,  $\diam(M)$ is uniformly bounded by non-collapsing, due to the Sobolev constant bound.  Therefore, if we have a sequence of $\mathcal{LM}_i \in \mathscr{K}(n,A)$
with $\lambda_i>\lambda_0>0$, we can always assume
\begin{align*}
  \lambda_i \to \bar{\lambda}>0, \quad  \mathcal{LM} \stackrel{\hat{C}^{\infty}}{\longrightarrow} \overline{\mathcal{LM}}
\end{align*}
without considering the base points.

\begin{proposition}[\textbf{KRS limit}]
 Suppose $\mathcal{LM}_i \in \mathscr{K}(n,A)$ satisfies
  \begin{align}
 \lambda_i>\lambda_0>0, \quad
 \mu\left(M_i, g(T_i), \frac{\lambda_i}{2} \right)-\mu\left(M_i, g(-T_i), \frac{\lambda_i}{2} \right) \to 0,
\label{eqn:HA03_2}
\end{align}
 where $\mu$ is Perelman's $W$-functional.   Suppose $\overline{\mathcal{LM}}$
 is the limit of $\mathcal{LM}$.
Then $\overline{\mathcal{M}}$ is a  gradient shrinking K\"ahler Ricci soliton.  In other words, there is a smooth real valued function $\hat{f}$
defined on $\mathcal{R}(\bar{M}) \times (-\bar{T},\bar{T})$ such that
\begin{align}
    \hat{f}_{jk}=\hat{f}_{\bar{j}\bar{k}}=0,  \quad  R_{j\bar{k}} + \hat{f}_{j\bar{k}}-\hat{g}_{j\bar{k}}=0.
 \label{eqn:HB07_1}
 \end{align}
\label{prn:HB07_2}
\end{proposition}

\begin{proof}
 Without loss of generality, we may assume $\lambda_i=1$.
 Let $\mathcal{LM} \in \mathscr{K}(n,A)$.  At time $t=1$, let $u$ be the minimizer of Perelman's $\mu$-functional.
 Then solve the backward heat equation $\square^* u= (-\partial_{t}-\Delta + R-n\lambda) u=0$.
 Let $f$ be the function such that $(2\pi)^{-n}e^{-f}=u$.  Then we have
 \begin{align*}
  &\quad  \int_{-1}^{1}\int_M (2\pi)^{-n} \left\{\left| R_{j\bar{k}}+f_{j\bar{k}}-g_{j\bar{k}}\right|^2 + |f_{jk}|^2+|f_{\bar{j}\bar{k}}|^2\right\} e^{-f} dv\\
  &\leq \mu\left(M,g(1),\frac{1}{2} \right)-\mu\left(M,g(-1),\frac{1}{2}\right)\\
  &\leq \mu\left(M,g(T),\frac{1}{2} \right)-\mu\left(M,g(-T),\frac{1}{2}\right) \to 0.
 \end{align*}
 At time $t=1$, $f$ has good regularity estimate for it is a solution of an elliptic equation.
 For $t \in (-1,1)$, we have estimate of $f$ from heat kernel estimate.
 It is not hard to see that, on the space-time domain $\mathcal{R} \times (-1,1)$, $f$ converges to a limit function $\hat{f}$ satisfying (\ref{eqn:HB07_1}).
 Clearly, the time interval of $(-1,1)$ can be replaced by $(-a,a)$ for every $a \in (1,\bar{T})$.  For each $a$, we have a limit function
 $\hat{f}^{(a)}$, which satisfies equation (\ref{eqn:HB07_1}) and therefore has enough a priori estimates.
 Then let $a \to \bar{T}$ and take diagonal sequence limit, we obtain a limit function $\hat{f}^{(\bar{T})}$ which satisfies (\ref{eqn:HB07_1})
 on $\mathcal{R} \times (-\bar{T}, \bar{T})$.  Without loss of generality, we still denote $\hat{f}^{(\bar{T})}$ by $\hat{f}$.
 Then $\hat{f}$ satisfies (\ref{eqn:HB07_1}) on $\mathcal{R} \times (-\bar{T}, \bar{T})$.
\end{proof}

\begin{remark}
It is an interesting problem to see whether $(\bar{M}, \bar{g}(0))$ is a conifold in Theorem~\ref{thm:HB07_1}.
This question has affirmative answer when we know $(\bar{M}, \bar{g}(0))$ has Einstein regular part, following the proof of
Theorem~\ref{thm:HE08_1} and Proposition~\ref{prn:SC30_1}.
In particular, the limit spaces in Proposition~\ref{prn:HB07_1} and Proposition~\ref{prn:HE15_1} are K\"ahler Einstein conifolds.
\label{rmk:HE20_1}
\end{remark}

\section{Applications}

 In this section, we will focus on the applications of our structure theory to the study
 of anti-canonical K\"ahler Ricci flows.

\subsection{Convergence of anti-canonical K\"ahler Ricci flows at time infinity}

Based on the structure theory, Theorem~\ref{thmin:SC24_3} can be easily proved.

\begin{proof}[Proof of Theorem~\ref{thmin:SC24_3}]
In view of the fundamental estimate of Perelman (c.f.~\cite{SeT}),  in order (\ref{eqn:SK20_1}) to hold, we only need a Sobolev constant bound,
which was proved by Q. Zhang (c.f.\cite{Zhq1}) and R. Ye (c.f. \cite{Ye}).
Therefore, the truncated flow sequences locate in $\mathscr{K}(n,A)$ for a uniform $A$.
It follows from Theorem~\ref{thmin:HC06_2} that the limit K\"ahler Ricci flow exists on a compact projective normal variety.
The limit normal variety is $Q$-Fano since it has a limit anti-canonical polarization.
According to Proposition~\ref{prn:HB07_2},
the boundedness and monotonicity of Perelman's $\mu$-functional force the limit flow to be a K\"ahler Ricci soliton.
The volume estimate of $r$-neighborhood of $\mathcal{S}$ follows from Corollary~\ref{cly:HA11_1}
and estimate (\ref{eqn:SB13_1}).
\end{proof}

We continue to discuss applications beyond Theorem~\ref{thmin:SC24_3}.
The following property is well known to experts, we write it down here for the convenience of the readers.

  \begin{proposition}[\textbf{Connectivity of limit moduli}]
    Suppose $\mathcal{M}=\left\{ (M^n, g(t)), 0 \leq t <\infty \right\}$ is an anti-canonical  K\"ahler Ricci flows on Fano manifold
 $(M, J)$.  Let $\mathscr{M}$ be the collection of all the possible limit space along this flow.  Then $\mathscr{M}$ is connected.
   \label{prn:SL06_1}
   \end{proposition}

  \begin{proof}

      If the statement was wrong, we have two limit spaces $\bar{M}_a$ and $\bar{M}_b$, locating in different
      connected components of $\mathscr{M}$.  Let $\mathscr{M}_a$ be the connected component containing
      $\bar{M}_a$. Since $\mathscr{M}_a$ is a connected component, it is open and closed.
      So its closure $\overline{\mathscr{M}_a}$ is the
      same as $\mathscr{M}_a$.  Clearly,  $\mathscr{M}_a$ is compact under the Gromov-Hausdorff topology.
      Define
      \begin{align}
       &d(X,\mathscr{M}_a) \triangleq \inf_{Y \in \mathscr{M}_a} d_{GH}(X,Y),  \label{eqn:SL06_7}\\
       &\eta_a \triangleq \inf_{X \in \mathscr{M} \backslash \mathscr{M}_a} d(X, \mathscr{M}_a).  \label{eqn:SL06_8}
      \end{align}
      Clearly, $\eta_a>0$ by the compactness of $\mathscr{M}_a$ and the fact that $\mathscr{M}_a$ is a connected component.

      Without loss of generality, we can assume $(M, g(t_i))$ converges to $\bar{M}_a$,
      $(M, g(s_i))$ converges to $\bar{M}_b$, for $t_i \to \infty$ and $s_i >t_i$.
      For simplicity of notation, we denote $(M,g(t_i))$ by $M_{t_i}$, $(M,g(s_i))$ by $M_{s_i}$.
      For large $i$, we have
      \begin{align}
        d_{GH}(M_{t_i}, \bar{M}_a)<\frac{\eta_a}{100}, \quad
        d_{GH}(M_{s_i}, \bar{M}_b) < \frac{\eta_a}{100}.  \label{eqn:SL06_6}
      \end{align}
      In particular,  the above inequalities imply that
      \begin{align*}
        d(M_{t_i}, \mathscr{M}_a)<\frac{\eta_a}{100}, \quad d(M_{s_i}, \mathscr{M}_a)>\frac{99}{100}\eta_a.
      \end{align*}
     By continuity of the flow, we can find $\theta_i \in (t_i, s_i)$ such that
     $d(M_{\theta_i}, \mathscr{M}_a)=\frac{1}{2}\eta_a$, whose limit form is
     \begin{align}
        d(\bar{M}_c, \mathscr{M}_a)=\frac{1}{2}\eta_a,
     \label{eqn:SL06_10}
     \end{align}
     where $\bar{M}_c$ is the limit of $M_{\theta_i}$.    However, (\ref{eqn:SL06_10}) contradicts with
     (\ref{eqn:SL06_8}) and the fact $\eta_a>0$.
  \end{proof}

 Proposition~\ref{prn:SL06_1} can be generalized as follows.

  \begin{proposition}[\textbf{KRS limit moduli}]
     Suppose $\mathcal{M}_s=\left\{ (M_s^n, g_s(t)), 0 \leq t <\infty,  s \in X \right\}$
      is a smooth family of anti-canonical  K\"ahler Ricci flows on Fano manifolds $(M_s, J_s)$, where
      $X$ is a connected parameter space.
     We call $(\bar{M}, \bar{g})$ as a limit space if  $(\bar{M},\bar{g})$ is the Gromov-Hausdorff limit of $(M, g_{s_i}(t_i))$ for some
     $t_i \to \infty$ and $s_i \to \bar{s} \in X$.

     Suppose $\displaystyle f(s)=\lim_{t \to \infty} \mu \left(g_s(t), \frac{1}{2} \right)$ is an upper semi-continuous function on $X$.
     Then we have the following properties.
     \begin{itemize}
     \item Every limit space is a K\"ahler Ricci soliton.
     \item  Let $\widetilde{\mathscr{M}}$ be the collection of all the limit spaces.
               Then $\widetilde{\mathscr{M}}$ is connected under the Gromov-Hausdorff topology.
     \end{itemize}
 \label{prn:SL07_2}
  \end{proposition}

  \begin{proof}
   We shall only show that every limit space is a K\"ahler Ricci soliton.
   The connectedness of  $\widetilde{\mathscr{M}}$ can be proved almost the same as Proposition~\ref{prn:SL06_1}.
   So we leave the details to the readers.

    Suppose $s_i \to \bar{s}$.  Fix $\epsilon$, we can choose $T_{\epsilon}$ such that
    \begin{align*}
         \mu\left(g_{\bar{s}}(T_{\epsilon}), \frac{1}{2} \right)>f_{\bar{s}}-\epsilon.
    \end{align*}
    By the smooth convergence of $g_{s_i}(T_{\epsilon})$ and the upper semi-continuity of $f$, we have
    \begin{align*}
        \mu\left(g_{s_i}(T_{\epsilon}), \frac{1}{2} \right) >f_{s_i}-\epsilon
    \end{align*}
    for large $i$.  Recall that  $t_i \to \infty$. Therefore, it follows from the monotonicity of Perelman's functional that
    \begin{align*}
         \mu\left(g_{s_i}(T_{\epsilon}), \frac{1}{2} \right)<\mu\left(g_{s_i}(t_i-1), \frac{1}{2} \right)<\lim_{t \to \infty}\mu\left(g_{s_i}(t), \frac{1}{2} \right)=f_{s_i}.
    \end{align*}
    Hence, we have
    \begin{align*}
      0\leq  \mu\left(g_{s_i}(t_i+1), \frac{1}{2} \right)-\mu\left(g_{s_i}(t_i-1), \frac{1}{2} \right)  <\epsilon
    \end{align*}
    for large $i$. By the arbitrary choice of $\epsilon$, we obtain
    \begin{align*}
       \lim_{i \to \infty} \left\{  \mu\left(g_{s_i}(t_i+1), \frac{1}{2} \right)-\mu\left(g_{s_i}(t_i-1), \frac{1}{2} \right) \right\}=0.
    \end{align*}
    Therefore, $(M, g_{s_i}(t_i))$ converges to a K\"ahler Ricci soliton, in light of  Proposition~\ref{prn:HB07_2}.
  \end{proof}

  The gap between singularity and regularity in Theorem~\ref{thmin:SC24_1} has a global version as follows.

  \begin{proposition}[\textbf{Gap around smooth KE}]
   Suppose $(\tilde{M},\tilde{g},\tilde{J})$ is a compact, smooth K\"ahler Einstein manifold which belongs to $\mathscr{K}(n,A)$
   when regarded as a trivial polarized K\"ahler Ricci flow solution.
   Then there exists an $\epsilon=\epsilon(n,A,\tilde{g})$ with the following properties.

   Suppose $\mathcal{LM} \in \mathscr{K}(n,A)$ and $ d_{GH}((\tilde{M},\tilde{g}), (M,g(0))) < \epsilon$, then we have
   \begin{align*}
     \mathbf{vcr}(M, g(0)) >\frac{1}{2}\mathbf{vcr}(\tilde{M},\tilde{g}).
   \end{align*}
  \label{prn:SL07_3}
  \end{proposition}
  \begin{proof}
    It follows from the continuity of canonical volume radius under the Cheeger-Gromov convergence.
  \end{proof}

  Proposition~\ref{prn:SL07_3}  means that  there is no singular limit space around any given smooth K\"ahler Einstein manifold.
  Clearly, the single smooth K\"ahler Einstein manifold in this Proposition
  can be replaced by a family of smooth K\"ahler Einstein manifolds with bounded geometry.
  The gap between smooth and singular K\"ahler Einstein metrics can be conveniently used to carry out topology argument.

   \begin{theorem}[\textbf{Convergence of KRF family}]
     Suppose $\mathcal{M}_s=\left\{ (M_s^n, g_s(t), J_s), 0 \leq t <\infty,  s \in X \right\}$
      is a smooth family of anti-canonical  K\"ahler Ricci flows on Fano manifolds $(M_s, J_s)$, where
      $X$ is a connected parameter space.      Moreover, we assume that
  \begin{itemize}
  \item The Mabuchi's K-energy is  bounded from below along each flow.
  \item Smooth K\"ahler Einstein metrics in all adjacent complex structures(c.f. Defintion 1.4 of~\cite{CS1}) have uniformly bounded Riemannian curvature.
  \end{itemize}
     Let $\Omega$ be the collection of $s$ such that the flow
      $g_s$ has bounded Riemannian curvature.  Then $\Omega=\emptyset$ or $\Omega=X$.
   \label{thm:HC03_2}
   \end{theorem}

   \begin{proof}
     It suffices to show that $\Omega$ is both open and closed in $X$.

    The openness follows from the stability of K\"ahler Ricci flow around a given smooth K\"ahler Einstein metric, due to Sun and Wang (c.f.~\cite{SW}).
    Suppose $s \in \Omega$, then the flow $g_s$ converges to some K\"ahler Einstein manifold $(M',g',J')$, which is the unique
    K\"ahler Einstein metric in its small smooth neighborhood.   By continuous dependence of flow on the initial data, and the stability of K\"ahler Ricci flow in a
    very small neighborhood of $(M',g',J')$,
     it is clear that $s$ has a neighborhood consisting of points in $\Omega$. Therefore, $\Omega$ is an open subset of $X$.

     The closedness follows from Proposition~\ref{prn:SL07_2}.
     Suppose $s_i \in \Omega$ and $s_i \to \bar{s} \in X$.
   Due to the fact that the Mabuchi's K-energy is bounded from below along each K\"ahler Ricci flow we are concerning now, the
   limit Perelman functional is always the same(c.f.\cite{CS1}).  Therefore, we can apply Proposition~\ref{prn:SL07_2} to show
   that every limit space is a possibly singular K\"ahler Einstein. However,  along every $g_{s_i}$, we obtain a smooth
    limit K\"ahler Einstein manifold
   $(M',g',J')$, which has uniformly bounded curvature, as a K\"ahler Einstein manifold in an adjacent complex structure.
    Note that the diameter of $M'$ is uniformly bounded by Myers theorem.
   The volume of $M'$ is a topological constant.  Therefore, the geometry of $(M',g')$ are uniformly bounded.
   By a generalized version of Proposition~\ref{prn:SL07_3}, $(M',g',J')$ is uniformly bounded away from singular K\"ahler Einstein metrics.
   Due to Proposition~\ref{prn:SL07_2}, the connectedness of $\mathscr{M}$ forces that the flow
   $g_{\bar{s}}$ must converge to a smooth $(M',g',J')$. In particular, $g_{\bar{s}}$ has bounded curvature.
   Threrefore,  $\bar{s} \in \Omega$ and $\Omega$ is closed.
   \end{proof}

The two assumptions in Theorem~\ref{thm:HC03_2} seem to be artificial. However, if $J_s$ is a trivial family or a test configuration family, by the unique
degeneration theorem of Chen-Sun(c.f.~\cite{CS1}), all the smooth K\"ahler Einstein metrics form an isolated family,
then the second condition is satisfied automatically. On the other hand, by the existence of K\"ahler Einstein metrics in the weak sense,
one can also obtain the lower bound of Mabuchi's $K$-energy(c.f.~\cite{BM},~\cite{DT},~\cite{C08}).
Consequently, Theorem~\ref{thm:HC03_2} can be applied to these special cases and obtain the following corollaries.

\begin{corollary}[\textbf{Convergence to given KE}, c.f. Tian-Zhu~\cite{TZ1}, Collins-Sz\'{e}kelyhidi~\cite{CoSz}]
 Suppose $(M,J)$ is a Fano manifold with a K\"ahler Einstein metric $g_{KE}$. Then every anti-canonical K\"ahler Ricci flow on $(M,J)$ converges to
 $(M,g_{KE},J)$.
\label{cly:HC04_0}
\end{corollary}
\begin{proof}
 Let $\omega_{KE}$ be the K\"ahler Einstein metric form. Then every metric form $\omega$ can be written as $\omega_{KE}+\sqrt{-1} \partial \bar{\partial}\varphi$
 for some smooth function $\varphi$. Define
 \begin{align*}
    \omega_s= \omega_{KE} + s \sqrt{-1} \partial \bar{\partial} \varphi, \quad s \in [0,1].
 \end{align*}
 It follows from Theorem~\ref{thm:HC03_2} that the K\"ahler Ricci flow from every $\omega_s$ has bounded curvature, and consequently
 converges to $\omega_{KE}$, by the uniqueness theorem of Chen-Sun(c.f.~\cite{CS1}).  In particular, the flow
 start from $\omega$ converges to $\omega_{KE}$.
\end{proof}

\begin{corollary}[\textbf{Convergence of a test configuration}]
   Suppose $\mathbf{M}$ is a smooth test configuration, i.e., a family of Fano manifolds $(M_s,J_s)$
   parameterized by  $s$ in unit disk $D \subset \C^1$ with a natural $C^*$-action.
   Suppose each fiber is smooth and the central fiber $(M_0,g_0,J_0)$
   admits K\"ahler Einstein metric $(M_0,g_{KE},J_0)$.  Then each K\"ahler Ricci flow starting from $(M_s, g_s,J_s)$ for arbitrary $s \in D$
    converges to $(M_0, g_{KE},J_0)$.
\label{cly:HC04_1}
\end{corollary}

\begin{proof}
Theorem~\ref{thm:HC03_2} can be applied for $X=D$.
 The central K\"ahler Ricci flow converges by Corollary~\ref{cly:HC04_0}.  Therefore,  the K\"ahler Ricci flow on each fiber has bounded curvature and converge to
 some smooth K\"ahler Einstein metric, which can only be $(M_0,g_{KE},J_0)$, due to the uniqueness theorem of Chen-Sun again.
\end{proof}

\begin{remark}
Corollary~\ref{cly:HC04_0} was announced by G.Perelman. The first written proof was given by Tian-Zhu in~\cite{TZ1} whenever there is no non-trivial holomorphic vector field.
The general case was proved by Collins-Sz\'{e}kelyhidi in~\cite{CoSz}.   
The strategy of Corollary~\ref{cly:HC04_0} was inspired by that in~\cite{TZ2}.
Corollary~\ref{cly:HC04_0}-Corollary~\ref{cly:HC04_1} have the corresponding K\"ahler Ricci soliton versions. These generalizations will be discussed in a separate paper.
\label{rmk:HC09_1}
\end{remark}

\subsection{Degeneration of anti-canonical K\"ahler Ricci flows}

In this subsection, we shall prove Theorem~\ref{thmin:HA01_1} and related corollaries.

The following Theorem is due to Jiang(c.f.~\cite{Jiang}).

\begin{theorem}[\textbf{Jiang's estimate}]
 Suppose $\mathcal{M}=\{(M^n, g(t), J), 0 \leq t <\infty\}$ is an anti-canonical K\"ahler Ricci flow solution satisfying \begin{align}
   \norm{Ric^{-}}{C^0(M)} + |\log \Vol(M)| + C_S(M, g(0)) \leq F
 \label{eqn:HA01_1}
 \end{align}
 at time $t=0$.  Then we have
 \begin{align}
    |R|+ |\nabla \dot{\varphi}|^2  \leq \frac{C}{t^{n+1}}
 \label{eqn:HA01_2}
 \end{align}
 for some constant $C=C(n,F)$.
\label{thm:HA01_1}
\end{theorem}

Note that (\ref{eqn:HA01_2}) implies a uniform bound of diameter at each time $t>0$, by the uniform bound of
Perelman's functional.  Then one can easily deduce a uniform bound (depending on $t$) of $\norm{\dot{\varphi}}{C^1(M)}$.
Combing this with the Sobolev constant estimate along the flow(c.f.~\cite{Zhq1},~\cite{Ye}), we see that
\begin{align}
   \norm{R}{C^0(M)} + \norm{\dot{\varphi}}{C^1(M)} + C_S(M, g(t)) \leq C(n,F,t)
\label{eqn:HA01_3}
\end{align}
for each $t>0$.   Therefore,  away from the initial time, we can always apply our structure theory.

\begin{theorem}[\textbf{Weak convergence with initial time}]
 Suppose $\mathcal{M}_i=\{(M_i^n,g_i(t), J_i), 0 \leq t <\infty\}$ is a sequence of anti-canonical K\"ahler Ricci flow
 solutions, whose initial time slices satisfy estimate (\ref{eqn:HA01_1}) uniformly.  Then we have
 \begin{align}
     (\mathcal{M}_i, g_i) \stackrel{G.H.}{\longrightarrow} (\bar{\mathcal{M}}, \bar{g}),
 \end{align}
 where the limit is a weak K\"ahler Ricci flow solution on a $Q$-Fano normal variety $\bar{M}$, for time $t>0$.
 Moreover, the convergence can be improved to
 be in the $\hat{C}^{\infty}$-Cheeger-Gromov topology for each $t>0$, i.e.,
 \begin{align}
    (M_i, g_i(t)) \stackrel{\hat{C}^{\infty}}{\longrightarrow}  (\bar{M}(t), \bar{g}(t))
 \end{align}
 for each $t>0$.
\label{thm:HA01_2}
\end{theorem}

Clearly, if $(M_i,g_i)$ is a sequence of almost K\"ahler Einstein metrics(c.f.~\cite{TW2}) in the anti-canonical classes, then
$(\bar{M}(0), \bar{g}(0))$ and $(\bar{M}(1), \bar{g}(1))$ are isometric to each other, due to Proposition~\ref{prn:HB07_1} and the estimate in~\cite{TW2}.
In this particular case, it is easy to see that partial-$C^0$-estimate holds uniformly at time $t=0$ for each $i$, at least intuitively.
Actually, by the work Jiang~\cite{Jiang}, it is now clear that partial-$C^0$-estimate  at time $t=0$ only requires a uniform Ricci lower bound.

Note that the evolution equation of the anti-canonical K\"ahler Ricci flow is
\begin{align}
  \dot{\varphi}=\log \frac{\omega_{\varphi}^n}{\omega^n} + \varphi - u_{\omega},
\label{eqn:HA01_5}
\end{align}
where $u_{\omega}$ is the Ricci potential satisfying the normalization condition
$\int_M e^{-u_{\omega}} \frac{\omega^n}{n!}=(2\pi)^n$.
By maximum principle and Green function argument, we have the following property(c.f.~\cite{Jiang}).

\begin{proposition}[\textbf{Potential equivalence}]
  Suppose $\mathcal{M}=\{(M^n, g(t), J), 0 \leq t <\infty\}$ is an anti-canonical K\"ahler Ricci flow solution satisfying
  (\ref{eqn:HA01_1}). At time $t=0$, let $\varphi=0$ and $u_{\omega}$ satisfy the normalization condition. Then we have
  \begin{align}
     C(1-e^{t}) \leq  \varphi  \leq C e^t
  \label{eqn:HA01_4}
  \end{align}
  for a constant $C=C(n,F)$.
\label{prn:HA01_3}
\end{proposition}

Let $\mathbf{b}(\cdot, t)$ be the Bergman function at time $t$.   By definition, at point $x \in M$ and time $t=0$,
we can find a holomorphic section $S \in H^0(M, K_M^{-1})$ such that
\begin{align*}
    \int_{M} \norm{S}{h(0)}^2 \frac{\omega^n}{n!}=1, \quad\mathbf{b}(x,0)=\log \norm{S}{h(0)}^2(x).
\end{align*}
Note that  $\norm{S}{h(1)}^2=\norm{S}{h(0)}^2e^{-\varphi(1)}$.
By (\ref{eqn:HA01_4}), it is clear that  $\norm{S}{h(1)}^2$ and $\norm{S}{h(0)}^2$ are uniformly equivalent.
On the other hand, $\Delta \norm{S}{}^2 \geq -n \norm{S}{}^2$. At time $t=0$, applying Moser iteration implies that $\norm{S}{h(0)}^2 \leq C$.
Hence we obtain $\norm{S}{h(1)}^2 \leq C$.
At time $t=1$, let $\tilde{S}$ be the normalization of $S$, i.e., $\tilde{S}=\lambda S$ such that
$\int_{M} \norm{\tilde{S}}{h(1)}^2 \frac{\omega_1^n}{n!}=1$. Then we have
\begin{align*}
  \lambda^{-2}= \int_{M} \norm{S}{h(1)}^2 \frac{\omega_1^n}{n!} \leq C.
\end{align*}
It follows that
\begin{align*}
   \mathbf{b}(x,1) &\geq  \log  \norm{\tilde{S}}{h(1)}^2(x)=  \log  \norm{S}{h(1)}^2(x)+ \log \lambda^2\\
    &=\log  \norm{S}{h(0)}^2(x) -\varphi(1) + \log \lambda^2\\
    &=\mathbf{b}(x,0)-\varphi(1) + \log \lambda^2\\
    &\geq  \mathbf{b}(x,0)-C.
\end{align*}
By reversing time, we can obtain a similar inequality with reverse direction.
Same analysis applies to $\mathbf{b}^{(k)}$ for each positive integer $k$.
So we have the following property.

\begin{proposition}[\textbf{Bergman function equivalence}]
Suppose $\mathcal{M}=\{(M^n, g(t), J), 0 \leq t <\infty\}$ is an anti-canonical K\"ahler Ricci flow solution satisfying
  (\ref{eqn:HA01_1}).  For each positive integer $k$, there exists $C=C(n,F,k)$ such that
  \begin{align}
   \mathbf{b}^{(k)}(x,0) -C \leq  \mathbf{b}^{(k)}(x,1) \leq \mathbf{b}^{(k)}(x,0)+C
  \label{eqn:HE27_1}
  \end{align}
  for all $x \in M$.
\label{prn:HA01_4}
\end{proposition}

In view of Theorem~\ref{thm:HA01_2}, partial-$C^0$-estimate holds at time $t=1$,
which induces the partial-$C^0$-estimate at time $t=0$, by Proposition~\ref{prn:HA01_4}.
Therefore, the following theorem is clear now.

\begin{theorem}[\textbf{Partial-$C^0$-estimate at initial time}]
Suppose $\mathcal{M}=\{(M^n, g(t), J), 0 \leq t <\infty\}$ is an anti-canonical K\"ahler Ricci flow solution satisfying
  (\ref{eqn:HA01_1}).  Then
  \begin{align*}
       \inf_{x \in M} \mathbf{b}^{(k_0)}(x,0) \geq -c_0
  \end{align*}
  for some positive integer $k_0=k_0(n,F)$ and positive number $c_0=c_0(n,F)$.
\label{thm:HA01_3}
\end{theorem}

By the Sobolev constant estimates for manifolds with uniform positive Ricci curvature, it is clear that
Theorem~\ref{thmin:HA01_1} follows from Theorem~\ref{thm:HA01_3} directly.
It is also clear that Corollary~\ref{clyin:HA01_2} follows from Theorem~\ref{thm:HA01_3}.

The proof of Corollary~\ref{clyin:HA03_1} is known in literature(c.f.~\cite{Sze}), provided the partial-$C^0$-estimate along
the K\"ahler Ricci flow.  We shall be sketchy here.
In fact, due to the work of S. Paul(\cite{Paul1209},\cite{Paul1308}) and the argument in  section 6 of  Tian and Zhang(\cite{TZZ2}), one obtains that the $I$-functional is bounded along the flow.
Then the K\"ahler Ricci flow converges to a K\"ahler-Einstein metric, on the same Fano manifold. 

It is an interesting problem to study the $K$-stability through the K\"ahler Ricci flow. 
Based on Theorem~\ref{thmin:SC24_3}, the weak compactness of polarized K\"ahler Ricci flow,  we are able to give an alternative K\"ahler Ricci flow proof of the stability theorem(Yau's conjecture) of Chen-Donaldson-Sun. 
Interested readers are referred to~\cite{CSW} for the details.

\appendixpage
\addappheadtotoc
\appendix
  
\section{Proof of weighted Sobolev inequality}  
\label{app:A}
The proof of Proposition~\ref{prn:HD07_2} follows exactly the same strategy as Proposition 2.1 of ~\cite{Bak}, which will be described below 
for the convenience of the readers. 

\begin{proof}[Strategy of the proof of Proposition~\ref{prn:HD07_2}]
Without loss of generality, we can assume $t=1$.  Then the weighted Sobolev inequality becomes
\begin{align}
  \int_X f^2(x) p(t,x,y) d\mu_x - \left(\int_X f(x) p(t,x,y)d\mu_x \right)^2 \leq 2 \int_X |\nabla f|^2(x) p(t,x,y) d\mu_x.   \label{eqn:HH09_3}
\end{align}
By density argument, we can assume $f \in C_c^{\infty}(\mathcal{R})$ without loss of generality.  
Let $u$ be the heat solution initiating from $f$, i.e., $ u(y,1)=\int_X f(x) p(1,x,y) d\mu_x$. 
Then we define
\begin{align*}
\Psi_0(s)(y)  \triangleq \int_X u^2(1-s,x) p(s,x,y)d\mu_x,    \quad \Psi(s)(y) \triangleq \int_X |\nabla u|^2(1-s, x) p(s,x,y)d\mu_x.
\end{align*}
According to this definition, we have
\begin{align*}
&\Psi_0(0)(y)=u^2(y,1)=\left( \int_X f(x) p(1,x,y) d\mu_x \right)^2,  &\Psi_0(1)(y)=\int_X f^2(x) p(1,x,y)d\mu_x, \\
&\Psi(0)(y)=|\nabla u|^2(y,1),             &\Psi(1)(y)=\int_X |\nabla f|^2(x) p(1,x,y)d\mu_x.
\end{align*}
Now the weighted-Sobolev inequality (\ref{eqn:HH09_3}) is the same as
\begin{align*}
  \Psi_0(1)-\Psi_0(0) \leq 2 \Psi(1), 
\end{align*}
which can be proved by the combination of the following three steps. 
\begin{enumerate}
\item  $\dot{\Psi}_0(s)=2\Psi(s)$ for each $s \in (0,1)$.
\item  $\Psi(a) \leq \Psi(b), \quad \forall \; 0<a<b<1$.   
\item  $\displaystyle \lim_{s \to 1^{-}} \Psi(s)=\Psi(1)$. 
\end{enumerate}
Actually, since $\Psi_0$ is a continuous function on $[0,1]$, the above three steps yield that
\begin{align*}
    \Psi_0(1)-\Psi_0(0) =\int_0^{1} \dot{\Psi}_0(s) ds=\int_0^{1} 2\Psi(s) ds \leq 2 \Psi(1).
\end{align*}
Consequently, (\ref{eqn:HH09_3}) is proved. 
\end{proof}

However, due to the existence of singularities, we need to check integrability and integration by parts very carefully in each step. 
The most delicate thing is to show that for a bounded heat solution $u$, we have $|\nabla u|^2 \in N_{loc}^{1,2}(X)$ for positive time. 
This is trivial when $X$ is smooth and known when $\dim_{\mathcal{M}} \mathcal{S}<2n-4$. 
We shall show that the same conclusion hold under the condition given by Definition~\ref{dfn:SC27_1}.  
In fact, we first  show that $\frac{|\nabla |\nabla u||^2}{|\nabla u|}$ is locally integrable whenever $|\nabla u|>1$, which is proved in Lemma~\ref{lma:HH18_1}.
Then in Lemma~\ref{lma:HH18_2}, by taking advantage of the weak convexity of $\mathcal{R}$, we show that actually $|\nabla |\nabla u||^2$ is
locally integrable whenever $|\nabla u|>1$.

\begin{lemma}
 Suppose $u$ is a bounded heat solution, i.e., $\square u=0$,  on $X \times [0,1]$ satisfying 
 \begin{align*}
  |u|+|\dot{u}| +\int_X |\nabla \dot{u}|^2 <K
 \end{align*}
 on $X \times [\frac{1}{2}, 1]$.  At time $t=1$, let $w=\max\{|\nabla u|, 1\}$. 
 Then we have
 \begin{align}
   \int_{B(x_0,1) \backslash \mathcal{S}} \frac{|\nabla w|^2}{w} < H,
 \label{eqn:HH10_2}  
 \end{align}
 where $H$ depends on $K$ and $n$, is independent of $x_0$. 
\label{lma:HH18_1} 
\end{lemma}

\begin{proof}
 Let $l=\dot{u}, v=|\nabla u|, m=2n$.  Direct calculation shows that
 \begin{align*}
   \Delta v&=\frac{|\nabla \nabla u|^2-|\nabla v|^2}{v} + \frac{\langle \nabla u, \nabla l \rangle}{v} \geq \frac{|\nabla \nabla u|^2-|\nabla v|^2}{v} -|\nabla l|. 
 \end{align*}
 In local frame, we can delicately compare $|\nabla \nabla u|$ and $|\nabla v|$.   
 Actually, by choosing normal coordinate such that $\nabla u=\frac{\partial}{\partial x^1}$, we then have
 $\Delta u=u_{11}+u_{22} + \cdots u_{mm}$ and $|\nabla v|^2=u_{11}^2+u_{12}^2+\cdots u_{1m}^2$.
 Therefore, we have
\begin{align*}
  &\quad|\nabla \nabla u|^2\\
  &=u_{11}^2+u_{22}^2 + \cdots + u_{mm}^2 +  2(u_{12}^2+\cdots+u_{1m}^2)
     \geq u_{11}^2+ \frac{(u_{22}+\cdots +u_{mm})^2}{m-1}+2(u_{12}^2+\cdots+u_{1m}^2)\\
  &\geq u_{11}^2+ \frac{(l-u_{11})^2}{m-1}+2(u_{12}^2+\cdots+u_{1m}^2)
     =\frac{m}{m-1}u_{11}^2 +\frac{l^2-2lu_{11}}{m-1}+2(u_{12}^2+\cdots+u_{1m}^2)\\
  &=\frac{m}{m-1}|\nabla v|^2 +\frac{l^2-2lu_{11}}{m-1}+\frac{m-2}{m-1}(u_{12}^2+\cdots+u_{1m}^2). 
\end{align*}
 Then it is easy to see that
 \begin{align*}
   |\nabla \nabla u|^2-|\nabla v|^2 \geq \frac{1}{m-1}|\nabla v|^2 + \frac{l^2-2lu_{11}}{m-1} \geq \frac{1}{m-1}|\nabla v|^2 -\frac{2K}{m-1} |\nabla v|
   \geq \frac{3}{4(m-1)} |\nabla v|^2 - \frac{4K^2}{(m-1)}. 
 \end{align*}
 Let $w=\max\{v,1\}$, we claim that
 \begin{align}
    \Delta w \geq \frac{3}{4(m-1)} \frac{|\nabla w|^2}{w} - \frac{4K^2}{(m-1)} -|\nabla l|,  \label{eqn:HH10_4}   
 \end{align}
 on $\mathcal{R}$ in the weak sense.  Actually, let $\varphi \in C_c^{\infty}(\mathcal{R})$ and $\varphi \geq 0$.   Let $\Omega$ be the domain consisting of points where $\varphi>0$ and $v>1$. 
 Clearly, $\nabla w \equiv 0$ outside $\overline{\Omega}$. Following from definition of weak Laplacian (Definition~\ref{dfn:HD29_1}) and $\varphi$ has compact support, we have
 \begin{align*}
    -\int_X \varphi \Delta w=\int_X \langle \nabla \varphi, \nabla w\rangle=\int_{\Omega} \langle \nabla \varphi, \nabla w\rangle=\int_{\Omega} \langle \nabla \varphi, \nabla v\rangle
    =-\int_{\Omega} \varphi \Delta v + \int_{\partial \Omega} \varphi \langle  \nabla v, \vec{n} \rangle,  
 \end{align*}
 where $\vec{n}$ is the outward normal vector field of $\Omega$ along $\partial \Omega$.  Note that $\nabla w \equiv 0$ outside $\Omega$. So we have
 \begin{align*}
   &\quad \int_X \varphi \Delta w + \int_X \varphi \left(  -\frac{3}{4(m-1)} \frac{|\nabla w|^2}{w} +\frac{4K^2}{(m-1)} +|\nabla l| \right)\\
   &=\int_{\Omega} \varphi \left( \Delta v  -\frac{3}{4(m-1)} \frac{|\nabla v|^2}{v} +\frac{4K^2}{(m-1)} +|\nabla l| \right)
       +\int_{X \backslash \Omega} \varphi \left( \frac{4K^2}{(m-1)} +|\nabla l| \right) 
       -\int_{\partial \Omega} \varphi \langle \nabla v, \vec{n} \rangle\\
    &\geq   -\int_{\partial \Omega} \varphi \langle \nabla v, \vec{n} \rangle \geq 0. 
 \end{align*}
 In the last step, $\vec{n}$ is the outward normal vector field along $\partial \Omega$ and $\varphi \langle \nabla v, \vec{n} \rangle \leq 0$.  By the arbitrary choice of $\varphi$, 
 the above inequality implies (\ref{eqn:HH10_4}). 
 
Suppose  $\eta$ is a radial cutoff function which vanishes outside $B(x_0,2)$ and equals $1$ in $B(x_0,1)$, 
$\psi_{\epsilon}$ is a cutoff function which vanishes outside the $2\epsilon$-neighborhood of $\mathcal{S}$ and equals $1$ in the $\epsilon$-neighborhood of $\mathcal{S}$. 
We can also require that $|\nabla \eta|<2, |\nabla \psi_{\epsilon}|<\frac{2}{\epsilon}$.   Let $\chi=\eta(1-\psi_{\epsilon})$. 
Multiplying both sides of the above inequality by $\chi^2$ and integrating by parts, we have
 \begin{align*}
  &\quad \frac{1}{2(m-1)} \int_X  \chi^2 \frac{|\nabla w|^2}{w} +2(m-1) \int_X w|\nabla \chi|^2\\
  &\geq
    -2\int_X \langle \nabla \chi, \chi \nabla w \rangle=\int_X \chi^2 \Delta w
     \geq \frac{3}{4(m-1)} \int_X \chi^2 \frac{|\nabla w|^2}{w} -\frac{4K^2}{m-1} \int_X \chi^2 -\int_X \chi^2 |\nabla l|.
 \end{align*}
 It follows that
 \begin{align}
      &\quad \frac{1}{4(m-1)} \int_X \chi^2 \frac{|\nabla w|^2}{w} \notag\\
      &\leq C \int_X \chi^2 (1+|\nabla l|) + 2(m-1) \int_X w|\nabla \chi|^2 \leq C \int_X \chi^2 (1+|\nabla l|^2) + 2(m-1) \int_X w|\nabla \chi|^2 \notag\\
      &\leq C \int_{B(x_0,2)} (1+|\nabla l|^2) + 2(m-1) \int_X (1+|\nabla u|) |\nabla \chi|^2  \notag\\
      &\leq C \int_{B(x_0,2)} (1+|\nabla l|^2) + 4(m-1) \int_X (1+|\nabla u|) \left\{|\nabla \eta|^2 + |\nabla \psi_{\epsilon}|^2 \right\}.  \label{eqn:GD02_3}
 \end{align}
 In the support of $\psi_{\epsilon}$, $u$ satisfies heat equation from time $t=\frac{1}{2}$ to $t=1$. Moreover, $|u|$ is bounded by $K$.
 By adding $K$ if necessary, we can assume $u$ to be positive.
 In light of classical Li-Yau gradient estimate for heat solutions(c.f.~\cite{LiYau},~\cite{Ha93}), we see that $|\nabla u|< \frac{C(K)}{\epsilon}$ at time $t=1$. 
 This can also be obtained from parabolic Moser iteration. Actually, let $y$ be a point such that $B(y,\epsilon)$ is regular.  
 Since $\square |\nabla u| \leq 0$ in $B(y, \epsilon) \times [0,1]$ as smooth functions,  Moser iteration implies that the value of $|\nabla u|(y,1)$ is
 dominated by the $L^2$-average of $u$ in $B(y, \epsilon) \times [\frac{1}{2},1]$, multiplying by a number which is about $\frac{C}{\epsilon}$. 
 Now we return to the main argument.  Recall that
 $|\nabla \psi_{\epsilon}|<\frac{C}{\epsilon}$ also.  The support of $\psi_{\epsilon}$ in $B(x_0, 2)$ has volume
 bounded above by $C\epsilon^{3+\delta}$ by the assumption $\dim_{\mathcal{M}} \mathcal{S}<2n-3$.  Consequently, we have
 \begin{align*}
   \int_{B(x_0,2)} |\nabla u||\nabla \psi_{\epsilon}|^2< C \epsilon^{3+\delta} \cdot \epsilon^{-1} \cdot \epsilon^{-2}=C\epsilon^{\delta}=o(\epsilon). 
 \end{align*}
 In inequality (\ref{eqn:GD02_3}),  let $\epsilon \to 0$, we obtain $\int_{\mathcal{R}} \eta^2 \frac{|\nabla w|^2}{w} <C$,  which of course implies (\ref{eqn:HH10_2}).  
 \end{proof}

\begin{lemma}
Same conditions as in Lemma~\ref{lma:HH18_1}. Then we have $w \in N_{loc}^{1,2}(X)$. 
\label{lma:HH18_2}
\end{lemma}

\begin{proof}
It suffices to prove 
\begin{align}
    \int_{B(x_0, \frac{1}{2})} |\nabla w|^2 <\infty
\label{eqn:HH10_3}    
\end{align}
for arbitrary point $x_0 \in X$. 

Let $\tilde{w}=(1-r^2)^2 w$, $E=\frac{4K^2}{m-1}$.   It follows from (\ref{eqn:HH10_4}) and the weak convexity of $\mathcal{R}$ (c.f. Proposition~\ref{prn:HC29_1}) that
\begin{align*}
     \Delta w \geq \frac{1}{2(m-1)} \frac{|\nabla w|^2}{w} - Ew-|\nabla l|, \qquad    \Delta (1-r^2)^2 \geq -4m. 
\end{align*}
Then
\begin{align*}
  \Delta \tilde{w}&=w \Delta (1-r^2)^2 + 2 \langle \nabla (1-r^2)^2, \nabla w\rangle +(1-r^2)^2 \Delta w \\
     &\geq -4mw -8(1-r^2)r \langle \nabla r, \nabla w\rangle +\frac{1}{2(m-1)} (1-r^2)^2|\nabla w|^2w^{-1} -E(1-r^2)^2 w -(1-r^2)^2 |\nabla l|\\
     &\geq -4mw -8(1-r^2)|\nabla w| +\frac{1}{2(m-1)} (1-r^2)^2|\nabla w|^2w^{-1} -E(1-r^2)^2 w -(1-r^2)^2 |\nabla l|\\
     &\geq -C w -|\nabla l|. 
\end{align*}
In short, we have $Cw+|\nabla l| \geq -\Delta \tilde{w}$. 
Let $\tilde{w}_k=\min\{\tilde{w}, k\}$. Clearly, $\tilde{w}_k|_{\partial \Omega}=0$ where $\Omega=B(x_0,1)$.  Let $\chi=1-\psi_{\epsilon}$. 
Multiplying both sides of the above inequality by $\chi^2 \tilde{w}_k$,  integration by parts implies that
\begin{align*}
  C \int_{\Omega} \chi^2 \tilde{w}_k w +\int_{\Omega} \chi^2 \tilde{w}_k |\nabla l|
 &\geq  -\int_{\Omega} \chi^2 \tilde{w}_k \Delta \tilde{w} 
 =\int_{\Omega} \chi^2 \langle \nabla \tilde{w}_k, \nabla \tilde{w} \rangle 
 + 2 \int_{\Omega} \langle  \tilde{w}_k \nabla \chi, \chi \nabla \tilde{w} \rangle\\
 &=\int_{\Omega} \chi^2 |\nabla \tilde{w}_k|^2+ 2 \int_{\Omega} \langle  (\tilde{w}_k-k) \nabla \chi, \chi \nabla \tilde{w} \rangle +2k \int_{\Omega} \langle \nabla \chi, \chi \nabla \tilde{w}\rangle\\
 &=\int_{\Omega} \chi^2 |\nabla \tilde{w}_k|^2+ 2 \int_{\Omega} \langle  (\tilde{w}_k-k) \nabla \chi, \chi \nabla \tilde{w}_k \rangle +2k \int_{\Omega} \langle \nabla \chi, \chi \nabla \tilde{w}\rangle.
\end{align*} 
Note that in the above inequality we used the fact that $\nabla \tilde{w}_k=0$ and $\tilde{w}_k-\tilde{w}=0$ whenever $w>k$.   
Applying an elementary inequality in the last step, we arrive
\begin{align*} 
 &\quad C \int_{\Omega} \chi^2 \tilde{w}_k w +\int_{\Omega} \chi^2 \tilde{w}_k |\nabla l|\\
 &\geq \frac{1}{2}\int_{\Omega} \chi^2 |\nabla \tilde{w}_k|^2-2 \int_{\Omega} (\tilde{w}_k-k)^2 |\nabla \chi|^2 
 - 2k \left(\int_{\Omega} w |\nabla \chi|^2 \right)^{\frac{1}{2}} \left( \int_{\Omega} \frac{\chi^2 |\nabla \tilde{w}|^2}{w} \right)^{\frac{1}{2}}\\
 &\geq \frac{1}{2}\int_{\Omega} \chi^2 |\nabla \tilde{w}_k|^2 - 2k^2 \int_{\Omega} |\nabla \chi|^2 
   -Ck\left(\int_{\Omega} w |\nabla \chi|^2 \right)^{\frac{1}{2}} \left( \int_{\Omega} \chi^2 \left( w + \frac{|\nabla w|^2}{w}  \right) \right)^{\frac{1}{2}}.  
\end{align*}
Recall that $\chi=1-\psi_{\epsilon}$.  Let $\epsilon \to 0$, the last two terms in the above inequality vanishes. Then we have
\begin{align*}
   \int_{\Omega} |\nabla \tilde{w}_k|^2 &\leq C \int_{\Omega} \tilde{w}_k + \int_{\Omega} \tilde{w}_k |\nabla l| \leq C\int_{\Omega} w + \left( \int_{\Omega} \tilde{w}_k^2 \right)^{\frac12} \left( \int_{\Omega} |\nabla l|^2 \right)^{\frac12}\\
     &\leq C\int_{\Omega} (1+|\nabla u|) + \left( \int_{\Omega} 1+|\nabla u|^2 \right)^{\frac12} \left( \int_{\Omega} |\nabla l|^2 \right)^{\frac12}  \leq C.
\end{align*}
Note that the last constant $C$ does not depend on $k$. Let $k \to \infty$, we obtain $\int_{\Omega} |\nabla \tilde{w}|^2 \leq C$. 
In particular, we can bound $\int_{B(x_0, \frac{1}{2})} |\nabla \tilde{w}|^2$ and consequently we have (\ref{eqn:HH10_3}). 
\end{proof}

\begin{remark}
If $\dim_{\mathcal{M}} \mathcal{S}<2n-4+\frac{2}{2n-1}$, we can obtain Lemma~\ref{lma:HH18_2} without using the weak convexity of $\mathcal{R}$. 
The ingredient is to show sub-solution property of $|\nabla u|^{q}$, for $q $ slightly bigger than $\frac{2n-2}{2n-1}$. 
Also, Lemma~\ref{lma:HH18_1} and Lemma~\ref{lma:HH18_2} has  versions for bounded solution of Poisson equation $\Delta u=c$ where $c$ is a constant.
\label{rmk:GD11_1}
\end{remark}

After we know $w=\max\{1, |\nabla u|\} \in N_{loc}^{1,2}(X)$, parabolic De-Giorgi iteration implies that  $w$'s point wise bound can be
dominated by $w$'s $L^2$-norm in the space-time.  By Lemma~\ref{lma:HH18_2}, we then have $w$ is bounded in Lemma~\ref{lma:HH10_1}.
This of course implies that $|\nabla u|$ is bounded.

\begin{lemma}
Suppose $f \in C_c^{\infty}(\mathcal{R})$, $u$ is the heat solution initiating from $f$. 
Let $h=|\nabla u|^2(\cdot, t)$ for some $t>0$. Then $\norm{h}{L^{\infty}(X)}<\infty$. 
\label{lma:HH10_1}
\end{lemma}

\begin{proof}
Without loss of generality, we assume $t=1$. 

Note that $\Delta u=\dot{u}$, which can be written down explicitly as
\begin{align*}
    \dot{u}(x, s)=\int_{X} f(y) \dot{p}(s,y,x) d\mu_y.
\end{align*}
By the exponential decay of $\dot{p}$ (c.f. Proposition~\ref{prn:HD07_1}), it is clear that $\dot{u}$ also decays exponentially fast.   Note that $\square \dot{u}=0$ on the regular part. 
Therefore,  for each $s>0$, we have $\dot{u} \in L^{\infty}(X) \cap N^{1,2}(X)$. 
Let $l=\dot{u}$, $v=|\nabla u|$, $w= \max \{v, 1\}$. 
By Lemma~\ref{lma:HH18_1} and Lemma~\ref{lma:HH18_2},  we know $w(\cdot, s) \in N_{loc}^{1,2}(X)$ and $w \in N_{loc}^{1,2}(X \times (0,\infty))$. 
Moreover, we have
\begin{align*}
  \dot{v}=\partial_t(|\nabla u|)=\frac{\langle \nabla u, \nabla \dot{u}\rangle}{v}, \quad
  |\dot{v}| \leq |\nabla \dot{u}|.
\end{align*} 
For each $0<t_1<s<t_2<\infty$, it is clear that $\dot{v}(\cdot, s) \in L^2(\mathcal{R})$ and
\begin{align*}
   \int_{t_1}^{t_2} \norm{\dot{v}}{L^2(\mathcal{R})}^2 ds \leq  \int_{t_1}^{t_2} \norm{\dot{u}}{N^{1,2}(\mathcal{R})}^2 ds<\infty. 
\end{align*}
Direct calculation shows that
\begin{align*}
  \square v=(\partial_s-\Delta) v=\frac{-|\nabla \nabla u|^2 +|\nabla |\nabla u||^2}{v} \leq 0. 
\end{align*}
Recall that $w=\max\{v,1\}$. So we have $\square w \leq 0$ on $\mathcal{R} \times (0,\infty)$ in the weak sense, i.e., for each nonnegative smooth cutoff function
$\varphi$ compactly supported on $\mathcal{R} \times (0, \infty)$, we have
\begin{align*}
   \iint_{\mathcal{R} \times (0,\infty)}  \varphi \square w \triangleq \iint_{\mathcal{R} \times (0,\infty)} \left\{ \dot{w} \varphi + \langle \nabla \varphi, \nabla w \rangle \right\} \geq 0. 
\end{align*}
This can be proved following similar argument as that in the proof of inequality (\ref{eqn:HH10_4}) in Lemma~\ref{lma:HH18_1}. 
Since $w \in N_{loc}^{1,2}(X \times (0,\infty))$, $\norm{\dot{w}(\cdot, s)}{L^2(X)}^2$ is locally integrable on $(0,\infty)$, 
the compactly supported smooth functions are dense in $N^{1,2}(X \times (0,\infty))$, 
we have $\square w \leq 0$ in the weak sense on $X \times (0,\infty)$. 
By parabolic version of De-Giorgi iteration, we then have $w$ is locally bounded.
Consequently, $|\nabla u|$ is locally bounded.   Similar to sub-harmonic extension theorem(c.f. Proposition~\ref{prn:HD16_2}),  we have a heat sub-solution extension theorem.
Since $|\nabla u|$ is locally bounded and it is a heat sub-solution on $\mathcal{R} \times (0,\infty)$, we obtain that $|\nabla u|$ is a heat sub-solution
on $X \times (0,\infty)$.  In particular, we have $|\nabla u|\in N_{loc}^{1,2}(X \times (0,\infty))$. 
Therefore, by parabolic De-Giorgi iteration again, $\norm{\nabla u}{L^{\infty}(B(x,1) \times [\frac{3}{4}, \frac{5}{4}])}$ is bounded by 
$\norm{\nabla u}{L^2(B(x,2) \times [\frac{1}{2}, 2])}$, which is uniformly bounded, independent of the choice of $x$.  Therefore, $|\nabla u|(\cdot, 1)$, and hence $h$, are globally bounded on $X$.
\end{proof}

We continue to show  the integrability of $\Delta |\nabla u|^2$ and $\partial_t |\nabla u|^2$. 

\begin{lemma}
Same conditions as that in Lemma~\ref{lma:HH10_1}. Then
$h \in N^{1,2}(X)$ and $|\Delta h|+|\dot{h}| \in L^1(X)$. 
\label{lma:HH09_1}
\end{lemma}

\begin{proof}
 Let us first show that $|\dot{h}| \in L^{1}(X)$.
 Note that $\dot{u}=\Delta u$ is also a bounded heat solution,  due to the exponential decay of $\dot{p}$. 
 Since both $u$ and $\Delta u$ decays exponentially fast at infinity, we see that  $u, \Delta u \in N^{1,2}(X)$. 
 Therefore, we have
\begin{align}
  \int_X |\dot{h}|=2\int_X \left| \langle \nabla \Delta u, \nabla u\rangle \right| < C \left( \int_X |\nabla u|^2 \right)^{\frac12} \left( \int_X |\nabla \Delta u|^2 \right)^{\frac12}<\infty. 
\label{eqn:GD03_1}  
\end{align}
This means that $|\dot{h}| \in L^1(X)$. 

 Then we continue to show that $|\Delta h| \in L^1(X)$.  Actually, we have
 \begin{align*}
     \Delta h=\Delta |\nabla u|^2=2|\nabla \nabla u|^2 +2 \langle \nabla u, \nabla \dot{u} \rangle.
 \end{align*}
 Fix $x_0 \in X$. 
 Let $\eta_k$ be a radial cutoff function which vanishes outside $B(x_0, k+1)$ and equals $1$ in $B(x_0,k)$, $|\nabla \eta_k|<2$. 
 Let $\psi_{\epsilon}$ be as usual. Let $\chi_k=\eta_k(1-\psi_{\epsilon})$.  Then we have
 \begin{align*}
    &\quad \int_X \chi_k^2|\nabla \nabla u|^2 +  \int_X \chi_k^2 \langle \nabla u, \nabla \dot{u} \rangle\\
    &= \frac{1}{2}\int_X \chi_k^2 \Delta |\nabla u|^2
    =-\frac{1}{2}\int_X \langle \nabla |\nabla u|^2, \nabla \chi_k^2 \rangle
    =-2\int_X \langle \chi_k \nabla |\nabla u|, |\nabla u| \nabla \chi_k \rangle\\
    &\leq \frac{1}{2} \int_X \chi_k^2 |\nabla |\nabla u||^2 + 2 \int_X |\nabla u|^2 |\nabla \chi_k|^2
      \leq \frac{1}{2} \int_X \chi_k^2 |\nabla \nabla u|^2 + 2\int_X |\nabla u|^2 |\nabla \chi_k|^2. 
 \end{align*}
 Recall that $|\nabla \chi_k|^2 \leq 2(|\nabla \eta_k|^2 + |\nabla \psi_{\epsilon}|^2)$ and $|\dot{h}|=2|\langle \nabla u, \nabla \dot{u}\rangle|$.
 Then we have
 \begin{align*}
    \int_X \chi_k^2|\nabla \nabla u|^2  \leq 8 \int_{B(x_0,k+1)} |\nabla u|^2 \left( |\nabla \eta_k|^2 + |\nabla \psi_{\epsilon}|^2 \right) + \int_X |\dot{h}|. 
 \end{align*}
 Note that $|\nabla u|$ is bounded here, due to Lemma~\ref{lma:HH10_1}.  Let $\epsilon \to 0$, we have
 \begin{align*}
        \int_{B(x_0,k)}  |\nabla \nabla u|^2 \leq  \int_X \eta_k^2 |\nabla \nabla u|^2  \leq 32 \int_{B(x_0,k+1)} |\nabla u|^2   + \int_X |\dot{h}|.
 \end{align*}
 Let $k \to \infty$, by (\ref{eqn:GD03_1}),  we obtain
 \begin{align}
    \int_X |\nabla \nabla u|^2 \leq 32 \int_X |\nabla u|^2+ \int_X |\dot{h}| < C.
 \label{eqn:HH09_4}   
 \end{align}
 It follows that
 \begin{align*}
    \int_X |\Delta h| \leq 2 \int_X |\nabla \nabla u|^2  +  \int_X |\dot{h}|<\infty. 
 \end{align*} 
 So we proved that $|\Delta h| \in L^1(X)$ and hence $|\dot{h}| + |\Delta h| \in L^1(X)$.

 Finally, we show that $h \in N^{1,2}(X)$.   Recall that $h=|\nabla u|^2$ is bounded. So we have
  \begin{align*}
    \int_X \left( h^2 + |\nabla h|^2 \right) &=\int_X  \left( |\nabla u|^4 + 4|\nabla u|^2 |\nabla |\nabla u||^2\right) \leq C \int_X \left( |\nabla u|^2 + |\nabla |\nabla u||^2 \right)
    \leq C \left( 1+ \int_X |\nabla \nabla u|^2 \right).
  \end{align*}
  Plugging (\ref{eqn:HH09_4}) into the above inequality, we have $h \in N^{1,2}(X)$.   
\end{proof}

After we obtain that $|\nabla u|^2 \in N^{1,2}(X)$,  in the following Lemma~\ref{lma:HH09_2},  Lemma~\ref{lma:HH08_1} and Lemma~\ref{lma:HH12_1}, 
we focus on the checking of integration by parts and continuity of integrals at boundary time.
The heat kernel's exponential decay will play an important role there. However, the following proof  will be by no means optimal. 
We only prove what we need by what we have. 

\begin{lemma}
Same conditions as that in Lemma~\ref{lma:HH10_1}. 
Suppose $x_0 \in X$. Then we have
 \begin{align}
  \int_{\mathcal{R}} h\Delta p(t,\cdot, x_0) = \int_{\mathcal{R}} p(t,\cdot, x_0) \Delta h
\label{eqn:HH08_1}   
\end{align}
\label{lma:HH09_2}
\end{lemma}

\begin{proof}

   For simplicity of notation, denote $p(t,\cdot, x_0)$ by $p$. 

   We first note that both sides of (\ref{eqn:HH08_1}) are finite integral.  Actually, we have
   \begin{align*}
   &\left| \int_{\mathcal{R}} h\Delta p\right| \leq \left| \int_{\mathcal{R}} |h||\Delta p| \right| \leq \norm{h}{L^2(X)} \norm{\Delta p}{L^2(\mathcal{R})}<\infty,\\ 
   &\left| \int_{\mathcal{R}} p\Delta h\right| \leq \left| \int_{\mathcal{R}} |p||\Delta h| \right| \leq \norm{p}{L^{\infty}(X)} \norm{\Delta h}{L^1(\mathcal{R})}<\infty.
   \end{align*}

   Then we show that both sides of (\ref{eqn:HH08_1}) can be approximated by integrations over compact supported sets. 
   Let $\eta_k$ be a radial cutoff function which vanishes outside $B(x_0, k+1)$ and equals $1$ in $B(x_0,k)$, $|\nabla \eta_k|<2$. 
   Then we have
   \begin{align*}
     \left|\int_{\mathcal{R}} h\Delta p -\int_{\mathcal{R}} \eta_k^2 h \Delta p \right| &=\left| \int_{\mathcal{R} \backslash B(x_0,k)} (1-\eta_k^2) h \Delta p \right|
     \leq \left| \int_{\mathcal{R} \backslash B(x_0,k)} |h| |\Delta p| \right|
     \leq \norm{\Delta p}{L^2(X \backslash B(x_0,k))} \norm{h}{L^2(X)}.
   \end{align*} 
   Clearly, $\norm{\Delta p}{L^2(X \backslash B(x_0,k))} \to 0$ as $k \to \infty$, due to the exponential decay of $\Delta p$.
   Note that $h \in L^2(X)$ by Lemma~\ref{lma:HH09_1}.  Thus we have proved that
   \begin{align}
    \lim_{k \to \infty} \int_{\mathcal{R}} \eta_k^2 h\Delta p= \int_{\mathcal{R}} h \Delta p.    \label{eqn:HH08_2}
   \end{align}
   Similarly, we calculate
    \begin{align*}
     \left|\int_{\mathcal{R}} p\Delta h -\int_{\mathcal{R}} \eta_k^2 p \Delta h \right| &=\left| \int_{\mathcal{R} \backslash B(x_0,k)} (1-\eta_k^2) p \Delta h \right|
     \leq \left| \int_{\mathcal{R} \backslash B(x_0,k)} |p| |\Delta h| \right|
     \leq \norm{p}{L^{\infty}(X \backslash B(x_0,k))} \norm{\Delta h}{L^1(X)}.
   \end{align*} 
   It follows from the exponential decay that $\norm{p}{L^{\infty}(X \backslash B(x_0,k))} \to 0$ as $k \to \infty$. 
   Also, we know $\Delta h \in L^1(X)$ by Lemma~\ref{lma:HH09_1}.
   So we have
   \begin{align}
      \lim_{k \to \infty} \int_{\mathcal{R}} \eta_k^2 p \Delta h= \int_{\mathcal{R}} p \Delta h.    \label{eqn:HH08_3}
   \end{align}
   Clearly,  $\eta_k^2 h \in N_c^{1,2}(X)$ and $p \in N_c^{1,2}(X)$, both of them are bounded functions.
   Furthermore, both $|\Delta p|$ and $|\Delta h|$ are integrable on $B(x_0,k+1)$. 
   Due to the fact that Minkowski codimension of $\mathcal{S}$ is greater than $2$, it is not hard to check that
   \begin{align*}
     -\int_{\mathcal{R}} \langle \nabla (\eta_k^2 h), \nabla p \rangle = \int_{\mathcal{R}} \eta_k^2 h \Delta p, \qquad
     -\int_{\mathcal{R}} \langle \nabla (\eta_k^2 p), \nabla h \rangle=\int_{\mathcal{R}} \eta_k^2 p \Delta h.
   \end{align*}
   It follows that
   \begin{align*}
     \int_{\mathcal{R}} \eta_k^2 h \Delta p-\int_{\mathcal{R}} \eta_k^2 p\Delta h=\int_{\mathcal{R}}  2\eta_k \left\langle \nabla \eta_k, p\nabla h -h\nabla p\right\rangle.
   \end{align*}
   Denoting $B(x_0,k+1) \backslash B(x_0,k)$ by $A_k$, we have
   \begin{align*}
     &\quad \left|  \int_{\mathcal{R}} \eta_k^2 h \Delta p -\int_{\mathcal{R}} \eta_k^2 p\Delta h\right|\\
     &=\left| \int_{\mathcal{R}}   2\eta_k \left\langle \nabla \eta_k, p\nabla h -h\nabla p\right\rangle \right|
     \leq 4\int_{B(x_0,k+1) \backslash B(x_0,k)} |p\nabla h-h\nabla p|\\
     &\leq 4\int_{A_k} |p\nabla h|+|h\nabla p|
       \leq 4  \left( \int_{A_k} p^2 \right)^{\frac12} \left( \int_{A_k} |\nabla h|^2\right)^{\frac12} 
       + 4 \left( \int_{A_k} h^2 \right)^{\frac12} \left( \int_{A_k} |\nabla p|^2\right)^{\frac12}\\ 
     &\leq  8 \norm{p}{N^{1,2}(A_k)} \norm{h}{N^{1,2}(X)}.  
   \end{align*}
   By the exponential decay of $p$ and $\Delta p$, it is clear that $\norm{p}{N^{1,2}(A_k)} \to 0$ as $k \to \infty$.  Therefore, we have
   \begin{align}
     \lim_{k \to \infty}  \int_{\mathcal{R}} \eta_k^2 f \Delta p = \lim_{k \to \infty}  \int_{\mathcal{R}} \eta_k^2 p \Delta h. 
   \label{eqn:HH08_4}  
   \end{align}
   Therefore, (\ref{eqn:HH08_1}) follows from the combination of (\ref{eqn:HH08_2}), (\ref{eqn:HH08_3}) and (\ref{eqn:HH08_4}). 
\end{proof}

\begin{lemma}
Same conditions as that in Lemma~\ref{lma:HH10_1}. 
Fix $T>0$, then for every pair $0<a<b<T$, we have 
 \begin{align}
     \int_X h(x,T-b) p(b,x,y)d\mu_x - \int_X h(x,T-a) p(a,x,y) d\mu_x = 2 \int_a^b \int_{\mathcal{R}} p(t,x,x_0) |\nabla \nabla u|^2(x) d\mu_x dt.
 \label{eqn:HH09_5}    
 \end{align}
\label{lma:HH08_1} 
\end{lemma}

\begin{proof}
  Applying H\"older inequality, then we see that each integral on the left hand side of (\ref{eqn:HH09_5}) is well defined and finite.
  Direct calculation shows that
  \begin{align*}
     \frac{d}{d t} \int_X h(x,T-t) p(t,x,y)d\mu_x =\int_X -\dot{h}p + h \dot{p} =\int_{\mathcal{R}} -\dot{h}p + h\Delta p
     =\int_{\mathcal{R}} \left( -\dot{h} +\Delta h\right) p.
  \end{align*}     
  Note that we have used the integrability of $-\dot{h}p+h\dot{p}$ (By Lemma~\ref{lma:HH10_1} and Lemma~\ref{lma:HH09_1}) and integration by parts (Lemma~\ref{lma:HH09_2}) in the above deduction. 
  Recall that $h(\cdot, s)=|\nabla u|^2(\cdot, s)$, which implies that
  $\displaystyle -\dot{h}+ \Delta h=2|\nabla \nabla u|^2$. Plugging this into the above equation and then integrating both sides of the equation over time,
  we obtain  (\ref{eqn:HH09_5}). 
\end{proof}

\begin{lemma}
Same conditions as that in Lemma~\ref{lma:HH10_1}. 
Fix $T>0$, $y \in X$, then we have 
 \begin{align}
     \lim_{b \to T^{-}}\int_X h(x,T-b) p(b,x,y)d\mu_x=\int_X |\nabla f|^2 p(T,x,y) d\mu_x.
 \label{eqn:HH12_1}    
 \end{align}
\label{lma:HH12_1} 
\end{lemma}

\begin{proof}
Choose an open set $\Omega$ such that $\supp f \Subset \Omega \Subset \mathcal{R}$. 
Note that $u$ is a smooth heat solution on $\Omega \times [0, \infty)$.  Then it is clear that
\begin{align*}
  \lim_{b \to T^{-}}\int_{\Omega} h(x,T-b) p(b,x,y)d\mu_x=\int_{\Omega} |\nabla f|^2 p(T,x,y) d\mu_x=\int_{X} |\nabla f|^2 p(T,x,y) d\mu_x.
\end{align*}
Therefore, in order to show (\ref{eqn:HH12_1}), it suffices to show $\displaystyle \lim_{b \to T^{-}} \int_{X \backslash \Omega}  h(x, T-b) p(b,x,y) d\mu_x=0$. 
However, by the uniform bound of $p(b,\cdot, \cdot)$ when $b \to T$, this equation can be deduced from 
\begin{align}
  \lim_{b \to T^{-}} \int_{X \backslash \Omega}  h(x, T-b) d\mu_x=0. 
\label{eqn:HH12_2}
\end{align}
Recalling that
\begin{align*}
   u(x,s)=\int_X f(z)p(s,z,x) d\mu_z, \quad
   h(x,s)=|\nabla u|(x,s) \leq \int_{\supp f} |f|(z) |\nabla_x p|(s,z,x) d\mu_z. 
\end{align*}
By the exponential decay of $p$ and $\Delta p$,  for every $w \in X$, $z \in \supp f$, it is not hard to see that 
\begin{align*}
  \int_{B(w,1)} |\nabla_x p|(s,z,x) d\mu_x  &\leq C \left(  \int_{B(w,1)} |\nabla_x p|^2(s,z,x) d\mu_x \right)^{\frac12}
  \leq C_1 s^{-n-1}e^{\frac{-d^2(w,z_0)+D^2}{C_2 s}}, 
\end{align*}
where $z_0$ is a fixed point in $\supp f$, $D$ is the diameter of $\Omega$, $0<s<1$. 
It follows that 
\begin{align}
  h(x,s)=|\nabla u|(x,s) \leq C  s^{-n-1}e^{-\frac{d^2(x,z_0) +D^2}{C_2 s}}.    \label{eqn:HH12_3}
\end{align}
Suppose $x \in X \backslash \Omega$, then  $d(x,z_0) \geq c_0>0$ always.  Then (\ref{eqn:HH12_2}) follows from (\ref{eqn:HH12_3}), the Euclidean volume growth estimate and direct calculation. 
\end{proof}

Now we can finish the proof of Proposition~\ref{prn:HD07_2}.

\begin{proof}[Proof of Proposition~\ref{prn:HD07_2}]
It suffices to check the three steps mentioned in the strategy. 

We first check that $\dot{\Psi}_0(s)=2\Psi(s)$ for each $s \in (0,1)$.  Formal calculation shows that
\begin{align*}
   \dot{\Psi}_0(s)&=\int_X \left( \frac{d}{ds} u^2(1-s,x) \right) p(s,x,y) d\mu_x  + \int_X u^2(1-s,x) \Delta p(s,x,y) d\mu_x\\
      &=\int_X \left\{  -2u(x,1-s) \dot{u}(x,1-s) \right\} p(s,x,y) d\mu_x  + \int_X u^2(1-s,x)  \Delta p(s,x,y) d\mu_x.       
\end{align*}
By the boundedness of $u, \dot{u}$ and exponential decay of $p$ and $\Delta p$, the above formal calculation is in fact rigorous for each $s \in (0,1)$.
Moreover, $u^2(1-s,\cdot) \in L^{\infty}(X) \cap N^{1,2}(X)$.  Since $\Delta p$ has exponential decay, similar to Lemma~\ref{lma:HH09_2}, one can have 
integration by parts to obtain
\begin{align*}
  \int_X u^2(1-s,x) \Delta p(s,x,y) d\mu_x=\int_X p(s,x,y) \Delta u^2(1-s,x) d\mu_x. 
\end{align*}
Therefore, we have
\begin{align*}   
       \dot{\Psi}_0(s)
      &=\int_X \left\{-2u(x,1-s) \dot{u}(x,1-s) \right\} p(s,x,y) d\mu_x \\
      &\qquad + \int_X \left\{ 2u(x,1-s) \Delta u(x,1-s) +2|\nabla u|^2(x,1-s) \right\} p(s,x,y) d\mu_x\\
      &=2 \int_X |\nabla u|^2 (x,1-s) p(s,x,y) d\mu_x
         =2\Psi(s).
\end{align*}
So we checked the first step.  
However, the second step follows from Lemma~\ref{lma:HH08_1},  the third step follows from Lemma~\ref{lma:HH12_1}. 
Therefore, the proof of Proposition~\ref{prn:HD07_2} is complete.  
\end{proof}

 \section{Perturbation technique} 
 \label{app:B}
 
 We often meet the problem of decomposing a manifold $M$ by regularity scales, e.g. $\mathbf{cvr}$. 
Although such regularity scale functions are not smooth in general, they satisfy local Harnack inequalities (c.f. inequality (\ref{eqn:SC24_1})).
In this section, we show that there is a general way to perturb the regularity scale functions to smooth functions, while keeping the major
properties of regularity scales. 
The perturbation method is a standard application of the proof of partition of unity. 
 
 \begin{proposition}[\textbf{Perturbation of general functions with local Harnack inequality}]
  Suppose $K$ is a big positive constant, $f$ is a map from Riemannian manifold $M^{2n}$ to $(0,K^{-1}]$ with the following local Harnack inequality
  \begin{align}
     K^{-1} f(x) < f(y)< Kf(x), \quad \forall \; y \in B(x, K^{-1}f(x)). 
  \label{eqn:MA25_1}   
  \end{align}
  Suppose each geodesic ball of radius $r$ has volume ratio in $(\kappa, \kappa^{-1})$ whenever $0<r<1$. 
  Then there exists a constant $C=C(n,\kappa,K)$ and smooth function $\tilde{f}$ such that
      \begin{align}
	 K^{-1} f< \tilde{f}<Cf, \quad     |\nabla \tilde{f}|<C.
      \label{eqn:GI26_1}	 
      \end{align}   
   Furthermore, $\tilde{f}$ also satisfies the local Harnack inequality
   \begin{align}
     \tilde{K}^{-1} \tilde{f}(x) < \tilde{f}(y)< \tilde{K} \tilde{f}(x), \quad \forall \; y \in B(x, \tilde{K}^{-1} \tilde{f}(x)) 
   \label{eqn:MA25_2}  
   \end{align}   
   for some $\tilde{K}=\tilde{K}(n,\kappa, K)$. 
  \label{prn:GI26_1}
  \end{proposition}

  \begin{proof}
    $M$ can be covered by $\displaystyle \cup_{x \in M} B(x, 0.01K^{-1} f(x))$. By Vitalli covering lemma, we can find countably many points $x_i \in M$ such that
    $\displaystyle M \subset \cup_{i} B(x_i, 0.1K^{-1} f(x_i))$ and $B(x_i, 0.01 K^{-1}f(x_i))$ are disjoint to each other.  
    Let $\eta_i$ be a smooth function supported on $B(x_i, 0.2 K^{-1}f(x_i))$ such that $\eta_i \equiv 1$ on $B(x_i, 0.1 K^{-1}f(x_i))$.  
    Moreover, $|\nabla \eta_i| \leq \frac{100K}{f(x_i)}$.    Fix $i$, let $B(x_j, 0.2 K^{-1}f(x_j))$ be a ball with non-empty intersection with $B(x_i, 0.2 K^{-1}f(x_i))$.
    Denote all such $j$'s by $J_{i}$.   By triangle inequality, we have 
    \begin{align*}
      d(x_i, x_j) < 0.2 K^{-1} \left( f(x_i) + f(x_j) \right)< K^{-1} \max \{f(x_i), f(x_j)\}.
    \end{align*}
    It follows from (\ref{eqn:MA25_1}) that $K^{-1}f(x_i)<f(x_j)<Kf(x_i)$. In particular, we have
    $d(x_i, x_j)<f(x_i)$ and consequently $\displaystyle B(x_j, 0.01K^{-1}f(x_j)) \subset B(x_i, 1.01 f(x_i))$.  
    By the disjoint property, we obtain
    \begin{align*}
      \sum_{j \in J_i} \left| B(x_j, 0.01K^{-1}f(x_j))  \right| < \left| B(x_i, 1.01 f(x_i)) \right|.
    \end{align*}
    Now we apply the lower bound $f(x_j)>K^{-1} f(x_i)$ and  the volume ratio's two-side-bound. The above inequality implies that
    \begin{align*}
      |J_i| \cdot  \kappa \left( 0.01 K^{-2} f(x_i)  \right)^{2n} < \kappa^{-1} \left( 1.01 f(x_i) \right)^{2n}.
    \end{align*}
    Therefore, $|J_i|<\kappa^{-2} (101 K^2)^{2n}$, which we denoted by $C=C(n,\kappa, K)$. 
    According to the definition of $\eta_j$, we know that any point in $M$ can at most locate in the support of $C$ number of $\eta_j$'s. 
    Now we define 
    \begin{align*}
      \tilde{f}(x) \triangleq \sum_{i} f(x_i) \eta_i(x), \quad \forall \; x \in M. 
    \end{align*}
    At every point $x$, there is a neighborhood of $x$ such that the above sum is a sum of at most $C$ non-zero terms of smooth functions.
    Therefore, $\tilde{f}$ is smooth.  
    Choose an arbitrary point $x \in M$ and assume $x \in B(x_i, 0.1K^{-1}f(x_i))$. Since $0 \leq \eta_i \leq 1$, recalling the definition of $J_i$, we have
    \begin{align*}
      \tilde{f}(x) = \sum_{j \in J_{i}} f(x_j) \eta_j(x) \leq \sum_{j \in J_{i}} f(x_j) \leq K \sum_{j \in J_{i}} f(x_i)=K|J_i| f(x_i) \leq CK f(x_i)< CK^2 f(x), 
    \end{align*}
    where we used (\ref{eqn:MA25_1}) and the fact $d(x,x_i)<0.1Kf(x_i)$ in the last step. Clearly, we have
    \begin{align*}
      \tilde{f}(x) \geq f(x_i) \eta_i(x)=f(x_i)> K^{-1} f(x). 
    \end{align*}
    Therefore, we obtain $\displaystyle K^{-1} f(x) < \tilde{f}(x) < CK^2 f(x)$.  
    In light of  the arbitrary choice of $x$, we obtain the first part of (\ref{eqn:GI26_1}), by adjusting $C$ if necessary. 
    The second part of (\ref{eqn:GI26_1}) follows from the following direct calculation. 
    \begin{align*}
      \left|\nabla \tilde{f}(x) \right| \leq \sum_{j \in J_i} f(x_j) |\nabla \eta_j|  \leq \sum_{j \in J_i} 100K \leq 100K|J_i|<C.
    \end{align*}
    The local Harnack inequality (\ref{eqn:MA25_2}) of $\tilde{f}$ follows from the combination of (\ref{eqn:MA25_1}) and the first part of (\ref{eqn:GI26_1}), by adjusting $K$ to $\tilde{K}=CK$ for some $C=C(n,\kappa,K)$.  
  \end{proof}
  
  In our application of  Proposition~\ref{prn:GI26_1}, we typically let $f=\min \left\{K^{-1}, \mathbf{cvr}(\cdot) \right\}$. 
  By Proposition~\ref{prn:SC24_1},  the function $f$ satisfies local Harnack inequality (\ref{eqn:MA25_1}). 
  Then Propsition~\ref{prn:GI26_1} guarantees the existence of a smooth function $\tilde{f}$, which is comparable to $f$ and also satisfies local Harnack inequality, with bounded gradient.
  Since $\tilde{f}$ has better regularity and its value is comparable to $\mathbf{cvr}$, it is convenient to use the level sets of $\tilde{f}$ to decompse the underlying manifold $M$.

  \begin{corollary}[\textbf{Perturbation of the level sets of $\mathbf{cvr}$}]
    Suppose $\mathbf{cr}(M)>1$, $\xi_0=\xi_0(n, \kappa)$ is a very small constant.     
    Suppose  $\xi=\mathbf{cvr}(x)<\xi_0$ for some $x \in B(x_0, 0.5)$.   Then there is a smooth $(2n-1)$-dimensional hyper-surface $\Sigma_{\xi}$ such that
    \begin{itemize}
    \item [(a).]   $C^{-1} \xi<\mathbf{cvr}(y)<C\xi$ for every $y \in \Sigma_{\xi}$.
    \item [(b).]    $\left|\Sigma_{\xi} \cap B(x_0,1) \right|_{\mathcal{H}^{2n-1}}< C\xi^{2p_0-1}$.  
    \end{itemize}
    Here $C=C(n,\kappa,K)=C(n,\kappa)$ since $K$ is the constant depending on $n, \kappa$ in Proposition~\ref{prn:SC24_1}. 
  \label{cor:GI26_1}  
  \end{corollary}

  \begin{proof}
    Let $f=\min \left\{K^{-1}, \mathbf{cvr}(\cdot) \right\}$, which satisfies (\ref{eqn:MA25_1})  by  Proposition~\ref{prn:SC24_1}. 
    Therefore, Proposition~\ref{prn:GI26_1} can be applied. We perturb $f$ to a smooth function $\tilde{f}$ such that  inequality (\ref{eqn:GI26_1}) hold.  
    At the given point $x$, we have $\tilde{f}(x)<C\xi<C\xi_0<K^{-2}$ since $\xi_0$ is chosen very small. 
    Recalling that $\mathbf{cr}(M)>1$,  the density estimate(c.f. Proposition~\ref{prn:SC02_2}) guarantees the existence of point $y \in B(x_0,1)$ such that $\mathbf{cvr}(y)>K^{-1}$.
    Clearly, $f(y)=K^{-1}$ by definition. 
    Therefore, $\tilde{f}(y)>K^{-2}$ by (\ref{eqn:GI26_1}).   In light of the continuity of $\tilde{f}$,  we have $\tilde{f}^{-1}(a) \cap B(x_0,1) \neq \emptyset$
    for each $a \in [C\xi, K^{-2}]$.  Applying coarea formula, we obtain
    \begin{align*}
           \int_{C\xi}^{2C \xi} \left|\tilde{f}^{-1}(a)  \cap B(x_0,1)  \right|_{\mathcal{H}^{2n-1}} dt \leq \int_{\tilde{f}^{-1}([C\xi, 2C\xi]) \cap B(x_0,1)}  \left| \nabla \tilde{f} \right| d\mu
            \leq C \int_{\tilde{f}^{-1}([0, 2C\xi]) \cap B(x_0,1)} 1 d\mu, 
    \end{align*}
    where we applied (\ref{eqn:GI26_1}) in the last step.   Note that $\tilde{f}$ is comparable to $\mathbf{cvr}$ on small values,
    our conditions provide $\mathbf{cr}(M)>1$, 
    the last term in the above inequality can be bounded by the density estimate(c.f. inequality (\ref{eqn:SC02_2})).  
    Therefore, we have
    \begin{align*}
       \int_{C\xi}^{2C \xi} \left|\tilde{f}^{-1}(a)  \cap B(x_0,1)  \right|_{\mathcal{H}^{2n-1}} dt< C \xi^{2p_0}. 
    \end{align*} 
    By Sard theorem and mean-value inequality, we can choose $a_0 \in [C\xi, 2C \xi]$ to be a regular value of $\tilde{f}$ and it satisfies
    \begin{align*}
      \left|\tilde{f}^{-1}(a_0) \right|_{\mathcal{H}^{2n-1}} \leq \frac{2}{C\xi}  \int_{C\xi}^{2C \xi} \left|\tilde{f}^{-1}(a) \right|_{\mathcal{H}^{2n-1}} dt \leq C \xi^{2p_0-1}. 
    \end{align*}
    Let $\Sigma_{\xi}$ be $\tilde{f}^{-1}(a_0)$. Then it satisfies all the requirements. 
  \end{proof}

  Because of the properties (a) and (b) of Corollary~\ref{cor:GI26_1},  we can regard $\Sigma_{\xi}$ as a perturbation of  $\partial \mathcal{F}_{\xi}$ in many applications. 
 
 \begin{corollary}[\textbf{Perturbation of distance function}]
  Let $X=\mathcal{R} \cup \mathcal{S} \in \widetilde{\mathscr{KS}}(n,\kappa)$. 
  Let $f=\min\{0.1, d(\cdot, \mathcal{S})\}$.  Then there is a smooth function $\tilde{d}$ such that
  \begin{align}
    0.1f < \tilde{f}<Cf, \quad |\nabla \tilde{f}|<C.
  \label{eqn:MA20_2}  
  \end{align}
  Furthermore, for each small positive number $\xi$, large positive number $H$ and  point $x_0 \in X$ satisfying $B(x_0, H) \cap \mathcal{S} \neq \emptyset$, we can find a smooth $(2n-1)$-dimensional hyper surface $\Sigma_{\xi} \subset \mathcal{R}$
  such that 
  \begin{itemize}
  \item [(a).]   $C^{-1} \xi<d(y, \mathcal{S})<C\xi$ for every $y \in \Sigma_{\xi}$.
  \item [(b).]    $\left|\Sigma_{\xi} \cap B(x_0,H) \right|_{\mathcal{H}^{2n-1}}< L\xi^2$.     
  \end{itemize}
  All the $C$ in this corollary depend only on $n$ and $\kappa$,  the constant $L$ depends on $n,\kappa$ and the ball $B(x_0, H)$. 
  \label{cor:MA20_1}
  \end{corollary}
  
  \begin{proof}
  The proof of (\ref{eqn:MA20_2}) follows from the  proof of inequality (\ref{eqn:GI26_1}) by letting $M=\mathcal{R}$ and $K=10$, with the following facts in mind.
  First, it is clear that $f$ satisfies the local Harnack inequality (\ref{eqn:MA25_1}) by triangle inequality.
  Second, for each $x \in \mathcal{R}$, we have
  \begin{align*}
   B(x, 0.2K^{-1}f(x)) \subset B(x, 0.02 d(x, \mathcal{S}))  \subset \mathcal{R}. 
  \end{align*}
  Therefore, it makes sense to construct smooth cutoff functions supported on $B(x, 0.2K^{-1}f(x))$. 
  
  The existence of such $\Sigma_{\xi}$ follows from the proof of Corollary~\ref{cor:GI26_1} and the Minkowski codimension assumption of $\mathcal{S}$. In other words, we have
  \begin{align*}
   \left| \{x \in B(x_0, H)| d(x, \mathcal{S})<\xi\} \right|< L \xi^{3}. 
  \end{align*}
  Similar to the proof of Corollary~\ref{cor:GI26_1}, the property of $\Sigma_{\xi}$ is the application of Sard theorem and co-area formula. 
  \end{proof}

\section{Direct estimate of distance by reduced distance}
\label{app:C}

 In this appendix, we provide an alternative approach to estimate distance by reduced distance. 

 \begin{lemma}[\textbf{Estimate distance by weighted integral of reduced distance}]
 There is a constant $C=C(n,\eta, D)$ with the following properties. 
 
  Suppose $\{(M^{2n}, g(t)), -1 \leq t \leq 1\}$ is a Ricci flow solution such that $R \geq -1$. 
  Suppose $\boldsymbol{\beta}$ is a space-time curve connecting $(y, 0)$ and $(z, -\delta)$. Moreover, we have
  \begin{align}
     \inf_{x \in \beta,  t \in [-\delta, 0]} \mathbf{cvr}(x, t) \geq \eta, \quad  |\beta|_{g(0)} \leq D.    \label{eqn:MB28_5}
  \end{align}
  Then for each $\delta <<\eta$, we have
  \begin{align}
      d_{g(0)}^2(y, z) \geq -C\delta + \left. \int_{B_{g(0)}\left(z, \delta^{\frac{1}{4}} \right)} (4\tau l) \cdot (4\pi \tau)^{-n} e^{-\tilde{l}}  dv \right|_{t=-\delta}.       \label{eqn:MB28_6}
  \end{align}
  Here $l$ is the reduced distance to base point $(y, 0)$, $\tilde{l}$ is the reduced distance to base point $(z, 0)$, $\tau=-t$.  
  \label{lma:MB27_1}
  \end{lemma}
  
  \begin{proof}
     Based on $l$ and $\tilde{l}$, we define two auxiliary functions $h=4\tau l$(c.f. $\bar{L}$ in Perelman's terminology) and $w=(4\pi \tau)^{-n} e^{-\tilde{l}}$.
     Recall that $n$ is complex dimension in our case. 
     According to Perelman's calculation(c.f. inequality (7.13) and (7.15) in section 7 and Corollary 9.5 in section 9 of Perelman~\cite{Pe1}. Note that there is a mistake of statement in the proof of Corollary 9.5~\cite{Pe1}, which is corrected in Corollary 29.23 in Kleiner-Lott~\cite{KL}), we have
     \begin{align}
      &\square h=\left( \partial_{t} - \Delta \right) h \geq -4n,    \label{eqn:MB27_0}\\
      &\square^* w=(-\partial_t - \Delta +R) w \leq 0,   \label{eqn:MB27_3}\\
      &\int_{\Omega} w dv \leq 1, \quad \forall \; \tau>0,     \label{eqn:MB28_0}
     \end{align}
     where $\tau=-t$, $n$ is the complex dimension of $M$.  Note that $w$ converges to $\delta$-function based at $(y, 0)$ and $h$ converges to $d_{g(0)}^2(y, z)$
     as $\tau \to 0^{+}$.  So we have
     \begin{align}
        \lim_{\tau \to 0^{+}} \int_{M} h w dv=d_{g(0)}^2(y, z).   \label{eqn:MB28_1}
     \end{align}
     Let $\chi=\phi\left(\frac{d_{g(0)}(z, \cdot)}{\sqrt{\lambda \delta}} \right)$ be a cutoff function, $\phi$ achieves $1$ on $(-\infty, 1)$ and $0$ on $(2, \infty)$,  $|\phi'| \leq 2$.
     Here $\lambda$ is some positive number greater than $1$ such that $\lambda \delta$ is very small, and it will be determined later. 
     Note that $\chi$ is independent of time.  Clearly, we have $|\nabla \chi|^2 \leq \frac{4}{\lambda \delta}$ at time $t=0$. 
     By uniformly bounded geometry around $(z, 0)$ and the fact $\delta <<\eta$, it is easy to obtain the metric equivalence and the estiamte $|\nabla \chi|^2 \leq \frac{40}{\lambda \delta}$
     for each time $t \in [-\delta, 0]$. 
    
     Recall that $R \geq -1$. It follows from definition of reduced length that
     \begin{align}
       l((y, 0), (x, -\tau))=\frac{1}{2\sqrt{\tau}} \int_{0}^{\tau}  \sqrt{s} \left(|\dot{\gamma}|^2 + R \right) ds \geq -\frac{1}{3} \tau+ \frac{1}{2\sqrt{\tau}} \int_{0}^{\tau}  \sqrt{s} |\dot{\gamma}|^2  ds \geq -\frac{1}{3} \tau
     \label{eqn:MC22_1}  
     \end{align}
     for arbitrary $x \in M$. It follows that  $h+\frac{4}{3}\tau^2$ is a nonnegative function. 
     Direct calculation shows that
     \begin{align*}
          \frac{d}{dt} \int_{M} h \chi^2 w dv
         &=\int_{M} \left\{ w \square (h \chi^2) - (h \chi^2) \square^*w \right\} dv\\
         &= \int_{M} \left\{ w \square (h \chi^2) - \left(h+\frac{4}{3} \tau^2 \right) \chi^2 \square^*w \right\} dv
          +\frac{4}{3}\tau^2  \int_{M} \chi^2 \square^* w dv.
     \end{align*}    
     Since $h+\frac{4}{3}\tau^2 \geq 0$, $0<\chi^2 \leq 1$ and $\square^* w \leq 0$, we have
     \begin{align}
        \frac{d}{dt} \int_{M} h \chi^2 w dv \geq \int_{M}  w \square (h \chi^2)  dv +\frac{4}{3}\tau^2  \int_{M} \square^* w dv.   \label{eqn:MB28_2}
     \end{align}
    Note that $\chi$ is independent of time. So we have     
    \begin{align*}     
         &\quad \int_{M} w \square (h \chi^2)  dv\\
         &=\int_{M} w \left\{ \chi^2 \square h +h \square \chi^2 - 2 \langle \nabla h, \nabla \chi^2 \rangle\right\} dv
         =\int_{M} w \left\{ \chi^2 \square h -h \Delta \chi^2 - 2 \langle \nabla h, \nabla \chi^2 \rangle\right\} dv\\
         &=\int_{M} w \left\{ \chi^2 \square h  -  \langle \nabla h, \nabla \chi^2 \rangle\right\} dv + \int_{M} h \langle \nabla \chi^2, \nabla w \rangle dv
           =\int_{M} w \chi^2 \square h dv +   \int_{M} \langle \nabla \chi^2,  h\nabla w -w \nabla h \rangle dv.  
     \end{align*}
     Plugging (\ref{eqn:MB27_0}) into the above equation and then in (\ref{eqn:MB28_2}), we have
     \begin{align*}
       \frac{d}{dt} \int_{M} h \chi^2 w dv &\geq -4n \int_{M} w \chi^2 dv + \int_{M} \langle \nabla \chi^2,  h \nabla w -w \nabla h \rangle dv +\frac{4}{3}\tau^2 \int_{M} \square^* w dv \\
         &\geq -4n + \int_{M} \left\langle \nabla \chi^2, h \nabla w - w \nabla h \right\rangle  dv +\frac{4}{3}\tau^2  \int_{M} \square^* w dv
     \end{align*}
     where we used (\ref{eqn:MB28_0}) in the last step. Integrating the above inequality from $t=-\delta$ to $t=0$, we have
     \begin{align*}
        d_{g(0)}^2(y, z)- \left. \int_{M} \chi^2 h w dv \right|_{t=-\delta}
        \geq -4n \delta + \frac{4}{3} \int_{-\delta}^{0} \int_{M} \tau^2 \square^* w dv dt
        + \int_{-\delta}^{0}  \int_{M} \left\langle \nabla \chi^2, h \nabla w - w \nabla h \right\rangle  dv dt.
     \end{align*}
     Recall that $\frac{d}{dt} \int_{M} w dv=-\int_{M} \square^* w dv$. So we have
     \begin{align*}
       \int_{-\delta}^{0} \int_{M} \tau^2 \square^* w dv dt &\geq \delta^2 \int_{-\delta}^{0} \int_{M}  \square^* w dv dt= \delta^2 \left\{  \left. \int_{M} w dv \right|_{t=-\delta} - \lim_{t \to 0^{-}} \int_{M} w dv \right\}\\
        &\geq -\delta^2  \lim_{t \to 0^{-}} \int_{M} w dv=-\delta^2. 
     \end{align*}
     Consequently, we obtain
      \begin{align*}
        \quad d_{g(0)}^2(y, z)- \left. \int_{M} \chi^2 h w dv \right|_{t=-\delta}
        \geq -4n \delta -\frac{4}{3} \delta^2+ \int_{-\delta}^{0}  \int_{M} \left\langle \nabla \chi^2, h \nabla w - w \nabla h \right\rangle  dv dt.
     \end{align*}
     Let $\Omega=B_{g(0)}(z, 2\sqrt{\lambda \delta}) \backslash B_{g(0)}(z, \sqrt{\lambda \delta})$.  
     Noting that the support of $\nabla \chi^2$ is contained in $\Omega$, we have
     \begin{align*}
      &\quad d_{g(0)}^2(y, z)- \left. \int_{M} \chi^2 h w dv \right|_{t=-\delta} \\ 
      &\geq -4n \delta -\frac{4}{3}\delta^2+ \int_{-\delta}^{0}   \int_{\Omega} \left\langle \nabla \chi^2, h \nabla w - w \nabla h \right\rangle  dv dt\\
      &\geq -5n \delta + \int_{-\delta}^{0}   \int_{\Omega} 2h \chi \left\langle  \nabla \chi,  \nabla w\right\rangle  dv dt
        - \int_{-\delta}^{0}   \int_{\Omega} 2w\chi \left\langle  \nabla \chi,  \nabla h\right\rangle  dv dt\\
      &=-5n \delta + \int_{-\delta}^{0}   \int_{\Omega} 2\left\langle  h \nabla \chi,  \chi \nabla w\right\rangle  dv dt
        - \int_{-\delta}^{0}   \int_{\Omega} 2 \left\langle w \nabla \chi,  \chi \nabla h\right\rangle  dv dt.
     \end{align*}
     Using H\"older inequality, the above inequality implies that
     \begin{align}
      &\quad 5n\delta + d_{g(0)}^2(y, z)- \left. \int_{M} \chi^2 h w dv \right|_{t=-\delta}  \notag\\
      &\geq -\int_{-\delta}^{0}   \int_{\Omega} \left\{ (h^2+w^2)|\nabla \chi|^2 + (|\nabla h|^2 + |\nabla w|^2) \chi^2\right\}  dv dt \notag\\
      &\geq -\frac{40}{\lambda \delta} \left\{  \underbrace{\int_{-\delta}^{0}   \int_{\Omega} h^2 dv dt}_{I} +  \underbrace{\int_{-\delta}^{0}   \int_{\Omega} w^2 dv dt}_{II} \right\}
        -\left\{ \underbrace{\int_{-\delta}^{0}   \int_{\Omega} |\nabla h|^2 dv dt}_{III} + \underbrace{\int_{-\delta}^{0}   \int_{\Omega} |\nabla w|^2 dv dt}_{IV}  \right\}.
       \label{eqn:MB29_0} 
     \end{align}
     
     Let's estimate terms $I, III$ first and then $II, IV$. 
     
     By our condition (\ref{eqn:MB28_5}), we know that $h \leq C |\beta|_{g(0)}^2$ and therefore is uniformly bounded.  Also by the uniform equivalence of volume ratio on $\Omega$(due to bounded geometry around $z$ and $\delta << \eta$), we have
     \begin{align}
       I \leq C \lambda^{n} \delta^{n+1}      \label{eqn:MB29_A}
     \end{align}
     for some depending on $n,\eta$ and $|\beta|_{g(0)}$. 
     
     For estimating term $III$, we need to transfrom $|\nabla h|^2$ to $h_t$.  Actually, it follows from Perelman's calculation(c.f. the inequality between (7.12) and (7.13) of Perelman~\cite{Pe1} and note the fact $X=\nabla l$)
     that
     \begin{align*}
        \frac{d}{d\tau} l= \frac{R+|\nabla l|^2}{2} - \frac{l}{2\tau}
     \end{align*} 
     along the reduced geodesic. This implies that
     \begin{align}
        l_{\tau}=\frac{d}{d\tau} l - |\nabla l|^2=\frac{R-|\nabla l|^2}{2} - \frac{l}{2\tau}.   \label{eqn:MC24_1}
     \end{align}
     Consequently, we have
   \begin{align*}
     h_{\tau}=(4\tau l)_{\tau}=4l+4\tau l_{\tau}=4l + 2\tau (R-|\nabla l|^2) - 2l=2\tau R +2l -2\tau |\nabla l|^2=2\tau R + \frac{h}{2\tau} - \frac{|\nabla h|^2}{8\tau},
   \end{align*}
   which implies that
   \begin{align*}
     |\nabla h|^2=-8\tau h_{\tau} +4h+16\tau^2 R=-8t h_t +4h+16t^2R. 
   \end{align*}
   Plugging this into the formula of term $III$, we obtain
   \begin{align*}
    III&=\int_{-\delta}^{0}   \int_{\Omega} |\nabla h|^2 dv dt = \int_{-\delta}^{0}   \int_{\Omega}    \left( -8t h_t +4h+16t^2 R\right) dv dt\\
       &=\int_{-\delta}^{0} \left\{ \frac{d}{dt} \int_{\Omega} -8th dv + \int_{\Omega} \left\{ 12h + 8R(2t^2 -th)  \right\} dv \right\}  dt\\
       &=-8\delta \left. \int_{\Omega} h dv \right|_{t=-\delta}  + \int_{-\delta}^{0} \int_{\Omega} \left\{ 12h + 8R(2t^2 -th)  \right\} dv dt.
   \end{align*}
   As explained before, $h$ is uniformly bounded, geometry on $\Omega$ is uniformly bounded.  
   Recall that  $\Omega=B_{g(0)}(z, 2\sqrt{\lambda \delta}) \backslash B_{g(0)}(z, \sqrt{\lambda \delta})$ and $\lambda \delta$ is very small. 
   So we obtain
   \begin{align}
     III \leq C |\Omega|_{g(0)} \delta \leq C \lambda^{n} \delta^{n+1}.      \label{eqn:MB29_B}
   \end{align}
  
   We now move on to estimate term $II$ and $IV$. 
   Recall that $w=(4\pi \tau)^{-n} e^{-\tilde{l}}$.  Using the uniformly bounded geometry and all high curvature derivative bounds around $(z, 0)$, we see that
     \begin{align*}
      w \leq C \tau^{-n}e^{- \frac{\tilde{d}^2}{5\tau}},  \quad   |\nabla w| \leq C \tau^{-n-1}  e^{- \frac{\tilde{d}^2}{5\tau}},  \quad \textrm{in} \; \Omega,    
     \end{align*}
    where $\tilde{d}$ is the distance to $z$ with respect to the metric $g(0)$.  Here $C$ depends only on $n$ and $\eta$.  Note that $\tilde{d} \geq \sqrt{\lambda \delta}$ in $\Omega$. 
    It follows that
     \begin{align}
      w \leq C \tau^{-n}e^{- \frac{\lambda \delta}{5\tau}},  \quad   |\nabla w| \leq C \tau^{-n-1}  e^{- \frac{\lambda \delta}{5\tau}},  \quad \textrm{in} \; \Omega.   \label{eqn:MB29_1}
     \end{align}
     Consequently,  we obtain
    \begin{align*}
      II &=\int_{-\delta}^{0}   \int_{\Omega} w^2 dv dt \leq   C|\Omega|_{g(0)} \int_{-\delta}^{0}  \tau^{-2n}e^{- \frac{2\lambda \delta}{5\tau}}  dt\\
         &\leq C (\lambda \delta)^{n} \int_{0}^{\delta} \tau^{-2n}e^{- \frac{2\lambda \delta}{5\tau}}  d\tau  \longeq{s=\frac{5\tau}{2\lambda \delta}}{}
         C (\lambda \delta)^{-n+1} \int_{0}^{\frac{5}{2\lambda}} s^{-2n} e^{-\frac{1}{s}} ds.
    \end{align*}  
    Similar to the estimate  we used in the proof of Proposition~\ref{prn:HC29_6}, we can bound $s^{-2n} e^{-\frac{1}{s}}$ by $C s^{\theta}$ on $(0,\frac{5}{2\lambda}) \subset (0, 1)$ for any positive $\theta$.
    We then have
    \begin{align}
      II \leq  C_{\theta} \delta^{-n+1} \lambda^{-n-\theta}.    \label{eqn:MB29_C}
    \end{align} 
    On the other hand, using (\ref{eqn:MB29_1}),  we obtain
    \begin{align*}
      IV&=\int_{-\delta}^{0}   \int_{\Omega} |\nabla w|^2 dv dt \leq C \int_{-\delta}^{0}   \int_{B(z, 2\sqrt{\lambda \delta}) \backslash B(z, \sqrt{\lambda \delta})}  \tau^{-2n-2} e^{-\frac{2 \lambda \delta}{5\tau}} dv dt\\
          &\leq C (\lambda \delta)^{n} \int_{0}^{\delta} \tau^{-2n-2}e^{- \frac{2\lambda \delta}{5\tau}}  d\tau
            \leq C (\lambda \delta)^{-n-1} \int_{0}^{\frac{5}{2\lambda}} s^{-2n-2} e^{-\frac{1}{s}} ds.
    \end{align*}
   Using  $s^{-2n-2} e^{-\frac{1}{s}}< C s^{\theta}$, we have 
   \begin{align}       
      IV&\leq C \delta^{-n-1} \lambda^{-n-2-\theta}.   \label{eqn:MB29_D}
    \end{align}
    Plugging (\ref{eqn:MB29_A}), (\ref{eqn:MB29_B}), (\ref{eqn:MB29_C}) and (\ref{eqn:MB29_D}) into (\ref{eqn:MB29_0}),  we obtain
    \begin{align*}
      5n\delta + d_{g(0)}^2(x, y)- \left. \int_{M} \chi^2 h w dv \right|_{t=-\delta}  \geq - C \left\{ \lambda^{n-1} \delta^{n} + \lambda^n \delta^{n+1} + \delta^{-n} \lambda^{-n-1-\beta} + \delta^{-n-1} \lambda^{-n-2-\beta} \right\},
    \end{align*}
    for some $C$ depending on $n, \eta$ and $\beta$.  Now we choose $\lambda=\delta^{-\frac{1}{2}}$. Note that by this choice of $\lambda$, we have $\lambda \delta=\delta^{\frac{1}{2}}$, which is very small. 
    Then we have
    \begin{align*}
      5n\delta + d_{g(0)}^2(x, y)- \left. \int_{M} \chi^2 h w dv \right|_{t=-\delta}  \geq - C \left\{  \delta^{\frac{n+1}{2}}  + \delta^{\frac{n+2}{2}} + \delta^{\frac{\beta+1-n}{2}} + \delta^{\frac{\beta-n}{2}} \right\}. 
    \end{align*} 
    Choosing $\beta=2n+1$ and noting that $\delta^{\frac{n+2}{2}} < \delta^{\frac{n+1}{2}} <\delta$, we obtain (\ref{eqn:MB28_6}). 
    
  \end{proof}
  
\begin{proposition}[\textbf{Estimate distance by reduced distance}]
Suppose $\mathcal{LM}_i \in \mathscr{K}(n,A;1)$ satisfies equation (\ref{eqn:SL06_1}) in Lemma~\ref{lma:SK27_4}.   
Suppose $x_i \in M_i$.
Let $(\bar{M}, \bar{x}, \bar{g})$ be the limit space of $(M_i, x_i, g_i(0))$, $\mathcal{R}$ be the regular part of $\bar{M}$ and $\bar{x} \in \mathcal{R}$.
For every two points $\bar{y}, \bar{z} \in \mathcal{R}$ and $\epsilon>0$, we have a smooth curve $\alpha$ connecting $\bar{y}, \bar{z}$ such that
\begin{align}
  |\alpha|<d(\bar{y}, \bar{z})+ \epsilon.      \label{eqn:MC01_1}
\end{align}
\label{prn:MC04_1}
\end{proposition}

\begin{proof}
There are only two possibilities: $(\bar{y}, 0)$ can be connected to $(\bar{z}, -1)$ by a smooth shortest reduced geodesic, or they cannot be connected by such a reduced geodesic. We shall treat them seperately.

\textit{Case 1. $(\bar{y}, 0)$ can be connected to $(\bar{z}, -1)$ by a smooth shortest reduced geodesic $\boldsymbol{\beta}$ with space projection curve $\beta$.}

Intuitively, $\beta$ should be the shortest Riemannian geodesic connecting $\bar{y}, \bar{z}$.  However, we do not know this directly right now. 
In the following discussion, we shall use the existence of $\boldsymbol{\beta}$ to apply Lemma~\ref{lma:MB27_1}. Therefore, for each small $\delta$,
we can bound $d^2(\bar{y}, \bar{z})$ from below by the weighted average of $4\delta$ times reduced distance in $B(\bar{z}, \delta^{\frac{1}{4}})$, 
at time slice $t=-\delta$.  Applying a rough mean value theorem, such a lower bound can be approximated by a shortest reduced geodesic $\boldsymbol{\gamma}$
connecting $(\bar{y}, 0)$ and $(e, -\delta)$ for some $e$ nearby $\bar{z}$.  By concatenating $e$ to $\bar{z}$ with a smooth curve and modifying the connection point a little bit, we obtain a smooth curve connecting $\bar{y}$ and $\bar{z}$, whose length approximates the lower bound of $d(\bar{y}, \bar{z})$.

Now we start the detailed discussion. Since $\beta \subset \mathcal{R}$, we can assume $\beta \subset \mathcal{R}_{2\eta}$ for some $\eta>0$.  
In other words, we can find  curves $\beta_i \subset \mathcal{F}_{\eta}(M_i, 0)$  such that $\beta_i \to \beta$.  
   Furthermore, $\beta_i$ is the space projection to time slice $t=0$ of the shortest reduced geodesic $\boldsymbol{\beta}_i$ connecting $(y_i, 0)$ and $(z_i, -1)$, where $y_i, z_i \in M_i$
   and $y_i \to \bar{y}$, $z_i \to \bar{z}$.  To be more precise, we have $\boldsymbol{\beta}_i(\tau)=(\beta_i(\tau), -\tau)$ for each $\tau \in [0, 1]$. 
   
   Fix $\delta>0$ small.  We see that $(y_i, 0)$ and $(z_i, -\delta)$ can be connected by a space-time curve $\tilde{\boldsymbol{\beta}}_i$, which comes from $\beta_i$ by the relationship
   $\tilde{\boldsymbol{\beta}}_i(\tau)=(\beta_i(\frac{\tau}{\delta}), -\tau)$ for each $\tau \in [0, \delta]$.    
   Clearly, we have uniform bound of the length 
   $|\beta_i|_{g_i(0)}$, say by $|\beta|_{\bar{g}}+1$.  We also have the uniform regularity bound around the curve $\beta_i$. In other words, the inequality (\ref{eqn:MB28_5}) in Lemma~\ref{lma:MB27_1} is satisfied, for curve
   $\beta_i$ and end points $(y_i, 0)$ and $(z_i, -\delta)$.  
   Therefore, we can apply inequality (\ref{eqn:MB28_6}) in Lemma~\ref{lma:MB27_1} to obtain
   \begin{align}
     d_{g_i(0)}^2(y_i, z_i)  \geq -C\delta + \left. \int_{B_{g_i(0)}\left(z_i, \delta^{\frac{1}{4}} \right)} (4\tau l) \cdot (4\pi \delta)^{-n} e^{-\tilde{l}}  dv \right|_{t=-\delta},     \label{eqn:MC01_3}
   \end{align}
   where $l$ is the reduced distance to $(y_i, 0)$ and $\tilde{l}$ is the reduced distance to $(z_i, 0)$.  Note that most shortest geodesics stay away from high curvature part by Lemma~\ref{lma:SL14_1}. 
   Away from a measure zero set $E$, $l(\cdot, -\delta)$ is achieved by smooth reduced geodesics.  
   If $l$ is achieved by a shortest reduced geodesic $\boldsymbol{\alpha}$ with uniform regular neighborhood, then $|\nabla l|$ is bounded in the small neighborhood of the end point of $\boldsymbol{\alpha}$, using
   the uniformly bounded geometry around $\bar{z}$(c.f. inequality (\ref{eqn:MA22_11}) and equation (\ref{eqn:MC24_1})). 
   Also, we have a rough estimate of  $l$ in Lemma~\ref{lma:SL13_1} even if $l$ cannot be achieved by a shortest reduced geodesic avoiding high curvature part. 
   Note also that $l$ is uniformly bounded from below by (\ref{eqn:MC22_1}).    
   Therefore, we can take limit on both sides of (\ref{eqn:MC01_3}) to obtain
   \begin{align}
     d^2(\bar{y}, \bar{z})  \geq -C\delta + \left. \int_{B\left(\bar{z}, \delta^{\frac{1}{4}} \right)} (4\tau l) \cdot (4\pi \delta)^{-n} e^{-\tilde{l}}  dv \right|_{t=-\delta}.   \label{eqn:MB01_2}
   \end{align}
   By weak convexity of reduced geodesic, i.e., Lemma~\ref{lma:SK27_4}, there exists a measure-zero set $E$ such that for each $x \in B\left(\bar{z}, \delta^{\frac{1}{4}} \right) \backslash E$,  
   we have $(x, -\delta)$ can be connected to $(\bar{y}, 0)$ by a smooth shortest reduced geodesic $\boldsymbol{\alpha}$.  
   Note that the reduced length of $\boldsymbol{\alpha}$ is uniformly bounded from below by a positive number.  Actually, let $\tau_0$ be the smallest parameter such that $\boldsymbol{\alpha}$ escape $B(\bar{y}, 0.01 \eta)$. Then we have
   \begin{align*}
     2\sqrt{\delta} l = \int_{0}^{\delta} \sqrt{s} |\dot{\alpha}|^2 ds \geq  \int_{0}^{\tau_0} \sqrt{s} |\dot{\alpha}|^2 ds \geq \frac{\left(\int_{0}^{\tau_0} |\dot{\alpha}| ds \right)^2}{\int_0^{\tau_0} \frac{1}{\sqrt{s}} ds} \geq \frac{10^{-4} \eta^2}{2\sqrt{\tau_0}}
     \geq \frac{\eta^2}{20000\sqrt{\delta}}. 
   \end{align*}
   So $l$ is uniformly bounded from below by a positive number $\frac{\eta^2}{40000\delta}$.    Let $\epsilon$ be a very small number in $(0, \frac{\eta^2}{10000})$, whose precise value will be determined later.
    We can choose a point $e \in B\left(\bar{z}, \delta^{\frac{1}{4}} \right) \backslash E$ such that
   \begin{align*}
     l((\bar{y}, 0), (e, -\delta)) \leq \inf_{x \in B\left(\bar{z}, \delta^{\frac{1}{4}} \right) \backslash E} l((\bar{y}, 0), (x, -\delta)) + \frac{0.01 \epsilon}{4\delta}. 
   \end{align*}
   In light of the definition of $E$, we see that $l((\bar{y}, 0), (e, -\delta))$ is achieved by some smooth reduced geodesic, say $\boldsymbol{\gamma}$, with space projection $\gamma$.  
   Intuitively, this $\boldsymbol{\gamma}$ is the shortest reduced reduced geodesic from $(\bar{y}, 0)$ to a ``neibhorood" of $(\bar{z}, -\delta)$. 
   Note that $\boldsymbol{\gamma}$ is different from $\boldsymbol{\beta}$, which is a shortest reduced geodesic from $(\bar{y}, 0)$ to $(\bar{z}, -1)$, for the ``exact" point $\bar{z}$. 
   Their space projection should be very close. But this relationship is not clear and not needed now.  We do know that $|\gamma|^2$ is bounded by some number determined by $|\beta|^2$, say $|\beta|^2+1$.   
   It follows from the choice of $\boldsymbol{\gamma}$ that
   \begin{align*}
     l((\bar{y}, 0), (x, -\delta)) \geq \frac{|\gamma|^2- 0.01 \epsilon}{4\delta}>0, \quad \forall \; x \in B\left(\bar{z}, \delta^{\frac{1}{4}} \right) \backslash E. 
   \end{align*} 
   Plugging the above inequality into (\ref{eqn:MB01_2}), we obtain
   \begin{align*}
     d^2(\bar{y}, \bar{z})  &\geq -C\delta + \left. \int_{B\left(\bar{z}, \delta^{\frac{1}{4}} \right)} \left(|\gamma|^2- 0.01 \epsilon \right) \cdot (4\pi \delta)^{-n} e^{-\tilde{l}}  dv \right|_{t=-\delta}\\
       &= -C\delta + \left(|\gamma|^2- 0.01 \epsilon \right) \left. \int_{B\left(\bar{z}, \delta^{\frac{1}{4}} \right)}  \cdot (4\pi \delta)^{-n} e^{-\tilde{l}}  dv \right|_{t=-\delta}. 
   \end{align*}
   Note that $(\bar{z}, 0)$ has a uniform regular space-time neighborhood, which is Ricci flat.  
   Therefore  $\tilde{l}=\frac{d^2(\bar{z}, \cdot)}{4\delta}$ on the ball  $B\left(\bar{z}, \delta^{\frac{1}{4}} \right)$, at time $t=-\delta$. 
   Then it is not hard to obtain
   \begin{align*}
     \left. \int_{B\left(\bar{z}, \delta^{\frac{1}{4}} \right)}  (4\pi \delta)^{-n} e^{-\tilde{l}}  dv \right|_{t=-\delta}>1-\delta,
   \end{align*}
   whenever $\delta$ is very small. It follows that
   \begin{align*}
     d^2(\bar{y}, \bar{z}) \geq -C\delta + \left(|\gamma|^2- 0.01 \epsilon \right) (1-\delta) \geq |\gamma|^2 - (C+|\gamma|^2) \delta -0.01 \epsilon(1-\delta). 
   \end{align*}
   Note $\delta$ is very small, compared with $10^{-4}\eta^2$. So we can choose $\epsilon=\delta$.  Also note that $4|\gamma|^2$ is bounded by $4|\beta|^2+4$. 
   Then we have
   \begin{align}
       d^2(\bar{y}, \bar{z}) \geq |\gamma|^2 - C\delta   \label{eqn:MC15_1}
   \end{align}
 for some $C$ depending on $\boldsymbol{\beta}$.  Recall that $\gamma$ is a smooth curve connecting $\bar{y}, e$.
 Since $e \in B(\bar{z}, \delta^{\frac{1}{4}})$ and $\bar{z}$ is a regular point,  there is a unique shortest geodesic $\tilde{\gamma}$ connecting
 $e$ to $\bar{z}$, whenever $\delta$ is very small. 
 Now we concatenate $\gamma$ and $\tilde{\gamma}$ and smoothen the concatenated curve around the connection point, we obtain
 a curve $\alpha$ such that
 \begin{align*}
   |\alpha|<|\gamma| +|\tilde{\gamma}| + \delta^{\frac{1}{4}}< |\gamma| + 2 \delta^{\frac{1}{4}},  \quad
   \Rightarrow \quad |\alpha|^2 < |\gamma|^2 +4\delta^{\frac{1}{4}} |\gamma| + 4\delta^{\frac{1}{2}}< |\gamma|^2 + C \delta^{\frac{1}{4}}. 
 \end{align*}
 Plugging the above inequality into (\ref{eqn:MC15_1}), we obtain
 \begin{align*}
  d^2(\bar{y}, \bar{z}) \geq |\alpha|^2 - C\delta^{\frac{1}{4}}, 
 \end{align*}
 which implies (\ref{eqn:MC01_1}) if we choose $\delta << \epsilon^{4}$.   We finish the proof of the first case. 
   
 \textit{Case 2. $(\bar{y}, 0)$ cannot be connected to $(\bar{z}, -1)$ by a smooth shortest reduced geodesic.}  
     
   For each fixed $\epsilon>0$, by the density of $\mathcal{R}(\bar{M}) \backslash E$ in $\mathcal{R}(\bar{M})$, we can find $\bar{z}' \in \mathcal{R}(\bar{M}) \backslash E$ such that 
   $\bar{z}' \in B(\bar{z}, 0.01\epsilon) \cap \mathcal{R}(\bar{M}) \backslash E$.  Furthermore, we can assume $d(\bar{z}', \bar{z})$ is much less
   that the regularity scale of $\bar{z}$.    
   By definition, we know $(\bar{z}', -1)$ can be connected to $(\bar{y}, 0)$ by a smooth shorted geodesic $\boldsymbol{\gamma}$. 
   In light of the proof in Case 1, we can obtain a smooth curve $\alpha_1$ connecting $\bar{y}$ and $\bar{z}'$ such that 
   \begin{align*}
      d(\bar{y}, \bar{z}') \geq |\alpha_1| - 0.01 \epsilon. 
   \end{align*}
   Recall that $\bar{z}$ is regular and $\bar{z}'$ is nearby $\bar{z}$.
   Similar to the last step in Case 1,  we have a smooth curve $\alpha_2$ connecting $\bar{z}'$ and $\bar{z}$ such that
   \begin{align*}
     |\alpha_2| =d(\bar{z}, \bar{z}') \leq 0.01 \epsilon. 
   \end{align*}
   Let $\alpha$ be the smooth curve obtained by concatenating $\alpha_1$ and $\alpha_2$. Then we have
   \begin{align*}
      |\alpha| =|\alpha_1| +|\alpha_2| \leq d(\bar{y}, \bar{z}') + 0.02 \epsilon \leq d(\bar{y}, \bar{z}) + d(\bar{z}, \bar{z}') + 0.02\epsilon \leq d(\bar{y}, \bar{z}) + 0.03\epsilon. 
    \end{align*}
    By smoothing around the connection point of $\alpha_1$ and $\alpha_2$ if necessary, we finish the proof of (\ref{eqn:MC01_1}) in Case 2. 
\end{proof}

\vspace{0.5in}

Xiuxiong Chen, School of Mathematics, Department of Mathematics, State University of New York, Stony Brook, NY 11794, USA;
University of Science and Technology of China, Hefei, Anhui, 230026, PR China;  xiu@math.sunysb.edu.\\

Bing  Wang, Department of Mathematics, University of Wisconsin-Madison,
Madison, WI 53706, USA;  bwang@math.wisc.edu.\\

\end{document}